\documentclass[reqno,11pt]{amsart}
\usepackage{amsmath,amssymb,mathrsfs,amsthm,amsfonts,mathtools}
\usepackage[inline]{enumitem}
\usepackage[usenames,dvipsnames]{xcolor}
\usepackage{hyperref}
\usepackage[percent]{overpic}
\usepackage{comment}
\usepackage{multirow,array,longtable,booktabs}
\usepackage{stmaryrd}
\usepackage{algorithm}
\usepackage{algorithmic}
\usepackage{makecell}
\usepackage{pgfplots}
\usepackage{subcaption}
\usepackage{tikz}
\usetikzlibrary{calc}
\usetikzlibrary{arrows.meta,backgrounds}
\usetikzlibrary{arrows}

\hypersetup{
	colorlinks=true, linkcolor=blue,
	citecolor=ForestGreen
}
\usepackage[paper=letterpaper,margin=1in]{geometry}
\DeclareMathAlphabet{\mathpzc}{OT1}{pzc}{m}{it}

\newtheorem{theorem}{Theorem}[section]
\newtheorem{lemma}[theorem]{Lemma}
\newtheorem{proposition}[theorem]{Proposition}
\newtheorem{corollary}[theorem]{Corollary}
\theoremstyle{definition}
\newtheorem{definition}[theorem]{Definition}
\newtheorem{observation}[theorem]{Lemma}
\newtheorem{assumption}[theorem]{Assumption}
\newtheorem{remark}[theorem]{Remark}
\numberwithin{equation}{section}
\allowdisplaybreaks
\usepackage{acronym}

\acrodef{KPZ}{Kardar--Parisi--Zhang}
\acrodef{SHE}{Stochastic Heat Equation}
\acrodef{LDP}{Large Deviation Principle}

\newcounter{dummy}
\renewcommand{\Pr}{\mathbb{P}}
\newcommand{\Ex}{\mathbb{E}}
\renewcommand{\d}{d}	
\newcommand{\ind}{\mathbf{1}}	

\newcommand{\calI}{\mathbb{Z}^{\operatorname{half}}}
\newcommand{\se}[2]{S_{#1}(#2)}
\newcommand{\pr}[3]{\Pr^{#1;#2 #3}}
\newcommand{\tpr}[3]{\til{\Pr}^{#1;#2 #3}}
\newcommand{\wprw}[2]{{\Pr}_{\operatorname{WPRW}}^{#1;#2}}
\newcommand{\prw}[2]{{\Pr}_{\operatorname{PRW}}^{#1;#2}}
\newcommand{\erw}[2]{{\Ex}_{\operatorname{PRW}}^{#1;#2}}
\newcommand{\SSS}{Q}
\newcommand{\ise}{S_1(1)}
\newcommand{\iise}{S_2(1)}
\newcommand{\iks}{S_1(k)}
\newcommand{\iiks}{S_2(k)}
\newcommand{\kis}{S_i(k)}
\newcommand{\isn}{S_1(n)}
\newcommand{\iisn}{S_2(n)}
\makeatletter
\newcommand\myitem[1][]{\item[#1]\refstepcounter{dummy}\def\@currentlabel{#1}}
\newcommand*{\mprime}{\ensuremath{'}}
\newcommand{\ihs}{\mathcal{I}^{(N)}}
\newcommand{\is}{\mathcal{I}_{\mathbf{sym}}^{(N)}}

\let\yt\ll
\renewcommand{\ll}{\llbracket}
\newcommand{\rr}{\rrbracket}

\newcommand{\fa}{\mathsf{f}_{\theta}}
\newcommand{\ffa}{\mathsf{f}}
\newcommand{\ga}{\mathsf{g}_{\zeta}}

\newcommand{\zl}{Z_{N}^{\operatorname{line}}}

\newcommand{\blue}{\textcolor{blue}{\textbf{blue}}}
\newcommand{\black}{\textcolor{black}{\textbf{black}}}
\newcommand{\red}{\textcolor{red}{\textbf{red}}}
\newcommand{\purple}{\textcolor{yellow!70!black}{\textbf{yellow}}}
\newcommand{\yellow}{\textcolor{gray}{\textbf{gray}}}
\newcommand{\green}{\textcolor{green!70!black}{\textbf{green}}}

\newcommand{\zh}{Z}
\newcommand{\hslg}{\mathpzc{HSLG}}
\newcommand{\btf}{bottom-free}

\newcommand{\Con}{\mathrm{C}} 

\newcommand{\R}{\mathbb{R}} 
\newcommand{\Z}{\mathbb{Z}} 
\newcommand{\wcr}{W_{\operatorname{cr}}}
\newcommand{\wsc}{W_{\operatorname{sc}}}
\newcommand{\qo}{\xi}

\newcommand{\e}{\varepsilon}

\renewcommand{\L}{\mathcal{L}}
\newcommand{\m}{\mathsf}

\newcommand{\til}{\widetilde}
\renewcommand{\bar}{\overline}

\renewcommand{\ni}{\mathsf{NI}}
\newcommand{\psa}{\Pr_{\alpha}^{\vec{y},\vec{z};k,T}}
\newcommand{\esa}{\Ex_{\alpha}^{\vec{y},\vec{z};k,T}}
\newcommand{\ps}{\Pr^{\vec{y},\vec{z};k,T}}
\newcommand{\es}{\Ex^{\vec{y},\vec{z};k,T}}

\usepackage{graphicx}

\title[KPZ exponents for the half-space log-gamma polymer]{KPZ exponents for the half-space log-gamma polymer}

\author[G.\ Barraquand]{Guillaume Barraquand}
\address{G.\ Barraquand,
	Laboratoire de physique de l'\'ecole normale sup\'erieure, ENS, Universit\'e PSL, ´
	CNRS, Sorbonne Universit\'e, Universit\'e de Paris, Paris, France
}
\email{guillaume.barraquand@ens.fr}
\author[I.\ Corwin]{Ivan Corwin}
\address{I.\ Corwin,
	Department of Mathematics, Columbia University,
	\newline\hphantom{\quad \ \ I. Corwin}
	2990 Broadway, New York, NY 10027 USA
}
\email{ivan.corwin@gmail.com}
\author[S.\ Das]{Sayan Das}
\address{S.\ Das,
	Department of Mathematics, University of Chicago,
	\newline\hphantom{\quad \ \ S. Das}
	5734 S.~University Avenue, Chicago, Illinois 60637 USA
}
\email{sayan.das@columbia.edu}

\begin{document}
	
	\begin{abstract} We consider the point-to-point log-gamma polymer of length $2N$ in a half-space with i.i.d. $\operatorname{Gamma}^{-1}(2\theta)$ distributed bulk weights and i.i.d. $\operatorname{Gamma}^{-1}(\alpha+\theta)$ distributed
boundary weights for $\theta>0$ and $\alpha>-\theta$. We establish the KPZ exponents ($1/3$ fluctuation and $2/3$ transversal) for this model when $\alpha=N^{-1/3}\mu$  for $\mu\in \R$ fixed (critical regime) and when $\alpha>0$ is fixed (supercritical regime). In particular, in these two regimes, we show that after appropriate centering, the free energy process with spatial coordinate scaled by  $N^{2/3}$ and fluctuations scaled by $N^{1/3}$ is tight. These regimes correspond to a polymer measure which is not pinned at the boundary.

This is the first instance of establishing the $2/3$ transversal exponent for a positive temperature half-space model, and the first instance of the $1/3$ fluctuation exponent besides precisely at the boundary where recent work of \cite{ims22} applies and also gives the exact one-point fluctuation distribution (our methods do not access exact fluctuation distributions).

Our proof relies on two inputs -- the relationship between the half-space log-gamma polymer and half-space Whittaker process (facilitated by the geometric RSK correspondence as initiated in \cite{cosz, osz}), and an identity in \cite{bw} which relates the point-to-line half-space partition function to the full-space partition function for the log-gamma polymer.

The primary technical contribution of our work is to construct the half-space log-gamma Gibbsian line ensemble and develop, in the spirit of work initiated in \cite{ch14}, a toolbox for extracting tightness and absolute continuity results from minimal information about the top curve of such half-space line ensembles. This is the first study of half-space line ensembles.
\end{abstract}
	
	\maketitle
	{
		\hypersetup{linkcolor=black}
		\setcounter{tocdepth}{1}
		\tableofcontents
	}	
	\section{Introduction}

	\subsection{The model and the main results}\label{sec:1.1} Fix $\vec{\theta}:=(\theta_i)_{i\in \mathbb{Z}_{\ge 1}}$ such that $\theta_i>0$ for all $i\in \mathbb{Z}_{\geq 1}$ and $\alpha>-\min\{\theta_i: i\in \mathbb{Z}_{\ge 1}\}$. Consider a family of independent random variables $(W_{i,j})_{(i,j)\in \calI}$ with $\calI:=\{(i,j) \in (\mathbb{Z}_{\ge 1})^2 :  j\le i\}$  such that
\begin{align}\label{eq:wt}
W_{i,j}\sim \operatorname{Gamma}^{-1}(\alpha+\theta_j)  \textrm{ for } i=j \qquad \textrm{and } \quad W_{i,j}\sim \operatorname{Gamma}^{-1}(\theta_i+\theta_j) \textrm{ for } j<i,
	\end{align}
where $X\sim \operatorname{Gamma}^{-1}(\beta)$ means  $X$ is a random variable with density $\ind_{x>0}\Gamma^{-1}(\beta)x^{-\beta-1}e^{-1/x}$. A directed lattice path $\pi=\big((x_i,y_i)\big)_{i=1}^k$ confined to the half-space index set $\calI$ is an up-right path with all $(x_i,y_i)\in \calI$, such that it only makes unit steps in the coordinate directions, i.e., $(x_{i+1},y_{i+1})=(x_i,y_i)+(0,1)$ or $(x_{i+1},y_{i+1})=(x_i,y_i)+(1,0)$; see Figure \ref{fig000}.  Given $(m,n)\in \calI$, we denote $\Pi_{m,n}$ to be the set of all directed paths from $(1,1)$ to $(m,n)$ confined to $\calI$.  Given the random variables from \eqref{eq:wt}, we define the weight of a path $\pi$ and the point-to-point partition function of the half-space log-gamma ($\hslg$) polymer as
	\begin{equation}\label{eq:wz}
		w(\pi):=\prod_{(i,j)\in \pi} W_{i,j},  \qquad Z_{(\alpha,\vec\theta)}(m,n):=\sum_{\pi\in \Pi_{m,n}} w(\pi).
	\end{equation}

	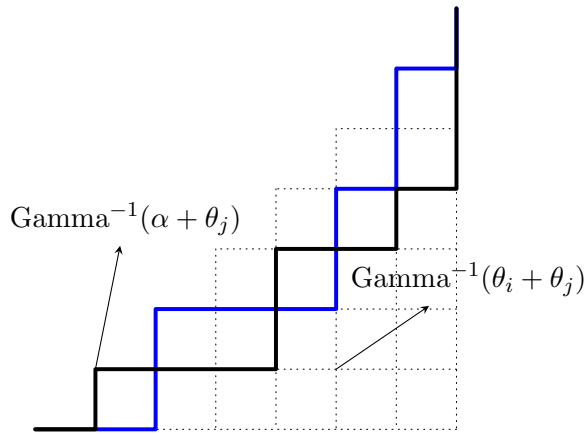
\begin{figure}[h!]
		\centering
		\begin{tikzpicture}[line cap=round,line join=round,>=triangle 45,x=0.8cm,y=0.8cm]
			\foreach \x in {0,1,2,...,7}
			{
				\coordinate (A\x) at ($(9,0)+(\x,0)$) {};
				\coordinate (B\x) at ($(9,0)+(0,\x)$) {};
				\draw[dotted] ($(A\x)+(0,\x)$) -- ($(A\x)$);
				\draw[dotted] ($(B\x)+(7,0)$) -- ($(B\x)+(\x,0)$);
			}
			\draw [line width=1.5pt,color=blue] (9,0)-- (11,0)--(11,2)--(14,2)--(14,4)--(15,4)--(15,6)--(16,6)--(16,7);
			\draw [line width=1.5pt] (9,0)-- (10,0)--(10,1)--(13,1)--(13,3)--(15,3)--(15,4)--(16,4)--(16,7);
			\node[] (A) at (10.5,3.5) {$\operatorname{Gamma}^{-1}(\alpha+\theta_j)$};
			\node[] (B) at (16.2,2.5) {$ \operatorname{Gamma}^{-1}(\theta_i+\theta_j)$};
			\draw[-stealth] (10,1) -- (A);
			\draw[-stealth] (14,1) -- (B);
		\end{tikzpicture}

		\caption{Vertex weights for the half-space log-gamma polymer (with $i=6,j=2$) and two possible paths (one marked in blue and the other in black) in $\Pi_{8,8}$.}
		\label{fig000}
	\end{figure}

\bigskip

{\bf Unless otherwise noted, all of our results and discussions below pertain to the \textit{homogeneous} polymer model where all the $\theta_i$'s are set equal to some $\theta>0$.} In that case, we write $Z_{(\alpha,\theta)}$ for $Z_{(\alpha,\vec{\theta})}$. We include the inhomogeneities when introducing the half-space log-gamma Gibbs property and line ensemble as well as when proving the key tool of stochastic monotonicity. As these key tools extend to the inhomogeneous case, we expect our methods and results should be likewise extendable, though do not pursue that here.

\bigskip

The parameter $\alpha$ controls the strength of the boundary weights and there is a phase transition in the behavior of this model at $\alpha=0$. In our current work we will probe the behavior in the critical regime where $\alpha$ is in a scaling window of order $N^{-1/3}$ of $0$, as well as in the supercritical regime when $\alpha$ is strictly positive. The subcritical regime may be probed in subsequent work as described in Section \ref{sec:relatedwork}. This phase transition has been the subject of quite a lot of previous work, some of which we review in Section \ref{sec:relatedwork}. The basic picture (some as of yet unproved) is as follows. For $\alpha\geq 0$ the free energy (i.e., log of the partition function) should demonstrate the KPZ  $1/3$ fluctuation  and $2/3$ transversal scaling exponents as well as certain universal limiting distributions. Here the transversal scaling references both the $N^{2/3}$ fluctuations of the endpoint of the length $2N$ half-space polymer as well as the $N^{2/3}$ correlation length of the free energy as a function of $(m,n)$ subject to $m+n=2N$. For $\alpha<0$ the situation is different -- the free energy fluctuations should be of order $N^{1/2}$, the endpoint should fluctuate transversally in an order one scale (i.e., not growing with $N$), while the free energy correlation length should be of order $N$ and the limiting distributions should be Gaussian. To be clear, in terms of the polymer measure, this phase transition relates to the pinning ($\alpha<0$) or unpinning ($\alpha\geq 0$) of the path from the diagonal.

Our main result captures the KPZ scaling exponents in the critical and subcritical regimes.
	
	\begin{theorem} \label{t:main0} Fix $\theta,r>0$. For each $\alpha>-\theta$, $s\in [0,r]$, and $N\ge \max\{3,r^3\}$ define the centered and scaled $\hslg$ free energy process
		\begin{equation}\label{eq:fnalpha}
			\mathcal{F}_N^{\alpha}(s):=
			\dfrac{\log Z_{(\alpha,\theta)}(N+sN^{2/3},N-sN^{2/3})+2N\Psi(\theta)}{N^{1/3}}.
		\end{equation}
		Here $\Psi$ denotes the digamma function defined on $\R_{>0}$ by
		\begin{align}\label{psidef}
			\Psi(z):=\partial_z\log \Gamma(z)=-\gamma+\sum_{n=0}^{\infty} \left(\frac1{n+1}-\frac1{n+z}\right),
		\end{align}
		where $\gamma$ is the Euler-Mascheroni constant. The function $\mathcal{F}_N^{\alpha}(\cdot)$ is linearly interpolated in between points where $Z_{(\alpha,\theta)}$ is defined. Let $\Pr^N_{\alpha}$ denotes the law of $\mathcal{F}_N^{\alpha}(\cdot)$ as a random variable in  $(C[0,r],\mathcal{C})$ -- the space of continuous functions on $[0,r]$ equipped with uniform topology and Borel $\sigma$-algebra $\mathcal{C}$. Then the  following holds.
		
		\begin{enumerate}[leftmargin=15pt]
			\item  The sequence $\Pr^N_{\alpha}$ is tight for each $\alpha\in (0,\infty)$.
			\item  For $\alpha_N=N^{-1/3}\mu$ with $\mu\in \R$ fixed (noting that for large enough $N$, $\alpha_N>-\theta$, and thus $\mathcal{F}_N^{\alpha_N}(\cdot)$ is well-defined), the sequence $\Pr^N_{\alpha_N}$ is tight.
		\end{enumerate}
	\end{theorem}

This theorem is proved at the beginning of Section \ref{sec:mc}.

As discussed below, it is possible to show (e.g. using the ideas of \cite{bcd}) absolute continuity of the limit points in Theorem \ref{t:main0} with respect to certain Brownian measures. We do not pursue this here, but remark further about this and related directions below (see the end of Section \ref{sec:hsle}).

The rest of this introduction is structured as follows. Section \ref{sec:hsle} introduces the idea of a half-space Gibbsian line ensemble, the study of which constitutes the key technical innovation responsible for the above theorem. Section \ref{sec:ideas} provides a precise definition of the half-space log-gamma line ensemble and Gibbs property, the key input from \cite{bw} and then a sketch of the steps to proving Theorem \ref{t:main0}. Finally, Section \ref{sec:relatedwork} reviews some related work in studying half-space polymer and related models (Section \ref{sec:hsle}  reviews the literature on Gibbsian line ensembles).

\subsection{Half-space Gibbsian line ensembles}\label{sec:hsle}

In order to prove Theorem \ref{t:main0} we develop a new probabilistic structure -- half-space Gibbsian line ensembles -- and introduce a toolbox through which to study limits of such ensembles. A remarkable fact, due to the geometric RSK correspondence \cite{cosz,osz, nz, bz19b} and the half-space Whittaker process \cite{barraquand_borodin_corwin_2020}, is that the free energy process $\log Z_{(\alpha,\theta)}(N+m,N-m)$ for the log-gamma polymer can be embedded as the top labeled curve of an ensemble of log-gamma increment random walks interacting through a soft version of non-intersection conditioning and subject to an energetic interaction at the left boundary (where $m=0$) depending on the value of $\alpha$. In particular, when $\alpha>0$ the interaction on the left boundary manifests itself as an attraction between the label $2i-1$ and $2i$ curves of the line ensemble for each relevant choice of $i$; for $\alpha<0$ the interaction is repulsive while for $\alpha=0$ it is not present. We describe this line ensemble embedding in Section \ref{sec:1.2} and Section \ref{sec2.2}.

The basic premise of Gibbsian line ensembles, as initiated in the study of full-space models in \cite{ch14}, is to use the resampling invariance of a sequence of such ensembles to propagate one-point tightness information (generally for the top curve of the ensemble) into tightness of the entire sequence of ensembles. In particular once the scale of one-point fluctuations (in this case $N^{1/3}$) is known, the Gibbs property implies transversal fluctuations are correlated in a diffusive scale (in this case $N^{2/3}$) and that lower curves also all fluctuate with these exponents in the same scale. In other words, one point tightness of the top curve translates into spatial tightness of the entire ensemble. Moreover, all subsequential limits of these line ensembles enjoy, themselves, a Gibbs property corresponding to the diffusive limit of that of the pre-limiting ensembles. This general approach has been applied widely in studying a variety of different Gibbs properties related to probabilistic models, e.g. \cite{ch16,cd18,dnv19,wu2,bcd,xd1,dff,serio,wu1}. Moreover, it has been leveraged to give fine information about the local behavior of these models \cite{ham1,ham2,ham3,ham4,chh19,chaos2,chaos1,ciw1,ciw2,cgh21,dg21,wu3,dz22a,dz22b,gh22}
and in studying related scaling limits such as the Airy sheet and directed landscape \cite{dv18,dov18,bgh1,bgh2,3by2,sv21,chhm,gh21,rv21,wu23}.

In this work we initiate the study of half-space Gibbsian line ensembles. These are measures on collections of curves in which there exists a left boundary around which the Gibbs property differs from its behavior in the bulk. As an illustrative example, consider curves $\L_1(s)\geq \L_2(s)\geq \cdots$ for $s\geq 0$ which enjoy the following resampling invariance. In the bulk, for $0<s<t$ and $1\leq k_1\leq k_2$ the law of $\L_{\ll k_1,k_2\rr}([s,t])$ (i.e., curves $k_1$ through $k_2$ on the interval $[s,t]$) conditioned on the values of $\L_{\ll k_1,k_2\rr}(s)$, $\L_{\ll k_1,k_2\rr}(t)$, $\L_{k_1-1}([s,t])$ (if $k_1=1$ then $\L_0\equiv +\infty$) and  $\L_{k_2+1}([s,t])$ is that of Brownian motions conditioned to start at $s$ and end at $t$ with the correct boundary values and to not intersect each other nor the curve $\L_{k_1-1}([s,t])$ above and  $\L_{k_2+1}([s,t])$ below. Around the left boundary, for $t>0$ and $1\leq k_1\leq k_2$ the law of $\L_{\ll k_1,k_2\rr}([0,t])$ conditioned on the values of  $\L_{\ll k_1,k_2\rr}(t)$, $\L_{k_1-1}([0,t])$ and  $\L_{k_2+1}([0,t])$ is the law of Brownian motions conditioned to end at values $\L_{\ll k_1,k_2\rr}(t)$  at time $t$, not intersect with each other or the $\L_{k_1-1}$ and $\L_{k_2+1}$ curves on the interval $[0,t]$ and to have values at zero such that $\L_{2i-1}(0)=\L_{2i}(0)$ for all $i$. This last condition that is quite novel to the half-space models. An example of such an ensemble is illustrated in Figure \ref{figO1} (B). This Gibbs property arise as a diffusive limit of the half-space log-gamma Gibbs property introduced and studied here.

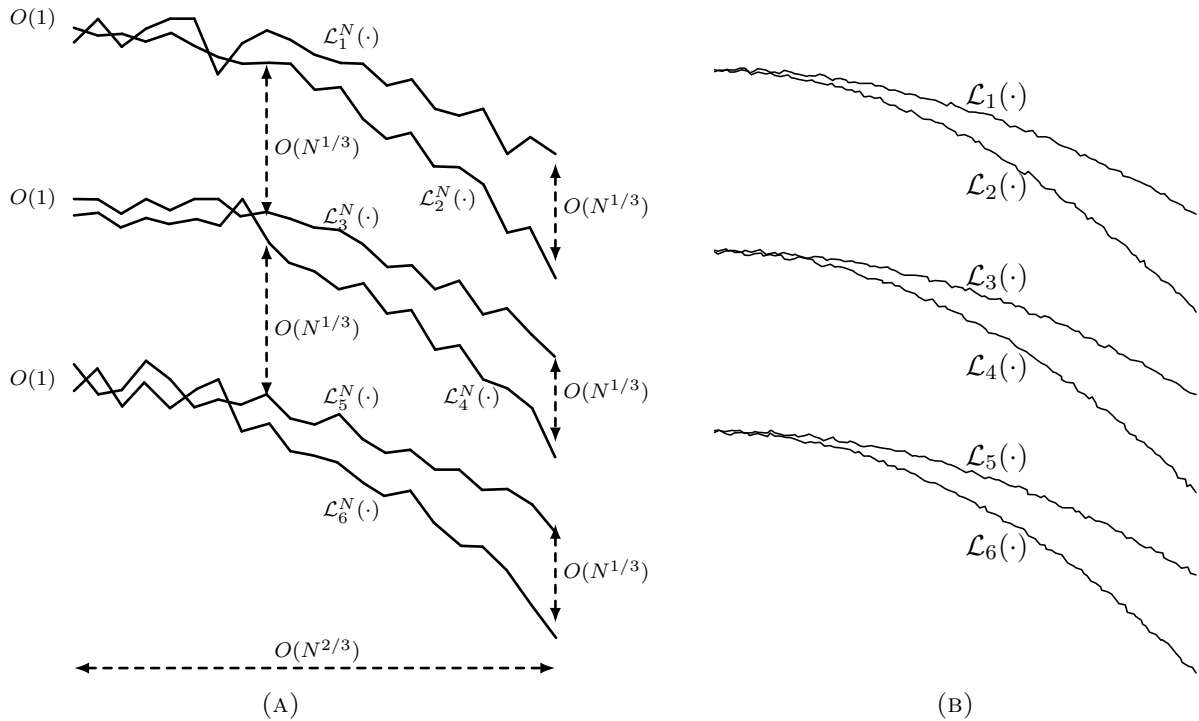
\begin{figure}[h!]
		\centering
		\begin{subfigure}[b]{0.45\textwidth}
			\centering
		\begin{tikzpicture}[scale=.8,line cap=round,line join=round,>=triangle 45,x=4cm,y=1.5cm]
			\draw [line width=1pt] (0,2.7361969564619426)-- (0.09968303324610706,2.9960253171531424);
			\draw [line width=1pt] (0.09968303324610706,2.9960253171531424)-- (0.19830773702771506,2.6884965795913667);
			\draw [line width=1pt] (0.19830773702771506,2.6884965795913667)-- (0.30006302893410286,2.889974785777928);
			\draw [line width=1pt] (0.30006302893410286,2.889974785777928)-- (0.4,3);
			\draw [line width=1pt] (0.4,3)-- (0.5,3);
			\draw [line width=1pt] (0.5,3)-- (0.5973644533669638,2.3832141676196814);
			\draw [line width=1pt] (0.5973644533669638,2.3832141676196814)-- (0.698588596048139,2.7266568810878273);
			\draw [line width=1pt] (0.698588596048139,2.7266568810878273)-- (0.8017593568577984,2.869758011699555);
			\draw [line width=1pt] (0.8017593568577984,2.869758011699555)-- (0.8990902632820055,2.7648171825842884);
			\draw [line width=1pt] (0.8990902632820055,2.7648171825842884)-- (0.9990053878836391,2.600795293991784);
			\draw [line width=1pt] (0.9990053878836391,2.600795293991784)-- (1.1079547104749525,2.5089745438145457);
			\draw [line width=1pt] (1.1079547104749525,2.5089745438145457)-- (1.2,2.5);
			\draw [line width=1pt] (1.2,2.5)-- (1.300093597749738,2.259193187756184);
			\draw [line width=1pt] (1.300093597749738,2.259193187756184)-- (1.4032643585593976,2.32597371537499);
			\draw [line width=1pt] (1.4032643585593976,2.32597371537499)-- (1.5,2);
			\draw [line width=1pt] (1.5,2)-- (1.5998727895362959,1.925290549662153);
			\draw [line width=1pt] (1.5998727895362959,1.925290549662153)-- (1.7,2);
			\draw [line width=1pt] (1.7,2)-- (1.8,1.5);
			\draw [line width=1pt] (1.8,1.5)-- (1.895758745065885,1.6867886653092736);
			\draw [line width=1pt] (1.895758745065885,1.6867886653092736)-- (2,1.5);
			\draw [line width=1pt] (0,2.8983782378219005)-- (0.09903021247502392,2.812517559454864);
			\draw [line width=1pt] (0.09903021247502392,2.812517559454864)-- (0.20025435515619922,2.8315977102030945);
			\draw [line width=1pt] (0.20025435515619922,2.8315977102030945)-- (0.2975852615804062,2.745737031836058);
			\draw [line width=1pt] (0.2975852615804062,2.745737031836058)-- (0.40398128310420617,2.841199415267653);
			\draw [line width=1pt] (0.40398128310420617,2.841199415267653)-- (0.5000335469427568,2.698036654965482);
			\draw [line width=1pt] (0.5000335469427568,2.698036654965482)-- (0.5993110714954479,2.5740156751019847);
			\draw [line width=1pt] (0.5993110714954479,2.5740156751019847)-- (0.7,2.5);
			\draw [line width=1pt] (0.7,2.5)-- (0.8116756285443952,2.510788449351374);
			\draw [line width=1pt] (0.8116756285443952,2.510788449351374)-- (0.9,2.5);
			\draw [line width=1pt] (0.9,2.5)-- (1.0003144059631806,2.2114928108856082);
			\draw [line width=1pt] (1.0003144059631806,2.2114928108856082)-- (1.1073784030298084,2.2401130370079536);
			\draw [line width=1pt] (1.1073784030298084,2.2401130370079536)-- (1.200816073197047,1.8871302481656922);
			\draw [line width=1pt] (1.200816073197047,1.8871302481656922)-- (1.2981469796212541,1.6677085145610433);
			\draw [line width=1pt] (1.2981469796212541,1.6677085145610433)-- (1.3993711223024294,1.7344890421798496);
			\draw [line width=1pt] (1.3993711223024294,1.7344890421798496)-- (1.4947554105981522,1.3624261025893578);
			\draw [line width=1pt] (1.4947554105981522,1.3624261025893578)-- (1.5998727895362959,1.3528860272152425);
			\draw [line width=1pt] (1.5998727895362959,1.3528860272152425)-- (1.701096932217471,1.162084519732939);
			\draw [line width=1pt] (1.701096932217471,1.162084519732939)-- (1.8003744567701623,0.6278402987824895);
			\draw [line width=1pt] (1.8003744567701623,0.6278402987824895)-- (1.8938121269374009,0.6850807510271805);
			\draw [line width=1pt] (1.8938121269374009,0.6850807510271805)-- (2.0008761240040287,0.12221630395438526);
			\draw [line width=1pt] (0,1)-- (0.09968303324610683,0.9960253171531422);
			\draw [line width=1pt] (0.09968303324610683,0.9960253171531422)-- (0.19636111889923094,0.8377219570130232);
			\draw [line width=1pt] (0.19636111889923094,0.8377219570130232)-- (0.3,1);
			\draw [line width=1pt] (0.3,1)-- (0.3988094042615815,0.875882258509484);
			\draw [line width=1pt] (0.3988094042615815,0.875882258509484)-- (0.5,1);
			\draw [line width=1pt] (0.5,1)-- (0.6,1);
			\draw [line width=1pt] (0.6,1)-- (0.6912407305290592,0.808874500983061);
			\draw [line width=1pt] (0.6912407305290592,0.808874500983061)-- (0.7998127387293144,0.8568021077612535);
			\draw [line width=1pt] (0.7998127387293144,0.8568021077612535)-- (0.8990902632820055,0.7804815047683322);
			\draw [line width=1pt] (0.8990902632820055,0.7804815047683322)-- (0.9964211697062124,0.6850807510271805);
			\draw [line width=1pt] (0.9964211697062124,0.6850807510271805)-- (1.1015385486443559,0.656460524904835);
			\draw [line width=1pt] (1.1015385486443559,0.656460524904835)-- (1.2,0.5);
			\draw [line width=1pt] (1.2,0.5)-- (1.3020402158782223,0.2462372838178825);
			\draw [line width=1pt] (1.3020402158782223,0.2462372838178825)-- (1.4032643585593976,0.26531743456611284);
			\draw [line width=1pt] (1.4032643585593976,0.26531743456611284)-- (1.5,0);
			\draw [line width=1pt] (1.5,0)-- (1.5998727895362959,0.09359607783203974);
			\draw [line width=1pt] (1.5998727895362959,0.09359607783203974)-- (1.699150314088987,-0.27846686175845203);
			\draw [line width=1pt] (1.699150314088987,-0.27846686175845203)-- (1.798427838641678,-0.2116863341396458);
			\draw [line width=1pt] (1.798427838641678,-0.2116863341396458)-- (1.9,-0.5);
			\draw [line width=1pt] (1.9,-0.5)-- (1.9969828877470606,-0.7459305550900955);
			\draw [line width=1pt] (0,0.8186418062647929)-- (0.10097683060350805,0.8472620323871384);
			\draw [line width=1pt] (0.10097683060350805,0.8472620323871384)-- (0.1944145007707468,0.6850807510271805);
			\draw [line width=1pt] (0.1944145007707468,0.6850807510271805)-- (0.29953187970889034,0.7900215801424474);
			\draw [line width=1pt] (0.29953187970889034,0.7900215801424474)-- (0.3949161680046132,0.7232410525236411);
			\draw [line width=1pt] (0.3949161680046132,0.7232410525236411)-- (0.50890466407512,0.7773440285883928);
			\draw [line width=1pt] (0.50890466407512,0.7773440285883928)-- (0.6071213261663665,0.7042691302093282);
			\draw [line width=1pt] (0.6071213261663665,0.7042691302093282)-- (0.7,1);
			\draw [line width=1pt] (0.7,1)-- (0.8116756285443949,0.5107884493513744);
			\draw [line width=1pt] (0.8116756285443949,0.5107884493513744)-- (0.893250408896553,0.2939376606884584);
			\draw [line width=1pt] (0.893250408896553,0.2939376606884584)-- (0.9983677878346965,0.19853690694730664);
			\draw [line width=1pt] (0.9983677878346965,0.19853690694730664)-- (1.1,0);
			\draw [line width=1pt] (1.1,0)-- (1.2027626913255312,0.06497585170969422);
			\draw [line width=1pt] (1.2027626913255312,0.06497585170969422)-- (1.300093597749738,-0.2689267863843368);
			\draw [line width=1pt] (1.300093597749738,-0.2689267863843368)-- (1.4013177404309134,-0.23076648488787616);
			\draw [line width=1pt] (1.4013177404309134,-0.23076648488787616)-- (1.5005952649836045,-0.6696099520971741);
			\draw [line width=1pt] (1.5005952649836045,-0.6696099520971741)-- (1.6018194076647798,-0.6219095752265983);
			\draw [line width=1pt] (1.6018194076647798,-0.6219095752265983)-- (1.7,-1);
			\draw [line width=1pt] (1.7,-1)-- (1.8003744567701623,-1.0989133439323568);
			\draw [line width=1pt] (1.8003744567701623,-1.0989133439323568)-- (1.9035452175798215,-1.318335077537006);
			\draw [line width=1pt] (1.9035452175798215,-1.318335077537006)-- (1.9989295058755447,-1.8621193738615707);
			\draw [line width=1pt] (2.002822742132513,-3.865535202425757)-- (1.9,-3.5);
			\draw [line width=1pt] (1.9,-3.5)-- (1.796481220513194,-3.1118692478706587);
			\draw [line width=1pt] (1.796481220513194,-3.1118692478706587)-- (1.6972036959605028,-2.854287212769549);
			\draw [line width=1pt] (1.6972036959605028,-2.854287212769549)-- (1.6076592620502324,-2.8447471373954336);
			\draw [line width=1pt] (1.6076592620502324,-2.8447471373954336)-- (1.4947554105981522,-2.587165102294324);
			\draw [line width=1pt] (1.4947554105981522,-2.587165102294324)-- (1.3974245041739453,-2.2341823134520626);
			\draw [line width=1pt] (1.3974245041739453,-2.2341823134520626)-- (1.2884138889788335,-2.2914227656967534);
			\draw [line width=1pt] (1.2884138889788335,-2.2914227656967534)-- (1.2027626913255312,-2.148321635085026);
			\draw [line width=1pt] (1.2027626913255312,-2.148321635085026)-- (1.0937520761304194,-1.9193598261062619);
			\draw [line width=1pt] (1.0937520761304194,-1.9193598261062619)-- (0.9944745515777283,-1.8430392231133403);
			\draw [line width=1pt] (0.9944745515777283,-1.8430392231133403)-- (0.8990902632820055,-1.7953388462427646);
			\draw [line width=1pt] (0.8990902632820055,-1.7953388462427646)-- (0.8116756285443951,-1.4892115506486259);
			\draw [line width=1pt] (0.8116756285443951,-1.4892115506486259)-- (0.6966419779196549,-1.5759171126381155);
			\draw [line width=1pt] (0.6966419779196549,-1.5759171126381155)-- (0.6,-1);
			\draw [line width=1pt] (0.6,-1)-- (0.5113633117794375,-1.1045969746536137);
			\draw [line width=1pt] (0.5113633117794375,-1.1045969746536137)-- (0.3988094042615815,-1.318335077537006);
			\draw [line width=1pt] (0.3988094042615815,-1.318335077537006)-- (0.29203043996444555,-1.034112711146331);
			\draw [line width=1pt] (0.29203043996444555,-1.034112711146331)-- (0.20025435515619922,-1.2992549267887756);
			\draw [line width=1pt] (0.20025435515619922,-1.2992549267887756)-- (0.09708359434653978,-0.879491610327708);
			\draw [line width=1pt] (0.09708359434653978,-0.879491610327708)-- (0,-1.1275335700547024);
			\draw [line width=1pt] (0,-0.831791233457132)-- (0.10097683060350805,-1.165693871551163);
			\draw [line width=1pt] (0.10097683060350805,-1.165693871551163)-- (0.19830773702771506,-1.1179934946805872);
			\draw [line width=1pt] (0.19830773702771506,-1.1179934946805872)-- (0.29953187970889034,-0.7936309319606714);
			\draw [line width=1pt] (0.29953187970889034,-0.7936309319606714)-- (0.4,-1);
			\draw [line width=1pt] (0.4,-1)-- (0.5000335469427568,-1.3087950021628907);
			\draw [line width=1pt] (0.5000335469427568,-1.3087950021628907)-- (0.5993110714954479,-1.222934323795854);
			\draw [line width=1pt] (0.5993110714954479,-1.222934323795854)-- (0.6927487416626866,-1.2801747760405453);
			\draw [line width=1pt] (0.6927487416626866,-1.2801747760405453)-- (0.7998127387293144,-1.165693871551163);
			\draw [line width=1pt] (0.7998127387293144,-1.165693871551163)-- (0.8990902632820055,-1.432815982026388);
			\draw [line width=1pt] (0.8990902632820055,-1.432815982026388)-- (1,-1.5);
			\draw [line width=1pt] (1,-1.5)-- (1.1015385486443559,-1.385115605155812);
			\draw [line width=1pt] (1.1015385486443559,-1.385115605155812)-- (1.198869455068563,-1.661777791005152);
			\draw [line width=1pt] (1.198869455068563,-1.661777791005152)-- (1.2942537433642858,-1.814418996990995);
			\draw [line width=1pt] (1.2942537433642858,-1.814418996990995)-- (1.3967412169145335,-1.7803544108111566);
			\draw [line width=1pt] (1.3967412169145335,-1.7803544108111566)-- (1.5,-2);
			\draw [line width=1pt] (1.5,-2)-- (1.6,-2);
			\draw [line width=1pt] (1.6,-2)-- (1.6952570778320186,-2.2246422380779474);
			\draw [line width=1pt] (1.6952570778320186,-2.2246422380779474)-- (1.8003744567701623,-2.215102162703832);
			\draw [line width=1pt] (1.8003744567701623,-2.215102162703832)-- (1.9035452175798215,-2.3868235194379053);
			\draw [line width=1pt] (1.9035452175798215,-2.3868235194379053)-- (1.9989295058755447,-2.692105931409591);
			\draw [line width=1pt,dashed,{Latex[length=2mm]}-{Latex[length=2mm]}] (0,-4.2)-- (2,-4.2);
			\draw [line width=1pt,dashed,{Latex[length=2mm]}-{Latex[length=2mm]}] (0.8,0.5)-- (0.8,-1.2);
			\draw [line width=1pt,dashed,{Latex[length=2mm]}-{Latex[length=2mm]}] (0.8,2.5)-- (0.8,0.8);
			\draw [line width=1pt,dashed,{Latex[length=2mm]}-{Latex[length=2mm]}] (2,1.4)-- (2,0.3);
			\draw [line width=1pt,dashed,{Latex[length=2mm]}-{Latex[length=2mm]}] (2,-0.75)-- (2,-1.7);
			\draw [line width=1pt,dashed,{Latex[length=2mm]}-{Latex[length=2mm]}] (2,-2.6)-- (2,-3.7);
			\begin{scriptsize}
				\draw (0.8,-4) node[anchor=west] {$O(N^{2/3})$};
				\draw (0.8,-0.2) node[anchor=north west] {$O(N^{1/3})$};
				\draw (0.8,1.8) node[anchor=north west] {$O(N^{1/3})$};
				\draw (2,1.15) node[anchor=north west] {$O(N^{1/3})$};
				\draw (2,-0.9) node[anchor=north west] {$O(N^{1/3})$};
				\draw (2,-2.9) node[anchor=north west] {$O(N^{1/3})$};
				\draw (-0.3,3) node[anchor=west] {$O(1)$};
				\draw (-0.3,1) node[anchor=west] {$O(1)$};
				\draw (-0.3,-1) node[anchor=west] {$O(1)$};
				\draw (1,2.8) node[anchor=west] {$\L_1^N(\cdot)$};
				\draw (1.4,1) node[anchor=west] {$\L_2^N(\cdot)$};
				\draw (1,0.8) node[anchor=west] {$\L_3^N(\cdot)$};
				\draw (1.5,-1.2) node[anchor=west] {$\L_4^N(\cdot)$};
				\draw (1,-1.2) node[anchor=west] {$\L_5^N(\cdot)$};
				\draw (1,-2.45) node[anchor=west] {$\L_6^N(\cdot)$};
			\end{scriptsize}
		\end{tikzpicture}
	\caption{}
		\end{subfigure}\qquad\qquad
\begin{subfigure}[b]{0.45\textwidth}
		\centering
		\begin{tikzpicture}[scale=.8,line cap=round,line join=round,>=triangle 45,x=4cm,y=1.5cm]
			\draw[semithick, decorate, decoration={random steps,segment length=2pt,amplitude=1pt}] plot[domain=0:2] (\x, {3-2*(\x)^2/5});
			\draw[semithick, decorate, decoration={random steps,segment length=2pt,amplitude=1pt}] plot[domain=0:2] (\x, {2.98-2*(\x)^2/3});
			\draw[semithick, decorate, decoration={random steps,segment length=2pt,amplitude=1pt}] plot[domain=0:2] (\x, {1-2*(\x)^2/5});
			\draw[semithick, decorate, decoration={random steps,segment length=2pt,amplitude=1pt}] plot[domain=0:2] (\x, {.98-2*(\x)^2/3});
			\draw[semithick, decorate, decoration={random steps,segment length=2pt,amplitude=1pt}] plot[domain=0:2] (\x, {-1-2*(\x)^2/5});
			\draw[semithick, decorate, decoration={random steps,segment length=2pt,amplitude=1pt}] plot[domain=0:2] (\x, {-1.02-2*(\x)^2/3});
			\draw (1,2.7) node[anchor=west] {$\L_1(\cdot)$};
			\draw (1,1.7) node[anchor=west] {$\L_2(\cdot)$};
			\draw (1,0.7) node[anchor=west] {$\L_3(\cdot)$};
			\draw (1,-.3) node[anchor=west] {$\L_4(\cdot)$};
			\draw (1,-1.3) node[anchor=west] {$\L_5(\cdot)$};
			\draw (1,-2.3) node[anchor=west] {$\L_6(\cdot)$};
		\end{tikzpicture}
			\caption{}
		\end{subfigure}
		
		\caption{(A) depicts the half-space log-gamma line ensemble for large $N$ along with the type of scalings that are deduced in proving Theorem \ref{t:main0}. This ensemble enjoys a half-space log-gamma Gibbs property. (B) depicts a potential limiting line ensemble which should enjoy a half-space non-intersecting Brownian Gibbs property.}			
\label{figO1}
	\end{figure}

Half-space Gibbsian line ensembles have not previously been studied. However, this structure exists implicitly in some previous literature studying half-space integrable probabilistic models. For instance, the half-space (or Pfaffian) Schur processes \cite{sis,br05, bbcs} have such a structure where the Brownian resampling is replaced by certain discrete random walks (geometric, exponential or Bernoulli), the non-intersection conditioning persists, and where the odd/even pairing at the boundary is replaced by an exponential interaction in the spirit of $e^{-\alpha (\L_{2i-1}(0)-\L_{2i}(0))}$. Half-space Whittaker processes  \cite{barraquand_borodin_corwin_2020} have a more complicated Gibbs property which is the one relevant to our current work. Essentially, the Brownian motion is replaced by log-gamma random walks, the non-intersection by a soft exponential energy reweighing, and the interaction at zero by the same sort of  $e^{-\alpha (\L_{2i-1}(0)-\L_{2i}(0))}$ reweighing. There are other half-space Gibbs properties that should be studied such as related to half-space version of Hall-Littlewood processes, $q$-Whittaker processes and their spin generalizations (see for instance, \cite{barraquand_borodin_corwin_2020,bbcw,jimmy2}). Furthermore, periodic or two-sided boundary versions of Gibbsian line ensembles (e.g. related to versions of Schur processes as in \cite{borodinperiodic, bete, bcy23}) will also likely play a key role in study of related integrable probabilistic models and hence warrant study in the spirit of what is done here.

As in the full-space setting, the challenge is to develop a route to take one-point fluctuation information about the top curve $\L_1^N$ of a sequence of line ensembles $\L^N$ and propagate that into fluctuation information about the whole ensemble. (Figure \ref{figO1} (A) illustrates the scalings that we prove to be associated with this sequence of line ensembles.) One-point information about the top curve for the half-space log-gamma polymer (and hence the top curve of our line ensemble) is in short supply with only two result due to (chronologically) \cite{bw} and then \cite{ims22}.

The core technical purpose and challenge of this paper is to extend the Gibbsian line ensemble methodology to address half-space models and provide tools to show tightness at the edge of such ensembles. We do this for the type of Gibbs property mentioned above that relates to half-space Whittaker processes which, owing to its relation to the log-gamma polymer, we call the half-space log-gamma Gibbs property. Our tools and method should extend to other Gibbs properties.

The challenge  in the half-space models comes from the impact of the pair interaction at the boundary. When $\alpha>0$ is fixed, in edge scaling limits $\L$ of the line ensemble  $\L_{2i-1}(0)=\L_{2i}(0)$ for all $i\geq 1$. Before taking a limit, the pairs of curves can be described in the vicinity of the origin as two (softly) non-intersecting log-gamma random walks whose left boundary endpoints are energetically conditioned to stay within $O(1)$. 
We call this law on pairs of paths the  \textit{weighted paired random walk} (WPRW) measure, see Definition \ref{prb} below. This is a discretization of two-particle Dyson Brownian motion with both particles started at the same point.

The fine and uniform information that we need to know about the WPRW measure does not follow from weak convergence to Dyson Brownian motion. Thus, we develop a variety of results herein to deal with WPRWs, in case with general underlying jump distributions, not just the log-gamma. See Section \ref{sec:133} for further discussion on WPRWs and their role in our analysis and their properties. Appendix \ref{app2} contains our general results on non-intersecting random walks and bridges.

\subsubsection*{Overall Strategy} As explained in Section \ref{sec:132}, we rely only on the  work of \cite{bw}. From \cite{bw} we are able to extract two vital pieces of information: after proper centering the process  $s\mapsto N^{-1/3}\L_1^N(sN^{2/3})$ stays bounded from positive infinity at $N\to\infty$, and at a random sequence of growing times $s^N_1,s^N_2,\ldots$ that stay tight as $N\to \infty$, the process has tight (bounded from positive and negative infinity) fluctuations around the parabola $-\nu s^2$ (for some explicit $\nu>0$). The slightly odd nature of these inputs comes from the fact that \cite{bw} studies a point-to-(partial)line partition function and not point-to-point directly. The work of \cite{ims22} does provide tightness (and a limit theorem) for the point-to-point free energy, but is restricted to precisely the left boundary $\L_1^N(1)$ which is insufficient information for our approach. Currently, there are no limit theorems proved for the point-to-point free energy process away from the left-boundary.

With the above input we proceed to show how the Gibbs property propagates tightness to the whole ensemble. The idea is to first argue that (with proper centering) the process $s\mapsto N^{-1/3} \L_2^N(sN^{2/3})$ must be tight at some random time  $s$. If not, the first curve would not follow a parabolic decay but rather a linear one in contradiction with our parabolic decay input. Now, we know that the (scaled) first and second curves are tight at some random times (not necessarily the same). The next step is to argue that this pair of scaled curves to the left of the random times (including the left-boundary) are likewise tight. This relies on showing (using the Gibbs property and some a priori bounds) that the third curve cannot rise much beyond the first two curves, and that the first two curves remain bounded from infinity (as follows from \cite{bw}). With this and a form of stochastic monotonicity associated to this Gibbs property, the control over the first two curves can be established by a fine analysis of the behavior of a pair of log-gamma random walks subject to soft non-intersection conditioning and attractive energetic pinning at zero. We call these \emph{weighted paired random walks} and a substantial amount of work is needed to develop tools and estimates regarding them. We give a more detailed overview of the steps of our proof in Section \ref{sec:133}.  The attractive nature of the boundary is directly linked to the choice here that $\alpha \geq 0$.

\subsubsection*{Extensions} In this paper we do not pursue showing that the tightness propagation process extends to the entire line ensemble, though it very likely can be done, e.g. in the spirit of \cite{xd1} for a full-space line ensemble. Any subsequential limit should enjoy the type of half-space Brownian Gibbs property discussed earlier. This would show that any such subsequential limit should also enjoy local comparison to Brownian motions away from the boundary, or two 2-particle Dyson Brownian motions started paired together when looking near the boundary. 
In fact, for the top two curves we can extract (though do not explicitly record here) such absolute continuity results without showing tightness of the whole ensemble, e.g. as in \cite{bcd}. The full-space Gibbs property in \cite{xd1,bcd} differs slightly from here since they consider point-to-point polymer endpoints varying along horizontal lines, while we consider endpoints varying along down-right zigzag paths.

Besides the directions alluded to above, we mention here a few more natural points of inquiry spurred by our work. Our analysis is restricted to $\alpha\geq 0$. When $\alpha<0$, the pair interaction at the boundary becomes repulsive, and thus, curves separate and behave quite differently. In particular, the log-gamma free energy (i.e., top curve) is expected to have $O(\sqrt{N})$ Gaussian fluctuations and $O(1)$ transversal fluctuation around $(N,N)$. The Gaussian fluctuations on the diagonal was recently proven in \cite{ims22}, while the $O(1)$ transversal fluctuations result appears in the subsequent work \cite{dz23}. The behavior in this $O(1)$ scale relates to a portion of the phase diagram for the half-space log-gamma stationary measure \cite{bc22}. Using our Gibbsian line ensemble techniques and modifications of the log-gamma polymer (i.e., adding a boundary condition on the first row too), it should be possible to access and re-derive the description of the entire phase diagram.

Beyond tightness, the half-space log-gamma line ensemble should converge to a universal limit, the half-space Airy line ensemble. This object, which should enjoy the type of Brownian Gibbs property discussed earlier, has not been constructed. While the log-gamma convergence result is currently out of reach, it should be possible to construct this from solvable last passage percolation, i.e. half-space Schur processes \cite{bbcs}. This should enjoy uniqueness characterization in the spirit of \cite{dm21,d21} and may even relate to a half-space Airy sheet in the spirit of \cite{dov18}. It is also a compelling challenge to identify a {\it strong characterization} of the half-space Airy line ensemble in the spirit of the recent work \cite{AH23} on the full-space Airy line ensemble.

A different scaling regime for the half-space log-gamma line ensemble involves weak-noise scaling in which $\theta$ goes to infinity while $\alpha$ remains fixed. In the full-space setting, \cite{wu} proved tightness of the full-space line ensemble and (via \cite{corwinnica}) convergence to the KPZ line ensemble \cite{OConWar,ch16}. A half-space analog of this result should hold and help in exploring questions related to the half-space KPZ equation and the corresponding half-space continuous directed random polymer.

\subsection{Ideas in the proof of Theorem \ref{t:main0}}  \label{sec:ideas}

In Section \ref{sec:1.2} we precisely define the half-space log-gamma Gibbs measure and line ensemble. In Section \ref{sec:132} we record the key input from \cite{bw} which we then combine with the Gibbs line ensemble structure in Section \ref{sec:133} to give the key deductions in the course of proving Theorem \ref{t:main0} (see Section \ref{sec:mc} for the full proof of this theorem).

Though the Gibbs measure and line ensemble definition holds for general $\alpha$, most of our discussion, especially around the proof, will focus on the case $\alpha>0$ which is harder than the $\alpha=N^{-1/3}\mu$ case. As noted earlier, we do not address the case of $\alpha<0$ here.

\subsubsection{$\hslg$ Gibbs measures and the $\hslg$ line ensemble}\label{sec:1.2}
The main technique that goes into the proof of Theorem \ref{t:main0} is our construction of the half-space log-gamma ($\hslg$) line ensemble whose law enjoys a property that we call the half-space log-gamma ($\hslg$) Gibbs measures. In what follows we construct these objects and describe how they relate to the $\hslg$ polymer free energy.

We will start by defining the fully-inhomogeneous $\hslg$ Gibbs measure whose state-space and associated weight function is indexed by the following directed and colored (and labeled) graph. Fix any parameters $\Theta:=\{\vartheta_{m,n}>0 \mid (m,n)\in \Z_{\ge 1}^2\}$ and $\alpha>-\min\{\vartheta_{m,n} : (m,n)\in \mathbb{Z}_{\ge 1}^2\}$. Note, we have used $\vartheta$ here to distinguish from $\theta$ used to define the polymer. In Theorem \ref{thm:conn} we will relate these parameters. Define the graph $G$ with vertices $V(G):=\{(m,n): m\in \Z_{\ge 1}, n\in \Z_{<0}+\frac12\ind_{m\in 2\Z} \}$ and with the following directed colored (and labeled) edges:
	\begin{itemize}[leftmargin=15pt]
		\item For each $(m,n)\in \Z_{\ge 1}^2$, we put two {\color{blue}{\textbf{blue}$(\vartheta_{m,n})$}} edges from $$(2m-1,-n)\to(2m,-n+\tfrac12)\mbox{ and }(2m+1,-n)\to(2m,-n+\tfrac12).$$
		\item For each $(m,n)\in \Z_{\ge 1}^2$, we put two \textbf{black} edges from $$(2m,-n-\tfrac12)\to (2m-1,-n)\mbox{ and }(2m,-n-\tfrac12)\to (2m+1,-n).$$
		\item For each $m\in \Z_{\ge 1}$, we put one \textcolor{red}{\textbf{red}} edge from $(1,-2m+1)\to(1,-2m)$.
	\end{itemize}
Note there is a parameter linked to the \blue\ edges, while the \black\ and \red\ edges do not have any associated parameters.
A portion of the corresponding graph is shown in Figure \ref{fig00} (A).	We write $E(G)$ for the set of edges of graph $G$ and $e=\{v_1\to v_2\}$ for a generic directed edge from $v_1$ to $v_2$ in $E(G)$ (the color of the edge is suppressed from the notation).

	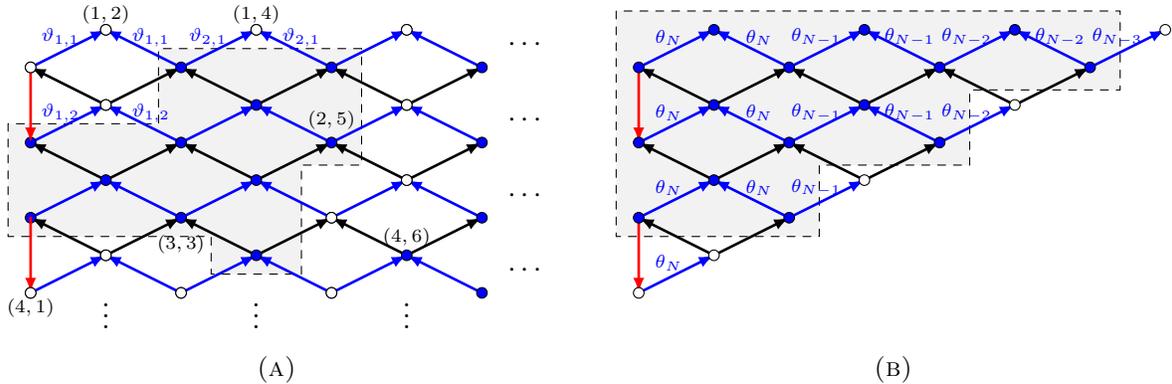
\begin{figure}[h!]
		\centering
		\begin{subfigure}[b]{0.45\textwidth}
			\centering
			\begin{tikzpicture}[line cap=round,line join=round,>=triangle 45,x=2cm,y=1cm]
				\draw[fill=gray!10,dashed] (1.7,-0.25)--(0.35,-0.25)--(0.35,-1.25)--(-0.65,-1.25)--(-0.65,-2.75)--(0.7,-2.75)--(0.7,-3.25)--(1.3,-3.25)--(1.3,-1.8)--(1.7,-1.8)--(1.7,-0.25);
				\foreach \x in {0,1,2}
				{
					\draw[line width=1pt,blue,{Latex[length=2mm]}-]  (\x,0) -- (\x-0.5,-0.5);
					\draw[line width=1pt,blue,{Latex[length=2mm]}-] (\x,0) -- (\x+0.5,-0.5);
					\draw[line width=1pt,black,{Latex[length=2mm]}-] (\x-0.5,-0.5) -- (\x,-1);
					\draw[line width=1pt,black,{Latex[length=2mm]}-] (\x+0.5,-0.5) -- (\x,-1);
					\draw[line width=1pt,blue,{Latex[length=2mm]}-]  (\x,-1) -- (\x-0.5,-1.5);
					\draw[line width=1pt,blue,{Latex[length=2mm]}-] (\x,-1) -- (\x+0.5,-1.5);
					\draw[line width=1pt,black,{Latex[length=2mm]}-] (\x-0.5,-1.5) -- (\x,-2);
					\draw[line width=1pt,black,{Latex[length=2mm]}-] (\x+0.5,-1.5) -- (\x,-2);
					\draw[line width=1pt,blue,{Latex[length=2mm]}-]  (\x,-2) -- (\x-0.5,-2.5);
					\draw[line width=1pt,blue,{Latex[length=2mm]}-] (\x,-2) -- (\x+0.5,-2.5);
					\draw[line width=1pt,black,{Latex[length=2mm]}-] (\x-0.5,-2.5) -- (\x,-3);
					\draw[line width=1pt,black,{Latex[length=2mm]}-] (\x+0.5,-2.5) -- (\x,-3);
					\draw[line width=1pt,blue,{Latex[length=2mm]}-]  (\x,-3) -- (\x-0.5,-3.5);
					\draw[line width=1pt,blue,{Latex[length=2mm]}-] (\x,-3) -- (\x+0.5,-3.5);
				}
				\foreach \x in {0,1,2}
				{
					\draw [fill=blue] (\x,0) circle (2pt);
					\draw [fill=blue] (\x,-1) circle (2pt);
					\draw [fill=blue] (\x,-2) circle (2pt);
					\draw [fill=blue] (\x-0.5,-0.5) circle (2pt);
					\draw [fill=blue] (\x-0.5,-1.5) circle (2pt);
					\draw [fill=blue] (\x-0.5,-2.5) circle (2pt);
					\draw [fill=blue] (\x-0.5,-3.5) circle (2pt);
					\draw [fill=blue] (\x,-3) circle (2pt);
				}
				\draw[line width=1pt,red,{Latex[length=2mm]}-] (-0.5,-1.5) -- (-0.5,-0.5);
				\draw[line width=1pt,red,{Latex[length=2mm]}-] (-0.5,-3.5) -- (-0.5,-2.5);
				\draw [fill=blue] (2.5,-0.5) circle (2pt);
				\draw [fill=blue] (2.5,-1.5) circle (2pt);
				\draw [fill=blue] (2.5,-2.5) circle (2pt);
				\draw [fill=blue] (2.5,-3.5) circle (2pt);
				\draw [fill=white] (2,0) circle (2pt);
				\draw[fill=white] (2,-1) circle (2pt);
				\draw[fill=white] (2,-2) circle (2pt);
				\draw[fill=white] (1.5,-2.5) circle (2pt);
				\draw[fill=white] (1.5,-3.5) circle (2pt);
				\draw[fill=white] (0.5,-3.5) circle (2pt);
				\draw[fill=white] (-0.5,-3.5) circle (2pt);
				\draw[fill=white] (0,-3) circle (2pt);
				\draw[fill=white] (0,0) circle (2pt);
				\draw[fill=white] (0,-1) circle (2pt);
				\draw[fill=white] (-0.5,-0.5) circle (2pt);
				\draw[fill=white] (1,0) circle (2pt);
				\node at (2.8,-0.2) {$\cdots$};
				\node at (2.8,-1.2) {$\cdots$};
				\node at (2.8,-2.2) {$\cdots$};
				\node at (2.8,-3.2) {$\cdots$};
				\node at (0,-3.7) {$\vdots$};
				\node at (1,-3.7) {$\vdots$};
				\node at (2,-3.7) {$\vdots$};
				\node at (0,0.2) {{\tiny${(1,2)}$}};
				\node at (1,0.2) {{\tiny${(1,4)}$}};
				\node at (1.5,-1.2) {{\tiny${(2,5)}$}};
				\node at (0.5,-2.85) {{\tiny${(3,3)}$}};
				\node at (-0.5,-3.7) {{\tiny${(4,1)}$}};
				\node at (2,-2.76) {{\tiny${(4,6)}$}};
				\begin{scriptsize}
					\node at (-0.3,-0.1) {{\tiny\color{blue}$\vartheta_{1,1}$}};
					\node at (-0.3,-1.1) {{\tiny\color{blue}$\vartheta_{1,2}$}};
					\node at (0.3,-0.1) {{\tiny\color{blue}$\vartheta_{1,1}$}};
					\node at (0.3,-1.1) {{\tiny\color{blue}$\vartheta_{1,2}$}};
					\node at (0.68,-0.1) {{\tiny\color{blue}$\vartheta_{2,1}$}};
					\node at (1.3,-0.1) {{\tiny\color{blue}$\vartheta_{2,1}$}};
				\end{scriptsize}
			\end{tikzpicture}
			\caption{}
		\end{subfigure}\qquad
		\begin{subfigure}[b]{0.45\textwidth}
			\centering
			\begin{tikzpicture}[line cap=round,line join=round,>=triangle 45,x=2cm,y=1cm]
				\draw[fill=gray!10,dashed] (-0.65,0.25)--(-0.65,-2.75)--(0.7,-2.75)--(0.7,-1.8)--(1.7,-1.8)--(1.7,-0.8)--(2.7,-0.8)--(2.7,0.25)--(-0.65,0.25);
				
				\foreach \x in {0,1}
				{
					\draw[line width=1pt,blue,{Latex[length=2mm]}-]  (\x,0) -- (\x-0.5,-0.5);
					\draw[line width=1pt,blue,{Latex[length=2mm]}-] (\x,0) -- (\x+0.5,-0.5);
					\draw[line width=1pt,black,{Latex[length=2mm]}-] (\x-0.5,-0.5) -- (\x,-1);
					\draw[line width=1pt,black,{Latex[length=2mm]}-] (\x+0.5,-0.5) -- (\x,-1);
					\draw[line width=1pt,blue,{Latex[length=2mm]}-]  (\x,-1) -- (\x-0.5,-1.5);
					\draw[line width=1pt,blue,{Latex[length=2mm]}-] (\x,-1) -- (\x+0.5,-1.5);
					\draw[line width=1pt,black,{Latex[length=2mm]}-] (\x-0.5,-1.5) -- (\x,-2);
					\draw[line width=1pt,black,{Latex[length=2mm]}-] (\x+0.5,-1.5) -- (\x,-2);
				}
				\draw[line width=1pt,blue,{Latex[length=2mm]}-]  (3,0) -- (3-0.5,-0.5);
				\draw[line width=1pt,blue,{Latex[length=2mm]}-]  (2,0) -- (2-0.5,-0.5);
				\draw[line width=1pt,blue,{Latex[length=2mm]}-]  (2,0) -- (3-0.5,-0.5);
				\draw[line width=1pt,blue,{Latex[length=2mm]}-]  (2,-1) -- (2-0.5,-1.5);
				\draw[line width=1pt,blue,{Latex[length=2mm]}-]  (1,-2) -- (1-0.5,-2.5);
				\draw[line width=1pt,blue,{Latex[length=2mm]}-]  (0,-2) -- (0-0.5,-2.5);
				\draw[line width=1pt,blue,{Latex[length=2mm]}-]  (0,-2) -- (1-0.5,-2.5);
				\draw[line width=1pt,blue,{Latex[length=2mm]}-]  (0,-3) -- (-0.5,-3.5);
				\draw[line width=1pt,black,{Latex[length=2mm]}-]  (-0.5,-2.5) -- (0,-3);
				\draw[line width=1pt,black,{Latex[length=2mm]}-]  (0.5,-2.5) -- (0,-3);
				\draw[line width=1pt,black,{Latex[length=2mm]}-]  (1.5,-0.5) -- (2,-1);
				\draw[line width=1pt,black,{Latex[length=2mm]}-]  (2.5,-0.5) -- (2,-1);
				\draw[line width=1pt,red,{Latex[length=2mm]}-] (-0.5,-1.5) -- (-0.5,-0.5);
				\draw[line width=1pt,red,{Latex[length=2mm]}-] (-0.5,-3.5) -- (-0.5,-2.5);
				\foreach \x in {0,1,2}
				{	
					\draw [fill=blue] (-0.5,-\x-0.5) circle (2pt);
					\draw [fill=blue] (0,-\x) circle (2pt);
					\draw [fill=blue] (0.5,-\x-0.5) circle (2pt);
				}
				\foreach \x in {0,1}
				{	
					\draw [fill=blue] (1,-\x) circle (2pt);
					\draw [fill=blue] (1.5,-\x-0.5) circle (2pt);
				}
				\draw [fill=blue] (2,0) circle (2pt);
				\draw [fill=blue] (2.5,-0.5) circle (2pt);
				\foreach \x in {0,1,2,3}
				{\draw[fill=white] (\x,-3+\x)  circle (2pt);}
				\draw[fill=white] (-0.5,-3.5)  circle (2pt);
				\node at (0,-3.7) {\phantom{$\vdots$}};
				\begin{scriptsize}
					\node at (-0.3,-0.1) {{\tiny\color{blue}$\theta_N$}};
					\node at (-0.3,-2.1) {{\tiny\color{blue}$\theta_N$}};
					\node at (-0.3,-1.1) {{\tiny\color{blue}$\theta_N$}};
					\node at (0.3,-0.1) {{\tiny\color{blue}$\theta_N$}};
					\node at (0.3,-1.1) {{\tiny\color{blue}$\theta_N$}};
					\node at (-0.3,-3.1) {{\tiny\color{blue}$\theta_N$}};
					\node at (0.3,-2.1) {{\tiny\color{blue}$\theta_N$}};
					\node at (0.68,-0.1) {{\tiny\color{blue}$\theta_{N-1}$}};
					\node at (1.3,-0.1) {{\tiny\color{blue}$\theta_{N-1}$}};
					\node at (1.3,-1.1) {{\tiny\color{blue}$\theta_{N-1}$}};
					\node at (0.68,-1.1) {{\tiny\color{blue}$\theta_{N-1}$}};
					\node at (0.68,-2.1) {{\tiny\color{blue}$\theta_{N-1}$}};
					\node at (1.68,-0.1) {{\tiny\color{blue}$\theta_{N-2}$}};
					\node at (2.3,-0.1) {{\tiny\color{blue}$\theta_{N-2}$}};
					\node at (1.68,-1.1) {{\tiny\color{blue}$\theta_{N-2}$}};
					\node at (2.68,-0.1) {{\tiny\color{blue}$\theta_{N-3}$}};
				\end{scriptsize}
			\end{tikzpicture}
			\caption{}
		\end{subfigure}
		
		\caption{(A)  The directed, colored (and labeled) graph $G$ associated to half-space log-gamma $\Theta$-Gibbs measures. A few of the vertices of $G$ have their $\phi$-induced index (i.e., the coordinates above the vertex), and a few of the blue edges are labeled by the $\vartheta_{i,j}$ parameters. A generic bounded connected domain $\Lambda$ is shown in the figure which contains all vertices in the shaded region. $\partial\Lambda$ consists of white vertices in the figure. (B) The domain $K_N$ considered in Theorem \ref{thm:conn}. $\Lambda_N^*$ consists of vertices in the shaded region. The assignment $\vartheta_{i,j}=\theta_{N-i+1}$ of $\Theta$ parameters from Theorem \ref{thm:conn} as shown here over the blue edges.}
		\label{fig00}
	\end{figure}
	
	We next define a bijection $\phi:V(G)\to \Z_{\ge 1}^2$ by $\phi((m,n))=(-\lfloor n \rfloor, m)$. This pushes the directed/colored edges in $G$ onto directed/colored edges on $\Z_{\ge 1}^2$ which we denote  by $E(\Z_{\ge 1}^2)$. We will always view $G$ as in Figure \ref{fig00} and will use the $\phi$-induced indexing when describing this graph. As in Figure \ref{fig00} (B), set $\Lambda_N^*:=\{(i,j)\in \Z_{\ge 1}^2 : i\in [1,N-1], j\in [1,2N-2i+1]\}$.

We associate to each $e\in E(\Z^2)$ a weight function based on the color of edge defined as follows: 
	\begin{align}\label{def:wfn}
		W_{e}(x):=\begin{cases}
			\exp(\vartheta x-e^x) & \mbox{ if $e$ is \blue \color{blue}{$(\vartheta)$}}, \\
			\exp(-e^x) & \mbox{ if $e$ is \black}, \\
			\exp(-\alpha x) & \mbox{ if $e$ is \red,}
		\end{cases}
	\end{align}
	
	\begin{definition}[Half-space log-gamma $\Theta$-Gibbs measure] \label{def:hslggibbs}
		Fix any $\Theta:=\{\vartheta_{m,n}>0 \mid (m,n)\in \Z_{\ge 1}^2\}$. Consider the graph  $\Z_{\ge 1}^2$ endowed with directed/colored edges $E(\Z_{\ge 1}^2)$ as above. Let $\Lambda$ be a bounded connected subset of $\Z_{\ge 1}^2$.  Set
		$$\partial \Lambda:=\big\{v\in \Z_{\ge 1}^2\cap \Lambda^c :\{v'\to v\}\in E(\Z_{\ge 1}^2) \mbox{ or }\{v\to v'\}\in E(\Z_{\ge 1}^2),\mbox{ for some } v'\in \Lambda\big\}.$$
		The half-space log-gamma ($\hslg$) $\Theta$-Gibbs measure  for the domain $\Lambda$, with boundary condition $\big(u_{i,j} \in \R: (i,j)\in \partial\Lambda\big)$, is a measure on $\R^{|\Lambda|}$ with density at $(u_{i,j})_{(i,j)\in \Lambda}$ proportional to
		\begin{align}\label{e:hsgibb}
			\prod_{e=\{v_1\to v_2\}\in E(\Lambda\cup \partial\Lambda)} W_{e}(u_{v_1}-u_{v_2}).
		\end{align}
Lemma \ref{b2} shows that the $\hslg$ $\Theta$-Gibbs measure is well-defined. When all $\vartheta_{m,n}$ are equal to a generic parameter $\theta>0$, we shall simply call the corresponding measure as $\hslg$ Gibbs measure.

Notationally, we will generally use $u_{v}$ for vertices $v=(i,j)\in \Lambda$ as dummy-variables when discussing the density of $\hslg$ Gibbs measures. When discussing multivariate random variables distributed jointly according to a $\hslg$ Gibbs measure we will typically write $L_i(j)$, or sometimes $L(v)$ for $v=(i,j)$, for the $(i,j)$ coordinate of these multivariate random variables.

An event $A$, i.e., elements of the Borel $\sigma$-algebra for $\R^{|\Lambda|}$, is {\it increasing} if it satisfies the condition that $\mathbf{u}'\in A$ implies $\mathbf{u}\in A$ provided $\mathbf{u}\preceq \mathbf{u}'$. Here $\mathbf{u}=\big(u_{i,j}\big)_{(i,j)\in \Lambda}$, $\mathbf{u}'=\big(u'_{i,j}\big)_{(i,j)\in \Lambda}$ and $\mathbf{u}\preceq \mathbf{u}'$ if $u_{i,j}\leq u_{i,j}'$ for all $(i,j)\in \Lambda$. An event is {\it decreasing} if $\mathbf{u}'\in A$ implies $\mathbf{u}\in A$ provided $\mathbf{u}'\preceq \mathbf{u}$
\end{definition}

The following  shows how the $\hslg$ free energy process can be embedded in a $\hslg$ $\Theta$-Gibbs measure. Its proof in Section \ref{sec2.2} relies on results of \cite{bw}  that build on the analysis of the log-gamma polymer via the geometric RSK correspondence \cite{cosz} on symmetrized domains \cite{osz, nz, bz19b}. 
In Section \ref{sec2.2}, for each $N>0$, we will define explicitly such a choice for $\big(\L_i^N(j): (i,j)\in\mathcal{K}_N\big)$ that will satisfy the two criterion of the theorem. We will call this the half-space log-gamma ($\hslg$) line ensemble. We will  use $\L$ when discussing the $\hslg$ line ensemble, while $L$ will be used when discussing general line ensembles that enjoy the $\hslg$ $\Theta$-Gibbs property.
	
	\begin{theorem}[Half-space log-gamma line ensemble]\label{thm:conn} Let $\vec\theta=(\theta_i)_{i\in \mathbb{Z}_{\ge 1}}$ be a sequence of positive parameters. Fix $\alpha>-\theta$ where $\theta:=\min\{\theta_i :i\in \mathbb{Z}_{\ge 1}\}$, and $N\in \Z_{\ge 1}$.  As in Figure \ref{fig00} (B), set $\mathcal{K}_N:=\{(i,j)\in \Z_{\ge 1}^2 : i\in [1,N], j\in [1,2N-2i+2]\}$. There exists random variables $\big(\L_i^N(j) : (i,j)\in \mathcal{K}_N\big)$, called here the $\hslg$ $\Theta$-line ensemble, on a common probability space such:
		\begin{enumerate}[label=(\roman*), leftmargin=20pt]
			\item \label{i01} We have the following equality in distribution
			\begin{align}\label{impi}
				(\L_1^N(2j+1))_{j\in \ll0,N-1\rr} \stackrel{(d)}{=} \big(\log Z_{(\alpha,\vec\theta)}(N+j,N-j)+2N\Psi(\theta)\big)_{j\in \ll0,N-1\rr}.
			\end{align}
			\item \label{i02} 
 The law of $\big(\L_i^N(j) : (i,j)\in \Lambda_N^*\big)$ conditioned on $\big(\L_i^N(j): (i,j)\in (\Lambda_N^*)^c\big)$ is given by the $\hslg$ $\Theta$-Gibbs measure for the domain $\Lambda_N^*$ with boundary condition $\big(\L_i^N(j) : (i,j)\in \partial\Lambda_N^*\big).$ Here the parameters in $\Theta$ are chosen as $\vartheta_{i,j}:=\theta_{N-i+1}$, see the blue edge labeling in Figure \ref{fig00} (B).
		\end{enumerate}
In the homogeneous case we set all $\theta_i\equiv \theta$.
	\end{theorem}

\begin{remark}
Theorem \ref{thm:conn} is stated for the polymer model using the inhomogeneous weights in \eqref{eq:wz}. In the homogeneous case (which we will focus upon here) where $\theta_{i}\equiv\theta$ (and hence $\vartheta_{i,j}\equiv\theta$) the $2N\Psi(\theta)$ centering term in \eqref{impi} is chosen to be adapted to our ultimate goal of taking scaling limits. However, this terms inclusion is ultimately inconsequential since it constitutes a constant shift of the Gibbs measures which does not impact the Gibbs property (see Lemma \ref{obs1}\ref{traninv}).
\end{remark}

We assume below that we are dealing with the homogeneous case of the $\hslg$ line ensemble.

It is useful to view  $\hslg$ Gibbs measures (in particular we focus here on the Gibbs measures from Theorem \ref{thm:conn}) in terms of the language of Gibbsian line ensembles. Consider $k$ and $T$ fixed and $N$ sufficiently large so that all of the random variables $\L_1^N\ll 1,T\rr, \L_2^N\ll 1,T\rr, \ldots, \L_{2k}^N\ll 1,T\rr$ are defined. We will think of $\L_i^N$ as the label $i$ `line' (rather, a piecewise linearly interpolated curve) in the ensemble. The values of $\big(\L_i^N(2T+1) : i\in \ll1,2k\rr\big)$ and $\L_{2k+1}^N(\cdot)$ constitute boundary data which, once known, uniquely identify (via the Gibbs description) the laws of $\L_1^N\ll 1,T\rr, \L_2^N\ll 1,T\rr, \ldots, \L_{2k}^N\ll 1,T\rr$.

Let us consider the three types of weights in the Gibbs measure. The weights corresponding to \textbf{black} edges $v_1 \to v_2$ contribute a factor of $e^{-e^{u_{v_1}-u_{v_2}}}$ (here $u_v$ is the dummy variable in the Gibbs density corresponding to a vertex $v$) in the Gibbs measure. Whenever $u_{v_1} \gg u_{v_2}$, this weight is very close to 0, whereas when $u_{v_1} \yt u_{v_2}$ the weight is close to 1 (between, there is a smooth monotone transition from $0$ to $1$). Thus, this weight produces a soft version of conditioning on the event that $\L^N(v_2) \ge \L^N(v_1)$ (recall the notational convention for a line ensemble that $L(v)=L_i(j)$ where $v=(i,j)$). \textbf{Black} edges arise between consecutive lines thus we expect that our measure will strongly favor configurations where $\L_1^N(\cdot) \gtrsim \L_2^N(\cdot) \gtrsim \L_3^N(\cdot) \gtrsim \cdots$, i.e., the curves are non-intersecting up to some error (Theorem \ref{t:order} provides a precise statement substantiating this). Of course, the soft nature of this conditioning will not rule out crossing, but a heavy penalty will be incurred so at a heuristic level it is useful to think in terms of non-intersecting lines.
		
The {\color{red} \textbf{red}} edges are $(2i-1,1)\to (2i,1)$ and come with a weight $e^{-\alpha (u_{2i-1,1}-u_{2i,1})}$. This weight is close to 0 when $u_{2i-1,1} \gg u_{2i,1}$ (since $\alpha>0$). This creates an attractive force between $\L_{2i-1}^N(1)$ and $\L_{2i}^N(1)$ which tries to establish the ordering $\L_{2i-1}^N(1)\leq \L_{2i}^N(1)$. Of course, this is in opposition to the soft non-intersecting influence already discussed. Combined, these forces ultimately (through our analysis of weighted paired random walks) result in the difference $\L_{2i-1}^N(1)-\L_{2i}^{N}(1)=O(1)$ as $N\to \infty$. In contrast, in the critical regime, when $\alpha_N=N^{-1/3}\mu$, the attraction weakens with $N$ and the forces result in $\L_{2i-1}^N(1)- \L_{2i}^{N}(1)=O(N^{1/3})$. It is the $O(1)$ distance between $\L_{2i-1}^N(1)$ and $\L_{2i}^{N}(1)$ that makes the supercritical case harder than the critical case.
		
Finally, consider the {\color{blue} \textbf{blue}} edges that encode the Gibbs weights between consecutive values of a given line, i.e. between $\L_{i}^N(j)$ and $\L_{i}^N(j+1)$. Alone, these weights define log-gamma increment random walks (with two-step periodicity in the law of the increments). Thus, putting these three factors together one arrives at the picture illustrated in Figure \ref{figO1} (A) -- an ensemble of softly non-intersecting log-gamma random walks with starting points $O(1)$ distance apart between the curves labeled $2i-1$ and $2i$ for each relevant $i$.
In order to prove Theorem \ref{t:main0} we essentially need to justify the  distance scales in Figure \ref{figO1} (A). To do that, we use the Gibbs property for the line ensemble described above along with some one-point control over $\L_1^N$ that we describe now.
	%
%
%
%
%
	
	\subsubsection{Point-to-line free energy fluctuations} \label{sec:132}
The $\hslg$ Gibbs measures machinery gives us access to the behavior of the $\hslg$ line ensemble conditioned on the boundary data. However, we still need to understand the behavior of the boundary data. The theory of (full-space) Gibbsian line ensembles that has been developed over the last decade has become proficient at taking very minimal seed information, such as the scale in which tightness occurs for the one-point fluctuations of the top curve of a Gibbsian line ensemble, and outputting the scaling and tightness for the entire edge of the line ensemble. We take the first step in developing such a half-space theory.

There are currently only two fluctuation results about the $\hslg$ polymer. The first (chronologically) is a result of \cite{bw} that we will recall below and appeal to, while the second is the work of \cite{ims22} that proves a limit theorem for $N^{-1/3} \L_1^N(1)$ (i.e. $\mathcal{F}_N^{\alpha}(0)$). Our work began prior to the release of \cite{ims22} and thus we rely only on the work of \cite{bw}. The control \cite{ims22} provides is for $\L_1^N(1)$ only and since we need some information away from the boundary too, most of the work herein is unavoidable and not significantly simplified by using \cite{ims22}. It is natural to wonder if \cite{ims22} could have been used alone, in place of \cite{bw}, at the seed for our analysis. While we do not rule this out, it would certainly require a very different type of argument since we rely heavily on the fact that \cite{bw} provides some information about $\L_1^N(j)$ as $j$ varies.

We recall the result of \cite{bw}.  For each $k>0$, define the point-to-(partial)line partition function
	\begin{align}\label{rfl}
		\zl (k):=\sum_{j=\lceil k \rceil}^{N} Z_{(\alpha,\theta)}(N+j,N-j).
	\end{align}
This sum is restricted to endpoints at least distance $2k$ from the boundary.
Set $p=\frac{N+k}{N-k}$, Let $\theta_c$ be the unique solution to $\Psi'(\theta_c)=p\Psi'(2\theta-\theta_c)$ and set (recall the digamma function $\Psi$ from \eqref{psidef})
	\begin{align*}
		f_{\theta,p}:=-\Psi(\theta_c)-p\Psi(2\theta-\theta_c), \quad \sigma_{\theta,p}:=\bigg(\tfrac12(-\Psi''(\theta_c)-p\psi''(2\theta-\theta_c)\bigg)^{1/3}.
	\end{align*}
	\begin{theorem}[Theorem 1.10 in \cite{bw}]\label{thm:bw} Suppose $(k_N)_{N\in \mathbb{Z}_{>0}}$ is such that for some $y\in \R\cup \{\infty\}$,  $\lim_{N\to\infty} (N-k_N)^{1/3}\sigma_{\theta,p}(\alpha+\theta-\theta_c)=y$. Then, as $N\to \infty$
		\begin{align*}
			\frac{\log \zl(k_N)-(N-k_N)f_{\theta,p}}{(N-k_N)^{1/3}\sigma_{\theta,p}} \stackrel{(d)}{\Longrightarrow} U_{-y}.
		\end{align*}
		where for $y\in \R$, $U_{-y}$ is distributed as the Baik-Ben Arous-P\'ech\'e distribution with parameter $y$ (see Eq.~(5.2) in \cite{bw}). When $y=\infty$, $U_{-\infty}$ is distributed as the GUE Tracy-Widom distribution.
	\end{theorem}
	
	The crucial deduction from Theorem \ref{thm:bw} is that there exists $\nu>0$ such that for each $M>0$,
	\begin{align}\label{pardec}
		V_N(M)+M^2 \xRightarrow[N\to\infty]{(d)}  X_{M}, \quad \mbox{where}\quad V_N(M):=\frac{\log \zl(MN^{2/3})+2\Psi(\theta)N}{N^{1/3}\nu}.
	\end{align}
	Here the BBP distributions of the limiting random variables $(X_M)_{M>0}$ form a tight sequence in $M$, in particular they converge in law to the GUE Tracy-Widom distribution as $M\to\infty$. A precise version of this deduction in given later in Lemma \ref{l:ut}.
 Essentially, the rescaled point-to-(partial)line free energy process $V_N(M)$ looks like an inverted parabola $-M^2$ with tight fluctuations around it.

\subsubsection{Using the Gibbs line ensemble structure to prove Theorem \ref{t:main0}}\label{sec:133}

We now give a brief overview of the steps of our proof and how it relies on combining the seed information from \cite{bw}, i.e. \eqref{pardec}, and the $\hslg$ line ensemble Gibbs property. Fixing a bit of notation, we will say that a sequence of random variables $X_N$ is \emph{upper-tight} if $\max(X_N,0)$ is tight, and \emph{lower-tight} if $\min(X_N,0)$ is tight. Recall that $X_N$ is tight if for all $\e>0$ there exists $K=K(\e)>0$ such that $\Pr(|X_N|\geq K)<\e$ for all $N\geq N_0$. If $X_N$ is both upper and lower tight, then it is tight.

We sketch the proof of the main theorem for $r=1$ ($r$ is as in the statement of the theorem). Fix any $N_0$ large enough so everything below is well-defined for $N\geq N_0$. We consider a time
\begin{equation}\label{eq:TT}
T=8\lfloor N^{2/3}\rfloor
\end{equation}
(the key point is that time window $[0,T]$ scales like $N^{2/3}$).
By virtue of the relation \eqref{impi} in Theorem \ref{thm:conn} \ref{i01}, to prove our main theorem for $r=1$ it suffices to establish tightness of the top curve of the $\hslg$ line ensemble $\L^N$ (after appropriate scaling) on the time window $[0,T/4]$.
The broad steps used in establishing our main theorem can be summarized as \ref{it1}-\ref{it3} below.
	\begin{enumerate}[label=(\roman*),leftmargin=20pt]
		\setlength\itemsep{1 em}
		\item\label{it1} Given any $M_1>0$, prove that there exists $M_2>M_1$ such that $N^{-\frac13}\L_1^N(2p^*-1)$ and $N^{-\frac13}\L_2^N(2p^*)$ are tight for some random $p^*\in [M_1N^{\frac23}, M_2N^{\frac23}]$. 
	\end{enumerate}
Here and below we consider staggered (i.e., even and odd) arguments for $\L_1^N$ and $\L_2^N$ (and $L_1$ and $L_2$) due to the diagonal Gibbs interaction. This is a technical point which can be ignored currently.

Owing to the Gibbs property, Theorem \ref{thm:conn} \ref{i02}, enjoyed by the line ensemble $\L^N$, the joint law of $\L_1^N(\ll 1, 2p^*-1\rr)$ and $\L_2^N(\ll 1 , 2p^*\rr)$ given the knowledge of $\L_1^N(2p^*-1),\L_2^N(2p^*)$ and $\L_3^N(\ll 1, 2p^*\rr)$ (where $p^*$ comes from \ref{it1}) is that of a two-curve $\hslg$ Gibbs line ensemble $(L_1,L_2)$ with a bottom boundary data given by $L_3=
\L^N_3$ and right-boundary data determined by $L_1(2p^*-1)=\L_1^N(2p^*-1)$ and $L_2^N(2p^*)=\L_2^N(2p^*)$. The point of this reduction is that we can now make use of a tool known as {\it stochastic monotonicity} (see Proposition \ref{p:gmc} and discussion later in the introduction). This implies that if we instead condition on lower boundary data (i.e., lower $L_3$, $L_1^N(2p^*-1)$ or $L_2^N(2p^*)$), the resulting measure is stochastically dominated by the original measure, i.e. the law of $(\L^N_1,\L^N_2)$.

By using stochastic monotonicity we see that conditioned on the values of $\L_1^N(2p^*-1),\L_2^N(2p^*)$, it is possible to couple on the same probability space $\big(\L_1^N(\ll 1, 2p^*-1\rr), \L_2^N(\ll 1 , 2p^*\rr)\big)$ along with $(L_1,L_2)$ distributed according to a {\it bottom-free} $\hslg$ Gibbs measure specified by $L_3\equiv -\infty$, $L_1(2p^*-1)=\L_1^N(2p^*-1)$ and $L_2^N(2p^*)=\L_2^N(2p^*)$ (see also \eqref{law} below, or Definition \ref{def:btf} for a precise definition) in such a way that $\L^N_1(\ll 1, 2p^*-1\rr)\geq L_1(\ll 1, 2p^*-1\rr)$ and $\L^N_2(\ll 1, 2p^*\rr)\geq L_2(\ll 1, 2p^*\rr)$ point-wise. In particular, this means that any increasing event (recall from Definition \ref{def:hslggibbs}) will have a larger probability under the bottom-free measure than under the original measure. This is an important tool in establishing lower-tightness as well as control over the modulus of continuity.

For the below three items we will assume that $(L_1,L_2)$ is specified in this bottom-free manner and that the right-boundary data given by $\L_1^N(2p^*-1)$ and $\L_2^N(2p^*)$ satisfies \ref{it1} above.

\begin{enumerate}[label=(\roman*),leftmargin=20pt]
		 \myitem[(ii)] \label{it2} Prove  that $N^{-\frac13}L_1(1)$ and $N^{-\frac13}L_2(2)$ are lower-tight.
		
		 \myitem[(iii)] \label{it2.5} Prove that for any $M^*>0$, with  positive probability (depending on $M^*$ and $r$ but not of $N$)
		$$L_1(p)\ge M^*N^{1/3}, \quad \textrm{and} \quad L_2(p) \ge M^*N^{1/3}, \quad \mbox{for all }p \in \ll1,T\rr.$$
		 \myitem[(iv)] \label{it3} Prove process-level tightness of $(N^{-1/3}L_1(xN^{2/3}))_{x\in [0,2]}.$
	\end{enumerate}
	
We shall describe how we establish the above broad steps in a moment. Let us first conclude how the above steps work together to yield our main theorem, Theorem \ref{t:main0}.
	
We first argue that $N^{-\frac13}\L_1^N(1)$ and $N^{-\frac13}\L_2^N(2)$ are tight. 
Indeed, since the point-to-line free energy is an upper bound for the point-to-point free energy process, utilizing \eqref{pardec} it follows immediately that $N^{-\frac13}\L_1^N(1)$ and $N^{-\frac13}\L_2^N(2)$ are upper-tight. To show that they are also lower-tight we utilize the above mentioned stochastic monotonicity of the $\hslg$ Gibbs line ensembles (Proposition \ref{p:gmc}) and instead show lower-tightness for the two-curve bottom-free line ensemble $(L_1,L_2)$ (i.e., under the condition $L_3\equiv\infty$), which is what we established in item \ref{it2}.
	
The next step to proving Theorem \ref{t:main0} is to argue that with strictly positive probability (i.e., not going to zero with $N\to\infty$) 
there is a uniform separation of length $cN^{1/3}$ (for sufficient small $c$) between the first two curves $\L_1^N$ and $\L_2^N$ and the third curve $\L_3^N$. The argument to show this (Proposition \ref{l:rpass} in the text) proceeds as follows. Once we have tightness at the left boundary, it is straight-forward  to show that $N^{-\frac13}\L_1^N(2v-1)$ and $N^{-\frac13}\L_2^N(2v)$ are tight for any choice of $v\in \ll1,p^*\rr$. Combining this with the soft non-intersection property of the line ensembles and \ref{it2}, we deduce in Theorem \ref{l:lhigh} that $\sup_{p\in \ll1,2T\rr} N^{-\frac13}\L_3^N(p)$ is upper tight. The result in \ref{it2.5} shows that the bottom-free line ensemble $(L_1,L_2$) has a strictly positive probability of being uniformly high on $[0,T]$ and thus by stochastic monotonicity so too does $(\L^N_1,\L^N_2)$. Together with upper-tightness of $\sup_{p\in \ll1,2T\rr} N^{-\frac13}\L_3^N(p)$, this shows that the probability that $(\L^N_1,\L^N_2)$ stay separate from $\L^N_3$ stays bounded from 0 as $N\to \infty$.
	
Finally, we prove the process-level tightness of the top curve of our ensemble.  Size biasing plays a key role in this deduction (see around equation \eqref{eq:clm2}). Indeed, once we know that there is a positive probability of uniform separation (as deduced above), we can use the fact that the Radon-Nikodym derivative defining our Gibbs measures highly penalize configurations where the top two curves are close to the third curve. Thus, the positive probability event of separation becomes a high probability event. Finally, we are able to establish process-level tightness (i.e., control on the modulus of continuity) by leveraging the separation and the process-level tightness of the first two curves with the third curve moved to $-\infty$ that was shown in item \ref{it3}. This establishes tightness of the first curve which, through identification with the free energy process, yields Theorem \ref{thm:conn}.
	
\begin{remark} \label{ims22rk}
The result of \cite{ims22} 
immediately implies the tightness of $N^{-1/3}\L_1^N(1)$.  However, to carry our proof outlined above we need tightness of $\L_2^N(1)$, and other fine information about $\L_1^N$ and $\L_2^N$ away from the boundary, as described in item \ref{it2} and \ref{it3}, which to our best understanding is beyond the scope of \cite{ims22}. 
\end{remark}
	
We return to steps \ref{it1}-\ref{it3} stated above and describe the main ideas in achieving them.

\smallskip

	\noindent{\it\textbf{Proof idea for \ref{it1}}:}
We start by proving (Theorem \ref{t:order}) that the curves $\L_i^N$ are typically non-intersecting (or at least do not overlap by much). Combining this with the fact that the point-to-line partition function (controlled in \cite{bw}) dominates the point-to-point partition function for any point along the line, it follows that  $\sup_{i,j} N^{-1/3}\L_i^N(j)$ is upper-tight. Lower-tightness is trickier.

From the parabolic decay of the point-to-(partial)line free energy \eqref{pardec}, we deduce that the point-to-point free energy process has to be in the $N^{1/3}$ fluctuation scale at some random $p_1^*$ in a $O(N^{2/3})$ window. We essentially (see Proposition \ref{p:high}) show that for $M_0$ large enough
		\begin{align}\label{partra}
			\sup_{p\in \ll \SSS N^{2/3},(M_0+2\SSS)N^{2/3}\rr} \frac{\L_1^N(2p+1)}{N^{1/3}\nu}+\SSS^2
		\end{align}
is tight as $N\to \infty$, uniformly over all $\SSS>0$. The parameter $\nu$ is an explicit function of $\theta$, see \eqref{nu}.
The crucial point here is the uniformity, i.e., the $K(\e)$ in the definition of tightness can be chosen independent of $\SSS>0$. Thus, in $N^{1/3}$ and $N^{2/3}$ scaling $\L_1^N$ follows an inverted parabola.
	
We next essentially show (see Proposition \ref{p:high2}) that there exists $M_1$ and $M_2$ large enough so that
$$\sup_{p_2 \in \ll M_1N^{2/3}, M_2N^{2/3}\rr} N^{-1/3}\L_2^N(2p_2)$$
is tight. The idea is if $\L_2^N$ is uniformly low in $[M_1N^{2/3}, M_2N^{2/3}]$, then, due to the Gibbs property of the line ensemble, the first curve $\L_1^N$ behaves like a random bridge, i.e., linearly,  in that interval. However, as we show in proving Proposition \ref{p:high2}, this violates the inverted parabolic trajectory \eqref{partra} for some $\SSS$ thus leaving us with a random $p_2^*\in [M_1N^{2/3},M_2N^{2/3}]$ so that $N^{-1/3}\L_2^N(2p_2^*)$ is tight. Owing to typical non-intersection (Proposition \ref{t:order}) we have that $N^{-1/3}\L_1^N(2p_2^*-1)$ is tight.

\medskip

\noindent{\it$\m{WPRW}$ machinery:}	
The remaining proofs of \ref{it2}, \ref{it2.5} and \ref{it3} rely heavily on understanding the  
\begin{equation}
	\label{law}
	\begin{aligned}
		\mbox{$\hslg$ Gibbs measure on } (L_1, L_2) \mbox{ with } L_1(2n-1)=x_n,\,\, L_2(2n)=y_n,\,\, \textrm{ and }L_3\equiv -\infty,
	\end{aligned}
\end{equation}
for $L_1$ with domain $\ll 1,2n-1\rr$, $L_2$ with domain $\ll 1,2n\rr$, $n$ of order $N^{2/3}$ and $x_n,y_n$ of order $N^{1/3}$ (i.e. order $\sqrt{n}$). We referred above to this as the bottom-free measure.

Set $M_1=16$ (so that $M_1N^{2/3} \ge 2T$, where $T$ is defined in \eqref{eq:TT}) in \ref{it1} and determine a random point $p^*$ from the same item. Essentially, we want to take $n$ in \eqref{law} to be this $p^*$. However, one caveat in taking $p^*$ as a choice for $n$ is that it is random. So, instead we analyze \eqref{law} for all fixed $n\in [M_1N^{2/3},M_2N^{2/3}]$. We shall show \ref{it2}, \ref{it2.51}, and \ref{it3} under the law in \eqref{law} with estimates uniform over all possible choices of $n\in [M_1N^{2/3},M_2N^{2/3}]$. Here \ref{it2.51} is given by
	\begin{enumerate}[leftmargin=25pt]
	\setlength\itemsep{1 em}
 \myitem[(iii\mprime)]\label{it2.51} Prove that any $M^*>0$, with strictly positive probability (depending only on $M^*$ and $M_1,M_2$)
$$\mbox{under \eqref{law}} \quad L_1(p)\ge M^*N^{1/3}, \quad \textrm{and} \quad L_2(p) \ge M^*N^{1/3} \quad \mbox{for all }p \in \ll1,n-1\rr.$$
\end{enumerate}
Note that as $n\ge 2T$, \ref{it2.51} implies \ref{it2.5}.

\smallskip

The law in \eqref{law} is closely related (see \eqref{e.reduction}) to the \textit{weighted paired random walk} ($\m{WPRW}$) law.

	\begin{definition}[Paired Random Walk and Weighted Paired Random Walk] \label{prb}
Let $\Omega^2_n=\R^{n}\times \R^{n}$ and $\mathcal{F}^2_n$ be the Borel $\sigma$-algebra associated to $\Omega^2_n$. Write $\omega\in \Omega^2_n$ as $\omega=(\omega_1(1),\ldots,\omega_1(N),\omega_2(1),\ldots,\omega_2(N))$. (For later purposes write $\Omega^1_n=\R^n$, let $\mathcal{F}^1_n$  be its Borel $\sigma$-algebra and write $\omega=(\omega_1,\ldots, \omega_n)$ for $\omega\in \Omega^1_n$.)
Let $\fa(x)$ denote the density at $x\in \R$ of $\log Y_1-\log Y_2$ were $Y_1,Y_2$ are independent $\operatorname{Gamma}(\theta)$ random variables and $\ga(x)=\Gamma(\alpha)^{-1}e^{\alpha x-e^x}$ (see also \eqref{def:faga} below). For $(x,y)\in \R^2$ and $n\in \mathbb{Z}_{\geq 2}$
The \textit{paired random walk} ($\m{PRW}$) law on $(\Omega^2_n,\mathcal{F}^2_n)$ is the probability measure $\Pr_{\operatorname{PRW}}^{n;(x,y)}$ proportional to the product of two Dirac delta functions $\delta_{\omega_1(n)=x}\delta_{\omega_2(n)=y}$ and a density (against Lebesgue on $\R^{2(n-1)}$) is given by
\begin{equation}
	\label{den}
	\begin{aligned} \ga\big(\omega_2(1)-\omega_1(1)\big)\prod_{k=2}^{n}\fa\big(\omega_1(k)-\omega_1(k-1)\big)\fa\big(\omega_2(k)-\omega_2(k-1)\big)\,d\omega_1(k)\, d\omega_2(k).
	\end{aligned}
\end{equation}
As a slight abuse of notation we will say that the coordinate functions (i.e., random variables)
$$\kis(\omega):=\omega_i(k), \quad k\in \ll1,n\rr, i\in \{1,2\}$$
under this measure $\Pr_{\operatorname{PRW}}^{n;(a,b)}$ are \textit{paired random walks}.
See Figure \ref{figr} for an illustration of the $\m{PRW}$.

The weighted paired random walk ($\m{WPRW}$) law $\Pr_{\operatorname{WPRW}}^{n;(x,y)}$ on $(\Omega^2_n,\mathcal{F}^2_n)$ is absolutely continuous with respect to $\Pr_{\operatorname{PRW}}^{n;(x,y)}$ and defined through a Radon-Nikodym derivative so that for all  $\m{A}\in \mathcal{F}_n$,
\begin{align}
	\label{wscint} \Pr_{\operatorname{WPRW}}^{n;(x,y)}(\m{A})=\frac{\Ex_{\operatorname{PRW}}^{n;(x,y)}[\wsc\ind_{\m{A}}]}{\Ex_{\operatorname{PRW}}^{n;(x,y)}[\wsc]},
\end{align}
where $\wsc=\wsc(\omega)$ is given by
\begin{align}\label{defw}
	\wsc:=\exp\bigg(-e^{\iise-\se{1}{2}}-\sum_{k=2}^{n-1} \left(e^{\iiks-\se{1}{k+1}}+e^{\iiks-\iks}\right)\bigg).
\end{align}
The `sc' here refers to `super-critical' as we will use a different representation of the bottom-free law in the critical case. For the purpose of this introduction we will just write $W$ in place of $\wsc$ below.
The $\m{WPRW}$ law can be seen as a `soft' version of the law that would result from conditioning on non-crossing. Crossing is now allowed but subject to substantial energetic penalization.
\end{definition}

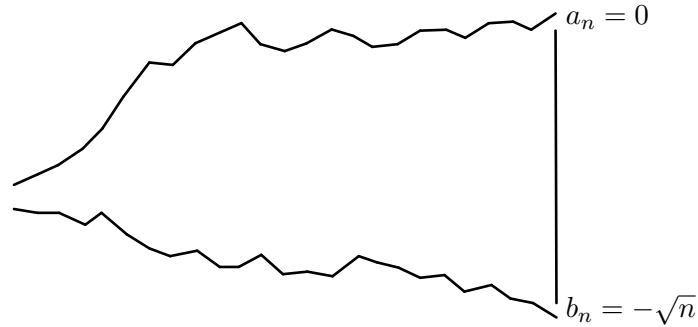
\begin{figure}[h!]
	\centering
		\begin{tikzpicture}[line cap=round,line join=round,>=triangle 45,x=0.6cm,y=0.6cm]
			\draw [line width=1pt] (8,7.2)-- (8.54,7.44);
			\draw [line width=1pt] (8.54,7.44)-- (8.98,7.64);
			\draw [line width=1pt] (8.98,7.64)-- (9.52,8);
			\draw [line width=1pt] (9.52,8)-- (9.96,8.45);
			\draw [line width=1pt] (9.96,8.45)-- (10.42,9.15);
			\draw [line width=1pt] (10.42,9.15)-- (11,9.91333333333334);
			\draw [line width=1pt] (11,9.91333333333334)-- (11.52,9.86);
			\draw [line width=1pt] (11.52,9.86)-- (12.02,10.34);
			\draw [line width=1pt] (12.02,10.34)-- (12.58,10.585714285714293);
			\draw [line width=1pt] (12.58,10.585714285714293)-- (13.04,10.78666666666674);
			\draw [line width=1pt] (13.04,10.786666666666674)-- (13.46,10.32);
			\draw [line width=1pt] (13.46,10.32)-- (14,10.166666666666673);
			\draw [line width=1pt] (14,10.166666666666673)-- (14.48,10.336);
			\draw [line width=1pt] (8,6.665)-- (8.52,6.58);
			\draw [line width=1pt] (8.52,6.58)-- (9,6.58);
			\draw [line width=1pt] (9,6.58)-- (9.58,6.316);
			\draw [line width=1pt] (9.58,6.316)-- (9.94,6.58);
			\draw [line width=1pt] (9.94,6.58)-- (10.5,6.103333333333338);
			\draw [line width=1pt] (10.5,6.103333333333338)-- (11,5.791428571428576);
			\draw [line width=1pt] (11,5.791428571428576)-- (11.46,5.62);
			\draw [line width=1pt] (11.46,5.62)-- (12.06,5.74);
			\draw [line width=1pt] (12.06,5.74)-- (12.56,5.38);
			\draw [line width=1pt] (12.56,5.38)-- (12.98,5.38);
			\draw [line width=1pt] (12.98,5.38)-- (13.48,5.648);
			\draw [line width=1pt] (13.48,5.648)-- (13.96,5.22);
			\draw [line width=1pt] (13.96,5.22)-- (14.5,5.276);
			\draw [line width=1pt] (14.5,5.276)-- (15.06,5.17428571428572);
			\draw [line width=1pt] (15.06,5.17428571428572)-- (15.64,5.62);
			\draw [line width=1pt] (15.64,5.62)-- (16.04,5.4828571428571475);
			\draw [line width=1pt] (16.04,5.4828571428571475)-- (16.52,5.365);
			\draw [line width=1pt] (16.52,5.365)-- (17,5.14);
			\draw [line width=1pt] (17,5.14)-- (17.54,5.2);
			\draw [line width=1pt] (17.54,5.2)-- (17.98,4.833333333333337);
			\draw [line width=1pt] (17.98,4.833333333333337)-- (18.58,4.98);
			\draw [line width=1pt] (18.58,4.98)-- (19,4.68);
			\draw [line width=1pt] (19,4.68)-- (19.5,4.58);
			\draw [line width=1pt] (19.5,4.58)-- (20.02,4.26);
			\draw [line width=1pt] (14.48,10.336)-- (15.04,10.646666666666674);
			\draw [line width=1pt] (15.04,10.646666666666674)-- (15.52,10.50333333333334);
			\draw [line width=1pt] (15.52,10.50333333333334)-- (15.94,10.26);
			\draw [line width=1pt] (15.94,10.26)-- (16.5,10.32);
			\draw [line width=1pt] (16.5,10.32)-- (17,10.62);
			\draw [line width=1pt] (17,10.62)-- (17.58,10.64);
			\draw [line width=1pt] (17.58,10.64)-- (18,10.46);
			\draw [line width=1pt] (18,10.46)-- (18.52,10.785);
			\draw [line width=1pt] (18.52,10.785)-- (19.06,10.82);
			\draw [line width=1pt] (19.06,10.82)-- (19.46,10.64);
			\draw [line width=1pt] (19.46,10.64)-- (20,11);
			\draw [line width=1pt] (20,10.62)-- (20.02,4.58);
			\draw (20,11.4) node[anchor=north west] {$a_n=0$};
			\draw (20,5) node[anchor=north west] {$b_n=-\sqrt{n}$};
		\end{tikzpicture}
	\caption{A  paired random walk ($\m{PRW}$) with top curve $(\iks)_{k=1}^{n}$ and the bottom curve $(\iiks)_{k=1}^{n}$ specified by the condition that $\isn=x_n=0$ and $\iisn=y_n=-\sqrt{n}$. The $\m{PRW}$ law is given in \eqref{den} and should be seen as a reweighting of independent random walks by $\ga(\ise-\iise)$. This explains why the paths approach each other on the left of this figure. Also illustrated here is the situation where the two random walks happen to also be non-intersecting.}
	\label{figr}
\end{figure}

It follows from the Gibbs property for the $\m{WPRW}$ law (see Lemma \ref{l:LSCrit}) that the
\begin{equation}\label{e.reduction}
\textrm{law of $(L_1(2k-1), L_2(2k))_{k=1}^n$ in \eqref{law} equals law of $(\iks,\iiks)_{k=1}^n$ under $\Pr_{\operatorname{WPRW}}^{n;(x_n,y_n)}$}
\end{equation}
where the latter depends only on $n$ and not $N$. Thus, hereon out we study the $\m{WPRW}$ law.

\begin{remark}\label{r.evenodd}
The $\m{WPRW}$ law described above only describes the behavior of points $(L_1(2k-1), L_2(2k))_{k=1}^n$ under the law in \eqref{law}. This leaves half of the points unaccounted for  -- even indexed points in $L_1\ll1,2n\rr$ and odd indexed points in $L_2\ll1,2n\rr$. However, once we have controlled the behavior of the points  $(L_1(2k-1), L_2(2k))_{k=1}^n$, the complementary points can easily be controlled by use of the Gibbs property as explained in Lemma \ref{l:LSCrit}.
\end{remark}

\smallskip

\noindent{\it \textbf{Proof idea for \ref{it2} and \ref{it2.5}}:}  We now illustrate the proof idea of \ref{it2}. The proof idea for \ref{it2.51} is quite similar and done in parallel in Section \ref{sec:rpe}.
To establish \ref{it2}, it suffices to show that for $\m{A}=\{\ise \le -M\sqrt{n}\}$ or $\m{A}=\{\iise \le -M\sqrt{n}\}$, $\Pr_{\operatorname{WPRW}}^{n;(x_n,y_n)}(\m{A})$  can be made arbitrarily small by choosing $M$ large enough in a manner that is uniform as $n\to \infty$. Let us consider the case $\m{A}=\{\ise \le -M\sqrt{n}\}$ as the argument for the other case is completely analogous. The event $\m{A}$ is increasing (recall from Definition \ref{def:hslggibbs}). Thus, by stochastic monotonicity of $\Pr_{\operatorname{WPRW}}^{n;(x_n,y_n)}$ (Proposition \ref{p:gmc}), decreasing the values of the endpoints $(x_n,y_n)$ can only increase the probability of $\m{A}$. Thus
\begin{align*}
	\Pr_{\operatorname{WPRW}}^{n;(x_n,y_n)}(\ise \le -M\sqrt{n}) & \le \Pr_{\operatorname{WPRW}}^{n;(\min\{x_n,y_n\},\min\{x_n,y_n\}-\sqrt{n})}(\ise \le -M\sqrt{n}) \\ & = \Pr_{\operatorname{WPRW}}^{n;(0,-\sqrt{n})}(\ise \le -M\sqrt{n}-\min\{x_n,y_n\}),
\end{align*}
where the last inequality follows from shift invariance of the Gibbs measures (Lemma \ref{obs1} \ref{traninv}). Recall that by the tightness afforded to us from \ref{it1} we were able to assume that $|\min\{x_n,y_n\}|$ is of order $\sqrt{n}$. Since in \ref{it2} and \ref{it2.5} we are likewise trying to prove tightness or that certain events occur with positive probability, it suffices to show that for all $C>0$, those results hold under the assumption $|\min\{x_n,y_n\}|< C\sqrt{n}$. We do not need uniformity in $C$ and the argument is the same for any such value, so we will currently assume $C=1$. Thus we aim now to bound $\Pr_{\operatorname{WPRW}}^{n;(0,-\sqrt{n})}(\ise \le -M\sqrt{n})$ (really $M+1$, but since $M$ is arbitrary we just write $M$ here) for large enough $M$, uniformly in $n$. To summarize, we have currently reduced our consideration to the boundary data  $x_n=0, y_n=-\sqrt{n}$. This type of reduction is also possible while dealing with corresponding events in \ref{it2.51} but not for the event \ref{it3}.



\medskip

We next claim that there exists a constant $\Con>0$ so that uniformly over all $M$ and $n$
$$
\Pr_{\operatorname{WPRW}}^{n;(0,-\sqrt{n})}(\m{A})\le \Con^2\,  \Pr_{\operatorname{PRW}}^{n;(0,-\sqrt{n})}(\m{A})\qquad \textrm{for }  \m{A}=\{\ise\le -M\sqrt{n}\}.
$$
The probability $\Pr_{\operatorname{PRW}}^{n;(0,-\sqrt{n})}(\m{A})$ can be readily shown to go to zero uniformly in $n$ as $M$ grows. Thus, it suffices to prove the above comparison.
In light of \eqref{wscint}, it suffices to show such a comparison for both the numerator and denominator. In particular, we show that
\begin{equation}
	\label{wprwest}
	\begin{aligned}
		 \Ex_{\operatorname{PRW}}^{n;(0,-\sqrt{n})}[W\ind_{\m{A}}] \le \Con \! \cdot \! n^{-1/2}\!\cdot\! \Ex_{\operatorname{PRW}}^{n;(0,-\sqrt{n})}[\ind_{\m{A}}], \quad \textrm{and} \quad\Ex_{\operatorname{PRW}}^{n;(0,-\sqrt{n})}[W] \ge \tfrac1{\Con}\!\cdot\! n^{-1/2},
	\end{aligned}
\end{equation}
where $\Con>0$ is a universal constant that does not depend on $M$ or $n$. Notice that while both numerator and denominator terms in \eqref{wscint} go to zero with $n\to\infty$, they do so at the same rate $n^{-1/2}$ which cancels and yields the desired control on $\Pr_{\operatorname{WPRW}}^{n;(0,-\sqrt{n})}(\m{A})$. Note also that the $n^{-1/2}$ decay behavior here is particular to $x_n,y_n$ of order $\sqrt{n}$ and that for general boundary values of $x_n, y_n$, the estimates in \eqref{wprwest} may not be true.
The inequalities in \eqref{wprwest} are established in the proofs of Lemmas \ref{l:rpass} and \ref{lem:tre}, and Corollary \ref{corb}. We will describe their main ideas here.

\smallskip

\noindent{\it Proof idea for upper bound in \eqref{wprwest}.}
We briefly explain the proof idea for the upper bound on $\Ex_{\operatorname{PRW}}^{n;(0,-\sqrt{n})}[W\ind_{\m{A}}]$ and, along the way, we explain why the $n^{-1/2}$ factor shows up. The starting point of our proof is to compare the soft non-intersection conditioning by $W$ to hard non-intersection conditioning in the following manner. Define
\begin{align}\label{defnipab}
	\ni_p:=\left\{ \iks-\iiks \ge -p, \mbox{ for all } k\in \ll 2,n-1\rr\right\}, \qquad\textrm{with}\quad \ni:=\ni_0.
\end{align}
Under the complement event $\ni_p^c$, we have $W\le e^{-e^p}$ and thus the following deterministic inequality:
\begin{align}\label{eqwineq}
	W & \le \ind_{\ni}+\sum_{p=0}^{\infty} e^{-e^p}\! \cdot\! \ind_{\ni_{p+1}\cap \ni_p^c} \le \ind_{\ni}+\sum_{p=0}^{\infty} e^{-e^p}\! \cdot\! \ind_{\ni_{p+1}}.
\end{align}
Note that if we condition on $(\ise,\iise)$, the $\m{PRW}$ law can be viewed as two independent random bridges from $(\ise,\iise)$ to $(0,-\sqrt{n})$. We denote this law as $\Pr^{n;(\ise,\iise),(0,-\sqrt{n})}$. In Lemma \ref{l:nipp}, we show that there is an absolute constant $\Con>0$, such that
$$\Pr^{n;(\ise,\iise),(0,-\sqrt{n})}(\ni_p) \le e^{\Con p}\cdot \Pr^{n;(\ise,\iise),(0,-\sqrt{n})}(\ni)\quad\textrm{for all } p\ge 0.$$
By this inequality and \eqref{eqwineq} along with the tower property of conditional expectations we find
$$\Ex_{\operatorname{PRW}}^{n;(0,-\sqrt{n})}[W\ind_{\m{A}}] \le \Con \cdot \Ex_{\operatorname{PRW}}^{n;(0,-\sqrt{n})}\left[\ind_{\m{A}}\cdot\Pr^{n;(\ise,\iise),(0,-\sqrt{n})}(\ni)\right]$$
for some $\Con>0$. Thus, to upper bound $\Ex_{\operatorname{PRW}}^{n;(0,-\sqrt{n})}[W\ind_{\m{A}}]$ it suffices to do so to $\Pr^{n;(\ise,\iise),(0,-\sqrt{n})}(\ni)$.

Due to the presence of the $g$ factor in \eqref{den}, under the $\m{PRW}$ law we expect a pinning effect in the left boundary, i.e., $\ise-\iise=O(1)$. Thus we expect the large scale behavior under the $\m{PRW}$ law should be comparable to that of two independent random walks started close to each other. It is well known (see for example \cite{spit,kozlov}) that when $\iks, \iiks$ are independent random walks with $\ise-\iise=0$, the non-intersection probability over a time horizon of $n$ step is of the order $n^{-1/2}$.  This is why we expect the $n^{-1/2}$ behavior of the non-intersection probability under the $\m{PRW}$ law as well. We confirm this expectation  with two lemmas. The first, Lemma \ref{as:el}, show that $|\ise-\iise|$, $\ise/\sqrt{n}$, and $\iise/\sqrt{n}$ all have exponential tails under the $\m{PRW}$ law. The second, Lemma \ref{genni}, bounds the non-intersection probability  as
\begin{equation}
	\label{genniineq}
	\begin{aligned}
		& \ind_{|\ise|+|\iise|\le \sqrt{n}(\log n)^{3/2}} \cdot \Pr^{n;(\ise,\iise),(0,-\sqrt{n})}\big(\ni\big) \\ & \hspace{3cm}\le \tfrac\Con{\sqrt{n}} \!\cdot\!{\max\{\ise-\iise,1\}\!\cdot\! \max\left\{\tfrac1{\sqrt{n}}|\ise|,2\right\}^{\frac32}}\!.
	\end{aligned}
\end{equation} 
This lemma allows us to control the probability when $|\ise|+|\iise| \le \sqrt{n}(\log n)^{3/2}$ (the complementary case probability is controlled by the exponential tails). Lemma \ref{as:el} follows from the description of the $\m{PRW}$ law in \eqref{den} and the exponential tails for the densities $\fa$ and $\ga$. Lemma \ref{genni} is more subtle and requires various estimates under the random bridge law that are uniform over a specified set of starting and ending points. Let us briefly explain here why we have such a bound in \eqref{genniineq}. Intuitively, the non-intersection probability should increase as the difference in starting points, $\ise-\iise$, increases. Thus we see a term of the form $\max\{\ise-\iise,1\}$ on the right-hand side of \eqref{genniineq}. The term involving $\tfrac1{\sqrt{n}}|\ise|$, on the other hand, arises due to the nature of our proof. In the course of proving Lemma \ref{genni}, we proceed by bounding the ratio of density of the random bridge and density of a (pure) random walk. Such bound naturally depends on the slope of  the random bridge and gets worse as the slope $|\ise|/\sqrt{n}$ increases. This is why we encounter $\tfrac1{\sqrt{n}}|\ise|$ term on the right-hand side of \eqref{genniineq}. The details of the proof are presented in Appendix \ref{app2}. 
From the above two lemmas the upper bound in \eqref{wprwest} follows readily.

\smallskip


\noindent{\it Proof idea for lower bound in \eqref{wprwest}.}
Lower bounding $\Ex_{\operatorname{PRW}}^{n;(0,-\sqrt{n})}[W]$ is more involved. The first step is to find a lower bound for $W$ in terms of the indicator function for an event which we call $\m{Gap}$. For simplicity we do this in a simpler setting to lower bound $\Ex_{\operatorname{PRW}}^{n;(0,-\sqrt{n})}[W']$ where
\begin{align}
	W':=\exp\bigg(-\sum_{k=2}^{n-1} e^{\iiks-\iks}\bigg)
\end{align}
is obtain by deleting various terms in the exponential defining of $W$ (see \eqref{defw}). Note that $W'\ge W$ and thus the argument for $W$ is necessarily more involved. It is known from \cite{ritter} that when a random walk 
$(\se{}{k})_{k=1}^n$ is conditioned to stay positive, with high probability $\se{}{k}$ has growth at least of the order $k^{\frac12-\delta}$ for any $\delta>0$. Taking $\delta=\frac14$ this implies that if we condition a random bridge $(\se{}{k})_{k=1}^n$ of length $n$ starting and ending at zero to stay positive, then $\se{}{k}$ should be at least of the order $\min\{k,n-k+1\}^{\frac14}$ with very high probability. Treating $\iks-\iiks$ as a random bridge, this inspires us to define
\begin{align*}
\m{Gap}_\beta':=\{\iks-\iiks \ge \beta\cdot \min\{k,n-k+1\}^{\frac14} \mbox{ for all } k\in \ll2,n\rr \}
\end{align*}
We note that when $\m{Gap}_{\beta'} \cap \{\ise-\iise\in [0,1]\}$ occurs, the sum in the exponent of $W'$ is bounded uniformly in $n$ and hence $W'$ is bounded below by a strictly positive constant, say $a_\beta'$. Thus we have $W' \ge a_{\beta}'\ind_{\m{Gap}_\beta'\cap \{\ise-\iise\in [0,1]\}}$.
For $W$, we define a similar, albeit more complicated, event $\m{Gap}_\beta$  (see \eqref{def:gp} for definition) that captures the above idea and we show in Lemma \ref{wgap} that $W \ge  a_{\beta}\ind_{\m{Gap}_\beta\cap \{\ise-\iise\in [0,1]\}}$ for some deterministic constant $a_{\beta}>0$. Thus to lower bound $\Ex_{\operatorname{PRW}}^{n;(0,-\sqrt{n})}[W]$ we lower bound $\Pr_{\operatorname{PRW}}^{n;(0,-\sqrt{n})}(\m{Gap}_\beta)$ and $\Pr_{\operatorname{PRW}}^{n;(0,-\sqrt{n})}(\ise-\iise\in [0,1])$.
\medskip

	\begin{figure}[h!]
	\centering
	\begin{tikzpicture}[line cap=round,line join=round,>=triangle 45,x=1cm,y=1cm]
		\draw [line width=1pt] (3,4)-- (3.5,3.14);
		\draw [line width=1pt] (3.5,3.14)-- (4,3.74);
		\draw [line width=1pt] (4,3.74)-- (4.5,4.72);
		\draw [line width=1pt] (4.5,4.72)-- (5,5.14);
		\draw [line width=1pt] (5,5.14)-- (5.52,4.88);
		\draw [line width=1pt] (5.52,4.88)-- (5.75,5.44);
		\draw [line width=1pt,color=red] (10,3)-- (10.46,3.18);
		\draw [line width=1pt,color=red] (10.46,3.18)-- (11.25,4.56);
		\draw [line width=1pt] (11.25,4.56)-- (11.54,5.52);
		\draw [line width=1pt] (11.54,5.52)-- (11.98,5.26);
		\draw [line width=1pt] (11.98,5.26)-- (12.52,5.02);
		\draw [line width=1pt] (12.52,5.02)-- (12.98,4.86);
		\draw [line width=1pt] (12.98,4.86)-- (13.52,5.06);
		\draw [line width=1pt] (13.52,5.06)-- (14,6);
		\draw [line width=1pt,color=red] (5.75,5.44)-- (6.44,4.88);
		\draw [line width=1pt,color=red] (6.44,4.88)-- (7.02,5.28);
		\draw [line width=1pt,color=red] (7.02,5.28)-- (7.54,4.64);
		\draw [line width=1pt,color=red] (7.54,4.64)-- (7.98,4.58);
		\draw [line width=1pt,color=red] (7.98,4.58)-- (8.5,4.16);
		\draw [line width=1pt,color=red] (8.5,4.16)-- (9.02,4.78);
		\draw [line width=1pt,color=red] (9.02,4.78)-- (9.46,3.86);
		\draw [line width=1pt,color=red] (9.46,3.86)-- (10,3);
		\begin{scriptsize}
			\draw [fill=black] (3,4) circle (2.5pt);
			\draw (3.16,4.43) node {$A$};
			\draw [fill=black] (14,6) circle (2.5pt);
			\draw (14.16,6.43) node {$D$};
			\draw [fill=blue] (5.75,5.44) circle (2.5pt);
			\draw (5.85,5.74) node {$B$};
			\draw [fill=blue] (11.25,4.56) circle (2.5pt);
			\draw (11.55,4.5) node {$C$};
			\draw[dashed,<->] (3,2)--(5.65,2);
			\draw[dashed,<->] (5.85,2)--(11.15,2);
			\draw[dashed,<->] (11.35,2)--(14,2);
			\node at (4.3,2.3) {$n/4$};
			\node at (8.4,2.3) {$n/2$};
			\node at (12.6,2.3) {$n/4$};
		\end{scriptsize}
	\end{tikzpicture}
	\caption{The modified random bridge is constructed by starting two random walks of length $n/4$, one from $A$ to $B$ and another (run backwards) from $D$ to $C$. The path between $B$ and $C$ is then chosen given the values there according to a random bridge of length $n/2$. 
	}
	\label{fmrb2}
\end{figure}
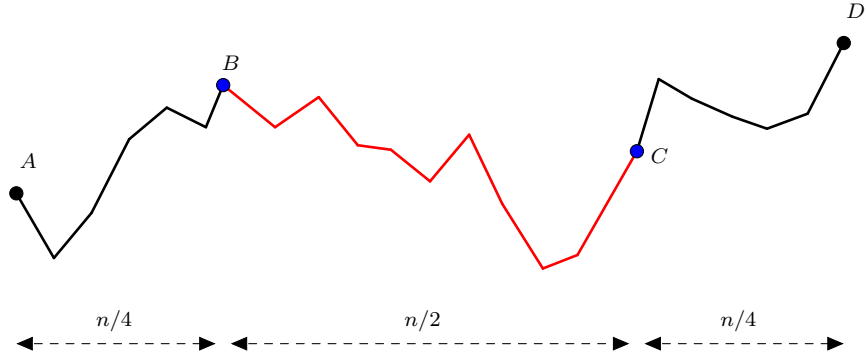

Recall that by the Gibbs property, conditioned on $(\ise,\iise)$, the law of $(\iks,\iiks)_{k=1}^{n}$ under $\Pr_{\operatorname{PRW}}^{n;(0,-\sqrt{n})}$
is that of two independent random bridges (with increment law $\fa$ as in all of our discussion above) started from $(\ise,\iise)$ and ended at $(0,-\sqrt{n})$. In estimating the $\m{Gap}$ event probability under this law we found it easier to work with the law of two independent \textit{modified random bridges}. These are described in Figure \ref{fmrb2} and composed of random walks (with increment law $\fa$) in the first and last $n/4$ portion of its domain, and then a bridge to connect the resulting values. Lemma \ref{lem:compare} shows that the density of the random bridge and modified random bridge are comparable provided the values at $k=1,n/4,3n/4,n$ are all of order $\sqrt{n}$.
In particular, if we set $$\m{E}:=\{|S_1(n/4)-S_1(3n/4)|+|S_2(n/4)-S_2(3n/4)|\le \sqrt{n}\}, \quad \m{F}:=\{|\ise|+|\iise|\le \sqrt{n}\},$$
the results from Lemma \ref{lem:compare} allows us to conclude that
\begin{equation}
	\label{eqrgv}
	\begin{aligned}
		& \Pr_{\operatorname{PRW}}^{n;(0,-\sqrt{n})}(\m{Gap}_\beta \cap \{\ise-\iise \in [0,1]\}) \ge \Pr_{\operatorname{PRW}}^{n;(0,-\sqrt{n})}(\m{Gap}_\beta \cap \{\ise-\iise \in [0,1]\}\cap \m{E}\cap \m{F}) \\ & \ge \tfrac1\Con  \cdot \Ex_{\operatorname{PRW}}^{n;(0,-\sqrt{n})}\left[\ind_{\{\ise-\iise\in [0,1]\}\cap \m{F}} \cdot \tpr{n}{(\ise,\iise)}{,(0,-\sqrt{n})}(\m{Gap}_\beta \cap \m{E})\right]
	\end{aligned}
\end{equation}
where $\til{\Pr}^{n;(\ise,\iise),(0,-\sqrt{n})}$ is the law of two independent modified random bridges started from $(\ise,\iise)$ and ended at $(0,-\sqrt{n})$.  Using shorthand $\til{\Pr}$ for $\til{\Pr}^{n;(\ise,\iise),(0,-\sqrt{n})}$ and Bayes rule, $$\til{\Pr}(\m{Gap}_\beta \cap \m{E})=\til{\Pr}(\ni)\cdot\til{\Pr}(\m{Gap}_\beta \cap \m{E}\mid \ni).$$
Since the modified random bridge has two true random walk portion (first and last quarter) we can now rely on standard non-intersecting random walk techniques to eventually obtain a lower bound on the probability $\til{\Pr}(\ni)$ above. In Appendix \ref{app2} we establish various uniform estimates and in particular (combining  Lemma \ref{l:class} and Corollary \ref{l:niexp}) show that for $x_n,y_n$ of order $n^{1/2}$,
\begin{align}\label{e3r}
	\til\Pr(\ni) \ge \tfrac1\Con \cdot n^{-1/2} \cdot \pr{n/4}{(x_n,y_n)}{}(\til{\ni}),
\end{align}
and $\til\Pr(\m{Gap}_\beta \cap \m{E}\mid \ni) \ge \frac1\Con$ (for small enough $\beta$) uniformly over $\ise,\iise\le M\sqrt{n}$ and $\ise-\iise \in [0,1]$. In \eqref{e3r}, $\pr{n/4}{(x_n,y_n)}{}(\til{\ni})$ denotes the probability of non-intersection of two random walks of length $n/4$ started from $x_n$ and $y_n$. As $x_n=0, y_n=-\sqrt{n}$, we can show that $\pr{n/4}{(x_n,y_n)}{}(\til{\ni})$ is bounded below. Finally, Lemma \ref{as:el} establishes that $\Ex_{\operatorname{PRW}}^{n;(0,-\sqrt{n})}\left[\ind_{\{\ise-\iise\in [0,1]\}\cap \m{F}}\right]$ is bounded below. Thus combining all the estimates leads  to an $n^{-1/2}$ order lower bound for the right hand side of \eqref{eqrgv}. Putting together the various bounds described above now yields the desired lower bound on
$\Ex_{\operatorname{PRW}}^{n;(0,-\sqrt{n})}[W]$ in \eqref{wprwest}.
Since in Section \ref{sec:pimc}, we prove \ref{it2} and \ref{it2.51} in parallel, some parts of the argument presented here in the introduction appear in a more general or slightly different flavor later. However, the core idea and features remain the same.

\smallskip

Non-intersecting random walks and random bridges that are pinned at the starting and/or ending points have been studied extensively (e.g. \cite{n1,n2,n3} and the reference therein) and are known to converge under diffusive scaling to Dyson Brownian motion and non-intersecting Brownian bridges. As demanded by our technical arguments, our work establishes \textit{uniform} (over starting and ending points) estimates  for non-intersection probabilities of pairs of random walks and random bridges in such scaling, i.e.~\textit{uniform} over \textit{all possible} $O(1)$ starting points that potentially can vary in a diffusive $O(\sqrt{n})$ window (precisely how  $(\ise,\iise)$ behaves). 
Appendix \ref{app2} develops the machinery to establish such uniform estimates under general assumptions on increments of the random bridges.

 \medskip

\noindent{\it \textbf{Proof idea for \ref{it3}}:}
The argument to prove \ref{it3} also uses the machinery developed in the proof of \ref{it2} and \ref{it2.5} and the reduction from \eqref{e.reduction} to the study of the weighted paired random walks $(\iks,\iiks)_{k=1}^n$. For $\gamma,\delta,M>0$, consider the events
\begin{equation*}
\m{B}  =\m{B}(\delta,\gamma) :=\bigg\{\sup_{\substack{i_1,i_2\in \ll1,n/4\rr \\ |i_1-i_2|\le \delta n}} \!\!\!|\se{1}{i_1}-\se{1}{i_2}| \ge \gamma\sqrt{n}\bigg\}, \quad
\m{G}  = \m{G}(M):=\big\{|\se{1}{1}|+|\se{2}{1}| \ge M\sqrt{n}\big\}.
\end{equation*}
To prove tightness we will show that for each $\gamma>0$, as $\delta\to 0$, we have $\Pr_{\operatorname{WPRW}}^{n;(x_n,y_n)}(\m{B})\to 0$. Recall that as an input we know that $x_n,y_n$ are of order $n^{1/2}$.
In \ref{it2}, we observed that $\Pr_{\operatorname{WPRW}}^{n;(x_n,y_n)}(\m{G})\to 0$ as $M\to \infty$ uniformly in $n$.  Thus it suffices to provide an upper bound for $\Pr_{\operatorname{WPRW}}^{n;(x_n,y_n)}(\m{B}\cap \m{G}^c)$ for each $M>0$. Thanks to \eqref{wscint}, it suffices to give a upper bound for $\Ex_{\operatorname{PRW}}^{n;(x_n,y_n)}[W\ind_{\m{B}\cap\m{G}^c}]$ and a lower bound for $\Ex_{\operatorname{PRW}}^{n;(x_n,y_n)}[W]$.
An important difference from the discussion regarding \ref{it2} is that now the event $\m{B}$ in question deals with two point differences of $\se{1}{\cdot}$ which is not an increasing event. Thus, the monotonicity of the Gibbs measure with respect to the boundary data does not help here and, unlike in \ref{it2}, we cannot use monotonicity to reduce consideration to $x_n=0, y_n=-\sqrt{n}$.

Instead, using the soft non-intersection property of our Gibbsian line ensemble we can deduce control on the difference of the exit points. We show in Theorem \ref{t:order} that for all large enough $N$
 $$\Pr\Big(\L_1^N(2n-1) \ge \L_1^N(2n)-(\log N)^{7/6}\textrm{ for all }n\in \ll M_1N^{2/3},M_2N^{2/3}\rr\Big) \ge 1-2^{-N}.$$
Note that in \eqref{law} we conditioned upon $L_1(2n-1)=x_n$ and $L_1(2n)=y_n$. In view of the above high probability event, we may thus assume $x_n-y_n\ge -(\log n)^{7/6}$.
Under these boundary conditions ($x_n, y_n$ of order $n^{1/2}$ and $x_n-y_n\ge -(\log n)^{7/6}$) the estimates in \eqref{wprwest} may not hold. 

Nonetheless, for the lower bound of $\Ex_{\operatorname{PRW}}^{n;(x_n,y_n)}[W]$, all the arguments up to  and including \eqref{e3r} hold under the present assumptions on $(x_n,y_n)$. We show in Lemma \ref{crude} that
\begin{align}\label{eq:crude0}
	\Ex_{\operatorname{PRW}}^{n;(x_n,y_n)}[W] \ge \tfrac1{\sqrt{n}}{\Con_1^{-1}}\cdot\pr{n/4}{(x_n,y_n)}{}(\til{\ni}) \ge {\Con_2^{-1}}e^{-\Con_2 (\log n)^{5/4}},
\end{align}
where the second bound above is true under our assumption $x_n-y_n\ge (\log n)^{7/6}$. For the upper bound on $\Ex_{\operatorname{PRW}}^{n;(x_n,y_n)}[W\ind_{\m{B}\cap\m{G}^c}]$  we first obtain a deterministic bound for $W$ (similar to \eqref{eqwineq})
\begin{align*}
	W & \le \Con \left(e^{-(\log n)^2}+\til{W}\right), \qquad \til{W}:=\sum_{p=0}^{\lfloor 2\log\log(n) \rfloor}   e^{-e^p}\! \cdot\! \ind_{\ni_{p+1}}.
\end{align*}
 Due to the $e^{-\Con_2 (\log n)^{5/4}}$ lower bound in \eqref{eq:crude0}, the $e^{-(\log n)^2}$ factor above can be ignored and we instead focus on upper bounding $\til{W}$. Using the Gibbs property and conditional expectations  
\begin{align*}
	\Ex_{\operatorname{PRW}}^{n;(x_n,y_n)}[\til{W}\ind_{\m{B}\cap\m{G}^c}]  & = \sum_{p=0}^{\lfloor 2\log\log(n) \rfloor-1} e^{-e^p} \Ex_{\operatorname{PRW}}^{n;(x_n,y_n)}\left[\ind_{\m{G}^c}\pr{n}{(\ise,\iise)}{,(x_n,y_n)}(\m{B}\cap \ni_{p+1})\right] \\ & = \sum_{p=0}^{\lfloor 2\log\log(n) \rfloor-1} e^{-e^p} \Ex_{\operatorname{PRW}}^{n;(x_n,y_n)}\left[\ind_{\m{G}^c}\pr{n}{(\ise+p+1,\iise)}{,(x_n+p+1,y_n)}(\m{B}\cap \ni)\right].
\end{align*}
where $\pr{n}{(\ise,\iise)}{,(x_n,y_n)}$ is the law of two independent random  bridges started from $(\ise,\iise)$ and ending at $(x_n,y_n)$. The last equality above follows by lifting the $\se{1}{\cdot}$ random walk by $p+1$ units. We then apply the density comparison (Lemma \ref{lem:compare}) to modified random bridges to obtain
\begin{align*}
\ind_{\m{G}^c}\cdot \pr{n}{(\ise+p,\iise)}{,(x_n+p,y_n)}(\m{B}\cap \ni) \le \Con \cdot \ind_{\m{G}^c}\cdot \tpr{n}{(\ise+p,\iise)}{,(x_n+p,y_n)}(\m{B}\cap \ni).
\end{align*}
where $\tpr{n}{(\ise+p,\iise)}{,(x_n+p,y_n)}$ is the law of two independent modified random bridges started from $(\ise,\iise)$ and ending at $(x_n,y_n)$. The above comparison is only possible when we have a control on the slopes of the random bridges. This slope control is precisely furnished by $\ind_{\m{G}^c}$.

Let us write $\til{\Pr}_p$ for $\tpr{n}{(\ise+p+1,\iise)}{,(x_n+p+1,y_n)}$. Using uniform estimates for non-intersection probability for random walks and bridges from Appendix \ref{app2} (combining Lemma \ref{l:class}, Lemma \ref{l:nipp} and Corollary \ref{l:niexp}) we obtain that
\begin{align*}
	\til{\Pr}_p(\ni) \le \tfrac{\Con}{\sqrt{n}}e^{\Con p} \cdot \max\{\ise-\iise,1\}\cdot \pr{n/4}{(x_n,y_n)}{}(\til{\ni}).
\end{align*}
Since $\m{B}$ depends only of the first quarter points, $\til{\Pr}_p(\m{B}\mid \ni)$ can be controlled by modulus of continuity estimates for (pure) random walks under non-intersection which we deduce in Lemma \ref{mret}. In particular, we obtain that $\sup_{p\in \ll0,\lfloor 2\log\log(n) \rfloor -1\rr}\til{\Pr}_p(\m{B}\mid \ni)\to 0$ as $\delta\to 0$. Combining all the above estimates, in view of the exponential tail bounds for $\se{1}{1}-\se{2}{1}$ under the $\m{PRW}$ law from Lemma \ref{as:el}, this leads to the desired estimate.

	\subsection{Related works on half-space polymers}\label{sec:relatedwork}
Half-space polymers are a particular variant of full-space polymers that are well-studied in the literature (introduced in \cite{huse,imb,bol}). Full-space polymers are widely believed to be in the KPZ universality class in the sense that they are expected to have $1/3$ fluctuation exponent and $2/3$ transversal exponent. However, besides a few solvable models, these exponents are not proven rigorously for general polymers. We refer to \cite{comets,timo,batesch,bcd,dz22a,dz22b} and references therein for more details.

	Half-space polymer models have been studied in the physics literature since the work of Kardar \cite{kar2}. They arise naturally in the context of modeling wetting phenomena \cite{phy1,phy2,phy3} where one studies directed polymers in the presence of a wall. They have been of great interest due to the presence of phase transition (called the `depinning transition') and a rich phase diagram for limiting distributions based on the diagonal strength. This phase diagram was first rigorously proven for geometric last passage percolation (LPP), i.e., polymers with zero temperature, in a series of works by Baik and Rains \cite{br1,br01,br3}. Multi-point fluctuations were studied then in \cite{sis} and similar results were later proven for exponential LPP in \cite{bbcs0,bbcs} using Pfaffian Schur processes. For further recents works on half-space LPP, we refer to \cite{bete,ale1,ale2,ale3}.
	
Positive temperature models such as polymers resisted rigorous treatment for longer compared to LPP since they are no longer directly related to Pfaffian point processes. For such class of models in the half-space geometry, the first rigorous proof of depinning transition appeared in \cite{bw} where the authors proved precise fluctuation results including the BBP phase transition \cite{bbap} for the point-to-line log-gamma free energy. For the point-to-point log-gamma free energy, the limit theorem along with Baik-Rains phase transition was conjectured in \cite{barraquand_borodin_corwin_2020} based on an uncontrolled steepest descent analysis of certain formulas coming from half-space Macdonald processes. This result was proved recently in \cite{ims22} using a new set of ideas, relating the half-space model to a free boundary version of the Schur process. In fact, \cite{ims22} also proves analogous results for the half-space KPZ equation which is the free energy of the continuum directed random polymer in half-space.
	The half-space KPZ equation arises as a limit of free energy of $\hslg$ polymer \cite{wu,bc22}. Since the early work by Kardar \cite{kar2}, the half-space KPZ equation has received significant attention, with a flurry of new results recently in
 in both mathematics \cite{cs1,bbcw,barraquand_borodin_corwin_2020,par,par2, bc22,ims22} and physics literature \cite{gld,bbc,ito,de,kr,bkld2,bld1,bkld}. Apart from log-gamma and continuum polymer, a half-space version of the beta polymer was recently introduced and studied in \cite{ryc}.
	
	
\subsection*{Organization} In Section \ref{sec:gibb}, we study several properties of $\hslg$ Gibbs measures and Gibbsian line ensemble, and prove Theorem \ref{thm:conn}. Section \ref{sec:2curve} is divided into three subsections that discuss three important probabilistic results for the line ensemble. In Section \ref{sec:cls}, we show a certain ordering of points on the line ensemble (Theorem \ref{t:order}). This is the precise technical form of the typical non-intersection property discussed at the end of Section \ref{sec:1.2}. In Section \ref{sec:high}, we show that there is a high point on the second curve (Theorem \ref{p:high2}) as discussed at the end of item \ref{it1} from Section \ref{sec:133}. In Section \ref{sec:high}, we provide high probability uniform upper bounds for the second and third curves (Theorem \ref{l:lhigh}). These bounds are used later in proving item \ref{it2} from Section \ref{sec:133}. In Section \ref{sec:rpe}, we prove one-point tightness on the left boundary and study the probability of a certain `region pass event'. The study of the region pass event is utilized in proving the lower bound on the uniform separation between the first two curves and the third curve (described earlier in \ref{it2} from Section \ref{sec:133}). Finally, in Section \ref{sec:mc}, we study the modulus of continuity under the $\m{WPRW}$ law and prove Theorem \ref{t:main0}.
Appendix \ref{appc} includes the proof of stochastic monotonicity for $\hslg$ Gibbsian line ensembles. Appendix \ref{app1} collects several basic facts about log-gamma random variables and related measures. Appendix \ref{app2} is devoted to proving several technical estimates related to non-intersecting random bridges which are required in studying the $\m{WPRW}$ law.

	\subsection*{Notations and Conventions} For $a,b\in \R$, we denote $\llbracket a,b \rrbracket := [a,b]\cap \Z$, $a\wedge b=\min(a,b)$, and $a\vee b=\max(a,b)$. Throughout this paper we work with three fixed parameters: $\theta >0$ (bulk parameter), $\zeta>0$ (supercritical boundary parameter), and $\mu\in \R$ (critical boundary parameter). All our constants appearing in the rest of the paper may depend on $\theta,\zeta,\mu$ and possibly other specified variables. We will only specify the dependency of the constants on the variables besides $\theta,\zeta,\mu$ by writing $\Con = \Con(a, b, c, \cdots) > 0$ to denote a generic deterministic positive finite constant that may change from line to line, but is dependent on the designated variables $a, b, c, \cdots$. We write l.h.s. or r.h.s. to denote the left- or right-hand side of an equation.
Given a density $f$, $X\sim f$ denotes a random variable $X$ whose distribution function has density $f$. We also write $X_i\stackrel{i.i.d.}{\sim}f$ if $\{X_i\}$ are i.i.d. with some common density $f$. We sometimes also use the notation $X\sim \bullet$ where $\bullet$ is the name of a distribution (e.g. $\operatorname{Gamma}^{-1}(\beta)$) to mean $X$ has distribution $\bullet$ and is independent of all other random variables being considered.  For two densities $f$ and $g$, we write $f\ast g(x)=\int_{\R} f(z)g(x-z)dz$ for the convolution density. We use the notation $\sigma(\bullet)$ for denote the $\sigma$-algebra generated by the random variables $\bullet$. We write $\stackrel{(d)}{\Longrightarrow}$ and $\stackrel{(d)}{=}$ for convergence and equality in distribution.
There is a glossary at the end of this text that recalls and points to the definitions of much of the notation introduced elsewhere.

	\subsection*{Acknowledgements} The project was initiated during the authors' participation in the `Universality and Integrability in Random Matrix Theory and Interacting Particle Systems' semester program at MSRI in fall 2021. The authors thank the program organizers for their hospitality and acknowledge the support from NSF DMS-1928930. GB was partially supported by ANR grant ANR-21-CE40-0019. IC was partially supported by the NSF through grants DMS-1937254, DMS-1811143, DMS-1664650, DMS-2246576, as well as through a Packard Fellowship in Science and Engineering, a Simons Fellowship, a Simons Investigator Award, a Miller Visiting Professorship from the Miller Institute for Basic Research in Science, and a W.M. Keck Foundation Science and Engineering Grant. SD's research was partially supported by Ivan Corwin's NSF grant DMS-1811143 and the Fernholz Foundation's ``Summer Minerva Fellows'' program.

\section{Half-space log-gamma objects and proof of Theorem \ref{thm:conn}} \label{sec:gibb}
In Section \ref{sec2.1}, we gather several useful properties of $\hslg$ Gibbs measures from Definition \ref{def:hslggibbs} including stochastic monotonicity (Proposition \ref{p:gmc}). The $\hslg$ line ensemble is defined in Section \ref{sec2.2} which includes the proof of Theorem \ref{thm:conn}.

Here, we introduce few important functions that will come up often in our arguments. We define
\begin{align}
	\label{def:wfns}
	W(a;b,c)=\exp(-e^{a-b}-e^{a-c}), \quad a,b,c\in \R,
\end{align}
For $\theta>0, m\in \Z$ we set
\begin{align}\label{def:gwt}
	G_{\theta,(-1)^m}(y):=e^{\theta (-1)^my-e^{(-1)^my}}/\Gamma(\theta), \quad y\in \R.
\end{align}
One can check that $G_{\theta,(-1)^m}$ is a density (i.e., it integrates to $1$). Using $G$, we define two more probability density functions:
\begin{align}\label{def:faga}
	\fa(x):=G_{\theta,+1}\ast G_{\theta,-1}(x), \qquad \ga(x):=G_{\zeta,1}(x), \qquad \theta,\zeta>0, x\in \R,
\end{align}
where $\ast$ denotes the convolution operation, i.e., $p\ast q(x):=\int_{\R} p(z)q(x-z)dz$. 	
	
	\subsection{Properties of $\hslg$ Gibbs measures} \label{sec2.1}
We start by writing down  several lemmas that all follow directly from the definition of $\hslg$ Gibbs measures (recall from Definition \ref{def:hslggibbs}).

	\begin{observation}\label{obs1}
	Consider the graph  $\Z_{\ge 1}^2$ endowed with directed/colored edges $E(\Z_{\ge 1}^2)$ as above. Let $\Lambda$ be a bounded connected subset of $\Z_{\ge 1}^2$.  For each $(i,j)\in \partial\Lambda$ fix some $u_{i,j}\in \R$. Fix any $c\in \R$. Let $\big(L(v) : v\in \Lambda\big)$ be a collection of random variables that are distributed as the $\hslg$ $\Theta$-Gibbs measure on the domain $\Lambda$ with boundary condition $\big(u_{i,j}:(i,j)\in \partial\Lambda\big)$.
		\begin{enumerate}[label=(\alph*),leftmargin=15pt]
			\item \label{traninv} (Translation invariance) The law of $\big(L(v)+ c : v\in \Lambda\big)$ is given by the $\hslg$ $\Theta$-Gibbs measure on the domain $\Lambda$ with boundary condition $\big(u_{i,j}+c : (i,j)\in \partial\Lambda\big)$.	
			\item \label{gpsd} (Gibbs property on smaller domain)	Take a bounded connected subset $\Lambda'\subset \Lambda$. The law of $\big(L(v) : v\in \Lambda'\big)$ conditioned on  $\big(L(v) : v\in \Lambda\setminus\Lambda'\big)$ is given by the $\hslg$ $\Theta$-Gibbs measure on the domain $\Lambda'$ with the boundary condition $\big(L(v) : v\in \partial\Lambda'\big)$ where we set $L(v)=u_v$ for $v\in \partial\Lambda$.
		\end{enumerate}
	\end{observation}	
	
	\begin{proof} Note that the density of a $\hslg$ $\Theta$-Gibbs measure given in \eqref{e:hsgibb} only involve terms of the form $u_{v_1}-u_{v_2}$. Thus adding a constant $c$ to every term does not change the law. The fact that Gibbs property carries to smaller domains follows from the explicit form of the density as well.		
	\end{proof}
	
	Although $\hslg$ Gibbs measures are defined for any bounded connected subset $\Lambda$, we will be mainly concerned with two kinds of domains $\Lambda$. Given $k\ge 1$ and $T\ge 2$, we define
	\begin{align}\label{def:kkt}
		\mathcal{K}_{k,T}:=\left\{(i,j) : i \in \ll 1,k \rr, j\in \ll 1,2T-1-\ind_{i=1} \rr \right\}, \quad \mathcal{K}_{k,T}':= \ll 1,k \rr \times \ll 1,2T-2 \rr.
	\end{align}
	The domains $\mathcal{K}_{k,T}$ and $\mathcal{K}_{k,T}'$ are shown as shaded regions in Figure \ref{fig4}. We state these results for the homogeneous Gibbs measures, though they could easily be adapted to $\Theta$-Gibbs measures.
	\begin{figure}[h!]
		\centering
		\begin{subfigure}[b]{0.45\textwidth}
			\centering
			\begin{tikzpicture}[line cap=round,line join=round,>=triangle 45,x=1.7cm,y=0.85cm]
				\draw[fill=gray!10,line width=0.5pt,dashed] (-0.65,0.25)--(-0.65,-2.75)--(2.7,-2.75)--(2.7,-0.9)--(2.25,-0.9)--(2.25,0.25)--(-0.65,0.25);
				\foreach \x in {0,1}
				{
					\draw[line width=1.5pt,blue,{Latex[length=2mm]}-]  (\x,0) -- (\x-0.5,-0.5);
					\draw[line width=1.5pt,blue,{Latex[length=2mm]}-] (\x,0) -- (\x+0.5,-0.5);
					\draw[line width=1.5pt,black,{Latex[length=2mm]}-] (\x-0.5,-0.5) -- (\x,-1);
					\draw[line width=1.5pt,black,{Latex[length=2mm]}-] (\x+0.5,-0.5) -- (\x,-1);
					\draw[line width=1.5pt,blue,{Latex[length=2mm]}-]  (\x,-1) -- (\x-0.5,-1.5);
					\draw[line width=1.5pt,blue,{Latex[length=2mm]}-] (\x,-1) -- (\x+0.5,-1.5);
					\draw[line width=1.5pt,black,{Latex[length=2mm]}-] (\x-0.5,-1.5) -- (\x,-2);
					\draw[line width=1.5pt,black,{Latex[length=2mm]}-] (\x+0.5,-1.5) -- (\x,-2);
					\draw[line width=1.5pt,blue,{Latex[length=2mm]}-]  (\x,-2) -- (\x-0.5,-2.5);
					\draw[line width=1.5pt,blue,{Latex[length=2mm]}-] (\x,-2) -- (\x+0.5,-2.5);
					\draw[line width=1pt,black!50!white,{Latex[length=2mm]}-] (\x-0.5,-2.5) -- (\x,-3);
					\draw[line width=1pt,black!50!white,{Latex[length=2mm]}-] (\x+0.5,-2.5) -- (\x,-3);
				}
				\foreach \x in {2}
				{
					\draw[line width=1.5pt,blue,{Latex[length=2mm]}-]  (\x,0) -- (\x-0.5,-0.5);
					\draw[line width=1.5pt,black,{Latex[length=2mm]}-] (\x-0.5,-0.5) -- (\x,-1);
					\draw[line width=1.5pt,blue,{Latex[length=2mm]}-]  (\x,-1) -- (\x-0.5,-1.5);
					\draw[line width=1.5pt,blue,{Latex[length=2mm]}-] (\x,-1) -- (\x+0.5,-1.5);
					\draw[line width=1.5pt,black,{Latex[length=2mm]}-] (\x-0.5,-1.5) -- (\x,-2);
					\draw[line width=1.5pt,black,{Latex[length=2mm]}-] (\x+0.5,-1.5) -- (\x,-2);
					\draw[line width=1.5pt,blue,{Latex[length=2mm]}-]  (\x,-2) -- (\x-0.5,-2.5);
					\draw[line width=1.5pt,blue,{Latex[length=2mm]}-] (\x,-2) -- (\x+0.5,-2.5);
					\draw[line width=1pt,black!50!white,{Latex[length=2mm]}-] (\x-0.5,-2.5) -- (\x,-3);
					\draw[line width=1pt,black!50!white,{Latex[length=2mm]}-] (\x+0.5,-2.5) -- (\x,-3);
				}
				\draw[line width=1pt,black!50!white,{Latex[length=2mm]}-] (2.5,-2.5)--(3,-3);
				\draw[line width=1pt,black!50!white,{Latex[length=2mm]}-] (2.5,-1.5)--(3,-2);
				\draw[line width=1pt,black!50!white,{Latex[length=2mm]}-] (2.5,-0.5)--(2,-1);
				\draw[line width=1pt,blue!50!white,{Latex[length=2mm]}-] (2,0)--(2.5,-0.5);
				\draw[line width=1pt,blue!50!white,{Latex[length=2mm]}-] (3,-3+2)--(3-0.5,-3.5+2);
				\draw[line width=1pt,blue!50!white,{Latex[length=2mm]}-] (3,-3+1)--(3-0.5,-3.5+1);
				\draw [fill=blue] (2.5,-1.5) circle (2pt);
				\draw [fill=blue] (2.5,-2.5) circle (2pt);
				\draw[line width=1.5pt,red,{Latex[length=2mm]}-] (-0.5,-1.5) -- (-0.5,-0.5);
				\draw[line width=1pt,red!50!white,{Latex[length=2mm]}-] (-0.5,-3.5) -- (-0.5,-2.5);
				\node at (2.7,-0.5) {$y_1$};
				\node at (3.2,-1) {$y_2$};
				\node at (3.2,-2) {$y_3$};
				\node at (-0.5,-3.8) {${z_0}$};
				\node at (0,-3.3) {$z_1$};
				\node at (1,-3.3) {$z_2$};
				\node at (2,-3.3) {$z_3$};
				\node at (3.2,-3) {$z_4$};
				\foreach \x in {0,1,2}
				{
					\draw [fill=blue] (\x,0) circle (2pt);
					\draw [fill=blue] (\x,-1) circle (2pt);
					\draw [fill=blue] (\x,-2) circle (2pt);
					\draw [fill=blue] (\x-0.5,-0.5) circle (2pt);
					\draw [fill=blue] (\x-0.5,-1.5) circle (2pt);
					\draw [fill=blue] (\x-0.5,-2.5) circle (2pt);
				}
				\foreach \x in {0,1,2}
				{\draw [fill=white] (3,-\x-1) circle (2pt);
					\draw [fill=white] (\x,-3) circle (2pt);}
				\draw [fill=white] (-0.5,-3.5) circle (2pt);
				\draw [fill=white] (2.5,-0.5) circle (2pt);
				\begin{scriptsize}
				\node at (0,0.2) {$(1,2)$};
				\node at (0.5,-1.2) {$(2,3)$};
				\node at (1.5,-1.2) {$(2,5)$};
				\node at (1,-0.7) {$(2,4)$};
				\node at (2,-1.7) {$(3,6)$};
				\end{scriptsize}
				
			\end{tikzpicture}
			\vspace{-0.4cm}
			\caption{$\mathcal{K}_{k,T}$}
		\end{subfigure}
		\begin{subfigure}[b]{0.45\textwidth}
			\centering
			\begin{tikzpicture}[line cap=round,line join=round,>=triangle 45,x=1.7cm,y=0.85cm]
				\draw[fill=gray!10,line width=0.5pt,dashed] (-0.65,0.25)--(-0.65,-2.75)--(2.25,-2.75)--(2.25,0.25)--(-0.65,0.25);
				\foreach \x in {0,1}
				{
					\draw[line width=1.5pt,blue,{Latex[length=2mm]}-]  (\x,0) -- (\x-0.5,-0.5);
					\draw[line width=1.5pt,blue,{Latex[length=2mm]}-] (\x,0) -- (\x+0.5,-0.5);
					\draw[line width=1.5pt,black,{Latex[length=2mm]}-] (\x-0.5,-0.5) -- (\x,-1);
					\draw[line width=1.5pt,black,{Latex[length=2mm]}-] (\x+0.5,-0.5) -- (\x,-1);
					\draw[line width=1.5pt,blue,{Latex[length=2mm]}-]  (\x,-1) -- (\x-0.5,-1.5);
					\draw[line width=1.5pt,blue,{Latex[length=2mm]}-] (\x,-1) -- (\x+0.5,-1.5);
					\draw[line width=1.5pt,black,{Latex[length=2mm]}-] (\x-0.5,-1.5) -- (\x,-2);
					\draw[line width=1.5pt,black,{Latex[length=2mm]}-] (\x+0.5,-1.5) -- (\x,-2);
					\draw[line width=1.5pt,blue,{Latex[length=2mm]}-]  (\x,-2) -- (\x-0.5,-2.5);
					\draw[line width=1.5pt,blue,{Latex[length=2mm]}-] (\x,-2) -- (\x+0.5,-2.5);
					\draw[line width=1pt,black!50!white,{Latex[length=2mm]}-] (\x-0.5,-2.5) -- (\x,-3);
					\draw[line width=1pt,black!50!white,{Latex[length=2mm]}-] (\x+0.5,-2.5) -- (\x,-3);
				}
				\foreach \x in {2}
				{
					\draw[line width=1.5pt,blue,{Latex[length=2mm]}-]  (\x,0) -- (\x-0.5,-0.5);
					\draw[line width=1.5pt,black,{Latex[length=2mm]}-] (\x-0.5,-0.5) -- (\x,-1);
					\draw[line width=1.5pt,blue,{Latex[length=2mm]}-]  (\x,-1) -- (\x-0.5,-1.5);
					\draw[line width=1.5pt,blue!50!white,{Latex[length=2mm]}-] (\x,-1) -- (\x+0.5,-1.5);
					\draw[line width=1.5pt,black,{Latex[length=2mm]}-] (\x-0.5,-1.5) -- (\x,-2);
					\draw[line width=1pt,black!50!white,{Latex[length=2mm]}-] (\x+0.5,-1.5) -- (\x,-2);
					\draw[line width=1.5pt,blue,{Latex[length=2mm]}-]  (\x,-2) -- (\x-0.5,-2.5);
					\draw[line width=1pt,blue!50!white,{Latex[length=2mm]}-] (\x,-2) -- (\x+0.5,-2.5);
					\draw[line width=1pt,black!50!white,{Latex[length=2mm]}-] (\x-0.5,-2.5) -- (\x,-3);
					\draw[line width=1pt,black!50!white,{Latex[length=2mm]}-] (\x+0.5,-2.5) -- (\x,-3);
				}
				
				\draw[line width=1pt,black!50!white,{Latex[length=2mm]}-] (2.5,-0.5)--(2,-1);
				\draw[line width=1pt,blue!50!white,{Latex[length=2mm]}-] (2,0)--(2.5,-0.5);
				\draw [fill=blue] (2.5,-1.5) circle (2pt);
				\draw [fill=blue] (2.5,-2.5) circle (2pt);
				\draw[line width=1.5pt,red,{Latex[length=2mm]}-] (-0.5,-1.5) -- (-0.5,-0.5);
				\draw[line width=1pt,red!50!white,{Latex[length=2mm]}-] (-0.5,-3.5) -- (-0.5,-2.5);
				\node at (2.7,-0.5) {$y_1$};
				\node at (2.7,-1.5) {$y_2$};
				\node at (2.7,-2.5) {$y_3$};
				\node at (-0.5,-3.8) {${z_0}$};
				\node at (0,-3.3) {$w_1$};
				\node at (1,-3.3) {$w_2$};
				\node at (2,-3.3) {$w_3$};
				\foreach \x in {0,1,2}
				{
					\draw [fill=blue] (\x,0) circle (2pt);
					\draw [fill=blue] (\x,-1) circle (2pt);
					\draw [fill=blue] (\x,-2) circle (2pt);
					\draw [fill=blue] (\x-0.5,-0.5) circle (2pt);
					\draw [fill=blue] (\x-0.5,-1.5) circle (2pt);
					\draw [fill=blue] (\x-0.5,-2.5) circle (2pt);
				}
				\foreach \x in {0,1,2}
				{\draw [fill=white] (2.5,-\x-0.5) circle (2pt);
					\draw [fill=white] (\x,-3) circle (2pt);}
				\draw [fill=white] (-0.5,-3.5) circle (2pt);
				\begin{scriptsize}
					\node at (2,0.2) {$(1,6)$};
					\node at (0.5,-2.2) {$(3,3)$};
					\node at (1.5,-1.2) {$(2,5)$};
					\node at (-0.5,-2.2) {$(3,1)$};
					\node at (0,-0.7) {$(2,2)$};
				\end{scriptsize}
			\end{tikzpicture}
			\vspace{-0.4cm}
			\caption{$\mathcal{K}_{k,T}'$}
		\end{subfigure}
		\caption{Two domains $\mathcal{K}_{k,T}$ and $\mathcal{K}_{k,T}'$ are shown in (A) and (B) with $k=3$, $T=4$ and boundary conditions $(\vec{y},\vec{z})$ and  $(\vec{y},\vec{w})$  respectively. They include all the vertices within the gray dashed box as well some labels for the points. The directed edges with lighter colors are edges connecting vertices from $\Lambda$ to $\partial\Lambda$ or viceversa ($\Lambda=\mathcal{K}_{k,T}$ or $\Lambda=\mathcal{K}_{k,T}'$).
The boundary variable $z_0$ does not actually play any role in the density of the corresponding $\hslg$ Gibbs measure  after normalizing it to be a probability density. This point is explained in the proof of Lemma \ref{def:hsgm1}.}
		\label{fig4}
	\end{figure}
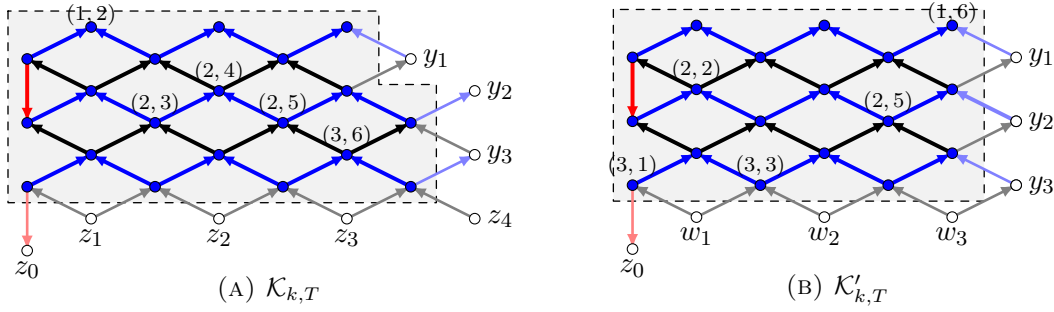

	\begin{observation}[One-sided boundary Gibbs property]\label{def:hsgm1}

		Fix $k,T\in \Z_{\ge2}$ and $\alpha>-\theta$. Fix $\vec{y}\in \R^k$, $\vec{z}\in \R^{T}$, and $\vec{w}\in \R^{T-1}$.
		\begin{enumerate}[label=(\alph*),leftmargin=15pt]
			\item \label{deff}	The $\hslg$ Gibbs measure on the domain $\mathcal{K}_{k,T}$ with boundary condition $(\vec{y},\vec{z})$ is a probability measure on $\R^{|\mathcal{K}_{k,T}|}$ whose density at  $\mathbf{u}=(u_{i,j})_{(i,j)\in \mathcal{K}_{k,T}}$ is proportional to
			\begin{align}\label{def:fhsgm}
				f_{k,T}^{\vec{y},\vec{z}}(\mathbf{u}) & :=\prod_{i=1}^k\left( e^{(-1)^{i}\alpha u_{i,1}} \prod_{j=1}^{T-\ind_{i=1}} W(u_{i+1,2j};u_{i,2j+1},u_{i,2j-1})\hspace{-0.2cm}\prod_{j=1}^{2T-1-\ind_{i=1}}\hspace{-0.3cm}G_{\theta,(-1)^{j+1}}(u_{i,j}-u_{i,j+1})\right)
			\end{align}
			where  $W(a;b,c)$ and $G_{\theta,(-1)^m}(y)$ are defined in \eqref{def:wfns}  and \eqref{def:gwt} respectively.
			Here $u_{k+1,2j}=z_{j}$ for each $j\in \ll1,T\rr$, $u_{1,2T-1}=y_1$, and $u_{i,2T}=y_i$, $u_{i,2T+1}:=+\infty$ (so that the factor $\exp(-e^{u_{i+1,2T}-u_{i,2T+1}})=1$) for each $i\in \ll2,k\rr$. \item \label{denQ}	The $\hslg$ Gibbs measure on the domain $\mathcal{K}_{k,T}'$ with boundary condition $(\vec{y},\vec{w})$ is a probability measure on $\R^{|\mathcal{K}_{k,T}'|}$ whose density at  $\mathbf{u}=(u_{i,j})_{(i,j)\in \mathcal{K}_{k,T}'}$ is proportional to
			\begin{align}\label{def:fhsgm2}
				Q_{k,T}^{\vec{y}',\vec{z}}(\mathbf{u}) & :=\prod_{i=1}^k\left( e^{(-1)^{i}\alpha u_{i,1}} \prod_{j=1}^{T-1} W(u_{i+1,2j};u_{i,2j+1},u_{i,2j-1})\prod_{j=1}^{2T-2}G_{\theta,(-1)^{j+1}}(u_{i,j}-u_{i,j+1})\right).
			\end{align}
			Here $u_{k+1,2j}=w_{j}$ for each $j\in \ll1,T-1\rr$, and $u_{i,2T-1}=y_i$ for each $i\in \ll1,k\rr$.
		\end{enumerate}
	\end{observation}

\begin{proof} We refer to Figure \ref{fig4} for a visual representation of the above measures. Recall the edge weights from \eqref{def:wfn}. The blue edges in the figure corresponds to $G_{\theta,(-1)^{j+1}}(\cdot)$ factors that appear in \eqref{def:fhsgm} and \eqref{def:fhsgm2}. The $(-1)^{j+1}$ factor is due to the alternate switching of the direction of blue weights as we read off from left to right. Here we have obtained the $G$ function from the blue edge weights by multiplying by a constant. This is done so that  the $G$ function becomes density (i.e., integrates to $1$), a fact that will be useful in the later analysis. The black edge weights from \eqref{def:wfn} corresponds to the $W$ factor in \eqref{def:fhsgm} and \eqref{def:fhsgm2}. Finally the red edge weights are of type $e^{-\alpha u_{2i-1,1}-u_{2i,1}}=e^{-\alpha u_{2i-1,1}}\cdot e^{\alpha u_{2i,1}}$. Note that only for odd $k$ we have $(k+1,1) \in \partial\mathcal{K}_{k,T}, \partial\mathcal{K}_{k,T}'$. In that case, the factor $e^{-\alpha u_{k+1,1}}$ can be absorbed into the proportionality constant. Thus, overall, the red weights contributes the factor $\prod_{i=1}^k e^{(-1)^i\alpha_{i,1}}$ in the above densities. This also explains why the $z_0$ value does not play any role in the definition of these densities.
	\end{proof}

\begin{definition}	\label{def:hslgfmeas}
We will mostly be concerned with the $\hslg$ Gibbs measure on $\mathcal{K}_{k,T}$ with boundary condition $(\vec{y},\vec{z})$ (see Lemma \ref{def:hsgm1}\ref{deff} for the probability density of this measure).
We will denote the probability and the expectation operator under this law as $\psa$ and $\esa$ respectively and a random variable with this law by $L:=\big(L(i,j):=L_i(j) : (i,j)\in \mathcal{K}_{k,T} \big)$. We may  drop $\alpha$ and write $\ps$ and $\es$ when clear from the context.
	\end{definition}

We now define the $\hslg$ Gibbs measure on $\mathcal{K}_{k,T}$  with boundary condition $\vec{y}\in \R^k$, $\vec{z}:=(-\infty)^T$.
	
	\begin{definition}[Bottom-free Gibbs measure] \label{def:btf} The \textit{bottom-free} measure on the domain $\mathcal{K}_{k,T}$ with boundary condition $\vec{y}$ is the $\hslg$ Gibbs measure on the domain $\mathcal{K}_{k,T}$ with boundary condition $(\vec{y},(-\infty)^T)$. By Lemma \ref{obs2.5} this the corresponding density $f_{k,T}^{\vec{y},(-\infty)^T}$ is integrable when $k$ is even and $\alpha\in \R$ (in that case the measure does not even depend on $\alpha$) or when $k$ is odd and $\alpha\in (-\theta,\theta)$.
In this case the \btf\ measure can be normalized to a probability measure so that for $\vec{z}\in \R^T$
		\begin{align}\label{e2gib}
			\psa(\m{A})=\frac{1}{V_{k}^T(\vec{y},\vec{z})}\Ex_{\alpha}^{\vec{y};(-\infty)^{T};k,T}\left[\ind_{\m{A}}\cdot \prod_{j=1}^{T} W(z_{2j};L_{k}(2j+1),L_{k}(2j-1))\right],
		\end{align}
for any event $\m{A}$, where we set $L_k(2T+1)=+\infty$ and the normalization is given by
		\begin{align}\label{defvkt}
			V_{k}^T(\vec{y},\vec{z}):=\Ex_{\alpha}^{\vec{y};(-\infty)^{T};k,T}\left[\prod_{j=1}^{T} W(z_{2j};L_{k}(2j+1),L_{k}(2j-1))\right].
		\end{align}
In other words, we can build the full Gibbs measure $\psa$ by reweighting the bottom-free measure by a Radon-Nikodym derivative given by the expression (except $\ind_{\m{A}}$) inside the expectation in \eqref{e2gib}, normalized by dividing by $V_{k}^T(\vec{y},\vec{z})$.
	\end{definition}

Besides one-sided conditioning as in Lemma \ref{obs1}, we can also use the Gibbs property when conditioning on boundary data on both sides as is standard in full-space discrete line ensembles \cite{dff,bcd,xd1}. We record here one such result that will be useful in our later proofs.
	
	\begin{observation}[Two-sided boundary Gibbs property] \label{i2a}
		Fix $1\le T_1 <T_2-1$. Suppose $L$ is distributed as $\mathbb{P}^{\vec{y},\vec{z};1,T_2}$. Let $(X(j))_{j=T_1-1}^{T_2-1}$ be a random bridge from $X(T_1-1)=a$ to $X(T_2-1)=b$ with i.i.d.~increments from the density $\fa$ defined in \eqref{def:faga}. The law of $\big(L_1(2j+1) : T_1 \le j \le T_2-2\big)$ conditioned on $\{L_1(2T_1-1)=a, L_1(2T_2-1)=b\}$ is absolutely continuous with respect to the law of $(X(j))_{j=T_1}^{T_2-2}$ with Radon-Nikodym derivative proportional to $$\til{W}:= \exp\bigg(-\sum_{j=T_1}^{T_2-1} (e^{z_{j}-X(j)}+e^{z_{j}-X(j-1)})\bigg).$$
	\end{observation}
	
\begin{proof} We utilize the form of the Gibbs measure density given in \eqref{def:fhsgm}.	The $G_{\theta,1}\ast G_{\theta,-1}$ function appears in the statement of Lemma \ref{i2a} as we focus on the marginal distribution of the odd points only and hence we integrate out the dummy variables on the even points (see Figure \ref{fig5}).
\end{proof}
	
	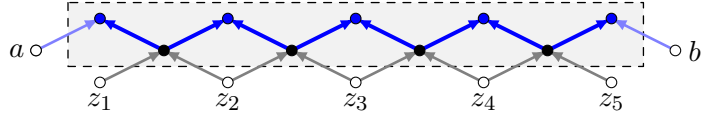
\begin{figure}[h!]
		\centering
		\begin{tikzpicture}[line cap=round,line join=round,>=triangle 45,x=1.7cm,y=0.85cm]
			\draw[fill=gray!10,line width=0.5pt,dashed] (-0.25,0.25)--(-0.25,-0.75)--(4.25,-0.75)--(4.25,0.25)--(-0.25,0.25);
			\foreach \x in {0,1,2,3}
			{
				\draw[line width=1.5pt,blue,{Latex[length=2mm]}-]  (\x+1,0) -- (\x+0.5,-0.5);
				\draw[line width=1.5pt,blue,{Latex[length=2mm]}-] (\x,0) -- (\x+0.5,-0.5);
				\draw[line width=1pt,black!50!white,{Latex[length=2mm]}-] (\x+0.5,-0.5) -- (\x+1,-1);
				\draw[line width=1pt,black!50!white,{Latex[length=2mm]}-] (\x+0.5,-0.5) -- (\x,-1);
			}
			\draw[line width=1pt,blue!50!white,{Latex[length=2mm]}-] (4,0)--(4+0.5,-0.5);
			\draw[line width=1pt,blue!50!white,{Latex[length=2mm]}-] (0,0)--(-0.5,-0.5);
			\node at (4.65,-0.5) {$b$};
			\node at (-0.65,-0.5) {$a$};
			\foreach \x in {1,2,3,4,5}
			{
				\draw [fill=blue] (\x-1,0) circle (2pt);
				\draw [fill=black] (\x-1.5,-0.5) circle (2pt);
				\node at (\x-1,-1.3) {$z_{\x}$};
			}
			\foreach \x in {0,1,2,3,4}
			{\draw [fill=white] (\x,-1) circle (2pt);}
			\draw [fill=white] (4.5,-0.5) circle (2pt);
			\draw [fill=white] (-0.5,-0.5) circle (2pt);
		\end{tikzpicture}
		\caption{The marginal distribution of the odd (black) points of the $\hslg$ Gibbs measure shown above with $T_1=1, T_2=6$ is described in Lemma \ref{i2a}.}
		\label{fig5}
	\end{figure}

As with full-space line ensemble Gibbs measures \cite{ch14,ch16,wu, bcd,xd1}, the $\hslg$ Gibbs measures satisfy \textit{stochastic monotonicity} with respect to the boundary data. The following, stated for the inhomogeneous $\Theta$-Gibbs measures provides a grand monotone coupling over all boundary data.

	\begin{proposition}[Stochastic monotonicity]\label{p:gmc} Fix $k_1\le k_2$, $a_i\le b_i$ for $k_1\le i\le k_2$. Fix $\Theta:=\{\vartheta_{i,j}>0 : (i,j)\in \Z_{\ge 1}^2\}$, and $\alpha>-\min\{\theta_{i,j} : (i,j)\in \Z_{\ge 1}^2\}$. Let
		\begin{align*}
			\Lambda:=\{(i,j): k_1\le i\le k_2, a_i\le j\le b_i\}.
		\end{align*}
		There exists a probability space that supports a collection of random variables
		\begin{align*}
			\big(L(v;(u_w)_{w\in \partial \Lambda}) : v\in \Lambda, (u_w)_{w\in\partial\Lambda} \in \R^{|\partial\Lambda|}\big)
		\end{align*}
		such that
		\begin{enumerate}
			\item For each $(u_w)_{w\in\partial\Lambda}\in \R^{|\partial\Lambda|}$,  the marginal law of $\big(L(v;(u_w)_{w\in \partial \Lambda}): v\in \Lambda\big)$ is given by the $\hslg$ $\Theta$-Gibbs measure for the domain $\Lambda$ with boundary condition $(u_w)_{w\in\partial\Lambda}\in \R^{|\partial\Lambda|}$.
			\item With probability $1$, for all $v\in \Lambda$ we have
$$L\big(v;(u_w)_{w\in \partial \Lambda}\big) \le L\big(v;(u_w')_{w\in \partial \Lambda}\big) \mbox{ whenever }u_w\le u_w' \mbox{ for all }w\in \partial\Lambda.$$
Consequently, the probability of increasing events (defined in Definition \ref{def:hslggibbs}) increase with  respect to decreasing the boundary conditions.
		\end{enumerate}
	\end{proposition}

The above proposition is stated for the general $\hslg$ $\Theta$-Gibbs measure introduced in Definition \ref{def:hslggibbs}. In light of the second part of the above proposition, we will sometimes say that an increasing event is `increasing with respect to the boundary data'. The proof of the above proposition follows a similar strategy as in \cite{bcd,xd1} and is provided in Appendix \ref{appc} for completeness.

	\subsection{The $\hslg$ line ensemble and Proof of Theorem \ref{thm:conn}} \label{sec2.2}	
In this section we define the half-space log-gamma ($\hslg$) line ensemble and prove Theorem \ref{thm:conn}. We work with the inhomogeneous polymer model determined by parameters $\vec{\theta}:=(\theta_i)_{i\in \mathbb{Z}_{\ge 1}}$. The construction of the line ensemble is based on the multi-path point-to-point partition functions. These are defined in \eqref{zsymr} as sums over multiple non-intersecting paths on the full quadrant $\Z_{\ge 1}^2$ (not just half-quadrant) of products of the symmetrized versions of the weights from \eqref{eq:wt}:
	\begin{align}
		\label{eq:symwt}
		\til{W}_{i,j} \sim \begin{cases}
		\tfrac{1}{2}W_{i,j} & \mbox{when }i=j, \\
			W_{j,i} & \mbox{when }j>i, \\
			W_{i,j} & \mbox{when }j<i.
		\end{cases}
	\end{align}
For $m,n,r\in \Z_{\ge 1}$ with $n\ge r$, let $\Pi_{m,n}^{(r)}$ be the set of $r$-tuples of non-intersecting upright paths in $\Z_{\ge 1}^2$ starting from $(1,r), (1,r-1),\cdots, (1,1)$ and going to $(m,n),(m,n-1), \ldots, (m,n-r+1)$ respectively. We define the multipath point-to-point symmetrized partition function as
	\begin{align}\label{zsymr}
		Z_{\operatorname{sym}}^{(r)}(m,n):=\sum_{(\pi_1,\ldots,\pi_r)\in \Pi_{m,n}^{(r)}}  \prod_{(i,j)\in \pi_1\cup \cdots \cup \pi_r} \til{W}_{i,j},
	\end{align}
with the convention that $Z_{\operatorname{sym}}^{(0)}(m,n)\equiv 1$ for all $m,n\in \Z_{\ge 1}$. The dependence on the $\vec{\theta}:=(\theta_i)_{i\in \mathbb{Z}_{\ge 1}}$ parameters that determine the weights through \eqref{eq:wt} is suppressed here and below.

	\begin{figure}[h!]
		\centering
		\begin{tikzpicture}[line cap=round,line join=round,>=triangle 45,x=1cm,y=1cm]
			\draw [line width=1pt] (7,4.78)-- (8,5);
			\draw [line width=1pt] (8,5)-- (9,3.7);
			\draw [line width=1pt] (9,3.7)-- (10,4.4);
			\draw [line width=1pt] (7,3.5)-- (8,4);
			\draw [line width=1pt] (7,7.6)-- (8,8);
			\draw [line width=1pt] (8,8)-- (9,6.9);
			\draw [line width=1pt] (9,6.9)-- (10,7.6);
			\draw [line width=1pt] (7,6)-- (8,6.5);
			\draw [line width=1pt] (8,6.5)-- (9,5.5);
			\draw [line width=1pt] (9,5.5)-- (10,6);
			\draw [line width=1pt] (10,6)-- (11,4.8);
			\draw [line width=1pt] (11,4.8)-- (12,5.5);
			\draw [line width=1pt] (7,9)-- (8,9.5);
			\draw [line width=1pt] (8,9.5)-- (9,8.6);
			\draw [line width=1pt] (9,8.6)-- (10,9.3);
			\draw [line width=1pt] (10,7.6)-- (11,6.6);
			\draw [line width=1pt] (11,6.6)-- (12,7.2);
			\draw [line width=1pt] (12,7.2)-- (13,5.8);
			\draw [line width=1pt] (13,5.8)-- (14,6.4);
			\draw [line width=1pt] (10,9.3)-- (11,8.3);
			\draw [line width=1pt] (11,8.3)-- (12,8.8);
			\draw [line width=1pt] (12,8.8)-- (13,8.1);
			\draw [line width=1pt] (13,8.1)-- (14,8.4);
			\draw [line width=1pt] (14,8.4)-- (15,7.3);
			\draw [line width=1pt] (15,7.3)-- (16,7.9);
			\draw [line width=1pt] (16,7.9)-- (17,6.9);
			\draw [line width=1pt] (17,6.9)-- (18,7);
			\draw [line width=1pt] (14,6.4)-- (15,5.7);
			\draw [line width=1pt] (15,5.7)-- (16,6.2);
			\draw [line width=1pt] (12,5.5)-- (13,6.3);
			\draw [line width=1pt] (13,6.3)-- (14,5);
			\draw [line width=1pt] (10,4.4)-- (11,3.9);
			\draw [line width=1pt] (11,3.9)-- (12,4.3);
			\draw [line width=1pt] (8,4)-- (9,4.2);
			\draw [line width=1pt] (9,4.2)-- (10,3.2);
			\draw [line width=1pt] (7,2.6)-- (8,2.7);
			\begin{scriptsize}
				\draw [fill=black] (17,6.9) circle (2pt);
				\draw [fill=blue] (18,7) circle (3pt);
				\draw [fill=blue] (7,2.6) circle (3pt);
				\draw [fill=blue] (8,2.7) circle (3pt);
				\draw [fill=black] (9,4.2) circle (2pt);
				\draw [fill=blue] (10,3.2) circle (3pt);
				\draw [fill=black] (11,3.9) circle (2pt);
				\draw [fill=blue] (12,4.3) circle (3pt);
				\draw [fill=blue] (14,5) circle (3pt);
				\draw [fill=black] (13,6.3) circle (2pt);
				\draw [fill=black] (15,5.7) circle (2pt);
				\draw [fill=blue] (16,6.2) circle (3pt);
				\draw [fill=black] (7,4.78) circle (2pt);
				\draw [fill=black] (7,9) circle (2pt);
				\draw [fill=black] (7,7.6) circle (2pt);
				\draw [fill=black] (7,6) circle (2pt);
				\draw [fill=black] (7,4.78) circle (2pt);
				\draw [fill=black] (7,3.5) circle (2pt);
				\draw [fill=black] (8,4) circle (2pt);
				\draw [fill=black] (8,5) circle (2pt);
				\draw [fill=black] (9,3.7) circle (2pt);
				\draw [fill=black] (10,4.4) circle (2pt);
				\draw [fill=black] (8,6.5) circle (2pt);
				\draw [fill=black] (9,5.5) circle (2pt);
				\draw [fill=black] (10,6) circle (2pt);
				\draw [fill=black] (11,4.8) circle (2pt);
				\draw [fill=black] (12,5.5) circle (2pt);
				\draw [fill=black] (8,8) circle (2pt);
				\draw [fill=black] (8,9.5) circle (2pt);
				\draw [fill=black] (9,6.9) circle (2pt);
				\draw [fill=black] (9,8.6) circle (2pt);
				\draw [fill=black] (10,7.6) circle (2pt);
				\draw [fill=black] (10,9.3) circle (2pt);
				\draw [fill=black] (11,6.6) circle (2pt);
				\draw [fill=black] (12,7.2) circle (2pt);
				\draw [fill=black] (14,6.4) circle (2pt);
				\draw [fill=black] (13,5.8) circle (2pt);
				\draw [fill=black] (12,8.8) circle (2pt);
				\draw [fill=black] (14,8.4) circle (2pt);
				\draw [fill=black] (16,7.9) circle (2pt);
				\draw [fill=black] (11,8.3) circle (2pt);
				\draw [fill=black] (13,8.1) circle (2pt);
				\draw [fill=black] (15,7.3) circle (2pt);
			\end{scriptsize}
			\foreach \x in {0,1,2,...,11}
			{
				\coordinate (A\x) at ($(7,2)+(\x*1cm,0)$) {};
				\draw ($(A\x)+(0,5pt)$) -- ($(A\x)-(0,5pt)$);
			}
			\draw[thick] (A0) -- (A11);
			\node at (7,1.5) {$1$};
			\node at (18,1.5) {$2N$};
			\node at (6.3,9) {$\L_{1}(\cdot)$};
			\node at (6.3,7.6) {$\L_{2}(\cdot)$};
			\node at (6.3,5.6) {$\vdots$};
			\node at (6.3,2.6) {$\L_{N}(\cdot)$};
		\end{tikzpicture}
		\caption{The half-space log gamma line ensemble $\L=(\L_i(\cdot))_{i=1}^N$ ($N=6$ in above figure). Each curve $\L_i(\cdot)$ has $2N-2i+2$ many coordinates. $\Lambda_N^*$ in Theorem \ref{thm:conn} is the set of all black points in the above figure. Theorem \ref{thm:conn} tells us that conditioned on the blue points, the law of the black points is given by the $\hslg$ Gibbs measures.}
		\label{fig:hgibbs}
	\end{figure}
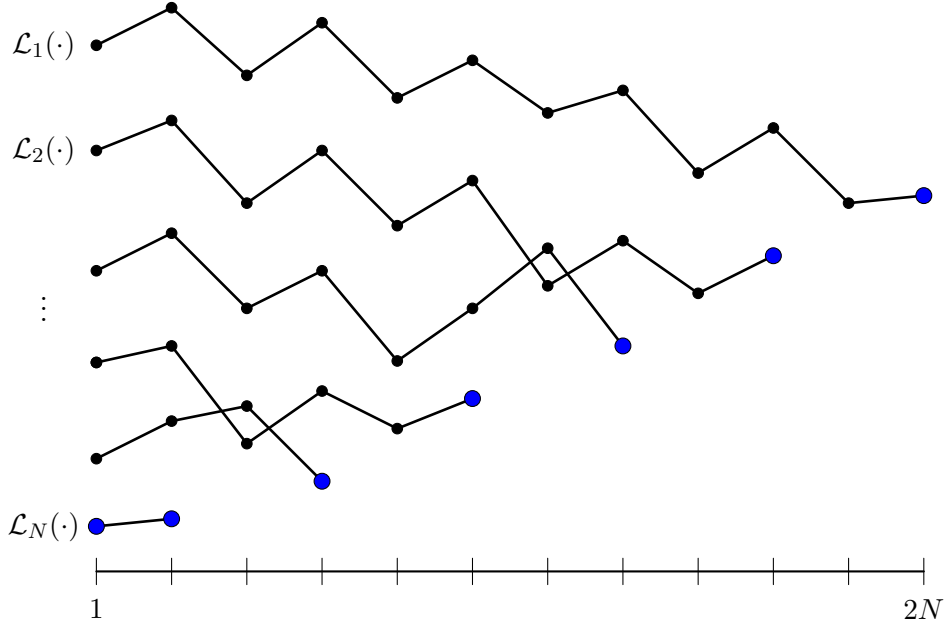	
	
	\begin{definition}[Half-space log-gamma line ensemble] \label{l:nz}
		Fix $N\in \Z_{\geq 1}$. For each $i\in \ll1,N\rr$ and $j\in \ll1,2N-2i+2\rr$, we set		
		\begin{equation*}
			\L_i^N(j)=\log\left(\frac{2 Z_{\operatorname{sym}}^{(i)}(p,q)}{ Z_{\operatorname{sym}}^{(i-1)}(p,q)}\right)+2\Psi(\theta)N.
		\end{equation*}
		where $p:=N+\lfloor j/2 \rfloor$ and $q:=N-\lceil j/2\rceil+1$. We call the collection of random variables
		\begin{align*}
			\big(\L_i^N(j): i\in \ll1,N\rr, j\in \ll1,2N-2i+2\rr\big)
		\end{align*}
the half-space log-gamma ($\hslg$) line ensemble with parameters $(\alpha,\theta)$, see Figure \ref{fig:hgibbs}.
	\end{definition}
	\begin{proof}[Proof of Theorem \ref{thm:conn}] Recalling the convention $Z_{\operatorname{sym}}^{(0)}(m,n)\equiv 1$, we can write
$$\L_1^N(j)=\log\left(2Z_{\operatorname{sym}}^{(1)}(N+\lfloor j/2\rfloor,N-\lceil j/2\rceil+1)\right)+2\Psi(\theta)N.$$
Assuming Part \ref{i02} of Theorem \ref{thm:conn} (verified below), Part \ref{i01} follows immediately from the identity
$2Z_{\operatorname{sym}}^{(1)}(p,q)=Z_{(\alpha,\vec\theta)}(p,q)$. The above identity is noted in Section 2.1 of \cite{bw} and follows easily due to symmetry of the weights (the factor of $2$ comes from a lack of double-counting the weight at $(1,1)$). This is an equality (not just in distribution).

Part \ref{i02} is a highly non-trivial deduction from first principles. However, the works of  \cite{cosz,osz, nz, bz19b, bw} have built a rich theory using the geometric RSK correspondence from which this part follows in a rather straightforward manner, as now described. We seek to determine the joint density of the $\hslg$ line ensemble defined above.
Let us start by defining
		$$K_N:=\{(i,j): i\in \ll1,N\rr, j\in \ll1,2N-2i+2\rr\}, \quad\ihs:=\{(i,j)\in \Z_{\geq 1}^2 :   i+j\le 2N+1\}.$$
Note that the map $(i,j) \mapsto (N+\lfloor j/2\rfloor-i+1,N-\lceil j/2\rceil -i+2)$ is a bijection from $K_N$ to $\ihs\cap \{i\ge j\}.$ For any $(i,j)\in K_N$, we then define
		$$T_{N+\lfloor j/2\rfloor-i+1,N-\lceil j/2\rceil -i+2}:=\frac{Z_{\operatorname{sym}}^{(i)}(N+\lfloor j/2\rfloor,N-\lceil j/2\rceil +1)}{Z_{\operatorname{sym}}^{(i-1)}(N+\lfloor j/2\rfloor,N-\lceil j/2\rceil +1)},$$
and then set $T_{j,i}:=T_{i,j}$ for $i\ge j$. From Proposition 2.6 in \cite{nz}, $(T_{i,j})_{(i,j)\in \ihs}$ is precisely the image under the geometric RSK map of the symmetrized weights \eqref{eq:symwt} with indices restricted to the $\ihs$ array. The density of this image has been computed in \cite{bw}. Indeed, setting $m=0, n=N$, $\alpha_i=\theta_i$ for $i\ge 1$, and $\alpha_0=\alpha$ in the final two (unnumbered) equations on page 28 in \cite{bw} (in the arXiv version see the second unnumbered equation on page 20), we see that the density of $(2T_{i,j})_{i\ge j}$ at $(t_{i,j})_{i\ge j}$ is proportional to
\begin{equation}
		\label{e:den1}
{e^{-\frac1{t_{1,1}}}\prod_{i=1}^{N} t_{i,i}^{(-1)^{N-i+1}\alpha}}\prod_{j=1}^N\Big(\frac{\tau_{2N-2j+2}\cdot \tau_{2N-2j}}{\tau_{2N-2j+1}^2}\Big)^{\theta_j} \!\!\exp\Big(\!-\!\sum_{i\ge j >1}\frac{t_{i,j-1}}{t_{i,j}} -\sum_{i>j}\frac{t_{i-1,j}}{t_{i,j}}\Big)\!\!\prod_{(i,j)\in \ihs} t_{i,j}^{-1}\ind_{t_{i,j}>0}
	\end{equation}
		where the $\tau$ variables are defined as
		$
			\tau_k=\prod \big( t_{i,j}:(i,j)\in \ihs,i-j=k\big) =\prod\big( t_{i+k,i}: 1\le i \le N - \frac{k-1}{2}\big).
		$
In fact, the density formula in \cite{bw} is for $(2T_{i,j})_{i\le j}$ at $(t_{i,j})_{i\le j}$, thus we needed to permute the indices in that formula to arrive at the above formula. The line ensemble $\L_i^N(j)$ defined in Definition \ref{l:nz} is related to $(2T_{i,j})_{(i,j)\in \ihs}$ via the relation
	\begin{align*}
		\L_i^N(j)-2\Psi(\theta)N = \log\big( T_{N+\lfloor j/2\rfloor-i+1,N-\lceil j/2 \rceil-i+2}\big).
	\end{align*}
Under the change of variables $u_{i,j}=\log\big( t_{N+\lfloor j/2\rfloor-i+1,N-\lceil j/2 \rceil-i+2}\big)$ for $(i,j)\in K_N$, we claim that 
the density of $(\L_i^N(j)-2\Psi(\theta)N)$ at $(u_{i,j})_{(i,j)\in K_N}$ is proportional to
	\begin{align}\label{e:den2}
		& 	{e^{-e^{-u_{N,1}}}\prod_{i=1}^{N} e^{(-1)^i u_{i,1}\alpha}}\prod_{i=1}^N \left(e^{-\theta_i u_{i,2N-2i+2}}\prod_{j=1}^{N-i+1} e^{\theta_{N-j+1}(u_{i,2j-1}-u_{i,2j})} \prod_{j=1}^{N-i} e^{-\theta_{N-j+1}(u_{i,2j}-u_{i,2j+1})}\right) \\ \label{e:den3} & \hspace{1.5cm} \cdot\exp\left(-\sum_{i=1}^{N}\sum_{j=1}^{N-i+1}e^{u_{i,2j-1}-u_{i,2j}}-\sum_{i=1}^{N-1}\sum_{j=1}^{N-i}e^{u_{i+1,2j}-u_{i,2j+1}}\right) \\ \label{e:den4} & \hspace{3cm} \cdot\exp\left(-\sum_{i=1}^{N-1}\sum_{j=1}^{N-i} e^{u_{i,2j+1}-u_{i,2j}}- \sum_{i=1}^{N-1} \sum_{j=1}^{N-i} e^{u_{i+1,2j}-u_{i,2j-1}}\right).
	\end{align}
The justification of going from \eqref{e:den1} to \eqref{e:den2}-\eqref{e:den4} is given in Appendix \ref{appd}.
Recall now that we are interested in the density conditioned on $(\L_i^N(j)-2\Psi(\theta)N)$ at $(u_{i,j})_{(i,j)\in K_N\setminus \Lambda_N^*}$. To compute this conditional density we may absorb all the $u_{i,j}$ terms with $(i,j)\in K_N\setminus \Lambda_N^*$ into the proportionality constant. Thus in \eqref{e:den2}, we may absorb the $e^{-u_{N,1}}$ term and $e^{-\theta_i u_{i,2N-2i+2}}$ terms and observe $$\prod_{i=1}^{N} e^{(-1)^i u_{i,1}\alpha} \propto \prod_{i\in \ll1,N/2\rr} e^{-\alpha(u_{2i-1,1}-u_{2i,1})}.$$
	Upon a quick inspection of the form of the weight function in \eqref{def:wfn}, one sees that these factors are precisely the red edge weights functions in the $\hslg$ Gibbs measure on the domain $\Lambda_N^*$; see Figure \ref{fig00} (B) and Definition \ref{def:hslggibbs}. Combining the terms which have  $(u_{i,2j-1}-u_{i,2j})$ and $(u_{i,2j+1}-u_{i,2j})$ in \eqref{e:den2}, \eqref{e:den3}, \eqref{e:den4} give rise to the following factor
	\begin{align*}
		& \prod_{i=1}^N\prod_{j=1}^{N-i+1} \exp\big(\theta_{N-j+1}(u_{i,2j-1}-u_{i,2j}) -e^{u_{i,2j-1}-u_{i,2j}}\big) \\ & \hspace{3cm}\cdot  \prod_{i=1}^{N-1}\prod_{j=1}^{N-i} \exp\big(\theta_{N-j+1}(u_{i,2j+1}-u_{i,2j}) -e^{u_{i,2j+1}-u_{i,2j}}\big).
	\end{align*}
	The above factor corresponds to the blue edge weight functions in the $\hslg$ Gibbs measure on the domain $\Lambda_N^*$. Finally, the remaining terms in \eqref{e:den3} and \eqref{e:den4} corresponds to black edge weight function in the $\hslg$ Gibbs measure on the domain $\Lambda_N^*$. Thus the density of $\{\L_i^N(j)-2\Psi(\theta)N : (i,j)\in \Lambda_N^*\}$ conditioned on $\{\L_i^N(j)-2\Psi(\theta)N : (i,j)\in K_N\setminus \Lambda_N^*\}$ is precisely given by the $\hslg$ Gibbs measure with boundary condition $\{\L_i^N(j)-2\Psi(\theta)N : (i,j)\in K_N\setminus \Lambda_N^*\}$ as in Definition \ref{def:hslggibbs}. By the Gibbs measures translation invariance (Lemma \ref{obs1} \ref{traninv}), we obtain Theorem \ref{thm:conn} \ref{i02} .
	\end{proof}

	\section{Properties of the first three curves}\label{sec:2curve}
	
	In this section we extract probabilistic information about the first few curves of $\hslg$ line ensemble $\L^N$ (Definition \ref{l:nz}). In Section \ref{sec:cls} we prove Theorem \ref{t:order}, which claims that there is a certain high probability ordering among the points of the curve. Section \ref{sec:high} contains Theorem \ref{p:high2} which asserts that with high probability there is a point $p=O(N^{2/3})$ such that $\L_2^N(p)$ is reasonably large.
Finally in Section \ref{sec:spcu}, we show Theorem \ref{l:lhigh} which argues that with high probability $(\L_2^N(s))_{s\in \ll1,kN^{2/3}\rr}$ and $(\L_3^N(s))_{s\in \ll1,kN^{2/3}\rr}$  lie below $MN^{1/3}$ for large enough $M$.
	
	\subsection{Ordering of the points in the line ensemble} \label{sec:cls} In this subsection we show that with high probability there is ordering among the points of the $\hslg$ line ensemble. Throughout this subsection we shall assume $\alpha \in (-\theta,\infty)$ is a fixed parameter. The results can be easily extended to the case where $\alpha=\alpha(N)$ satisfying
	\begin{align*}
		-\theta< \liminf_{N\to\infty} \alpha(N) \le \limsup_{N\to\infty} \alpha(N)<\infty.
	\end{align*}
	We consider the $\hslg$ line ensemble $\L^N$ from Definition \ref{l:nz} with parameter $(\alpha,\theta)$.
	
	\begin{theorem}\label{t:order} Fix any $k\in \Z_{\ge 1}$ and $\rho\in (0,1)$. There exists $N_0=N_0(\rho,k)>0$ such that for all $N\ge N_0$, $i\in \ll 1,k\rr$ and $p\in \ll 1,N-k-2\rr$ the following inequalities holds:
		\begin{equation}\label{t41}
			\begin{aligned}
				\Pr\big(\L_i^N(2p+1)\le \L_i^N(2p)+(\log N)^{7/6}\big) & \ge 1-\rho^{N} , \\ \Pr\big(\L_i^N(2p-1)\le \L_i^N(2p)+(\log N)^{7/6}\big) & \ge 1-\rho^{N}, \\
				\Pr\big(\L_{i+1}^N(2p)\le \L_i^N(2p+1)+(\log N)^{7/6}\big) & \ge 1-\rho^{N} , \\ \Pr\big(\L_{i+1}^N(2p)\le \L_i^N(2p-1)+(\log N)^{7/6}\big) & \ge 1-\rho^{N}.
			\end{aligned}
		\end{equation}
	\end{theorem}
	We refer to the caption of Figure \ref{fig:order} for a visual interpretation of the above Theorem. The $7/6$ appearing above can be replaced with any $\gamma>1$. $N_0$ will also depend on $\gamma$ in that case. In order to prove the above theorem, we first provide an apriori loose bound for the entries of the first $k$ curves of the line ensemble $\L^N$.
	
	\begin{proposition}\label{t:aest0} Fix any $\rho\in (0,1)$ and $k\in \Z_{\ge 1}$. There exists a constant $\Con=\Con(\rho,k)>0$  and $N_0(\rho,k)>0$ such that for all $N\ge N_0$, $i\in \ll1,k\rr$, $j\in \ll 1,2N-2i+2\rr$ we have
		\begin{align}\label{e:aest0}
			\Pr\left( |\L_i^N(j)| \le \Con \cdot N\right) \ge 1-\rho^{N}.
		\end{align}
	\end{proposition}
	
	We first prove Theorem \ref{t:order} assuming Proposition \ref{t:aest0}.

	\begin{figure}[t]
		\centering
		\begin{tikzpicture}[line cap=round,line join=round,>=triangle 45,x=2.4cm,y=1.2cm]
			\foreach \x in {0,1,2}
			{
				\draw[line width=1pt,blue,{Latex[length=2mm]}-]  (\x,0) -- (\x-0.5,-0.5);
				\draw[line width=1pt,blue,{Latex[length=2mm]}-] (\x,0) -- (\x+0.5,-0.5);
				\draw[line width=1pt,black,{Latex[length=2mm]}-] (\x-0.5,-0.5) -- (\x,-1);
				\draw[line width=1pt,black,{Latex[length=2mm]}-] (\x+0.5,-0.5) -- (\x,-1);
				\draw[line width=1pt,blue,{Latex[length=2mm]}-]  (\x,-1) -- (\x-0.5,-1.5);
				\draw[line width=1pt,blue,{Latex[length=2mm]}-] (\x,-1) -- (\x+0.5,-1.5);
				\draw[line width=1pt,black,{Latex[length=2mm]}-] (\x-0.5,-1.5) -- (\x,-2);
				\draw[line width=1pt,black,{Latex[length=2mm]}-] (\x+0.5,-1.5) -- (\x,-2);
				\draw[line width=1pt,blue,{Latex[length=2mm]}-]  (\x,-2) -- (\x-0.5,-2.5);
				\draw[line width=1pt,blue,{Latex[length=2mm]}-] (\x,-2) -- (\x+0.5,-2.5);
				\draw[line width=1pt,black,{Latex[length=2mm]}-] (\x-0.5,-2.5) -- (\x,-3);
				\draw[line width=1pt,black,{Latex[length=2mm]}-] (\x+0.5,-2.5) -- (\x,-3);
				\draw[line width=1pt,blue,{Latex[length=2mm]}-]  (\x,-3) -- (\x-0.5,-3.5);
				\draw[line width=1pt,blue,{Latex[length=2mm]}-] (\x,-3) -- (\x+0.5,-3.5);
			}
			\foreach \x in {0,1,2}
			{
				\draw [fill=blue] (\x,0) circle (2pt);
				\draw [fill=blue] (\x,-1) circle (2pt);
				\draw [fill=blue] (\x,-2) circle (2pt);
				\draw [fill=blue] (\x-0.5,-0.5) circle (2pt);
				\draw [fill=blue] (\x-0.5,-1.5) circle (2pt);
				\draw [fill=blue] (\x-0.5,-2.5) circle (2pt);
				\draw [fill=blue] (\x-0.5,-3.5) circle (2pt);
				\draw [fill=blue] (\x,-3) circle (2pt);
			}
			\draw[line width=1pt,red,{Latex[length=2mm]}-] (-0.5,-1.5) -- (-0.5,-0.5);
			\draw[line width=1pt,red,{Latex[length=2mm]}-] (-0.5,-3.5) -- (-0.5,-2.5);
			\draw [fill=blue] (2.5,-0.5) circle (2pt);
			\draw [fill=blue] (2.5,-1.5) circle (2pt);
			\draw [fill=blue] (2.5,-2.5) circle (2pt);
			\draw [fill=blue] (2.5,-3.5) circle (2pt);
			\draw [fill=blue] (2,0) circle (2pt);
			\draw[fill=blue] (2,-1) circle (2pt);
			\draw[fill=blue] (2,-2) circle (2pt);
			\draw[fill=blue] (1.5,-2.5) circle (2pt);
			\draw[fill=blue] (1.5,-3.5) circle (2pt);
			\draw[fill=blue] (0.5,-3.5) circle (2pt);
			\draw[fill=blue] (-0.5,-3.5) circle (2pt);
			\draw[fill=blue] (0,-3) circle (2pt);
			\draw[fill=blue] (0,0) circle (2pt);
			\draw[fill=blue] (0,-1) circle (2pt);
			\draw[fill=blue] (-0.5,-0.5) circle (2pt);
			\draw[fill=blue] (1,0) circle (2pt);
			\node at (2.8,-0.2) {$\cdots$};
			\node at (2.8,-1.2) {$\cdots$};
			\node at (2.8,-2.2) {$\cdots$};
			\node at (2.8,-3.2) {$\cdots$};
			\node at (0,-3.7) {$\vdots$};
			\node at (1,-3.7) {$\vdots$};
			\node at (2,-3.7) {$\vdots$};
			\node at (0,0.2) {{\tiny$\L_1^N(2)$}};
			\node at (1,0.2) {{\tiny$\L_1^N(4)$}};
			\node at (1.5,-1.2) {{\tiny$\L_2^N(5)$}};
			\node at (0.5,-2.85) {{\tiny$\L_3^N(3)$}};
			\node at (-0.5,-3.72) {{\tiny$\L_4^N(1)$}};
			\node at (2,-2.73) {{\tiny$\L_4^N(6)$}};
			
		\end{tikzpicture}
		\caption{The ordering of points within the $\hslg$ line ensemble. The figure depicts of first four curves of the line ensemble $\L$. Points $v$ and $v'$ connected by a black or blue arrow from $v\to v'$ satisfy that $\L^N(v')\geq \L^N(v)+ (\log N)^{7/6}$ with high probability (recall that for $v=(i,j)$ we write $\L^N(v)=\L^N_i(j)$). The blue arrows thus imply ordering within a particular indexed curve while the black arrow imply ordering between the two consecutive curves.}
		\label{fig:order}
	\end{figure}
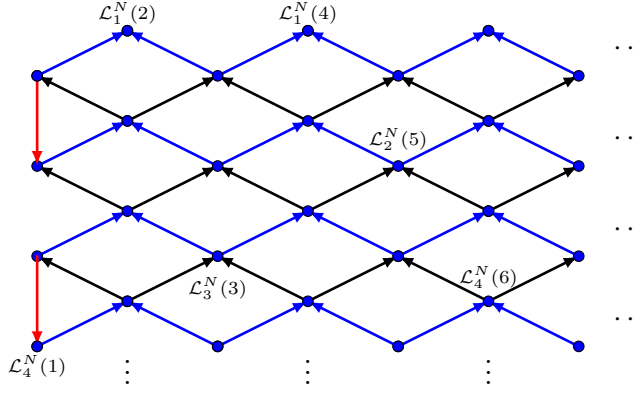

	\begin{proof}[Proof of Theorem \ref{t:order}]  Fix any $\rho \in (0,1)$ and $k\in \Z_{\ge 1}$. Set $T:=N-k$. Fix $i_0\in \ll 1,k\rr$ and $p\in\ll 1,T-2 \rr$.  We will show only the first of the inequalities in \eqref{t41}, as the rest are all proved analogously. For simplicity, we write $\L$ for $\L^N$. Consider the event
		\begin{align*}
			\m{V}:=\left\{\L_{i_0}(2p+1) \ge \L_{i_0}(2p)+(\log N)^{7/6}\right\}.
		\end{align*}
		We apply Proposition \ref{t:aest0} with $k\mapsto k+1$ and $\rho\mapsto \rho/2$ to get $\Con>0$ so that for all large enough $N$, by union bound we have $\Pr(\m{A}) \ge 1-2Nk\cdot (\rho/2)^N$	where
		\begin{align*}
			\m{A}:=\Big\{|\L_{k+1}(j)|,|\L_i(2T-1)| \le \Con \cdot N, \textrm{ for all } j\in \ll 1,2T \rr, i\in \ll 1,k\rr\Big\}.
		\end{align*}
		Thus if we consider the $\sigma$-algebra $$\mathcal{F}:=\sigma\Big(\L_{k+1}(j),\L_i(2T-1) : \ j\in \ll 1,2T \rr, i\in \ll 1,k\rr\Big),$$
		by union bound and tower property of the conditional expectation we have
		\begin{align}\label{etst}
			\Pr(\m{V})\le \Pr(\neg\m{A})+\Pr(\m{V}\cap \m{A}) \le 2Nk \cdot (\rho/2)^N+\Ex\left[\ind_{\m{A}}\Ex[\ind_{\m{V}}\mid \mathcal{F}]\right].
		\end{align}
		Recall $\mathcal{K}_{k,T}'$ from \eqref{def:kkt}. From Theorem \ref{thm:conn} and Lemma \ref{obs1} \ref{gpsd}, the law of $\{\L(v): v\in \mathcal{K}_{k,T}'\}$ conditioned on $\mathcal{F}$ is given by the $\hslg$ Gibbs measure on the domain $\mathcal{K}_{k,T}'$ with boundary condition $\vec{y}:=\{\L_j(2T-1)\}_{j=1}^k$. $\vec{z}:=\{\L_{k+1}(2i)\}_{i=1}^{T-1}$. In view of Lemma \ref{def:hsgm1} \ref{denQ} we see that
		\begin{align}\label{e:ratio}
			\Ex[\ind_{\m{V}}\mid \mathcal{F}]= \frac{\int_{\m{V}} Q_{k,T}^{\vec{y},\vec{z}}(\mathbf{u}) d\mathbf{u}}{\int_{\R^{|\mathcal{K}_{k,T}'|}} Q_{k,T}^{\vec{y},\vec{z}}(\mathbf{u}) d\mathbf{u}}
		\end{align}
		where $Q_{k,T}^{\vec{y},\vec{z}}(\mathbf{u})$ is defined in \eqref{def:fhsgm2}. We will now bound the numerator and denominator of \eqref{e:ratio} respectively. We claim that there exists $R,\tau>0$ depending only on $k,\alpha,\theta,\Con$ such that
		\begin{align}\label{toshowone}
			\ind_{\m{A}}\cdot \int_{\m{V}} Q_{k,T}^{\vec{y},\vec{z}}(\mathbf{u}) d\mathbf{u} \le 	\ind_{\m{A}}\exp(-\tfrac12e^{(\log N)^{7/6}}) \cdot  R^N, \quad \mbox{ and } \quad 	\ind_{\m{A}}\int_{\R^{|\mathcal{K}_{k,T}'|}} Q_{k,T}^{\vec{y},\vec{z}}(\mathbf{u}) d\mathbf{u} \ge 	\ind_{\m{A}}\cdot\tau^N.
		\end{align}
		Clearly plugging this bounds back in \eqref{e:ratio} and then back in \eqref{etst} leads to $\Pr(\m{V})\le \rho^N$ for all large enough $N$, as desired. Thus we focus on proving the two inequalities in \eqref{toshowone}.

		\medskip
		
		\noindent\textbf{Proof of the first inequality in \eqref{toshowone}.} Recall $G$ defined in \eqref{def:gwt}. Set 	
		$$H_{\theta,(-1)^k}(y):=e^{\frac12e^{(-1)^ky}}\cdot G_{\theta,(-1)^k}(y)=\frac1{\Gamma(\theta)}\exp({\theta (-1)^ky-\tfrac12e^{(-1)^ky}}).$$ Set $\sqrt{W}(a;b,c):=\sqrt{W(a;b,c)}$ where $W$ is defined in \eqref{def:wfns}. From \eqref{def:fhsgm2} we have
		\begin{align*}
			Q_{k,T}^{\vec{y},\vec{z}}(\mathbf{u}) & = \prod_{i=1}^k\left( e^{(-1)^{i}\alpha u_{i,1}} \prod_{j=1}^{T-1} \sqrt{W}(u_{i+1,2j};u_{i,2j+1},u_{i,2j-1})\prod_{j=1}^{2T-2}H_{\theta,(-1)^{j+1}}(u_{i,j}-u_{i,j+1})\right) \\ & \hspace{1cm} \cdot \prod_{i=1}^k\left(  \prod_{j=1}^{T-1} \sqrt{W}(u_{i+1,2j};u_{i,2j+1},u_{i,2j-1})\prod_{j=1}^{2T-2} \exp(-\tfrac12e^{(-1)^{j+1}(u_{i,j}-u_{i,j+1})})\right).
		\end{align*}
		On $\m{V}$, among the terms appearing in the last line of the above equation, the term $\exp(-\frac12e^{u_{2p+1,i_0}-u_{2p,i_0}})$ is at most $\exp(-\frac12e^{(\log N)^{7/6}})$. We bound the rest of the terms in the above last line just by $1$, so that on $\m{V}$, we have $Q_{k,T}^{\vec{y},\vec{z}}(\mathbf{u}) \le e^{-\frac12e^{(\log N)^{7/6}}}\til{Q}_{k,T}^{\vec{y},\vec{z}}(\mathbf{u})$ where
		\begin{align*}
			\til{Q}_{k,T}^{\vec{y},\vec{z}}(\mathbf{u}):=\prod_{i=1}^k\left( e^{(-1)^{i}\alpha u_{i,1}} \prod_{j=1}^{T-1} \sqrt{W}(u_{i+1,2j};u_{i,2j+1},u_{i,2j-1})\prod_{j=1}^{2T-2}H_{\theta,(-1)^{j+1}}(u_{i,j}-u_{i,j+1})\right).
		\end{align*}
		By Lemma  \ref{b2} it follows that $\int_{\R^{|\mathcal{K}_{k,T}'|}} \til{Q}_{k,T}^{\vec{y},\vec{z}}(\mathbf{u})d\mathbf{u} \le R^N$ for some $R>0$ depending on $k,\alpha,\theta$ and $\Con$ only. This verifies the first inequality in \eqref{toshowone}.
		
		\medskip

		\noindent\textbf{Proof of the second inequality in \eqref{toshowone}.} We define the event
		\begin{align*}
			\m{D}:=\bigcap_{i=1}^k\bigcap_{j=1}^{2T-2}\big\{\L_i(j)-CN-2N+2i\in [0,1]\big\}.
		\end{align*}	
		Note that on $\m{D}$, $|\L_i(1)| \le CN+2N+3$ and $\L_{i+1}(2j) \le L_i(2j+1),\L_i(2j-1)$. Hence on $\m{D}$ we have $$W(\L_{i+1}(2j);\L_i(2j+1),\L_i(2j-1))=\exp\left(-e^{\L_{i+1}(2j)-\L_i(2j+1)}-e^{\L_{i+1}(2j)-\L_i(2j-1)}\right) \ge e^{-2}.$$
		Hence on $\m{D}$ we have
		\begin{align*}
			Q_{k,T}^{\vec{y},\vec{z}}(\mathbf{u}) & =\prod_{i=1}^k\left( e^{(-1)^{i}\alpha u_{i,1}} \prod_{j=1}^{T-1} W(u_{i+1,2j};u_{i,2j+1},u_{i,2j-1})\hspace{-0.2cm}\prod_{j=1}^{2T-2}G_{\theta,(-1)^{j+1}}(u_{i,j}-u_{i,j+1})\right) \\ & \ge e^{-\alpha k(CN+2N)} e^{-2kT}\prod_{i=1}^{k}\prod_{j=1}^{2T-2}G_{\theta,(-1)^{j+1}}(u_{i,j}-u_{i,j+1}).
		\end{align*}
		Again note that on $\m{D}$, $|\L_{i}(j)-\L_{i}(j+1)| \le 2$ for all $i\in \ll 1,k\rr$ and $j\in\ll 1,2T-3\rr$, whereas on $\m{A}\cap\m{D}$,  $$\L_i(2T-2)-\L_i(2T-1) \in [0,2CN+2N].$$
		Thus, on $\m{A}\cap\m{D}$
		\begin{align*}
			Q_{k,T}^{\vec{y},\vec{z}}(\mathbf{u})  \ge e^{-\alpha k(CN+2N)-2kT}\left(\inf_{|x|\le 2} G_{\theta,1}(x)\right)^{k(2T-3)} \left(\inf_{x\in [0,2CN+2N]} G_{\theta,1}(-x)\right)^{k}.
		\end{align*}
		Note that the lower tail of $G_{\theta,1}(x)$ is exponential. Thus $\inf_{x\in [0,2CN+2N]} G_{\theta,1}(-x) \ge \tau_1^N$ for some $\tau_1>0$ depending on $\alpha,\theta,$ and $C$. Thus overall on $\m{A}\cap\m{D}$, $Q_{k,T}^{\vec{y},\vec{z}}(\mathbf{u})\ge \tau^N$ for some $\tau$ depending on $\alpha,\theta,k,$ and $C$.
		Since the Lebesgue measure of $\m{D}$ is $1$ we have
		\begin{align*}
			\ind_{\m{A}}\int_{\R^{|\mathcal{K}_{k,T}'|}} Q_{k,T}^{\vec{y},\vec{z}}(\mathbf{u})d\mathbf{u} \ge \ind_{\m{A}}\int_{\m{D}} Q_{k,T}^{\vec{y},\vec{z}}(\mathbf{u})d\mathbf{u} \ge \ind_{\m{A}}\cdot\tau^N\int_{\m{D}} d\mathbf{u} = \ind_{\m{A}}\cdot\tau^N.
		\end{align*}
		This proves the second inequality in \eqref{toshowone} completing the proof.
	\end{proof}
	
	\begin{proof}[Proof of Proposition \ref{t:aest0}]
		Recall $\L_i^N(j)$ from Definition \ref{l:nz}. Fix any $k\in \Z_{\ge 1}$ and $\rho\in (0,1)$. For all $r\in \ll1,k\rr$ and $j\in \ll1,2N-2i+2\rr$  set
		\begin{equation*}
			\begin{aligned}
				\mathcal{B}_r(j):=\sum_{i=1}^r \L_i^N(j) & = r\log 2+2r\Psi(\theta) N+\log Z_{\operatorname{sym}}^{(r)}(N+\lfloor j/2 \rfloor,N-\lceil j/2 \rceil+1),
			\end{aligned}
		\end{equation*}
		where recall $Z_{\operatorname{sym}}^{(r)}(\cdot,\cdot)$ defined in \eqref{zsymr}. Set $\mathcal{B}_0(j)\equiv 0$. We claim that there exist $\Con=\Con(\rho,k)>0$ and $N_0=N_0(\rho,k)>0$, such that for all $N\ge N_0$ and $r\in \ll1,k\rr$
		\begin{align}\label{bji}
			\Pr\left(\left|\log Z_{\operatorname{sym}}^{(r)}(N+\lfloor j/2 \rfloor,N-\lceil j/2 \rceil+1)\right|\le \Con \cdot N\right)\ge 1-\rho^N.
		\end{align}
		Setting $\Con'=\Con+2k|\Psi(\theta)|+k\log 2$ we see from above that, by triangle inequality and union bound
		$$\Pr(|\L_r(j)|\le 2\Con'\cdot N) \ge \Pr(|\mathcal{B}_{r-1}(j)|\le \Con' \cdot N) +\Pr(|\mathcal{B}_r(j)|\le \Con'\cdot N)-1 \ge 1-2\cdot \rho^N.$$
		Adjusting $\rho,N_0$ the above inequality yields \eqref{e:aest0}.  The rest of the proof is devoted in proving \eqref{bji}. \\
		
		Recall that $Z_{\operatorname{sym}}^{(r)}(\cdot,\cdot)$, defined in \eqref{zsymr}, can be viewed as sum of weights of $r$-tuple of non-intersecting paths. We first provide concentration bound for weight of a given path $\pi$ with endpoints in $\is:=\{(i,j): i+j\le 2N+1\}$ via standard Chernoff bound for i.i.d.~random variables. Then we provide an upper bound on the number of $r$-tuple of non-intersecting paths. Via union bound, this gives a concentration bound of type \eqref{bji} for $Z_{\operatorname{sym}}^{(r)}(\cdot,\cdot)$. \\
		
		Recall the symmetric weight $\til{W}_{i,j}$ from \eqref{eq:symwt}. Note that for an upright path $\pi$, $(i,j)\in \pi$ and $(j,i)\in \pi$ cannot happen simultaneously provided $i\neq j$. Thus $(\til{W}_{i,j})_{(i,j)\in \pi}$ forms an independent collection.	Set \begin{align*}
			R_1 & :=\max \{\log\Gamma(\theta)-\log\Gamma(2\theta),\log\Gamma(\alpha) -\theta\log 2-\log \Gamma(\alpha+\theta)\}, \\
			R_2 & :=\max\{\log\Gamma(3\theta)-\log\Gamma(2\theta),\log\Gamma(\alpha+2\theta) +\theta\log 2-\log \Gamma(\alpha+\theta)\}.
		\end{align*}
		Using moments of Gamma distribution and Markov inequality for each $s>0$ we have
		\begin{align*}
			\Pr\left(\sum_{(i,j)\in \pi} \log \til{W}_{i,j} \ge \tfrac{s+R_1}{\theta}|\pi| \right) & \le e^{-(s+R_1)|\pi|}\prod_{(i,j)\in \pi} \Ex[\til{W}_{i,j}^{\theta}] \\ & = e^{-(s+R_1)|\pi|}\prod_{(i,j)\in \pi,i\neq j} \frac{\Gamma(\theta)}{\Gamma(2\theta)} \prod_{(i,i)\in \pi} \frac{\Gamma(\alpha)}{2^{\theta}\Gamma(\alpha+\theta)} \le e^{-s|\pi|},
		\end{align*}
		and
		\begin{align*}
			\Pr\left(\sum_{(i,j)\in \pi} \log \til{W}_{i,j} \le -\tfrac{s+R_2}{\theta}|\pi| \right) & \le e^{-(s+R_2)|\pi|}\prod_{(i,j)\in \pi} \Ex[\til{W}_{i,j}^{-\theta}] \\ & = e^{-(s+R_2)|\pi|}\prod_{(i,j)\in \pi,i\neq j} \frac{\Gamma(3\theta)}{\Gamma(2\theta)} \prod_{(i,i)\in \pi} \frac{2^{\theta}\Gamma(\alpha+2\theta)}{\Gamma(\alpha+\theta)} \le e^{-s|\pi|}.
		\end{align*}
		This leads to the following concentration bound
		\begin{align}\label{e:conc}
			\Pr\bigg(\bigg|\sum_{(i,j)\in \pi} \log \til{W}_{i,j}\bigg| \le \tfrac{s+R_1+R_2}{\theta}|\pi| \bigg) \ge 1-2e^{-s|\pi|}.
		\end{align}
		To upgrade the above bound to \eqref{bji}, we need an upper bound for the number of $r$-tuples of non-intersecting upright paths. To do this, we introduce a few notations. Set $m:=N+\lfloor j/2\rfloor$, $n:=N-\lceil j/2\rceil +1$. Given two points $(i_1,j_1),(i_2,j_2)\in \is$, let $F_N((i_1,j_1)\to (i_2,j_2))$ be the set of all upright paths from $(i_1,j_1)$ to $(i_2,j_2)$. For any $\pi\in \Pi_{(m,n)}^{(r)}$ we have $N\le |\pi|\le 2N$. Furthermore, $|F_N((i_1,j_1)\to(i_2,j_2))|\le 4^N$ for all  $(i_1,j_1),(i_2,j_2)\in \is$. Thus $|\Pi_{(m,n)}^{(r)}| \le 4^{kN}$ as $r\le k$. Fix $s=s(\rho,k)>0$ such that $4^{kN}\cdot 2e^{-sN} \le \rho^N$ and consider the event
		\begin{align*}
			\m{A}:=\bigg\{\bigg|\log\prod_{(i,j)\in \pi_1\cup \cdots \cup \pi_r}\til{W}_{i,j} \bigg|\le \tfrac{s+R_1+R_2}{\theta}\cdot 2rN \mbox{ for all } (\pi_q)_{q=1}^r \in \Pi_{(m,n)}^{(r)}\bigg\}.
		\end{align*}
		Applying the concentration bound \eqref{e:conc} for each path in $\Pi_{(m,n)}^{(r)}$, an union bound yields
		\begin{align}\label{e:rh}
			\Pr\left(\m{A}\right) \ge 1-4^{kN}\cdot 2e^{-sN}\ge 1-\rho^N.
		\end{align}
		Next set $\Con=\Con(\rho,k):=k\log 4+\frac{s+R_1+R_2}{\theta}2k$. Note that on $\m{A}$  we have
		\begin{equation}\label{zsymup}
			\begin{aligned}
				\log Z_{\operatorname{sym}}^{(r)}(m,n) & \le \log \left(\sum_{(\pi_1,\ldots,\pi_r)\in \Pi_{(m,n)}^{(r)}} \prod_{(i,j)\in \pi_1\cup \cdots \cup \pi_r} \til{W}_{i,j}\right) \\ & \le \log \left(4^{kN}\cdot e^{\frac{s+R_1+R_2}{\theta}2rN}\right)\le kN\log 4+\tfrac{s+R_1+R_2}{\theta}2kN \le \Con \cdot N.
			\end{aligned}
		\end{equation}
		Similarly for the lower bound we consider any $(\pi_1,\ldots,\pi_r)\in \Pi_{(m,n)}^{(r)}$ which forms a disjoint collection of paths. Then on $\m{A}$ we have
		\begin{align}\label{zsymdn}
			\log Z_{\operatorname{sym}}^{(r)}(m,n) \ge \log  \left(\prod_{(i,j)\in \pi_1\cup \cdots \cup \pi_r} \til{W}_{i,j} \right)\ge -\tfrac{s+R_1+R_2}{\theta}2kN \ge -\Con\cdot N.
		\end{align}
Now	\eqref{bji} follows from \eqref{zsymup}, \eqref{zsymdn} and the bound in \eqref{e:rh}.
	\end{proof}

	\subsection{High point on the second curve}\label{sec:high}  The goal of this subsection is to show there is a point $p=O(N^{2/3})$ such that with high probability $\L_2^N(2p) \ge -\Con N^{1/3}$ where $\L^N$ is the $\hslg$ line ensemble defined in Definition \ref{l:nz}. For the rest of this section we work with the boundary parameter fixed in the critical or supercritical phase. We assume $\alpha$ equals $\alpha_1$ or $\alpha_2$ where
	\begin{align}\label{acric}
		\begin{cases}
			\alpha_1 :=\alpha_1(N):=N^{-1/3}\mu & \mbox{(Critical)} \\
			\alpha_2 :=\zeta & \mbox{(Super-Critical)}
		\end{cases}
	\end{align}
	where $\mu\in \R$ and $\zeta>0$ are fixed numbers. The labeling of the parameter might seem a bit unnatural at this moment. Essentially, when the boundary parameter is $\alpha_i$, we shall resample the top $i$ curves of the $\hslg$ line ensemble in the arguments of Section \ref{sec:rpe}.
	
	\begin{theorem}[High point on the second curve] \label{p:high2} Fix any $\e\in (0,1)$ and $k>0$. There exist $R_0(k,\e)>0$ such that for all $R\ge R_0$
		\begin{align}\label{e:thigh}
			\liminf_{N\to \infty}\Pr\left(\sup_{p\in [kN^{2/3},RN^{2/3}]}\mathcal{L}_2^N(2p) \ge -\big(\tfrac18R^2\nu+2\sqrt{R}\big)N^{1/3} \right)>1-\e.
		\end{align}
		where
		\begin{align}\label{nu}
			\nu:=\frac{(\Psi'(\theta))^2}{(-\Psi''(\theta))^{4/3}}.
		\end{align}
	\end{theorem}
	
The factor $1/8$ appearing in \eqref{e:thigh} can be replaced by any constant $\gamma>0$. $R_0$ will depend on $\gamma$ in that case. The proof of Theorem \ref{p:high2} relies on two results related to the first curve.
	
	\begin{proposition}[High point on the first curve]\label{p:high} Fix any $\e\in (0,1)$.  There exists $M_0(\e)>0$ such that for all  $M_1,M_2\ge M_0$ and $k>0$  we have
		\begin{align}\label{e:high0}
			\liminf_{N\to\infty}\Pr\left(\sup_{p\in \ll kN^{2/3},(M_1+2k)N^{2/3}\rr} \frac{\mathcal{L}_1^N(2p+1)}{N^{1/3}} +k^2\nu \le M_2 \right)>1-\e,
		\end{align}
		\begin{align}\label{e:high}
			\liminf_{N\to\infty}\Pr\left(\sup_{p\in \ll kN^{2/3},(M_1+2k)N^{2/3}\rr} \frac{\mathcal{L}_1^N(2p+1)}{N^{1/3}} +k^2\nu \ge -M_2 \right)>1-\e.
		\end{align}
		where $\nu$ is defined in \eqref{nu}.
	\end{proposition}
	Figure \ref{fig:arg2} depicts the high probability events considered in Proposition \ref{p:high}.

	\begin{figure}[h!]
		\centering
		\begin{overpic}[width=7cm]{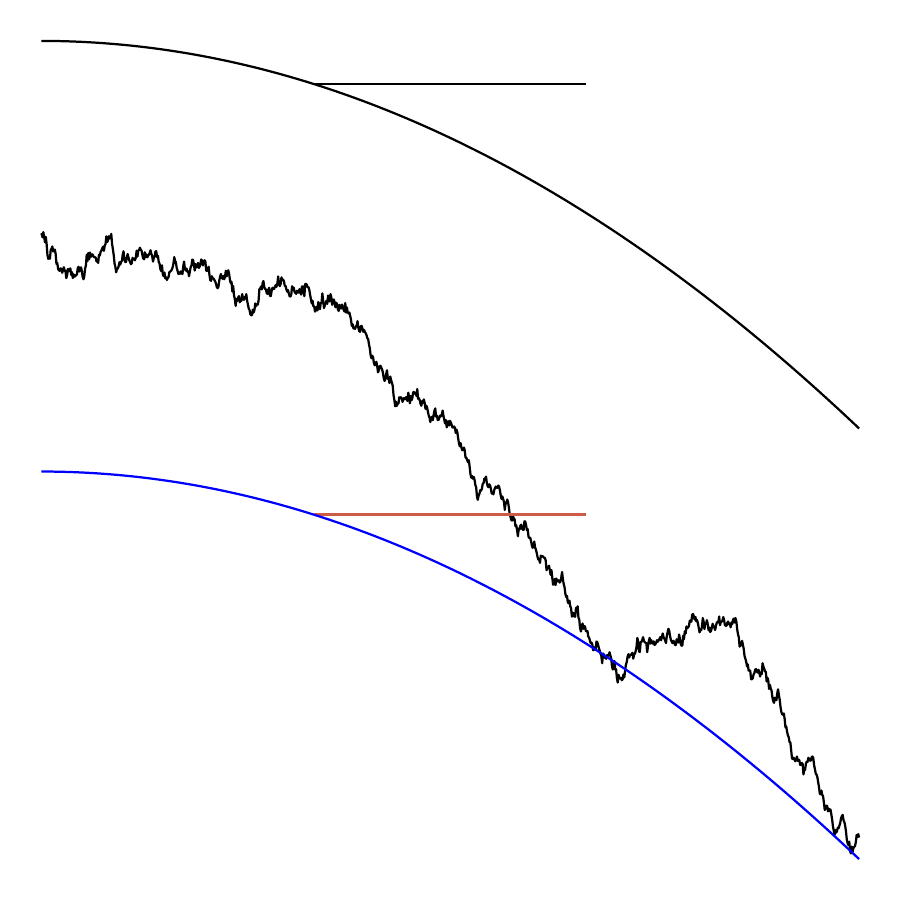}
			\put(33,10){
				\begin{tikzpicture}
					\draw[dashed] (0,1) -- (0,7.1);
			\end{tikzpicture}}
			\put(63,10){
				\begin{tikzpicture}
					\draw[dashed] (0,1) -- (0,7.1);
			\end{tikzpicture}}
			\put(33,12){
				\begin{tikzpicture}
					\draw[dashed] (0,0)--(2.1,0);
			\end{tikzpicture}}
			\put(63,68) {$f_+(\cdot)$}
			\put(53,25) {$f_-(\cdot)$}
			\put(47,15) {${\mathcal{I}_k}$}
			\put(31,8) {\tiny${kN^{2/3}}$}
			\put(50,8) {\tiny${(M_1+2k)N^{2/3}}$}
		\end{overpic}
		\vspace{0.5cm}
		\caption{Events considered in Proposition \ref{p:high}. Here $\L_1^N(2p+1)$ is given by the black rough curve. The parabolic curves $f_{\pm}(x):=-(N\nu)^{-1}x^2\pm M_2N^{1/3}$ are also depicted. Horizontal lines eminate from these parabolas starting at $x=kN^{2/3}$. The event in \eqref{e:high0} tells us that on the horizontal interval $\mathcal{I}_k:=\ll kN^{2/3},(M_1+2k)N^{2/3}\rr$ the black rough curve stays entirely below the black horizontal line while the event in \eqref{e:high} tells us that there is a point in $\mathcal{I}_k$ at which the black rough curve exceeds the red horizontal curve.}
		\label{fig:arg2}
	\end{figure}
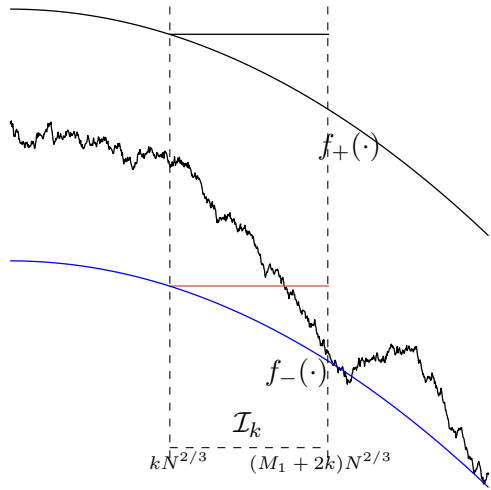
	
	\begin{proposition}[Low point on the first curve]\label{p:low} Fix any $\e\in (0,1)$. There exists $M_0(\e)$ such that for all $M\ge M_0$,
		\begin{align}\label{e:low}
			\liminf_{N\to \infty}\Pr\left(\mathcal{L}_1^N(2MN^{2/3}+1) \le -\tfrac18M^2N^{1/3}\nu \right)>1-\e,
		\end{align}
		where $\nu$ is defined in \eqref{nu}.
	\end{proposition}

	\begin{definition}[$n$-step random walk and bridge measures] \label{def:rws}
Recall the spaces $(\Omega^p_n,\mathcal{F}^p_n)$ for $p\in \{1,2\}$ from Definition \ref{prb}. For $p=2$ that definition provided coordinate function notation $S_i(k):= \omega_i(k)$ for $i\in \{1,2\}$ and $k\in \ll 1,n\rr$. For $p=1$ similarly define coordinate functions $S(k):=\omega(k)$ for $k\in \ll 1,n\rr$.  Recall $\fa$ from \eqref{def:faga}.

For $a\in \R$ define the probability measure $\Pr^{n;a}$ on  $(\Omega^1_n,\mathcal{F}^1_n)$ for a single $n$-step random walk started at $a$ to be proportional to the product of the Dirac delta function $\delta_{\omega(1)=a}$ and a density (against Lebesgue on $\R^{n-1}$) given by
		\begin{align*}			\prod_{k=2}^{n}\fa\big(\omega(k)-\omega(k-1)\big)\,d\omega(k).
		\end{align*}
Similarly, for $(a_1,a_2)\in \R^2$ define the probability measure $\Pr^{n;(a_1,a_2)}$ on  $(\Omega^2_n,\mathcal{F}^2_n)$ for a pair of independent $n$-step random walk started at $a_1$ and $a_2$ by taking the product of
$\Pr`^{n;a_1}$ and $\mathbb{P}^{n;a_2}$.

For $a,b\in \R$ define the probability measure $\Pr^{n;a;b}$
$(\Omega^1_n,\mathcal{F}^1_n)$ for a single $n$-step random bridge started at $a$ and ended at $b$ to be proportional to the product of two Dirac delta function $\delta_{\omega(1)=a}\delta_{\omega(n)=b}$ and a density (against Lebesgue on $\R^{n-2}$) given by
		\begin{align*}			\prod_{k=2}^{n}\fa\big(\omega(k)-\omega(k-1)\big) \prod_{k=2}^{n-1}d\omega(k).
		\end{align*}
Similarly, for $(a_1,a_2),(b_1,b_2)\in \R^2$ define the probability measure $\Pr^{n;(a_1,a_2);(b_1,b_2)}$ on  $(\Omega^2_n,\mathcal{F}^2_n)$ for a pair of independent $n$-step random briges started at $a_1$ and $a_2$ and ended (respectively) at $b_1$ and $b_2$ by taking the product of
$\Pr^{n;a_1;b_1}$ and $\mathbb{P}^{n;a_2;b_2}$.
	\end{definition}

	The proofs of Propositions \ref{p:high} and \ref{p:low} rely on the fluctuation results from \cite{bw}, as restated earlier in Theorem \ref{thm:bw}, and are postponed to the next subsection. Assuming their validity, we complete the proof of Theorem \ref{p:high2}.
	
	\begin{proof}[Proof of Theorem \ref{p:high2}] For clarity we divide the proof into two steps.
		
		\medskip
		
		\noindent\textbf{Step 1.} In this step we define notation and events used in the proof. Fix $\e\in (0,1)$ and $k>0$. Take $M_0$ from Proposition \ref{p:high}. We set $R$ large enough so that
		\begin{align}\label{eR}
			2^{-5}R\ge 2k+1, \qquad M_0-2^{-5}(\tfrac18R^2\nu+M_0)+R^{3/2} \le -M_0-2^{-10}R^2\nu, \qquad R\ge 2M_0
		\end{align}
		and $\SSS:=2^{-5}R$. 
		We will assume some additional conditions on $R$ later, which will depend on certain probability bounds that will be specified in the next step. For convenience, we will also assume $kN^{2/3}$ and $RN^{2/3}$ are integers (instead of using floor functions below).  We set
		\begin{align*}
			a:=M_0N^{1/3}, \ \ b:=-\tfrac18R^2N^{1/3}\nu, \ \ n:=RN^{2/3}-kN^{2/3}+1, \ \ v:=-\big(\tfrac18R^2\nu+2\sqrt{R}\big)N^{1/3}.
		\end{align*}
		Let us define the sets $\mathcal{I}:=\ll \SSS N^{2/3},(M_0+2\SSS)N^{2/3}\rr$ and $\mathcal{J}:=\ll kN^{2/3},RN^{2/3}\rr$. Due to \eqref{eR}, we have $\mathcal{I} \subset \mathcal{J}$. Next we define the following events:
		\begin{align*}
			\mathsf{A} &:=\bigg\{\sup_{p\in \mathcal{J}}{\mathcal{L}_2(2p)} \le v  \bigg\}, \quad \mathsf{B} :=\bigg\{ \mathcal{L}_1(2kN^{2/3}+1) \le a,\mathcal{L}_1(2RN^{2/3}+1) \le b \bigg\} .
		\end{align*}
		The $\mathsf{A}$ event demands that the second curve $\L^N_2(2p)$ does not rise above $v$ for any $p\in \mathcal{J}$. The $\m{B}$ event requires both $\L^N_1(2kN^{2/3}+1)$ and $\L^N_1(2RN^{2/3}+1)$ to be less than $a$ and $b$ respectively. Finally we set
		\begin{align*}
			\mathsf{C}:=\bigg\{ \sup_{p\in \mathcal{I}} {\L^N_1(2p+1)}+\SSS^2\nu N^{1/3}\ge -a\bigg\}
		\end{align*}
		In words, $\mathsf{C}$ ensures there exists some $p\in \mathcal{I}$ such that $\L^N_1(2p+1)$ is greater than $-a-\SSS^2\nu N^{1/3}$. 	
		
		\medskip
		
		Note that by Proposition \ref{p:high} we have $\Pr(\m{C}) \ge 1-\e$ . Furthermore, by Propositions \ref{p:high} and \ref{p:low} for large enough $R$ we also have $\Pr(\neg\m{B})\le 2\e$. We claim that
		for all large enough $R$ we have
		\begin{align}\label{e:tosh}
			\Pr(\m{A}\cap\m{B}\cap \m{C}) \le \e.
		\end{align}
		We prove \eqref{e:tosh} in the next step. Assuming this, note that by union bound we have
		\begin{align*}
			\Pr(\neg \m{A}) \ge \Pr(\m{C})-\Pr(\neg \m{B})-\Pr(\m{A}\cap\m{B}\cap \m{C}) \ge 1-4\e.
		\end{align*}
		Changing $\e \mapsto \e/4$ we arrive at \eqref{e:thigh}. This completes the proof modulo \eqref{e:tosh}.

		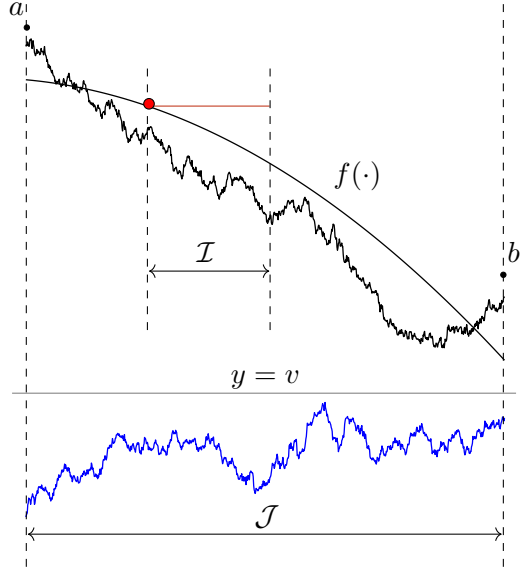
\begin{figure}[t]
			\centering
			\begin{overpic}[width=7cm]{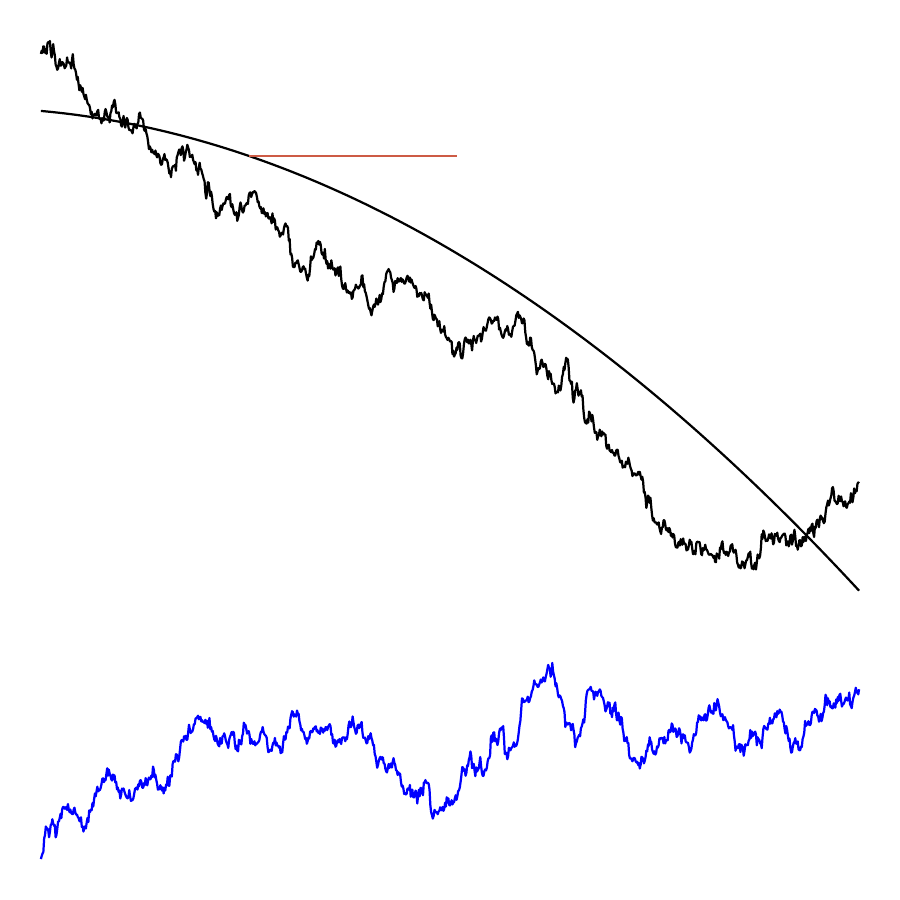}
				\put(2.7,-5){
					\begin{tikzpicture}
						\draw[dashed] (0,0) -- (0,7.5);
				\end{tikzpicture}}
				\put(25.7,40){
					\begin{tikzpicture}
						\draw[dashed] (0,3) -- (0,6.5);
				\end{tikzpicture}}
				\put(49,40){
					\begin{tikzpicture}
						\draw[dashed] (0,3) -- (0,6.5);
				\end{tikzpicture}}
				\put(93.3,-5){
					\begin{tikzpicture}
						\draw[dashed] (0,0) -- (0,7.5);
				\end{tikzpicture}}
				\put(2.3,97){
					\begin{tikzpicture}
						\draw[fill=black] (0,0) circle (1pt);
				\end{tikzpicture}}
				\put(92.8,50){
					\begin{tikzpicture}
						\draw[fill=black] (0,0) circle (1pt);
				\end{tikzpicture}}
				\put(25,82){
					\begin{tikzpicture}
						\draw[fill=red] (0,0) circle (2	pt);
				\end{tikzpicture}}
				\put(0,28){
					\begin{tikzpicture}
						\draw[gray] (0.2,0)--(6.9,0);
				\end{tikzpicture}}
				\put(43.5,30) {$y=v$}
				\put(26,50){
					\begin{tikzpicture}
						\draw[<->] (0,0)--(1.6,0);
				\end{tikzpicture}}
				\put(3,0){
					\begin{tikzpicture}
						\draw[<->] (0.2,0)--(6.5,0);
				\end{tikzpicture}}
				\put(1.3,100) {$a$}
				\put(96,53) {$b$}
				\put(63,68) {$f(\cdot)$}
				\put(37,53) {${\mathcal{I}}$}
				\put(48,3) {${\mathcal{J}}$}
			\end{overpic}
			\vspace{0.5cm}
			\caption{In this figure $\L^N_1(2p+1)$ (black curve) and $\L^N_2(2p)$ (blue curve) are plotted for $p\in \mathcal{J}$. $\m{A}$ denotes the event that the blue curve lies below the horizontal line $y=v$. $\m{B}$ denotes the event that the black curve starts below $a$ and ends below $b$. The curve $f$ in the figure is given by $f(x)=-(N\nu)^{-1}x^2-a$. The event $\m{C}$ denotes that there is a point $p'\in \mathcal{I}$ where the black rough curve stays above the red horizontal line (this event does not occur in the above figure). The key idea is that on $\m{A}\cap \m{B}$, the blue curve lies below $y=v$ completely, and the black curve behaves like a simple random bridge and follows a linear trajectory with starting and ending points less than $a$ and $b$ respectively.  As a result, the event $\m{C}$ (which requires the black curve to follow parabolic trajectory) does not occur with high probability. But we know both $\m{B}$ and $\m{C}$ occurs with high probability. Thus the event $\m{A}$ occurs with low probability.}
			\label{fig:arg}
		\end{figure}

		\medskip
		
		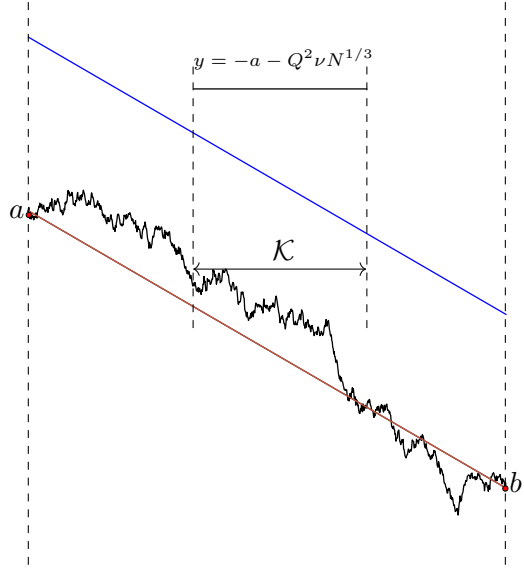
\begin{figure}
			\centering
			\begin{overpic}[width=7cm]{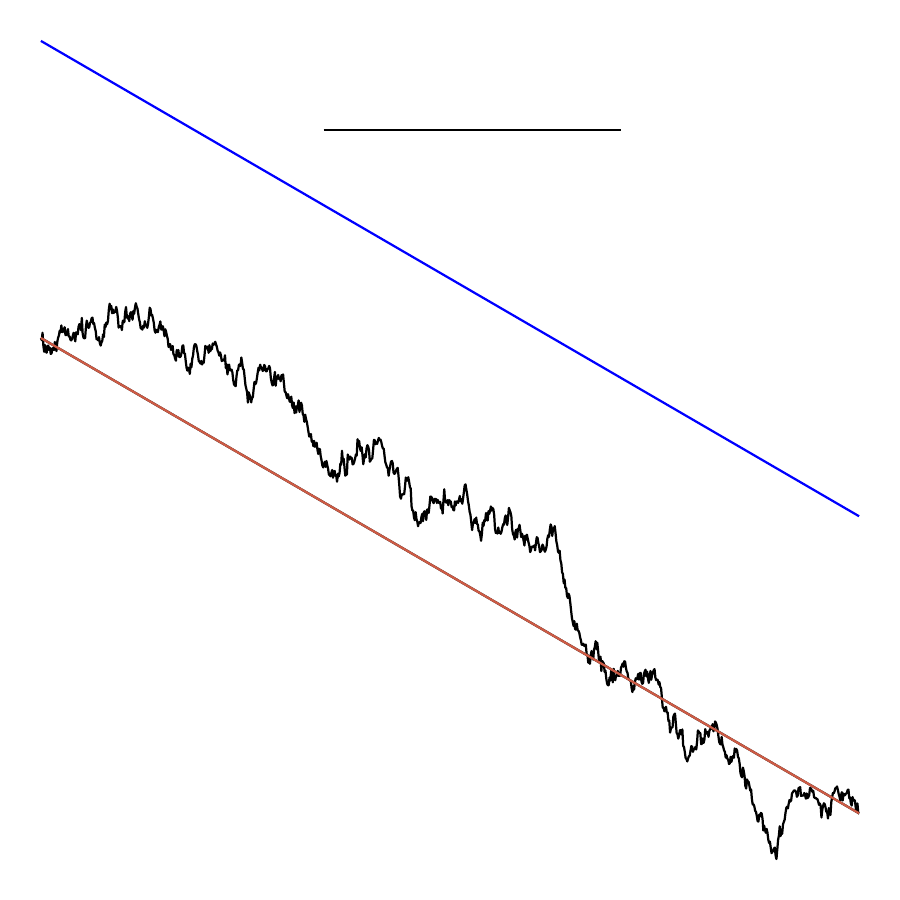}
				\put(2.7,-5){
					\begin{tikzpicture}
						\draw[dashed] (0,0) -- (0,7.5);
				\end{tikzpicture}}
				\put(34,40){
					\begin{tikzpicture}
						\draw[dashed] (0,3) -- (0,6.5);
				\end{tikzpicture}}
				\put(67,40){
					\begin{tikzpicture}
						\draw[dashed] (0,3) -- (0,6.5);
				\end{tikzpicture}}
				\put(93.3,-5){
					\begin{tikzpicture}
						\draw[dashed] (0,0) -- (0,7.5);
				\end{tikzpicture}}
				\put(2.3,61){
					\begin{tikzpicture}
						\draw[fill=red] (0,0) circle (1pt);
				\end{tikzpicture}}
				\put(92.8,9){
					\begin{tikzpicture}
						\draw[fill=red] (0,0) circle (1pt);
				\end{tikzpicture}}
				\put(34,50){
					\begin{tikzpicture}
						\draw[<->] (0,0)--(2.3,0);
				\end{tikzpicture}}
				\put(1,61) {$a$}
				\put(96,9) {$b$}
				\put(51,53) {${\mathcal{K}}$}
				\put(36,90) {\tiny{$y=-a-\SSS^2\nu N^{1/3}$}}
			\end{overpic}
			\vspace{0.5cm}
			\caption{In the above figure the random bridge $S(i)$ from $a$ to $b$ is depicted by the black curve. The event $\m{D}$ ensures the random bridge lies below the blue line $y=a+\frac{x}{n}(b-a)+\sqrt{Rn}$. The event $\m{C}$ requires $S(i)\ge -\big(M_0+\SSS^2\nu\big)N^{1/3}$ for some $i\in \mathcal{K}:= \ll(\SSS-k)N^{2/3},(M_0+2\SSS-k)N^{2/3}\rr$. One can choose $R$ large enough so that the horizontal black line $y=-a-\SSS^2\nu N^{1/3}=-\big(M_0+\SSS^2\nu\big)N^{1/3}$ lies above the blue line $y=a+\frac{x}{n}(b-a)+\sqrt{Rn}$ for all $x\ge (\SSS-k)N^{2/3}$. This forces $\m{D}\subset \neg\m{C}$.}
			\label{fig:arg1}
		\end{figure}

		\noindent\textbf{Step 2.} In this step we will prove \eqref{e:tosh}. The reader is encouraged to consult with Figure \ref{fig:arg} and its caption to get an overview of the key idea behind the proof.

		We consider the $\sigma$-algebra: $$\mathcal{F}:=\sigma\big(\L^N_2\ll1,2N-2i\rr, \L^N_1(\ll1,2kN^{2/3}+1 \rr \cup \ll 2RN^{2/3}+1,2N\rr)\big).$$ Note that $\m{A}\cap \m{B}$ is measurable with respect to  $\mathcal{F}$. Hence
		\begin{align*}
			\Pr(\m{A}\cap \m{B}\cap \m{C})=\Ex\left[\ind_{\m{A}\cap\m{B}} \Ex\left[\ind_{\m{C}}\mid \mathcal{F}\right]\right].
		\end{align*}
		Using the Gibbs property for two-sided boundaries (see Lemma \ref{i2a}), the conditional law is determined by the boundary data and is monotone with respect to  the boundary data (see Proposition \ref{p:gmc}). On the event $\m{A}\cap \m{B}$, $\L^N_2$ (on even points) is at most $v$, $\L^N_1(2kN^{2/3}+1)$ is at most $a$ and $\L^N_1(2RN^{2/3}+1)$ is at most $b$.  Thus by stochastic monotonicity we have
		\begin{align} \label{cl0}
			\ind_{\m{A}\cap\m{B}}\cdot\Ex\big( \ind_{\m{C}}\mid\mathcal{F}\big) \le \ind_{\m{A}\cap\m{B}}\cdot\frac{\Ex^{n;a;b}\left(W(S,v)\ind_{\mathsf{C}}\right)}{\Ex^{n;a;b}\left(W(S,v)\right)} \le \ind_{\m{A}\cap\m{B}}\cdot\frac{\Ex^{n;a;b}\left(\mathsf{C}\right)}{\Ex^{n;a;b}\left(W(S,v)\right)}.
		\end{align}
where $S=(S(1),\ldots,S(n))$ is distributed according to $\Pr^{n;a;b}$, the $n$-step random bridge measure from $a$ to $b$, and where $W(S,v):=\exp\left(-2\sum_{i=2}^{n-1} e^{v-S(i)}\right)$. The event $\mathsf{C}$ should now be treated as being defined in terms of $S$ as
		\begin{align*}
			\mathsf{C}= \bigg\{ \sup_{p\in \ll \SSS N^{2/3},(M_0+2\SSS)N^{2/3}\rr} S(p-kN^{2/3}+1)+\SSS^2\nu N^{1/3} \ge -a\bigg\}.
		\end{align*}
		Note that
		\begin{align}
			\nonumber	\Ex^{n;a;b}\left(W(S,v)\right) & \ge \exp\left(-2ne^{-\sqrt{n}}\right)\Pr^{n;a;b}\left(S(i) \ge v+\sqrt{n} \mbox{ for all } i\in \llbracket 1,n \rrbracket\right) \\ & \ge \exp\left(-2ne^{-\sqrt{n}}\right)\Pr^{n;a;b}\left(S(i)-a-\tfrac{i(b-a)}{{n}} \ge -\sqrt{n} \mbox{ for all } i\in \llbracket 1,n \rrbracket\right). \label{cal3}
		\end{align}
where the last inequality follows by noting that  $S(i) -a-\frac{i(b-a)}{n} \ge -\sqrt{n}$ implies $S(i) \ge b-\sqrt{n} \ge v+\sqrt{n}$. Since random bridges weakly converge to Brownian bridges (see \cite{liggett} and \cite{xd} for a quantitative version), using estimates for Brownian bridges, we see that the r.h.s.~\eqref{cal3} is uniformly bounded below by some absolute constant $\delta$. We now claim that for all large enough $R$
		\begin{align}\label{edc}
			\m{D} \subset \neg\m{{C}}, \quad \Pr^{n;a;b}(\m{D})\ge 1-\e\delta, \ \ \mbox{ where } \  \m{D}:=\left\{\sup_{i\in \ll 1,n\rr} \left(S(i)-a-\tfrac{i(b-a)}{n} \right) \le \sqrt{R}\sqrt{n}\right\}.
		\end{align}
		Note that \eqref{edc} implies $\Pr^{n;a;b}(\mathsf{C}) \le \e\delta$. Plugging this back in \eqref{cl0} along with the bound $\Ex^{n;a;b}\left(W(S,z)\right)\ge \delta$, yields that r.h.s.~\eqref{cl0} is at most $\e$. This proves \eqref{e:tosh}.
		
		\medskip
		
		Let us now verify \eqref{edc}. Indeed, $\Pr^{n;a;b}(\m{D})$ can be made arbitrarily close to $1$ by choosing $R$ large enough (as random bridges weakly converge to Brownian bridges \cite{liggett}). We choose $R$ so large that $\Pr^{n;a;b}(\m{D})$ is at least $1-\e \delta$.
		Let us now verify $\m{D}\subset \neg\m{C}$ (see also Figure \ref{fig:arg1} and its caption). For $q\ge \SSS$ we see that
		\begin{align*}
			a+\tfrac{(q-k)(b-a)}{R-k}+\sqrt{R}\sqrt{n} & \le \left(M_0 - \tfrac{\SSS-k}{R-k}(\tfrac18R^2\nu+M_0)+R^{3/2}\right)N^{1/3}  \\ & \le \left(M_0 - 2^{-5}(\tfrac18R^2\nu+M_0)+R^{3/2}\right)N^{1/3} \le -\left(M_0 +\SSS^2\nu\right)N^{1/3}
		\end{align*}
		The penultimate inequality follows by observing that as $\SSS=2^{-5} R$, we have $\SSS-k \ge 2^{-5}(R-k) >0$. Finally the last inequality follows from \eqref{eR}. Thus for all $p\ge \SSS N^{2/3}$,
		\begin{align*}
			x+\tfrac{(p-kN^{2/3})(y-x)}{(R-k)N^{2/3}}+\sqrt{R}\sqrt{n} \le M_0N^{1/3}-\SSS^2\nu N^{1/3}
		\end{align*}
		Clearly this implies $\m{D}\subset \neg\m{C}$, completing the proof \eqref{edc}.
	\end{proof}

	\subsubsection{Proof of Propositions \ref{p:high} and \ref{p:low}}
	 The proofs of Propositions \ref{p:high} and  \ref{p:low} uses the following.

	\begin{lemma}[Uniform tightness] \label{l:ut} Recall $\zl (m)$, the point-to-(partial)line partition function defined in  \eqref{rfl}. Fix $\e\in (0,1)$. There exists $K_0=K_0(\e)>0$, such that for all $M>0$ and $K\ge K_0$ we have
		\begin{align*}
			\liminf_{N\to\infty}\Pr\left(-K\le \frac{\log \zl(MN^{2/3})+2\Psi(\theta)N}{N^{1/3}}+M^2\nu \le K \right) >1-\e
		\end{align*}
		where $\nu$ is defined in \eqref{nu}.
	\end{lemma}
	
	We remark that the above lemma was alluded in the introduction in the form of \eqref{pardec}.
	\begin{proof} We recall the notations introduced in Section \ref{sec:132}. Fix any $M>0$. Set $k=MN^{2/3}$ and $p:=1+\frac{2k}{N-k}$. Let $\theta_c$ be the unique solution to $\Psi'(\theta_c)-p\Psi'(2\theta-\theta_c)=0$. Set $f_{\theta,p}=-\Psi(\theta_c)-p\Psi(2\theta-\theta_c)$ and $\sigma_{\theta,p}^3=\frac12(-\Psi''(\theta_c)-\Psi''(2\theta-\theta_c))$ where $\Psi$ is the digamma function defined in \eqref{psidef}. A straightforward calculation (done at the end of Appendix \ref{appd}) shows
		\begin{equation}\label{calcc}
			(N-k)f_{\theta,p}=-2N\Psi(\theta)+M^2N^{1/3}(\Psi'(\theta))^2/\Psi''(\theta)+O(1),\,\,\, \mbox{ and } \,\,\, \sigma_{\theta,p}/(-\Psi''(\theta))^{1/3}\stackrel{N\to\infty}{\to} 1,
		\end{equation}
		where $O(1)$ terms depend on $M,\theta$, but are bounded in $N$.
		When $\alpha=\alpha_2>0$, we have that $\lim_{N\to\infty} (N-k)^{1/3}\sigma_{\theta,p}(\alpha_2+\theta-\theta_c)=\infty$ for each fixed $M>0$. Thus by Theorem \ref{thm:bw} we get
		\begin{align*}
			\frac{\log \zl(MN^{2/3})+2\Psi(\theta)N}{(-N\Psi''(\theta))^{1/3}}+M^2\nu \stackrel{(d)}{\Longrightarrow} \operatorname{TW}_{\operatorname{GUE}},
		\end{align*}
		where $\operatorname{TW}_{\operatorname{GUE}}$ is the GUE Tracy-Widom distribution \cite{tw} and $\nu$ is defined in \eqref{nu}. For $\alpha=\alpha_1=N^{-1/3}\mu$, we have  $\lim_{N\to\infty} (N-k)^{1/3}\sigma_{\theta,p}(\alpha_1+\theta-\theta_c)=y:=\sigma_{\theta,1}(\mu-M\Psi'(\theta)/\Psi''(\theta)).$ Another application of Theorem \ref{thm:bw} yields
		\begin{align*}
			\frac{\log \zl(MN^{2/3})+2\Psi(\theta)N}{(-N\Psi''(\theta))^{1/3}}+M^2\nu \stackrel{(d)}{\Longrightarrow} U_{-y}.
		\end{align*}
		where $U_{-y}$ is the Baik-Ben Arous-P\'ech\'e distribution \cite{bbap} (see  \cite[(5.2)]{bw} for definition). As $M\to \infty$, so does $y\to \infty$. Since $U_{-y} \stackrel{(d)}{\Longrightarrow} \operatorname{TW}_{\operatorname{GUE}}$ as $y\to \infty$ (see \cite[(2.36)]{br01}), we thus get tightness uniformly in $M$.
	\end{proof}
	\begin{proof}[Proof of Proposition \ref{p:high}] Fix $k>0$, $\e\in (0,1)$. Since for any $M_1>0$
		\begin{align*}
			\sup_{j\in \ll kN^{2/3}, (M_1+2k)N^{2/3}\rr} \zh(N+j,N-j) \le \zl(kN^{2/3}),
		\end{align*}
		appealing to Lemma \ref{l:ut} with $M\mapsto k$ we see that
		\begin{align*}
			\Pr\left(\sup_{j\in \ll kN^{2/3}, (M_1+2k)N^{2/3}\rr} \frac{\log\zh(N+j,N-j)+2\Psi(\theta)N}{N^{1/3}}+k^2\nu \le M_2\right) \ge 1-\e,
		\end{align*}
		where $M_2$ can be chosen to be any $M \ge K_0$ where $K_0(\e)$ comes from Lemma \ref{l:ut}. Recalling that $\L^N_1(2j+1)=\log \zh(N+j,N-j)+2\Psi(\theta)N$ from \eqref{impi}, we get \eqref{e:high0}.
		
		\medskip
		
		The remainder of the proof is now devoted in proving \eqref{e:high}. Towards this end, set $K_1=\frac{1}{2}(M_1+2k)^2\nu$. Choose $M_1$ large enough so that $K_1\ge K_0(\e/4)$ where $K_0$ comes from Lemma \ref{l:ut}. Applying Lemma \ref{l:ut} with $M\mapsto M_1+2k$, $K\mapsto K_1$, and $\e\mapsto \e/4$ we have
		\begin{align}\label{e1}
			\liminf_{N\to\infty}\Pr\left(\frac{\log \zl((M_1+2k)N^{2/3})+2\Psi(\theta)N}{N^{1/3}} \le -\tfrac{1}{2}(M_1+2k)^2\nu \right)>1-\tfrac14\e.
		\end{align}
		Now we take $K_2=(\frac{(M_1+2k)^2}{4}-k^2)\nu-\log 2\ge \frac14M_1^2\nu$. We again choose $M_1$ large enough so that $K_2\ge K_0(\e/4)$. Then applying Lemma \ref{l:ut} with $M\mapsto k$, $K\mapsto K_2$, and $\e\mapsto \e/4$ we have
		\begin{align}\label{e2}
			\liminf_{N\to\infty}\Pr\left(\frac{\log \zl(kN^{2/3})+2\Psi(\theta)N}{N^{1/3}} \ge -\tfrac14(M_1+2k)^2\nu+\log 2 \right)>1-\tfrac14\e.
		\end{align}
		By union bound the above two estimates implies for all large enough $M_1$ we have \begin{align}\label{e:12}
			\liminf_{N\to\infty}\Pr\left(\zl(kN^{2/3}) > 2\cdot \zl((M_1+2k)N^{2/3}) \right)>1-\tfrac12\e.
		\end{align}
		Let us temporarily set $A=\zl(kN^{2/3})- \zl((M_1+2k)N^{2/3}$ and $B=\zl((M_1+2k)N^{2/3}$. Observe that $A+B>2B$ implies $2A>A+B$. Recall from \eqref{rfl} that
$$A=\!\!\!\!\!\!\sum_{\lceil
			kN^{2/3}\rceil}^{\lceil (M_1+2k)N^{2/3}\rceil -1}\!\!\!\!\!\! \zh(N+j,N-j) \le (M_1+k)N^{\frac23}\!\!\!\!\!\!\sup_{j\in \ll kN^{\frac23},(M_1+2k)N^{\frac23}\rr}\!\!\!\!\!\! \zh(N+j,N-j).$$ We thus have
		\begin{equation}\label{e:relat}
			\begin{aligned}
				& \big\{\zl(kN^{2/3}) > 2\cdot \zl((M_1+2k)N^{2/3})\big\} \\ & \subset \bigg\{ \sup_{j\in \ll kN^{\frac23},(M_1+2k)N^{\frac23}\rr}\!\!\!\!\!\! \log\zh(N+j,N-j) > \log \zl(kN^{\frac23})-\log (2(M_1+k)N^{\frac23}) \bigg\}.
			\end{aligned}
		\end{equation}
		By Lemma \ref{l:ut}, one can choose $M_2$ large enough (but free of $k$) so that
$$\liminf_{N\to\infty}\Pr\Big( \log \zl(kN^{2/3})+2\Psi(\theta)N+k^2\nu N^{1/3} \ge -M_2N^{\frac13}+\log (2(M_1+k)N^{\frac23})\Big)>1-\tfrac12\e.$$
Using this, in view of \eqref{e:relat} and \eqref{e:12}, and using $\L^N_1(2j+1)=\log \zh(N+j,N-j)+2\Psi(\theta)N$ (see \eqref{impi}) we  arrive at \eqref{e:high}. This proves Proposition \ref{p:high}.
	\end{proof}
	
	\begin{proof}[Proof of Proposition \ref{p:low}] We use the same notations as from the proof of Proposition \ref{p:high} and utilize \eqref{e1} and \eqref{e2} obtained there with $k=1$. Set $M=M_1+2$. Combining \eqref{e1} and \eqref{e2} implies
		\begin{align*}
			\liminf_{N\to \infty} \Pr\left(\log \zl(N^{2/3})>\tfrac14M^2N^{1/3}\nu+\log \zl(MN^{2/3})\right) \ge 1-\tfrac12\e.
		\end{align*}
		As $\zl(MN^{2/3})\ge \zh(N+MN^{2/3};N-MN^{2/3})$,	this leads to
		\begin{align*}
			& \liminf_{N\to \infty} \Pr\left(\log \zl(N^{2/3})>\tfrac14M^2N^{1/3}\nu+\log \zh(N+MN^{2/3},N-MN^{2/3})\right) \ge 1-\tfrac12\e.
		\end{align*}
		Again by Lemma \ref{l:ut}, one can choose $M$ large enough so that $$\liminf_{N\to\infty}\Pr\bigg(\log\zl(N^{2/3})\le \tfrac18M^2N^{1/3}\nu -2\Psi(\theta)N\bigg)>1-\tfrac12\e,$$ which forces
		\begin{align*}
			\liminf_{N\to \infty} \Pr\left(\log \zh(N+MN^{2/3};N-MN^{2/3}) <-2N\Psi(\theta)-\tfrac18M^2N^{1/3}\nu\right) \ge 1-\e.
		\end{align*}
		By \eqref{impi}, $\L^N_1(2MN^{2/3}+1)=\log \zh(N+MN^{2/3};N-MN^{2/3})-2\Psi(\theta)N$ hence \eqref{e:low} follows.	\end{proof}

	\subsection{Spatial properties of the lower curves} \label{sec:spcu}
	
	In this subsection, we study spatial properties of the lower curves of the $\hslg$ line ensemble. The main result of this section is the following.

	\begin{theorem}\label{l:lhigh} Fix any $p\in \{1,2\}$. Set $\alpha:=\alpha_p$ according to \eqref{acric}. Consider the $\hslg$ line ensemble from Definition \ref{l:nz} with parameters $(\alpha,\theta)$. Given any $k,\e>0$, there exist constants $M=M(k,\e)\ge 1$ and $N_0(k,\e)\ge 1$ such that for all $N\ge N_0(k,\e)$ and $v\in \{2,3\}$ we have
		\begin{align}\label{e:2high}
			\Pr\left(\sup_{s\in \ll 1,kN^{2/3}\rr} \L_v^N(s)\ge MN^{1/3} \right) \le \e.
		\end{align}
	\end{theorem}
	In plain words, Theorem \ref{l:lhigh} argues that with high probability on the domain $\ll1,kN^{2/3}\rr$, the entire second curve and third curve lies below a threshold $MN^{1/3}$. The proof of Theorem \ref{l:lhigh} can be easily extended to include other lower indexed curves as well. However, for the proofs of our main  results, it suffices to consider the first three curves.
	
	Recall from Theorem \ref{thm:conn} that the conditional laws of the $\hslg$ line ensemble are given by $\hslg$ Gibbs measures introduced in Definition \ref{def:hslggibbs}. The key technical ingredient in proving Theorem \ref{l:lhigh} is the tightness of left boundary points of the first two curves under the \btf\ measure defined in Definition \ref{def:btf}.
	
	\begin{proposition}\label{lem:ep} Fix any $p\in \{1,2\}$. Set $\alpha:=\alpha_p$ according to \eqref{acric}. Fix any $r \ge 1$ and $\e >0$. Set $T=\lfloor rN^{2/3}\rfloor$. Define
		\begin{align}\label{adef}
			A:=\begin{cases}
				1+\sqrt{r}|\mu|\Psi'(\frac12\theta) & \mbox{if }p=1, \\
				1 & \mbox{if }p=2.
			\end{cases}
		\end{align}
		There exists $M=M(\e)>0$ and $N_0(\e)>0$ such that for all $N\ge N_0$ we have
		\begin{align}\label{eq:l5.1}
			\Pr_{\alpha_p}^{(0,-A\sqrt{T}),(-\infty)^{T};2,T}\big(|L_1(1)|+|L_2(2)|\ge M\sqrt{T}\big)\le \e.
		\end{align}
		where the law $	\Pr_{\alpha_p}^{\vec{y},(-\infty)^{2T};2,T}$ is defined in Definition \ref{def:btf}. Furthremore, there exists $\til{M}=\til{M}(\e)>0$ and $\til{N}_0(\e)>0$ such that for all $N\ge \til{N}_0$ we have
		\begin{align}\label{eq:l5.w}
			\Pr_{\alpha_1}^{0,(-\infty)^{T};1,T}\big(|L_1(1)|\ge \til{M}\sqrt{T}\big)\le \e.
		\end{align}
	\end{proposition}
	
	As we shall see in the next section, the proof of the above lemma can be extended to include $L_2(1)$ instead of $L_2(2)$. For technical reasons we work with $L_2(2)$ here.
	
	As mentioned in the introduction, the proof of Proposition \ref{lem:ep} relies on several ingredients related to non-intersecting random walks. We postpone its proof to Section \ref{sec:rpe}. We now complete the proof of Theorem \ref{l:lhigh} assuming Proposition \ref{lem:ep}.
	
	\begin{proof}[Proof of Theorem \ref{l:lhigh}] We prove the $v=2$~case and then use it to show the $v=3$~case.
		
		\medskip
		
		\noindent\textbf{Part I: $v=2$ case.} For clarity we divide the proof into two steps.
		
		\medskip
		
		\noindent\textbf{Step 1.} Recall that the points in the line ensemble satisfy certain high probability ordering due to Theorem \ref{t:order}. In particular, if we know the even points on $\L^N_2$ are not too high, Theorem \ref{t:order} will force that with high probability the odd points are not too high as well. Thus it suffices to control the even points on $\L^N_2$. In this step, we flesh out the details of the above idea. The proof of control on even points on $\L^N_2$ appears in the second step of the proof.
		
		\medskip
		
		We begin by defining a few events that will appear in the rest of the proof. Fix $k,\e>0$. For any $r\in \ll 1,kN^{2/3}\rr\cap 2\Z$, define
		\begin{align*}
			\mathsf{A}(r,M):=\big\{ \L^N_2(r)\ge MN^{1/3}\big\}, \quad \mathsf{F}(r,M):= \{\L^N_1(r-1)\ge \tfrac{3M}4 N^{1/3}\}.
		\end{align*}
		Define
		\begin{align*}
			\mathsf{B}(r,M):= \mathsf{A}(r,M) \cap \hspace{-0.6cm}\bigcap_{s\in \ll r+2,kN^{2/3} \rr \cap 2\Z} \hspace{-0.6cm}\neg\mathsf{A}(s,M),
		\end{align*}
		so that $(\mathsf{B}(r,M))_{r\in \ll 1,kN^{2/3}\rr}$ forms a disjoint collections of events. Note that
		$$\bigsqcup_{r\in \ll 1,kN^{2/3}\rr\cap 2\Z}\hspace{-0.6cm} \mathsf{B}(r,M)=\hspace{-0.6cm}\bigcup_{r\in \ll 1,kN^{2/3}\rr\cap 2\Z}\hspace{-0.6cm} \mathsf{A}(r,M) = \bigg\{\sup_{r\in \ll 1,kN^{2/3}\rr\cap 2\Z} \L^N_2(r) \ge MN^{1/3}\bigg\}.$$
		In the above equation, we use $\sqcup$ instead of $\cup$ to stress on the fact that it is an union of disjoint events. Thus the above union demands at least one of the even points in $\ll1,kN^{2/3}\rr$ of $\L^N_2$ to exceed $MN^{1/3}$. We next define
		\begin{align*}
			\mathsf{G}^+(M):=\hspace{-0.6cm}\bigsqcup_{r\in \ll 1,kN^{2/3}\rr\cap 2\Z}\hspace{-0.6cm} \mathsf{B}(r,M) \cap \mathsf{F}(r,M), \quad
			\mathsf{G}^-(M):=\hspace{-0.6cm}\bigsqcup_{r\in \ll 1,kN^{2/3}\rr\cap 2\Z}\hspace{-0.6cm} \mathsf{B}(r,M) \cap \neg\mathsf{F}(r,M).
		\end{align*}
		 Finally set $\m{G}(M):=\m{G}^+(M)\sqcup \m{G}^{-}(M).$  Observe that the event \begin{align*}
			\neg\m{G}(M)=\left\{\sup_{s\in \ll 1,kN^{2/3}\rr \cap 2\Z} \L^N_2(s) < MN^{1/3}\right\}
		\end{align*} controls the supremum of the second curve over the even points.
		Take $0<k'<k$.  By the union bound we get that
		\begin{align}\label{e:ubg}
			\Pr\left(\sup_{s\in \ll 1,k'N^{2/3}\rr} \L^N_2(s)\ge 3MN^{\frac13} \right) \le \Pr(\m{G}(2M))+\Pr\bigg(\sup_{\substack{s\in \ll 1,k'N^{2/3}\rr \\ s\in (2\Z+1)}} \L^N_2(s) \ge 3MN^{\frac13}, \neg \m{G}(2M)\bigg).
		\end{align}
		Note that on $-\m{G}(2M)$ the supremum of $\L^N_2(s)$ over all $s\in \ll1,kN^{2/3}\rr$ is at most $2MN^{1/3}$. Then by the ordering of the line ensemble (Theorem \ref{t:order}) on $\neg\m{G}(2M)$ it is exponentially unlikely that any odd point within $\ll1,k'N^{2/3}\rr$ will exceed $2MN^{1/3}+(\log N)^{7/6}$. In particular the second term in r.h.s.~\eqref{e:ubg} can be made smaller than $\frac{\e}{2}$ by choosing $N$ large enough and taking $M\ge 1$. For the first term we claim that there exists $M_0, N_0$ depending on $k,\e$ such that for all $N\ge N_0$ and $M\ge M_0$ we have
		\begin{align}\label{e:gmshow}
			\Pr(\m{G}(2M)) \le \tfrac\e2. 
		\end{align}
		Clearly plugging this bound back in r.h.s.~\eqref{e:ubg} proves \eqref{e:2high} with $M\mapsto 3M$ and $k'\mapsto k$. For the remainder of the proof we focus on proving \eqref{e:gmshow}.
		
		\medskip
		
		\noindent\textbf{Step 2.} In this step we prove  \eqref{e:gmshow}. Observe that from the definition of $\m{G}^{-}(2M)$ we have
		$$\Pr(\m{G}^{-}(2M)) \le \Pr \Big(\L^N_1(r-1)-\L^N_2(r)\ge -\tfrac{M}{2}N^{1/3} \mbox{ for some }r\in \ll 1,kN^{1/3}\rr \cap 2\Z\Big).$$
		However by Theorem \ref{t:order} the right-hand side of the above equation can be made smaller that $\frac\e4$ for all $N\ge N_0$ and $M\ge 1$, by choosing $N_0:=N_0(k,\e)>0$ appropriately. We next claim that
		\begin{align}\label{e:gmplus}
			\Pr(\m{G}^+(2M))\le 2\Pr(\m{A}(2,M)) \le \tfrac\e4.
		\end{align}
		As $\m{G}(2M) =\m{G}^{-}(2M)\cup \m{G}^{+}(2M)$, in view of the above claim,  \eqref{e:gmshow} follows via a union bound.
		
		Let us now prove \eqref{e:gmplus}. Observe that by definition of $G^{+}(2M)$ we have
		\begin{align}\label{e:gmp0}
			\Pr\big(\mathsf{A}(2,M)\big)  \ge \Pr\big(\mathsf{G}^+(2M) \cap \mathsf{A}(2,M)\big)  & = \sum_{r\in  \ll 1,kN^{2/3}\rr \cap 2\Z} \Pr\big(\mathsf{B}(r,2M)\cap \m{F}(r,2M) \cap \mathsf{A}(2,M)\big).
		\end{align}
		We focus on each of the terms in the above sum. Using the tower property we have
		\begin{equation}\label{e:gmp1}
			\begin{aligned}
				&	\Pr\big(\mathsf{B}(r,2M)\cap \m{F}(r,2M) \cap \mathsf{A}(2,M)\big) \\ & =  \Ex \left[\ind_{\mathsf{B}(r,2M)\cap \m{F}(r,2M)} \Ex\left(\ind_{\mathsf{A}(2,M)} \mid \sigma\big(\L^N_3, \L^N_1\ll r-1, kN^{2/3}\rr, \L^N_2\ll r, kN^{2/3}\rr\big) \right)\right].
			\end{aligned}
		\end{equation}
		Using the Gibbs property (see Theorem \ref{thm:conn} and Lemma \ref{def:hsgm1} \ref{deff}) we have almost surely that
		\begin{equation}\label{e:gmp2}
			\begin{aligned}
				& \ind_{\mathsf{B}(r,2M)\cap \m{F}(r,2M)}\Ex\left(\ind_{\mathsf{A}(2,M)} \mid\sigma\big( \L^N_3, \L^N_1\ll r-1, kN^{2/3}\rr, \L^N_2\ll r, kN^{2/3}\rr\big) \right) \\ & =\ind_{\mathsf{B}(r,2M)\cap \m{F}(r,2M)}\Pr_{\alpha_p}^{\vec{y},\vec{z};2,r/2}(L_2(2)>MN^{1/3}) \\ & \ge \ind_{\mathsf{B}(r,2M)\cap \m{F}(r,2M)}\Pr_{\alpha_p}^{\vec{w},(-\infty)^{r};2,r/2}(L_2(2)>MN^{1/3}),
			\end{aligned}
		\end{equation}
		where $\vec{y}=(\L^N_1(r-1),\L^N_2(r))$, $\vec{z}=(\L^N_3(2v))_{v=1}^{r/2}$ and $\vec{w}:=(\frac{3M}{2}N^{1/3},\frac{3M}{2}N^{1/3}-A\sqrt{r/2})$ ($A\ge 1$ is defined in \eqref{adef}). The last inequality above follows by stochastic monotonicity (Proposition \ref{p:gmc}).  We now briefly explain how stochastic monotonicity works here. Note that the event $\{L_2(2)>MN^{1/3}\}$ is decreasing thus by stochastic monotonicity to achieve a lower bound, we can reduce the boundary $\vec{z}$ to $(-\infty)^{r}$. Furthermore, on $\m{B}(r,2M)\cap \m{F}(r,2M)$, we may reduce $\vec{y}$ to $\vec{w}$  as $y_i\ge w_i$ on $\m{B}(r,2M)\cap \m{F}(r,2M)$.
		
		Note that $MN^{1/3} \ge Mk^{-\frac12}\sqrt{r/2}$. By translation invariance (Lemma \ref{obs1} \ref{traninv}) and Proposition \ref{lem:ep}, we may choose $M_0(k,\e)$ large enough so that for all $M\ge M_0$ and $r\in \ll1,kN^{2/3}\rr\cap 2\Z$ we have
		\begin{align*}
			\Pr_{\alpha_p}^{\vec{w},(-\infty)^{r};2,r/2}(L_2(2)>MN^{1/3})=\Pr_{\alpha_p}^{(0,-A\sqrt{r/2}),(-\infty)^{r};2,r/2}\big(L_2(2)>-\tfrac12MN^{1/3}\big) \ge \tfrac12.
		\end{align*}
		Inserting the above bound in \eqref{e:gmp2} and then going back to \eqref{e:gmp1} we get
		\begin{align*}
			\mbox{r.h.s.~\eqref{e:gmp1}} \ge \tfrac12\Pr\big(\mathsf{B}(r,2M)\cap \m{F}(r,2M) \big).
		\end{align*}
		Recall that $\mathsf{B}(r,2M)\cap \m{F}(r,2M)$ are all disjoint events whose union over $r\in \ll 1,kN^{2/3}\rr\cap 2\Z$ is given by $\m{G}^+(2M)$. Summing the above inequality over $r\in \ll 1,kN^{2/3}\rr\cap 2\Z$, in view of \eqref{e:gmp0}, we thus arrive at $\Pr(\m{A}(2,M)) \ge \frac12\Pr(\m{G}^+(2M))$. This proves the first inequality in \eqref{e:gmplus}. For the second one observe that by union bound
		\begin{align*}
			\Pr(\m{A}(2,M)) \le \Pr\big(\L^N_1(3)-\L^N_2(2) \le -N^{1/3}\big)+\Pr\big(\L^N_1(3) \ge (M-1)N^{1/3}\big).
		\end{align*}
		By Theorem \ref{t:order} the first term on the right-hand side of the above equation can be made arbitrarily small by choosing $N$ large enough. As for the second term, recall the point-to-line partition function $\zl(\cdot)$ from \eqref{rfl}. From Theorem \ref{thm:bw} we know $N^{-1/3}[\log\zl(1)+2\Psi(\theta)N]$ is tight.  Since $\L^N_1(3) \le \log\zl(1)+2\Psi(\theta)N$ (see \eqref{impi}), one can make the second term arbitrarily small enough by choosing $M,N$ large enough. This completes the proof of \eqref{e:gmplus}.
		
		\medskip
		
		\noindent\textbf{Part II: $v=3$ case.} Fix $k>0$. Let us define
		$$\m{E}:=\bigg\{\sup_{s\in \ll 1,kN^{2/3}\rr} \L^N_3(s)\ge MN^{1/3}\bigg\}, \quad \m{F}:=\bigg\{\sup_{s\in \ll 1,kN^{2/3}\rr} \L^N_2(s)\ge \tfrac12MN^{1/3}\bigg\}.$$ By repeated application of the union bound we have
		\begin{align}
			\nonumber
			\Pr(\m{E}) & \le \Pr(\m{F})+\Pr(\m{E}\cap \neg \m{F}) \\ & \nonumber \le
			\Pr(\m{F})+\Pr\bigg(\L^N_2(s)-\L^N_3(s)\le -\tfrac12MN^{1/3}, \mbox{ for some }s\in \ll1,kN^{2/3}\rr\bigg) \\  & \le \Pr(\m{F})+\sum_{\ll1,kN^{2/3}\rr} \Pr\big(\L^N_2(s)-\L^N_3(s)\le -\tfrac12MN^{1/3}\big). \label{lastwo}
		\end{align}
		By Theorem \ref{t:order}, there exists an absolute constant $N_0$ such that for all $s\ge 1$, and $M\ge 1$, we have  $\Pr\left(\L^N_2(s)-\L^N_3(s)<-\tfrac12 MN^{1/3}\right)\le 2^{-N}$. Since we have established $v=2$ case of Theorem \ref{l:lhigh}, we may directed use \eqref{e:2high} with $v\mapsto 2$, $M\mapsto \frac12M$ and $\e\mapsto \frac12\e$, to get that $\Pr(\m{F})\le \frac12\e$ for all large enough $N,M$. Thus for all $N,M$ large enough we have $\eqref{lastwo} \le \frac12\e+kN^{2/3}2^{-N} \le \e$.
	\end{proof}
	
	Theorem \ref{l:lhigh} and Proposition \ref{lem:ep} can be used to deduce left boundary tightness for the $\hslg$ line ensemble. We shall refer to this property as \textit{endpoint tightness}.
	
	\begin{theorem}[Endpoint Tightness] \label{thm:eptight} Fix any $p\in \{1,2\}$. Set $\alpha:=\alpha_p$ according to \eqref{acric}. Recall the $\hslg$ line ensemble from Definition \ref{l:nz} with parameters $(\alpha,\theta)$. The sequences $\{{N^{-1/3}}\L_1^N(1)\}_{N\geq 1}$ and $\{{N^{-1/3}}\L_2^N(2)\}_{N\geq 1}$ are tight.
	\end{theorem}
	
	Again the proof can be extended to include tightness of $N^{-1/3}\L_2^N(1)$ as well, once we have the corresponding version in Proposition \ref{lem:ep}. We again refrain from doing so, as it is inconsequential to the proofs of our main theorem. We refer to the discussion in the introduction (Remark \ref{ims22rk}) about how Theorem \ref{thm:eptight} relates to the work of \cite{ims22}.
	

	\begin{proof}[Proof of Theorem \ref{thm:eptight}] Fix an $\e>0$. We shall show that for all large enough $N,M$ we have
		\begin{align}\label{ethr}
			\Pr(\L_1^N(1)\le MN^{1/3}) \ge 1-3\e, \quad \Pr(\L_2^N(2)\le -MN^{1/3}) \le 3\e.
		\end{align}
		In view of the ordering of points in the line ensemble (Theorem \ref{t:order}), we know $\L_1^N(1) \ge \L_2^N(2)-(\log N)^{7/6}$ with probability at least $1-2^{-N}$. This along with the above equation ensures endpoint tightness. We thus focus on proving \eqref{ethr}.
		
		\medskip
		
		\noindent{\textbf{Proof of the first inequality in \eqref{ethr}.}} Recall the point-to-line partition function $\zl(\cdot)$ from \eqref{rfl}. From Theorem \ref{thm:bw}, we know $N^{-1/3}\big(\log\zl(1)+2\Psi(\theta)N\big)$ is tight.  Since $\L_1^N(3) \le \log\zl(1)+2\Psi(\theta)N$, there exists $M_1(\e)>0$ such that for all $N\ge 3$ we have $\Pr(\L_1^N(3) \le M_1N^{1/3}) \ge 1-\e.$ Thanks to Theorem \ref{t:order}, there exists $M_2(\e)>M_1(\e)$ such that for all $N\ge 3$
		\begin{align*}
			\Pr(\m{A})\ge 1-2\e, \quad \m{A}:=\bigg\{\L_1^N(3)\le M_1N^{1/3}, \sup_{j\in\ll1,4\rr}\L_2^N(j) \le M_2N^{1/3}\bigg\}.
		\end{align*}
		Define $\mathcal{F}:=\sigma\big((\L_1^N(j))_{j\ge 3}, (\L_i^N\ll1,2N-2i+2\rr)_{i\ge 2}\big)$. By the union bound and tower property of the conditional expectation, for any $M_3>0$ we have
		\begin{align*}
			\Pr(\L_1^N(1)\ge M_2N^{1/3}+M_3) \le 2\e + \Ex\left[\ind_{\m{A}}\Ex\big[\ind_{\L_1^N(1)\ge M_2N^{1/3}+M_3}\mid \mathcal{F}\big]\right]
		\end{align*}
		Using Theorem \ref{thm:conn} we have
		\begin{align*}
			\Ex\big[\ind_{\L_1^N(1)\ge M_2N^{1/3}+M_3}\mid \mathcal{F}\big] = \Pr_{\alpha_p}^{\L_1^N(3),(\L_2^N(2),\L_2^N(4));1,2}(L_1(1)\ge  M_2N^{1/3}+M_3)
		\end{align*}
		On event $\m{A}$, the boundary data are at most $M_2N^{1/3}$. By stochastic monotonicity (Proposition \ref{p:gmc}) and translation invariance of the Gibbs measure (Lemma \ref{obs1} \ref{traninv})  we have
		\begin{align*}
			\ind_{\m{A}}\cdot	\Pr_{\alpha_p}^{\L_1^N(3),(\L_2^N(2),\L_2^N(4));1,2}(L_1(1)\ge  M_2N^{1/3}+M_3) \le \ind_{\m{A}}\cdot\Pr_{\alpha_p}^{0,(0,0);1,2}(L_1(1)\ge M_3).
		\end{align*}
		The last probability can be made less than $\e$ by taking $M_3$ large enough. Thus setting $M_4=M_4(\e):=M_3+M_2$, we see that for all $N\ge 3$, the first inequality in \eqref{ethr} holds with $M=M_4$.
		
		\medskip
		
		\noindent{\textbf{Proof of the second inequality in \eqref{ethr}.}} We start by defining two high probability events $\m{B}_1$ and $\m{B}_2$. The idea is to then show $\Pr\big(\{\L_2^N(2)\le -MN^{1/3}\}\cap \m{B}_1\cap \m{B}_2\big)$ can be made arbitrarily small by choosing $N,M$ large enough.
		
		\medskip
		
		We shall use Theorem \ref{p:high2} (high point on the second curve) with $k\mapsto 1$. Consider $R_0=R_0(1,\e)>0$ from Theorem \ref{p:high2}. Set $R=\max\{R_0,1\}$. By Theorem \ref{p:high2} with $k\mapsto 1$, there exists $M_5(\e)>0$ such that for all large enough $N$
		\begin{align*}
			\Pr(\m{B}_1) \ge 1-\e, \quad \m{B}_1 & := \bigcup_{q\in \ll N^{2/3},RN^{2/3}\rr} \m{B}_1(p), \quad \m{B}_1(q):=\big\{ \L_2^N(2q) \ge -M_5N^{1/3}\big\}.
		\end{align*}
		We write the set $\m{B}_1$ as a union of disjoint sets as follows:
		\begin{align*}
			\m{C}_1(q):= \m{B}_1(q) \cap \bigcap_{s\in \ll q+1,RN^{2/3}\rr}\neg \m{B}_1(s), \quad \m{C}_1:=\bigsqcup_{q\in \ll N^{2/3},RN^{2/3}\rr} \m{C}_1(q)=\m{B}_1.
		\end{align*}
		By Theorem \ref{t:order}, for large enough $N$ we have
		\begin{align*}
			\Pr(\m{B}_2) \ge 1-\e, \quad 	\m{B}_2 & :=\bigcap_{q\in \ll N^{2/3},RN^{2/3}\rr} \m{B}_2(q), \quad \m{B}_2(q):=\big\{\L_2^N(2q)-\L_1^N(2q-1) \le N^{1/3} \big\}.
		\end{align*}
		Set $\mathcal{F}_q:=\sigma\big((\L_1^N(j-1),\L_2^N(j))_{j\ge 2q}, (\L_i^N\ll1,2N-2i+2\rr)_{i\ge 3}\big)$. Observe that $\m{B}_2(q)\cap \m{C}_1(q)$ is measurable with respect to $\mathcal{F}_q$. Note that for any $M_6>0$ we have
		\begin{align}
			\nonumber \Pr\bigg(\left\{\L_2^N(2)\le -{M}_6N^{1/3}\right\} \cap \m{B}_1\cap\m{B}_2\bigg) & \le \sum_{q\in \ll N^{2/3},RN^{2/3}\rr}\Pr\bigg(\left\{\L_2^N(2)\le -{M}_6N^{1/3}\right\}\cap \m{C}_1(p)\cap\m{B}_2(p)\bigg) \\ & = \sum_{q\in \ll N^{2/3},RN^{2/3}\rr} \Ex\left[\ind_{\m{B}_2(q)\cap\m{C}_1(q)}\Ex\left[\ind_{\L_2^N(2)\le -{M}_6N^{1/3}}\mid \mathcal{F}_q\right]\right]. \label{abd}
		\end{align}
		By the Gibbs property (Theorem \ref{thm:conn})  we have
		\begin{align*}
			\ind_{\m{B}_2(q)\cap\m{C}_1(q)} \cdot \Ex\left[\ind_{\L_2^N(2)\le -{M}_6N^{1/3}}\mid \mathcal{F}_q\right] & =\ind_{\m{B}_2(q)\cap\m{C}_1(q)} \cdot \Pr_{\alpha_p}^{(\L_1^N(2q-1),\L_2^N(2q)),(\L_3^N(2i))_{i=1}^{q};2,q}(L_2(2)\le -{M}_6N^{\frac13}) \\ & \le \ind_{\m{B}_2(q)\cap\m{C}_1(q)} \cdot \Pr_{\alpha_p}^{(y_1,y_2),(-\infty)^{q};2,q}(L_2(2)\le -{M}_6N^{1/3}),
		\end{align*}
		where $y_1=-(M_5+1)N^{1/3}$, $y_2=-M_5N^{1/3}$. The last inequality follows due to stochastic monotonicity (Proposition \ref{p:gmc}) as on the event $\m{B}_2(q)\cap\m{C}_1(q)$ we have $\L_2^N(2q) \ge -M_5N^{1/3}$ and $\L_1^N(2q-1) \ge -(M_5+1)N^{1/3}$. By translation invariance (Lemma \ref{obs1} \ref{traninv}) and stochastic monotonicity (Proposition \ref{p:gmc}) we have
		\begin{align*}
			\Pr_{\alpha_p}^{(y_1,y_2),(-\infty)^{q};2,q}(L_2(2)\le -{M}_6N^{1/3}) \le \Pr_{\alpha_p}^{(0,-A\sqrt{q}),(-\infty)^{q};2,q}(L_2(2)\le (M_5+1-{M}_6)N^{1/3}) \le \e,
		\end{align*}
		where the last inequality is uniform over $q\in \ll N^{2/3},RN^{2/3}\rr$ and follows from Proposition \ref{lem:ep} by taking $M_6$ large enough ($A\ge 1$ is defined in \eqref{adef}). Plugging the above bound back in \eqref{abd}, and noting that $(\m{B}_2(q))_{q\in \ll N^{2/3},RN^{2/3}\rr}$ forms a disjoint collection of events we have that $\eqref{abd} \le \e$. Using the fact that $\Pr(\neg \m{B}_i)\le \e$ for $i=1,2$, an application of the union bound yields the second inequality in \eqref{ethr} with $M={M}_6$.
	\end{proof}
	
	\section{Properties of the first two curves of Gibbs measures with no bottom curve}
	\label{sec:rpe}
	In this section, we prove Proposition \ref{lem:ep} that asserts endpoint tightness of \btf\ measures defined in Definition \ref{def:btf}. Along with Proposition \ref{lem:ep}, we also study probabilities of a certain event which we call \textit{region pass event} 	under the \btf\ measure.

	\begin{proposition}\label{l:rpass} Fix any $r,M>0$ and $p\in \{1,2\}$. Set $T=\lfloor rN^{2/3}\rfloor$.  We set $\alpha=\alpha_p$ according to \eqref{acric}. Recall the \btf\ measure from Definition \ref{def:btf}. Let $\vec{y}\in \R^p$ with $y_i=-(M+i-1)N^{1/3}$. There exists $\phi=\phi(r,M)>0$ and $N_0(r,M)>0$ such that for all $N\ge N_0$ we have
		\begin{align}\label{eq:l5.2}
			\Pr_{\alpha_p}^{\vec{y},(-\infty)^{2T};p,2T}(\m{RP}_{p,M})\ge \phi,
		\end{align}
	where the region pass event is defined as
	\begin{align}\label{defrp}
		\m{RP}_{p,M}:= \left\{\inf_{j\in \ll1,2T+p-2\rr} L_p(j) \ge 2MN^{1/3}\right\}.
	\end{align}
	\end{proposition}
	
	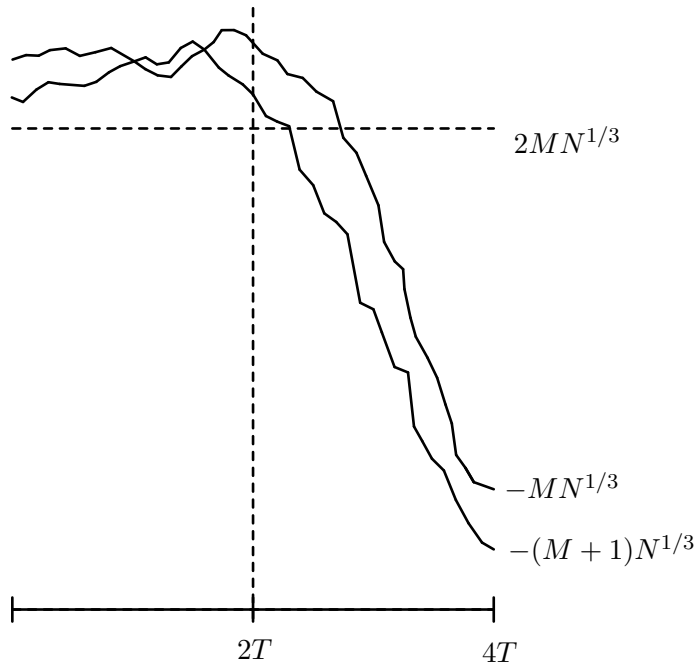
\begin{figure}[h!]
		\centering
		\begin{tikzpicture}[line cap=round,line join=round,>=triangle 45,x=0.8cm,y=0.8cm]
			\draw [line width=1pt] (5,-1)-- (13,-1);
			\draw [line width=1pt,dashed] (5,7)-- (13,7);
			\draw [line width=1pt,dashed] (5,-1)-- (13,-1);
			\draw [line width=1pt] (5,-0.8)-- (5,-1.2);
			\draw [line width=1pt] (9,-0.8)-- (9,-1.2);
			\draw [line width=1pt] (13,-0.8)-- (13,-1.2);
			\draw [line width=1pt] (4.991016453694423,7.515903485791182)-- (5.1773598421639715,7.4409322209817805);
			\draw [line width=1pt] (5.1773598421639715,7.4409322209817805)-- (5.400971908327431,7.647015847315514);
			\draw [line width=1pt] (5.400971908327431,7.647015847315514)-- (5.599738189361617,7.769215770654496);
			\draw [line width=1pt] (5.599738189361617,7.769215770654496)-- (5.786081577831166,7.7405723858819275);
			\draw [line width=1pt] (5.786081577831166,7.7405723858819275)-- (6.196037032464175,7.708820705557649);
			\draw [line width=1pt] (6.196037032464175,7.708820705557649)-- (6.394803313498361,7.775307256603095);
			\draw [line width=1pt] (6.394803313498361,7.775307256603095)-- (6.605992487097184,7.934078733558778);
			\draw [line width=1pt] (6.792335875566733,8.036267974016276)-- (6.605992487097184,7.934078733558778);
			\draw [line width=1pt] (6.792335875566733,8.036267974016276)-- (7.027014017121064,8.115892894703272);
			\draw [line width=1pt] (7.027014017121064,8.115892894703272)-- (7.214714222764378,8.18320017392581);
			\draw [line width=1pt] (7.401057611233927,8.066660137615377)-- (7.599823892268113,8.09871345828387);
			\draw [line width=1pt] (7.401057611233927,8.066660137615377)-- (7.214714222764378,8.18320017392581);
			\draw [line width=1pt] (7.599823892268113,8.09871345828387)-- (7.7985901733022995,8.338211987149002);
			\draw [line width=1pt] (7.7985901733022995,8.338211987149002)-- (7.997356454336486,8.446027041226872);
			\draw [line width=1pt] (7.997356454336486,8.446027041226872)-- (8.190240919375054,8.311989525650649);
			\draw [line width=1pt] (8.190240919375054,8.311989525650649)-- (8.432157694098768,8.011701574788901);
			\draw [line width=1pt] (8.432157694098768,8.011701574788901)-- (8.593655297439044,7.880906406872265);
			\draw [line width=1pt] (8.593655297439044,7.880906406872265)-- (8.82969025616714,7.732916025304531);
			\draw [line width=1pt] (8.82969025616714,7.732916025304531)-- (9,7.566179660144361);
			\draw [line width=1pt] (9,7.566179660144361)-- (9.214799925670874,7.207350679472336);
			\draw [line width=1pt] (9.214799925670874,7.207350679472336)-- (9.401143314140425,7.114299467114728);
			\draw [line width=1pt] (9.401143314140425,7.114299467114728)-- (9.59990959517461,7.037955926656306);
			\draw [line width=1pt] (9.59990959517461,7.037955926656306)-- (9.773830091079523,6.315838385379071);
			\draw [line width=1pt] (9.773830091079523,6.315838385379071)-- (10,6.056501059527714);
			\draw [line width=1pt] (10,6.056501059527714)-- (10.183785545712531,5.58873593358919);
			\draw [line width=1pt] (10.183785545712531,5.58873593358919)-- (10.382551826746719,5.44521690259839);
			\draw [line width=1pt] (10.382551826746719,5.44521690259839)-- (10.568895215216267,5.240384665491709);
			\draw [line width=1pt] (10.568895215216267,5.240384665491709)-- (10.780084388815089,4.102094270590119);
			\draw [line width=1pt] (10.780084388815089,4.102094270590119)-- (11,3.9932274470869396);
			\draw [line width=1pt] (11,3.9932274470869396)-- (11.140348273189552,3.608182865419951);
			\draw [line width=1pt] (11.140348273189552,3.608182865419951)-- (11.351537446788374,3.033384599580616);
			\draw [line width=1pt] (11.351537446788374,3.033384599580616)-- (11.575149512951834,2.9424609433807682);
			\draw [line width=1pt] (11.575149512951834,2.9424609433807682)-- (11.674532653468926,2.0438601711192814);
			\draw [line width=1pt] (11.674532653468926,2.0438601711192814)-- (11.81118447167993,1.8020196500116272);
			\draw [line width=1pt] (11.81118447167993,1.8020196500116272)-- (11.972682075020206,1.5085599529477771);
			\draw [line width=1pt] (11.972682075020206,1.5085599529477771)-- (12.171448356054391,1.3117721799118174);
			\draw [line width=1pt] (12.171448356054391,1.3117721799118174)-- (12.370214637088578,0.8222560561561116);
			\draw [line width=1pt] (12.370214637088578,0.8222560561561116)-- (12.5814038106874,0.4341744267582974);
			\draw [line width=1pt] (12.5814038106874,0.4341744267582974)-- (12.80501587685086,0.11365207232298165);
			\draw [line width=1pt] (12.80501587685086,0.11365207232298165)-- (13,0);
			\draw [line width=1pt] (5.00343934625906,8.146132387316179)-- (5.227051412422519,8.220005022587138);
			\draw [line width=1pt] (5.227051412422519,8.220005022587138)-- (5.438240586021341,8.210403572958194);
			\draw [line width=1pt] (5.438240586021341,8.210403572958194)-- (5.624583974490891,8.302691319256892);
			\draw [line width=1pt] (5.624583974490891,8.302691319256892)-- (5.88546471834826,8.332297677021353);
			\draw [line width=1pt] (5.88546471834826,8.332297677021353)-- (6.121499677076356,8.207041264402111);
			\draw [line width=1pt] (6.121499677076356,8.207041264402111)-- (6.382380420933725,8.26849365061283);
			\draw [line width=1pt] (6.382380420933725,8.26849365061283)-- (6.6432611647910935,8.33854453791707);
			\draw [line width=1pt] (6.6432611647910935,8.33854453791707)-- (7.027014017121064,8.115892894703272);
			\draw [line width=1pt] (7.027014017121064,8.115892894703272)-- (7.20995031934465,7.98730190802151);
			\draw [line width=1pt] (7.20995031934465,7.98730190802151)-- (7.415006177681988,7.883343759704072);
			\draw [line width=1pt] (7.415006177681988,7.883343759704072)-- (7.636688186695326,7.856478554030042);
			\draw [line width=1pt] (7.636688186695326,7.856478554030042)-- (7.808491743680663,8.031221642799224);
			\draw [line width=1pt] (7.808491743680663,8.031221642799224)-- (7.985837350891334,8.202116655464364);
			\draw [line width=1pt] (7.985837350891334,8.202116655464364)-- (8.190240919375054,8.311989525650649);
			\draw [line width=1pt] (8.190240919375054,8.311989525650649)-- (8.346070615538007,8.437552686038327);
			\draw [line width=1pt] (8.346070615538007,8.437552686038327)-- (8.47907982094601,8.634966211551848);
			\draw [line width=1pt] (8.47907982094601,8.634966211551848)-- (8.684135679283349,8.636273837528606);
			\draw [line width=1pt] (8.684135679283349,8.636273837528606)-- (8.883649487395353,8.55007795072708);
			\draw [line width=1pt] (8.883649487395353,8.55007795072708)-- (9.160751998662025,8.243304356069553);
			\draw [line width=1pt] (9.160751998662025,8.243304356069553)-- (9.404602208576696,8.133176139045947);
			\draw [line width=1pt] (9.404602208576696,8.133176139045947)-- (9.570863715336701,7.905714799348354);
			\draw [line width=1pt] (9.570863715336701,7.905714799348354)-- (9.825798025701992,7.840405282105841);
			\draw [line width=1pt] (9.825798025701992,7.840405282105841)-- (10.053022084940663,7.612823212038233);
			\draw [line width=1pt] (10.053022084940663,7.612823212038233)-- (10.330124596207334,7.4539789012551445);
			\draw [line width=1pt] (10.330124596207334,7.4539789012551445)-- (10.5,6.845805361027928);
			\draw [line width=1pt] (10.5,6.845805361027928)-- (10.718068111980674,6.5964635527823585);
			\draw [line width=1pt] (10.718068111980674,6.5964635527823585)-- (10.878787568515344,6.213836242473325);
			\draw [line width=1pt] (10.878787568515344,6.213836242473325)-- (11.08384342685268,5.722227825812045);
			\draw [line width=1pt] (11.08384342685268,5.722227825812045)-- (11.178058280683349,5.113143988825);
			\draw [line width=1pt] (11.35540388789399,4.788603455949127)-- (11.493955143527325,4.659801457397594);
			\draw [line width=1pt] (11.178058280683349,5.113143988825)-- (11.35540388789399,4.788603455949127);
			\draw [line width=1pt] (11.493955143527325,4.659801457397594)-- (11.51612334442866,4.327278443877564);
			\draw [line width=1pt] (11.51612334442866,4.327278443877564)-- (11.615880248484661,3.8529641893566153);
			\draw [line width=1pt] (11.615880248484661,3.8529641893566153)-- (11.704553052089995,3.536710778837294);
			\draw [line width=1pt] (11.704553052089995,3.536710778837294)-- (11.898524809976665,3.1891927771308306);
			\draw [line width=1pt] (11.898524809976665,3.1891927771308306)-- (12.059244266511334,2.8541994787218625);
			\draw [line width=1pt] (13,1)-- (12.668869791297977,1.1181950321951888);
			\draw [line width=1pt] (12.668869791297977,1.1181950321951888)-- (12.530318535664643,1.3548649368496712);
			\draw [line width=1pt] (12.530318535664643,1.3548649368496712)-- (12.375141129355306,1.5709126446947848);
			\draw [line width=1pt] (12.375141129355306,1.5709126446947848)-- (12.303094476425972,2.0937547673411867);
			\draw [line width=1pt] (12.303094476425972,2.0937547673411867)-- (12.197795522144638,2.409651630185193);
			\draw [line width=1pt] (12.197795522144638,2.409651630185193)-- (12.059244266511334,2.8541994787218625);
			\draw [line width=1pt] (12.668869791297977,1.1181950321951888)-- (12.530318535664643,1.3548649368496712);
			\draw [line width=1pt,dashed] (9,9)-- (9,-1.2);
			\draw (13.166135065639724,7.203395294452487) node[anchor=north west] {$2MN^{1/3}$};
			\draw (13.010692073419078,1.45472991623822) node[anchor=north west] {$-MN^{1/3}$};
			\draw (13.072869270307336,0.441062507469147) node[anchor=north west] {$-(M+1)N^{1/3}$};
			\draw (8.565022495908623,-1.319032459012575) node[anchor=north west] {$2T$};
			\draw (12.63762889208953,-1.3812096559008333) node[anchor=north west] {$4T$};
		\end{tikzpicture}
		\caption{The above figure depicts the event $\m{RP}_{2,M}$ under the law $\Pr_{\alpha_2}^{\vec{y},(-\infty)^{2T};2,2T}$.}
		\label{rpass}
	\end{figure}
	
	The event $\m{RP}_{p,M}$ requires the first $2T+p-2$ points of the $p$-th curve to lie above $2MN^{1/3}$. Although this is a low probability event, Proposition \ref{l:rpass} says that this event always has positive probability (independent of $N$) under the \btf\ measure. (see Figure \ref{rpass} for $p=2$ case).

	Recall from \eqref{acric} that $\alpha_1$ and $\alpha_2$ are the boundary parameters corresponding to critical and supercritical phases respectively. Depending on the phase being critical or supercritical, the arguments for proving Proposition \ref{lem:ep} and Proposition \ref{l:rpass} are markedly different. We first give interpretation of the \btf\ laws under the two phases in Section \ref{sec5.1}. In Section \ref{sec5.3} and \ref{sec5.5}, we provide proofs of the aforementioned lemmas for critical and supercritical phases respectively.
	
	In the critical phase, the left boundary attraction between the first two curves is weak as $\alpha_1=O(N^{-1/3})$ -- it is the soft-intersection that only comes into the effect. The analysis of the Gibbs measures in this case is similar to the one done in studying full-space line ensembles and relies on KMT coupling type results. In the supercritical phase, we have $\alpha_2=O(1)$ and the soft-intersection and attraction acts as two opposite forces: one tries to repel the curves and another tries to attract. This situation has asymptotically zero probability. The KMT coupling is no longer suitable to analyze events under this setting. This makes the argument for the supercritical phase more involved.
	
	Let us first introduce a piece of notation that we will use frequently for the remainder of the paper. Consider any probability measure $\Pr^{\bullet}$ on $\R^{|\mathcal{K}_{k,T}|}$ equipped with Borel $\sigma$-algebra on $\R^{|\mathcal{K}_{k,T}|}$ where $\mathcal{K}_{k,T}$ is defined in \eqref{def:kkt}. For $\omega\in \R^{|\mathcal{K}_{k,T}|}$, we denote the coordinate functions as $L_i(j)(\omega):=\omega_i(j)$ for $(i,j)\in \mathcal{K}_{k,T}$. We will simply write  $(L_i(j))_{(i,j)\in \mathcal{K}_{k,T}} \sim \Pr^{\bullet}$ for the random variables $(L_i(j))_{(i,j)\in \mathcal{K}_{k,T}}$ under the measure $\Pr^{\bullet}$.

	\subsection{The \btf\ laws under critical and supercritical phase} \label{sec5.1} In this section, we provide two alternative (and ultimately equivalent) representations of the \btf\ laws defined in Definition \ref{def:btf}. The first representation, provided in Lemma \ref{l:LCrit}, is most suitable for studying the critical phase while the second representation, provided in Lemma \ref{l:LSCrit} (see also Definition \ref{prb} and the discussion in the introduction), is most suitable for studying the supercritical phase.

We begin with the following lemma where we mention how the \btf\ measure on the domain $\mathcal{K}_{k,T}$ (see \eqref{def:kkt}) with boundary condition $\vec{y}$ is well-defined under certain cases.

	\begin{observation}[Well-definedness of \btf\ measures] \label{obs2.5} Take $\vec{y}\in \R^k$. For $k\ge 1$ and $\alpha\in (-\theta,\theta)$, $f_{k,T}^{\vec{y},(-\infty)^{T}}(\mathbf{u})$ (see \eqref{def:fhsgm}) is proportional to
		\begin{align}\label{intr1}
			\prod_{i=1}^{k-1} \prod_{j=1}^{T-\ind_{i=1}} W(u_{i+1,2j};u_{i,2j+1},u_{i,2j-1})\prod_{i=1}^k\prod_{j=1}^{2T-1-\ind_{i=1}}G_{\theta+(-1)^{i+j-1}\alpha,(-1)^{j+1}}(u_{i,j}-u_{i,j+1}).
		\end{align}
		where $W$ and $G$ are defined in \eqref{def:wfns} and \eqref{def:gwt}. Furthermore,	 $f_{2,T}^{\vec{y},(-\infty)^{T}}(\mathbf{u})$ is proportional to
		\begin{equation}
			\label{intr2}\begin{aligned}
				& \exp\left(-e^{u_{2,2}-u_{1,3}}\right)G_{\alpha,1}(u_{2,2}-u_{1,1}) G_{\theta,1}(u_{1,1}-u_{1,2})G_{\alpha+\theta,1}(u_{2,1}-u_{2,2})\\ & \hspace{2cm}\prod_{j=2}^{T-1} W(u_{2,2j};u_{1,2j+1},u_{1,2j-1})\prod_{i=1}^2\prod_{j=2}^{2T-1-\ind_{i=1}} G_{\theta,(-1)^{j+1}}(u_{i,j}-u_{i,j+1}).
			\end{aligned}
		\end{equation}
	Moreover, the above two densities are integrable.
	\end{observation}
	\begin{proof}  Recall the definition of the $\hslg$ Gibbs measure from \eqref{def:fhsgm} and the corresponding graphical representation from Figure \ref{fig00}. Note that the red colored edges comes with a weight of the form  $\exp(-\alpha(u_{2i-1,1}-u_{2i,1}))$ which can be written as a product $e^{-\alpha u_{2i-1,1}}\cdot e^{\alpha u_{2i,1}}$. Thus we will think of each red edge as two red rings on the endpoints of the edges that comes with the (vertex) weight $e^{(-1)^i\alpha u_{i,1}}$ (see Figure \ref{fig12} (C)).  Upon doing this vertex weight identification, the case $k=1$ and $T=3$ corresponds to the top graph of Figure \ref{fig12} (C). One can check that the weights corresponding two graphs in Figure \ref{fig12} (C) are equal. The red vertex weight on the right boundary of Figure \ref{fig12} (C) can be absorbed in the constant of proportionality of the Gibbs measure. Thus ignoring this weight, the remaining weight is precisely given by the expression in \eqref{intr1}. The general $k$ odd case follows by redistributing the weights according to Figure \ref{fig12} (A). For the case when $k$ is even and $\alpha>0$, we redistribute according to Figure \ref{fig12} (B). This leads to the $k=2$ density given in \eqref{intr2}. One can compute also the explicit density for the general even case from Figure \ref{fig12} (B). Since $0\le W\le 1$ and $G$s are densities, it is clear that the expressions in \eqref{intr1} and \eqref{intr2} are integrable.
	\end{proof}
	
	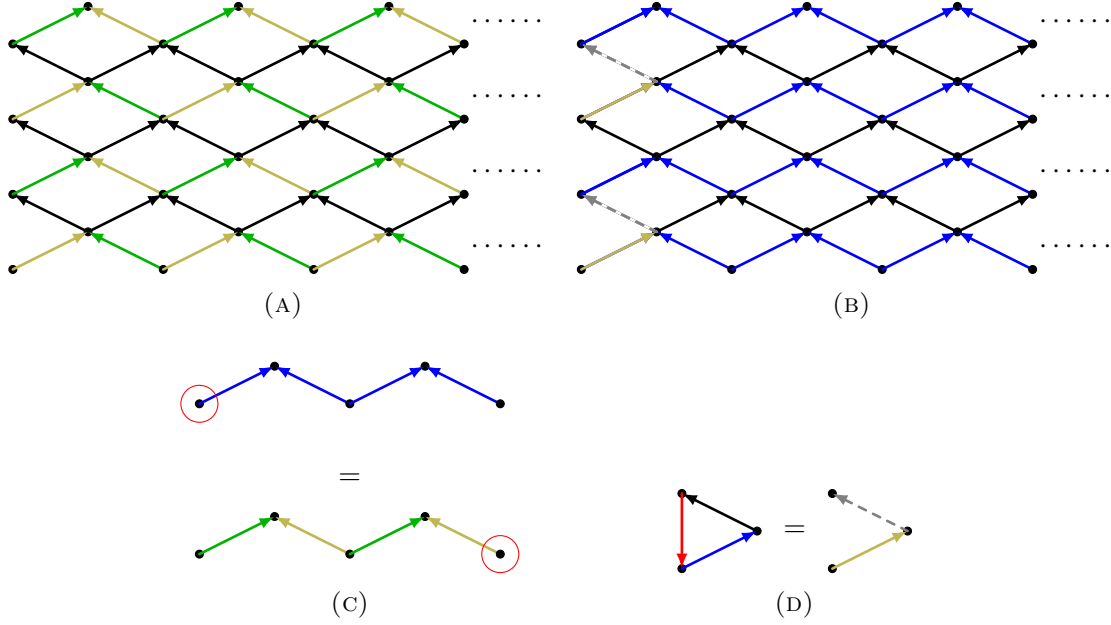
\begin{figure}[h!]
		\centering
		\begin{subfigure}[b]{0.45\textwidth}
			\centering
			\begin{tikzpicture}[line cap=round,line join=round,>=triangle 45,x=2cm,y=1cm]
				\foreach \x in {0,1,2}
				{
					\draw [fill=black] (\x,0) circle (1.5pt);
					\draw [fill=black] (\x,-1) circle (1.5pt);
					\draw [fill=black] (\x,-2) circle (1.5pt);
					\draw [fill=black] (\x-0.5,-0.5) circle (1.5pt);
					\draw [fill=black] (\x-0.5,-1.5) circle (1.5pt);
					\draw [fill=black] (\x-0.5,-2.5) circle (1.5pt);
					\draw [fill=black] (\x-0.5,-3.5) circle (1.5pt);
					\draw [fill=black] (\x,-3) circle (1.5pt);
					\draw[line width=1pt,green!70!black,{Latex[length=2mm]}-]  (\x,0) -- (\x-0.5,-0.5);
					\draw[line width=1pt,yellow!70!black,{Latex[length=2mm]}-] (\x,0) -- (\x+0.5,-0.5);
					\draw[line width=1pt,black,{Latex[length=2mm]}-] (\x-0.5,-0.5) -- (\x,-1);
					\draw[line width=1pt,black,{Latex[length=2mm]}-] (\x+0.5,-0.5) -- (\x,-1);
					\draw[line width=1pt,yellow!70!black,{Latex[length=2mm]}-]  (\x,-1) -- (\x-0.5,-1.5);
					\draw[line width=1pt,green!70!black,{Latex[length=2mm]}-] (\x,-1) -- (\x+0.5,-1.5);
					\draw[line width=1pt,black,{Latex[length=2mm]}-] (\x-0.5,-1.5) -- (\x,-2);
					\draw[line width=1pt,black,{Latex[length=2mm]}-] (\x+0.5,-1.5) -- (\x,-2);
					\draw[line width=1pt,green!70!black,{Latex[length=2mm]}-]  (\x,-2) -- (\x-0.5,-2.5);
					\draw[line width=1pt,yellow!70!black,{Latex[length=2mm]}-] (\x,-2) -- (\x+0.5,-2.5);
					\draw[line width=1pt,black,{Latex[length=2mm]}-] (\x-0.5,-2.5) -- (\x,-3);
					\draw[line width=1pt,black,{Latex[length=2mm]}-] (\x+0.5,-2.5) -- (\x,-3);
					\draw[line width=1pt,yellow!70!black,{Latex[length=2mm]}-]  (\x,-3) -- (\x-0.5,-3.5);
					\draw[line width=1pt,green!70!black,{Latex[length=2mm]}-] (\x,-3) -- (\x+0.5,-3.5);
				}
				\draw [fill=black] (2.5,-0.5) circle (1.5pt);
				\draw [fill=black] (2.5,-1.5) circle (1.5pt);
				\draw [fill=black] (2.5,-2.5) circle (1.5pt);
				\draw [fill=black] (2.5,-3.5) circle (1.5pt);
				\node at (2.8,-0.2) {$\cdots\cdots$};
				\node at (2.8,-1.2) {$\cdots\cdots$};
				\node at (2.8,-2.2) {$\cdots\cdots$};
				\node at (2.8,-3.2) {$\cdots\cdots$};
			\end{tikzpicture}
			\caption{}
		\end{subfigure}
		\begin{subfigure}[b]{0.45\textwidth}
			\centering
			\begin{tikzpicture}[line cap=round,line join=round,>=triangle 45,x=2cm,y=1cm]
				\foreach \x in {0,1,2}
				{
					\draw [fill=black] (\x,0) circle (1.5pt);
					\draw [fill=black] (\x,-1) circle (1.5pt);
					\draw [fill=black] (\x,-2) circle (1.5pt);
					\draw [fill=black] (\x-0.5,-0.5) circle (1.5pt);
					\draw [fill=black] (\x-0.5,-1.5) circle (1.5pt);
					\draw [fill=black] (\x-0.5,-2.5) circle (1.5pt);
					\draw [fill=black] (\x-0.5,-3.5) circle (1.5pt);
					\draw [fill=black] (\x,-3) circle (1.5pt);
					\draw[line width=1pt,blue,{Latex[length=2mm]}-]  (\x,0) -- (\x-0.5,-0.5);
					\draw[line width=1pt,blue,{Latex[length=2mm]}-] (\x,0) -- (\x+0.5,-0.5);
					\draw[line width=1pt,black,{Latex[length=2mm]}-] (\x-0.5,-0.5) -- (\x,-1);
					\draw[line width=1pt,black,{Latex[length=2mm]}-] (\x+0.5,-0.5) -- (\x,-1);
					\draw[line width=1pt,blue,{Latex[length=2mm]}-]  (\x,-1) -- (\x-0.5,-1.5);
					\draw[line width=1pt,blue,{Latex[length=2mm]}-] (\x,-1) -- (\x+0.5,-1.5);
					\draw[line width=1pt,black,{Latex[length=2mm]}-] (\x-0.5,-1.5) -- (\x,-2);
					\draw[line width=1pt,black,{Latex[length=2mm]}-] (\x+0.5,-1.5) -- (\x,-2);
					\draw[line width=1pt,blue,{Latex[length=2mm]}-]  (\x,-2) -- (\x-0.5,-2.5);
					\draw[line width=1pt,blue,{Latex[length=2mm]}-] (\x,-2) -- (\x+0.5,-2.5);
					\draw[line width=1pt,black,{Latex[length=2mm]}-] (\x-0.5,-2.5) -- (\x,-3);
					\draw[line width=1pt,black,{Latex[length=2mm]}-] (\x+0.5,-2.5) -- (\x,-3);
					\draw[line width=1pt,blue,{Latex[length=2mm]}-]  (\x,-3) -- (\x-0.5,-3.5);
					\draw[line width=1pt,blue,{Latex[length=2mm]}-] (\x,-3) -- (\x+0.5,-3.5);
				}
			\draw[line width=1pt,white,{Latex[length=2mm]}-] (-0.5,-2.5) -- (0,-3);
				\draw[line width=1pt,white,{Latex[length=2mm]}-] (-0.5,-0.5) -- (0,-1);
				\draw[line width=1pt,gray,{Latex[length=2mm]}-,dashed] (-0.5,-2.5) -- (0,-3);
				\draw[line width=1pt,blue,{Latex[length=2mm]}-] (0,-2) -- (-0.5,-2.5);
				\draw[line width=1pt,yellow!70!black,{Latex[length=2mm]}-] (0,-3) -- (-0.5,-3.5);
				\draw[line width=1pt,gray,{Latex[length=2mm]}-,dashed] (-0.5,-0.5) -- (0,-1);
				\draw[line width=1pt,blue,{Latex[length=2mm]}-] (0,0) -- (-0.5,-0.5);
				\draw[line width=1pt,yellow!70!black,{Latex[length=2mm]}-] (0,-1) -- (-0.5,-1.5);
				\draw [fill=black] (2.5,-0.5) circle (1.5pt);
				\draw [fill=black] (2.5,-1.5) circle (1.5pt);
				\draw [fill=black] (2.5,-2.5) circle (1.5pt);
				\draw [fill=black] (2.5,-3.5) circle (1.5pt);
				\node at (2.8,-0.2) {$\cdots\cdots$};
				\node at (2.8,-1.2) {$\cdots\cdots$};
				\node at (2.8,-2.2) {$\cdots\cdots$};
				\node at (2.8,-3.2) {$\cdots\cdots$};
			\end{tikzpicture}
			\caption{}
		\end{subfigure}
	\vspace{0.5cm}
	
	\begin{subfigure}[b]{0.35\textwidth}
		\centering
		\begin{tikzpicture}[line cap=round,line join=round,>=triangle 45,x=2cm,y=1cm]
			\foreach \x in {0,1}
			{
				\draw [fill=black] (\x,0) circle (1.5pt);
				\draw [fill=black] (\x,-2) circle (1.5pt);
				\draw [fill=black] (\x-0.5,-0.5) circle (1.5pt);
				\draw [fill=black] (\x-0.5,-2.5) circle (1.5pt);
				\draw[line width=1pt,blue,{Latex[length=2mm]}-]  (\x,0) -- (\x-0.5,-0.5);
				\draw[line width=1pt,blue,{Latex[length=2mm]}-] (\x,0) -- (\x+0.5,-0.5);
				\draw[line width=1pt,green!70!black,{Latex[length=2mm]}-]  (\x,-2) -- (\x-0.5,-2.5);
				\draw[line width=1pt,yellow!70!black,{Latex[length=2mm]}-] (\x,-2) -- (\x+0.5,-2.5);
			}
			\draw [fill=black] (1.5,-0.5) circle (1.5pt);
			\draw [fill=black] (1.5,-2.5) circle (1.5pt);
			\draw[red] (-0.5,-0.5) circle (7pt);
			\draw[red] (1.5,-2.5) circle (7pt);
			\node at (0.5,-1.5) {$=$};
		\end{tikzpicture}
		\caption{}
	\end{subfigure}
			\begin{subfigure}[b]{0.35\textwidth}
		\centering
		\begin{tikzpicture}[line cap=round,line join=round,>=triangle 45,x=2cm,y=1cm]
			\draw [fill=black] (1,-1) circle (1.5pt);
			\draw [fill=black] (0.5,-0.5) circle (1.5pt);
			\draw [fill=black] (0.5,-1.5) circle (1.5pt);
			\draw[line width=1pt,gray,{Latex[length=2mm]}-,dashed] (0.5,-0.5) -- (1,-1);
			\draw[line width=1pt,yellow!70!black,{Latex[length=2mm]}-] (1,-1) -- (0.5,-1.5);
			\draw [fill=black] (0,-1) circle (1.5pt);
			\draw [fill=black] (-0.5,-0.5) circle (1.5pt);
			\draw [fill=black] (-0.5,-1.5) circle (1.5pt);
			\draw[line width=1pt,black,{Latex[length=2mm]}-] (-0.5,-0.5) -- (0,-1);
			\draw[line width=1pt,blue,{Latex[length=2mm]}-] (0,-1) -- (-0.5,-1.5);
			\draw[line width=1pt,red,{Latex[length=2mm]}-] (-0.5,-1.5) -- (-0.5,-0.5);
			\node at (0.25,-1) {$=$};
		\end{tikzpicture}
		\caption{}
	\end{subfigure}

		\caption{Redistribution of edge weights for $\alpha\in (-\theta,\theta)$ (A) and for $\alpha>0$ and $k$ even (B). The weights of \green, \yellow\ (dashed), and \purple\ edges are $e^{(\theta-\alpha)x-e^{x}}$, $e^{\alpha x-e^x}$, and $e^{(\theta+\alpha)x-e^{x}}$ respectively. The weights of \black, \blue\ edges are defined in \eqref{def:wfn}. (A) can be derived from Figure \ref{fig00} by observing the equality (as weight functions) in (C). The red circle around vertex $v$ signifies a vertex weight of the form $e^{\alpha x}$. The red vertex weight on the right boundary can be absorbed in the constant of proportionality of the Gibbs measure. (B) can be derived from Figure \ref{fig00} by observing the equality (as weight functions) in (D).}
		\label{fig12}
	\end{figure}	
	
	This redistribution described in Figure \ref{fig12} will allows us to view the bottom free laws as laws which are absolutely continuous with respect to random walks. We give two such representations which will be useful in our critical and supercritical phase analysis. Towards this end, we introduce  $\qo$-distributions. Given $\theta_1,\theta_2>0$ and $a,b\in \R$, we consider the following two probability density functions
	\begin{align}\label{qdist}
		\qo_{\theta_1,\theta_2;\pm 1}^{(a,b)}(x) \propto G_{\theta_1,\pm1}(a-x)G_{\theta_2,\pm1}(b-x).
	\end{align}
	The graphical representation of the above two distributions are given in Figure \ref{fig13} (B).
	
	\begin{figure}[h!]
		\centering
		\begin{subfigure}[b]{0.45\textwidth}
			\centering
			\begin{tikzpicture}[line cap=round,line join=round,>=triangle 45,x=2cm,y=1cm]
				\draw[fill=gray!10,dashed] (-0.65,-0.25)--(-0.65,-0.75)--(2.65,-0.75)--(2.65,-0.25)--(-0.65,-0.25);
				\foreach \x in {0,1,2}
				{
					\draw [fill=white] (\x,0) circle (1.5pt);
					\draw [fill=black] (\x-0.5,-0.5) circle (1.5pt);
					\draw[line width=1pt,green!70!black,{Latex[length=2mm]}-]  (\x,0) -- (\x-0.5,-0.5);
					\draw[line width=1pt,yellow!70!black,{Latex[length=2mm]}-] (\x,0) -- (\x+0.5,-0.5);
				}
				\draw [fill=black] (2.5,-0.5) circle (1.5pt);
				\node at (2.6,-0.2) {$y_1$};
			\end{tikzpicture}
			\caption{$\Pr_{\alpha_1}^{y_1,(-\infty)^{4};1,4}$}
		\end{subfigure}
		\begin{subfigure}[b]{0.45\textwidth}
			\centering
			\begin{tikzpicture}[line cap=round,line join=round,>=triangle 45,x=2cm,y=1cm]
				\foreach \x in {0}
				{
					\draw [fill=white] (\x,0) circle (1.5pt);
					\draw [fill=black] (\x-0.5,-0.5) circle (1.5pt);
					\draw[line width=1pt,black,{Latex[length=2mm]}-]  (\x,0) -- (\x-0.5,-0.5);
					\draw[line width=1pt,black,{Latex[length=2mm]}-] (\x,0) -- (\x+0.5,-0.5);
				}
				\draw [fill=black] (0.5,-0.5) circle (1.5pt);
				\node at (0.6,-0.4) {$b$};
				\node at (-0.6,-0.4) {$a$};
				\node at (-0.3,-0.1) {$\theta_1$};
				\node at (0.32,-0.1) {$\theta_2$};
				\foreach \x in {2}
				{
					\draw [fill=black] (\x-0.5,0) circle (1.5pt);
					\draw[line width=1pt,black,{Latex[length=2mm]}-]  (\x-0.5,0) -- (\x,-0.5);
					\draw[line width=1pt,black,{Latex[length=2mm]}-] (\x+0.5,0) -- (\x,-0.5);
				}
				\draw [fill=white] (2,-0.5) circle (1.5pt);
				\draw [fill=black] (2.5,0) circle (1.5pt);
				\node at (2.6,0) {$b$};
				\node at (1.4,0) {$a$};
				\node at (1.67,-0.4) {$\theta_1$};
				\node at (2.35,-0.4) {$\theta_2$};
			\end{tikzpicture}
			\caption{$\qo_{\theta_1,\theta_2;+1}^{(a,b)}$ and $\qo_{\theta_1,\theta_2;-1}^{(a,b)}$}
		\end{subfigure}
		\vspace{0.5cm}
		
		\begin{subfigure}[b]{0.9\textwidth}
			\centering
			\begin{tikzpicture}[line cap=round,line join=round,>=triangle 45,x=1.6cm,y=0.8cm]
				\draw[fill=gray!10,dashed] (2.4,-0.35)--(2.4,-0.65)--(4.6,-0.65)--(4.6,-0.35)--(2.4,-0.35);
				\draw[fill=blue!10,dashed] (2.9,-0.85)--(2.9,-1.15)--(5.1,-1.15)--(5.1,-0.85)--(2.9,-0.85);
				\draw (2.35,0.25)--(2.35,-1.75)--(5.23,-1.75)--(5.23,0.25)--(2.35,0.25);
				\draw (5.35,0.25)--(5.35,-1.75)--(8.23,-1.75)--(8.23,0.25)--(5.35,0.25);
				\foreach \x in {-0.2,0.8}
				{
					\draw [fill=black] (\x,0) circle (1.5pt);
					\draw [fill=black] (\x,-1) circle (1.5pt);
					\draw [fill=black] (\x-0.5,-1.5) circle (1.5pt);
					\draw [fill=black] (\x-0.5,-0.5) circle (1.5pt);
					\draw[line width=1pt,green!70!black,{Latex[length=2mm]}-]  (\x,0) -- (\x-0.5,-0.5);
					\draw[line width=1pt,yellow!70!black,{Latex[length=2mm]}-] (\x,0) -- (\x+0.5,-0.5);
					\draw[line width=1pt,black,{Latex[length=2mm]}-] (\x-0.5,-0.5) -- (\x,-1);
					\draw[line width=1pt,black,{Latex[length=2mm]}-] (\x+0.5,-0.5) -- (\x,-1);
					\draw[line width=1pt,yellow!70!black,{Latex[length=2mm]}-]  (\x,-1) -- (\x-0.5,-1.5);
					\draw[line width=1pt,green!70!black,{Latex[length=2mm]}-] (\x,-1) -- (\x+0.5,-1.5);
				}
				\draw[line width=1pt,yellow!70!black,{Latex[length=2mm]}-] (1.8,-1) -- (1.3,-1.5);
				\draw [fill=black] (1.3,-0.5) circle (1.5pt);
				\draw [fill=black] (1.3,-1.5) circle (1.5pt);
				\draw [fill=black] (1.8,-1) circle (1.5pt);
				\node at (1.4,-0.2) {$y_1$};
				\node at (1.9,-0.7) {$y_2$};
				\foreach \x in {3,4}
				{
					\draw [fill=white] (\x,0) circle (1.5pt);
					\draw [fill=black] (\x,-1) circle (1.5pt);
					\draw [fill=white] (\x-0.5,-1.5) circle (1.5pt);
					\draw [fill=black] (\x-0.5,-0.5) circle (1.5pt);
					\draw[line width=1pt,green!70!black,{Latex[length=2mm]}-]  (\x,0) -- (\x-0.5,-0.5);
					\draw[line width=1pt,yellow!70!black,{Latex[length=2mm]}-] (\x,0) -- (\x+0.5,-0.5);
					\draw[line width=1pt,yellow!70!black,{Latex[length=2mm]}-]  (\x,-1) -- (\x-0.5,-1.5);
					\draw[line width=1pt,green!70!black,{Latex[length=2mm]}-] (\x,-1) -- (\x+0.5,-1.5);
				}
				\draw[line width=1pt,yellow!70!black,{Latex[length=2mm]}-] (5,-1) -- (4.5,-1.5);
				\draw [fill=black] (4.5,-0.5) circle (1.5pt);
				\draw [fill=white] (4.5,-1.5) circle (1.5pt);
				\draw [fill=black] (5,-1) circle (1.5pt);
				\node at (4.6,-0.2) {$y_1$};
				\node at (5.1,-0.7) {$y_2$};
				\foreach \x in {6,7}
				{
					\draw [fill=black] (\x,0) circle (1.5pt);
					\draw [fill=black] (\x,-1) circle (1.5pt);
					\draw [fill=black] (\x-0.5,-1.5) circle (1.5pt);
					\draw [fill=black] (\x-0.5,-0.5) circle (1.5pt);
					\draw[line width=1pt,black,{Latex[length=2mm]}-] (\x-0.5,-0.5) -- (\x,-1);
					\draw[line width=1pt,black,{Latex[length=2mm]}-] (\x+0.5,-0.5) -- (\x,-1);
				}
				\draw [fill=black] (7.5,-0.5) circle (1.5pt);
				\draw [fill=black] (7.5,-1.5) circle (1.5pt);
				\draw [fill=black] (8,-1) circle (1.5pt);
				\node at (7.6,-0.2) {$y_1$};
				\node at (8.1,-0.7) {$y_2$};
				\node at (5.3,-0.7) {$\cdot$};
				\node at (2.2,-0.7) {$=$};
			\end{tikzpicture}
			\caption{$\Pr_{\alpha_1}^{(y_1,y_2),(-\infty)^{3};2,3}$}
		\end{subfigure}
		\caption{Figures (A) and (B) are graphical representations of the probability distributions $\Pr_{\alpha_1}^{y_1,(-\infty)^{4};1,4}$ and $\qo_{\theta_1,\theta_2;\pm 1}^{(a,b)}$ respectively. We use the representation from Figure \ref{fig12} (A) here. The black edge labeled as $\theta_i$ in (B) represents that the edge carries a weight of $e^{\theta_i x-e^x}$. (C) shows the decomposition of $\Pr_{\alpha_1}^{(y_1,y_2),(-\infty)^{4};2,4}$ into $\widehat\Pr^{(y_1,y_2)}$ (middle figure) and $\wcr$ (right figure). The marginal law of the gray (blue resp.) shaded region is a random walk started at $y_1$ ($y_2$ resp.) with increment $G_{\theta+\alpha_1,-1}\ast G_{\theta-\alpha_1,+1}$ ($G_{\theta+\alpha_1,+1}\ast G_{\theta-\alpha_1,-1}$ resp.). }
		\label{fig13}
	\end{figure}
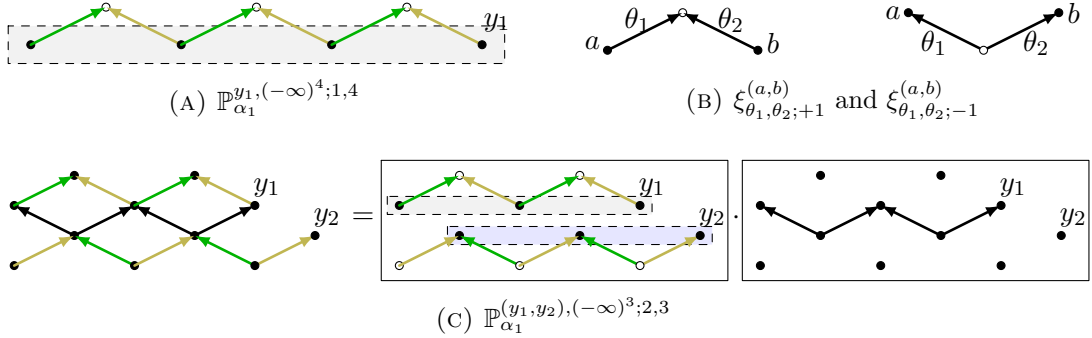

	\begin{observation}[Critical phase representation]\label{l:LCrit}
		Consider an	independent collection of random variables $Y_{i}(j) \stackrel{i.i.d.}{\sim} G_{\theta+\alpha_1,1}$ and $U_{i}(j) \stackrel{i.i.d.}{\sim} \operatorname{Beta}(\theta-\alpha_1,2\alpha_1)$ for $i=1,2$ and $j\in \Z_{\ge 1}$. Define
		\begin{align}\label{def:Vij}
			V_{i}(j):=Y_{i}(2j)+\log U_{i}(2j)-\Ex[\log U_{i}(2j)]-Y_{i}(2j-1).
		\end{align}
		so that $V_{i}(j)$ form an i.i.d.~collection of mean zero random variables. Set $L_i(2T+i-2)=y_i$ and $\mbox{for }k\in \ll1,T-1\rr$ define
		\begin{align}\label{def:LCrit2}
			L_i(2T+i-2k-2):=\left(y_i+(-1)^{i+1}(k-1)\Ex[\log U_{1}(1)]\right)+(-1)^{i+1}\sum_{j=1}^{k-1} V_{i}(j),
		\end{align}
		and set
		\begin{align}\label{def:wcr}
			\wcr:=\exp\left(-\sum_{k=1}^{T-1} \left(e^{L_2(2k)-L_1(2k-1)}+e^{L_2(2k)-L_1(2k+1)}\right)\right),
		\end{align}
		where `cr' stands for critical. Conditioned on $(L_i(2j+i-2))_{i\in \{1,2\},j\in \ll1,T\rr}$, we set $$L_i(2k+i-1) \sim \qo_{\theta-\alpha_1,\theta+\alpha_1,(-1)^{i+1}}^{(L_i(2k+i-2),L_i(2k+i))}\mbox{ for }i\in \{1,2\}, k\in \ll1,T-1\rr$$ and $L_2(1) = X+L_2(2)$ where $X\sim  G_{\theta+\alpha_1,1}$. We have
		
		\begin{enumerate}[label=(\alph*),leftmargin=18pt]
			\setlength\itemsep{0.5em}
			\item $(L_1(j))_{(1,j)\in \mathcal{K}_{1,T}}$ is distributed as $\Pr_{\alpha_1}^{y_1,(-\infty)^{T};1,T}$.
			\item Let $\bar{\Pr}^{(y_1,y_2)}$ denotes the joint law of $\{(L_i(j))_{(i,j)\in \mathcal{K}_{2,T}}\}$. This law has graphical representation given by the middle figure in Figure \ref{fig13} (C). The law $\Pr_{\alpha_1}^{(y_1,y_2),(-\infty)^{T};2,T}$ is absolutely continuous with respect to  $\bar{\Pr}^{(y_1,y_2)}$ with
			\begin{align*}
				\frac{d \Pr_{\alpha_1}^{(y_1,y_2),(-\infty)^{S};2,T}}{d\bar{\Pr}^{(y_1,y_2)}} \propto \wcr.
			\end{align*}
		\end{enumerate}	
	\end{observation}

\begin{proof}   Let us consider $(L_1(j))_{(1,j)\in \mathcal{K}_{1,T}} \sim \Pr_{\alpha_1}^{y_1,(-\infty)^T;1,T}$. See Figure \ref{fig13} (A) for the graphical representation of the law. We focus on the odd points (shaded inside the gray box in the figure). Note that $(L_1(2T-1-2k))_{k=0}^{T-1}$ is a random walk starting at $L_1(2T-1)=y_1$ with increments distributed as $G_{\theta+\alpha_1,1}\ast G_{\theta-\alpha_1,-1}$. Conditioned on the odd points, we have $L_1(2k)\sim \qo_{\theta-\alpha_1,\theta+\alpha_1;1}^{(L(2k-1),L(2k+1))}$. Since $V_{1,j}+\Ex[\log U_{1,1}] \sim G_{\theta+\alpha,1}\ast G_{\theta-\alpha,-1}$, Part (a) of the lemma follows.
		
		\medskip
		
		 Let us now consider the $\Pr_{\alpha_1}^{(y_1,y_2),(-\infty)^T;2,T}$ law whose graphical representation is given in Figure \ref{fig13} (C). We view the graph as superimposition of two graphs where in one graph we collect all the non-black edges and the other graph we include only the black edges (see Figure \ref{fig13} (C)). We denote the law of the Gibbs measure formed by deleting the black edges as $\widehat{\Pr}^{(y_1,y_2)}$ (middle figure in Figure \ref{fig13} (C)). The law $\Pr_{\alpha_1}^{(y_1,y_2),(-\infty)^T;2,T}$ can be recovered from $\widehat{\Pr}^{(y_1,y_2)}$ by viewing the black edges as a Radon-Nikodym derivative. Note that $\wcr$, defined in \eqref{def:wcr}, precisely contains all the effect of the black edges in the Gibbs measure.
		
	     If $(L_i(j))_{(i,j)\in \mathcal{K}_{2,T}} \sim \widehat{\Pr}^{(y_1,y_2)}$, we have $L_1(\cdot)$ independent of $L_2(\cdot)$ and $L_1$ is distributed as $\Pr_{\alpha_1}^{y_1,(-\infty)^T;1,T}$. $L_2$ has a similar representation with even points $(L_2(2T-2k))_{k=0}^{T-1}$ forming a random walk starting at $y_2$ with increments distributed as $G_{\theta+\alpha,-1}\ast G_{\theta-\alpha,1}$. Conditioned on the even points, we have $L_2(2k+1)\sim \qo_{\theta-\alpha_1,\theta+\alpha_1;-1}^{(L(2k),L(2k+2))}$ and $L_2(1)\sim G_{\theta+\alpha_1,1}+L_2(2)$. Since $-V_{2,j}-\Ex[\log U_{1,1}] \sim G_{\theta+\alpha,-1}\ast G_{\theta-\alpha,1}$, we see that $\widehat{\Pr}^{(y_1,y_2)}$ is equal to the law $\bar{\Pr}^{(y_1,y_2)}$ defined in Part~(b) of the lemma. This completes the proof of Part~(b).
\end{proof}

	In the supercritical phase, the weighted paired random walk ($\m{WPRW}$) measure (recall this and the $\m{PRW}$ measure from Definition \ref{prb}) provides a useful way to describe the measure $\Pr_{\alpha_2}^{\vec{y},(-\infty)^{T};2,T}$.
\begin{observation}[Supercritical phase representation]\label{l:LSCrit} Fix any $\vec{y}\in \R^2$ and $T\in \Z_{\ge 2}$. Suppose $(L_1(2j-1),L_2(2j))_{j\in \ll1,T-1\rr} \sim \Pr_{\operatorname{WPRW}}^{T;\vec{y}}$. Conditioned on $(L_1(2i-1),L_2(2i))_{i\in \ll1,T-1\rr}$, set $L_2(1) \sim X+L_2(2)$ where $X\sim G_{\alpha_2+\theta,1}$ and $$L_1(2k) \sim \qo_{\theta,\theta;1}^{(L_1(2k-1),L_1(2k+1))}, \ \ L_2(2k+1)\sim \qo_{\theta,\theta;-1}^{(L_2(2k),L_2(2k+2))}\mbox{ for }k\in \ll1, T-1\rr.$$ Then $(L_i(j))_{(i,j)\in \mathcal{K}_{2,T}}$ is distributed as $\Pr_{\alpha_2}^{\vec{y},(-\infty)^{T};2,T}$.
	\end{observation}
	
	\begin{proof} We use the alternative graph representation of $\Pr_{\alpha_2}^{\vec{y},(-\infty)^{T};2,T}$ law from Figure \ref{fig12} (B). We decompose this graph into two graphs: one without the black edges (middle figure of Figure \ref{fig14}), and one with the black edges, (right figure of Figure \ref{fig14}). However, unlike the critical phase, the Gibbs measure corresponding to the middle figure does not split into two independent parts because of the \yellow\ (dashed) edge. For this measure, the marginal law of the odd points of the first curve and even points of the second curve together form the paired random walk. Upon taking the black edges into consideration (which corresponds to the $\wsc$ weight), the odd points of the first curve and even points of the second curve jointly follow the $\m{WPRW}$ law.
	\end{proof}
	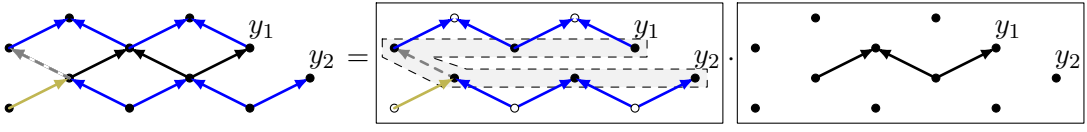
\begin{figure}[h!]
		\centering
		\begin{tikzpicture}[line cap=round,line join=round,>=triangle 45,x=1.6cm,y=0.8cm]
			\draw[fill=gray!10,dashed] (2.4,-0.35)--(2.4,-0.65)--(2.9,-1.15)--(5.1,-1.15)--(5.1,-0.85)--(3.1,-0.85)--(2.8,-0.65)--(4.6,-0.65)--(4.6,-0.35)--(2.4,-0.35);
			\draw (2.35,0.25)--(2.35,-1.75)--(5.23,-1.75)--(5.23,0.25)--(2.35,0.25);
			\draw (5.35,0.25)--(5.35,-1.75)--(8.23,-1.75)--(8.23,0.25)--(5.35,0.25);
			\foreach \x in {-0.2,0.8}
			{
				\draw [fill=black] (\x,0) circle (1.5pt);
				\draw [fill=black] (\x,-1) circle (1.5pt);
				\draw [fill=black] (\x-0.5,-1.5) circle (1.5pt);
				\draw [fill=black] (\x-0.5,-0.5) circle (1.5pt);
				\draw[line width=1pt,blue,{Latex[length=2mm]}-]  (\x,0) -- (\x-0.5,-0.5);
				\draw[line width=1pt,blue,{Latex[length=2mm]}-] (\x,0) -- (\x+0.5,-0.5);
				\draw[line width=1pt,black,{Latex[length=2mm]}-] (\x-0.5,-0.5) -- (\x,-1);
				\draw[line width=1pt,black,{Latex[length=2mm]}-] (\x+0.5,-0.5) -- (\x,-1);
				\draw[line width=1pt,blue,{Latex[length=2mm]}-]  (\x+1,-1) -- (\x+0.5,-1.5);
				\draw[line width=1pt,blue,{Latex[length=2mm]}-] (\x,-1) -- (\x+0.5,-1.5);
			}
		\draw[line width=1pt,white,{Latex[length=2mm]}-] (-0.7,-0.5) -- (-0.2,-1);
			\draw[line width=1pt,yellow!70!black,{Latex[length=2mm]}-] (-0.2,-1) -- (-0.7,-1.5);
			\draw[line width=1pt,gray,{Latex[length=2mm]}-,dashed] (-0.7,-0.5) -- (-0.2,-1);
			\draw[line width=1pt,gray,{Latex[length=2mm]}-,dashed] (2.5,-0.5) -- (3,-1);
			\draw [fill=black] (1.3,-0.5) circle (1.5pt);
			\draw [fill=black] (1.3,-1.5) circle (1.5pt);
			\draw [fill=black] (1.8,-1) circle (1.5pt);
			\node at (1.4,-0.2) {$y_1$};
			\node at (1.9,-0.7) {$y_2$};
			\foreach \x in {3,4}
			{
				\draw [fill=white] (\x,0) circle (1.5pt);
				\draw [fill=black] (\x,-1) circle (1.5pt);
				\draw [fill=white] (\x-0.5,-1.5) circle (1.5pt);
				\draw [fill=black] (\x-0.5,-0.5) circle (1.5pt);
				\draw[line width=1pt,blue,{Latex[length=2mm]}-]  (\x,0) -- (\x-0.5,-0.5);
				\draw[line width=1pt,blue,{Latex[length=2mm]}-] (\x,0) -- (\x+0.5,-0.5);
				\draw[line width=1pt,blue,{Latex[length=2mm]}-]  (\x+1,-1) -- (\x+0.5,-1.5);
				\draw[line width=1pt,blue,{Latex[length=2mm]}-] (\x,-1) -- (\x+0.5,-1.5);
			}
			\draw[line width=1pt,yellow!70!black,{Latex[length=2mm]}-] (3,-1) -- (2.5,-1.5);
			\draw [fill=black] (4.5,-0.5) circle (1.5pt);
			\draw [fill=white] (4.5,-1.5) circle (1.5pt);
			\draw [fill=black] (5,-1) circle (1.5pt);
			\node at (4.6,-0.2) {$y_1$};
			\node at (5.1,-0.7) {$y_2$};
			\foreach \x in {6,7}
			{
				\draw [fill=black] (\x,0) circle (1.5pt);
				\draw [fill=black] (\x,-1) circle (1.5pt);
				\draw [fill=black] (\x-0.5,-1.5) circle (1.5pt);
				\draw [fill=black] (\x-0.5,-0.5) circle (1.5pt);
			}
			\draw[line width=1pt,black,{Latex[length=2mm]}-] (6+0.5,-0.5) -- (6,-1);
			\draw[line width=1pt,black,{Latex[length=2mm]}-] (7-0.5,-0.5) -- (7,-1);
			\draw[line width=1pt,black,{Latex[length=2mm]}-] (7+0.5,-0.5) -- (7,-1);
			\draw [fill=black] (7.5,-0.5) circle (1.5pt);
			\draw [fill=black] (7.5,-1.5) circle (1.5pt);
			\draw [fill=black] (8,-1) circle (1.5pt);
			\node at (7.6,-0.2) {$y_1$};
			\node at (8.1,-0.7) {$y_2$};
			\node at (5.3,-0.7) {$\cdot$};
			\node at (2.2,-0.7) {$=$};
		\end{tikzpicture}
		\caption{$\Pr_{\alpha_2}^{(y_1,y_2),(-\infty)^{3};2,3}$ law is decomposed into two parts. We use the representation from Figure \ref{fig12} (B) here. The first part (middle figure) shaded region corresponds to a paired random walk. The second part (right figure) corresponds to $\wsc$.}
		\label{fig14}
	\end{figure}

	\subsection{Proof of Proposition \ref{lem:ep} and Proposition \ref{l:rpass} in the critical phase}  \label{sec5.3} We continue with the notations from Lemma \ref{l:LCrit}. By the KMT coupling for random walks \cite{kmt}, we may find independent Brownian motions $B_1,B_2$ defined on the same probability space such that the following holds. There exists a constant $\Con>0$ depending only on $\theta$ and $\mu$ such that for all $T$ large enough,
	\begin{align}\label{coupl}
		{\Pr}\left(\max_{k\le T-2} \big|\sum_{j=1}^k V_{i}(j)-\sigma B_i(k)\big| \ge \Con \log T\right) \le 1/T.
	\end{align}
	 $V_{1}(1)$ defined in \eqref{def:Vij} and $\sigma^2:=\operatorname{Var}(V_{1}(1))$. Recall that in the critical phase we have $\alpha_1=N^{-1/3}\mu$. Set $\kappa:=\tfrac14|\mu| \Psi'(\tfrac12\theta)\ge 0$. As $\Psi'$ is a decreasing nonnegative function on $[0,\infty)$, for large enough $N$
	\begin{align}\label{logu}
		\big|\Ex[\log U_1(1)]\big| = \big|\Psi(\theta-\alpha_1)-\Psi(\theta+\alpha_1)\big| \le 2|\alpha_1|\Psi'(\tfrac12\theta) = \tfrac12\kappa N^{-1/3}.
	\end{align}
	Propositions \ref{lem:ep} and \ref{l:rpass} can now be proven using the above coupling and estimate for $|\Ex[\log U_1(1)]|$.
	
\begin{proof}[Proof of Proposition \ref{lem:ep} in the case $p=1$ (critical phase)]   Fix $\e\in (0,1)$. Set
$$\beta_1:=\Pr\Big(\sup_{x\in [0,T]} B_1(x)\le \tfrac{\sqrt{T}}8\Big)=\Pr\Big(\sup_{x\in [0,1]} B_1(x)\le \tfrac{1}8\Big)>0, \quad \beta_2:=\inf_{n\in \mathbb{N}} \exp\big(-2(n-1)e^{-\frac12\sqrt{n}}\big) >0.$$
		Set $T:=\lfloor rN^{2/3}\rfloor$. Continuing with the notations from Lemma \ref{l:LCrit}, let us assume $$(L_i(j))_{(i,j)\in \mathcal{K}_{2,T}}\sim\bar{\Pr}^{(0,-A\sqrt{T})}.$$ 	Observe that $|(T-1)\Ex[\log U_1(1)]|\le \sqrt{r}\kappa\sqrt{T}$. Following the relation in \eqref{def:LCrit2} and the estimate in \eqref{coupl} we get that
		\begin{align*}
		\bar{\Pr}^{(0,-A\sqrt{T})} \bigg(	L_1(2k-1) & \ge -\tfrac18\sqrt{T}-\sqrt{r}\kappa\sqrt{T} -\Con \log T \mbox{ for all } k\in \ll1,T\rr, \mbox{ and} \\ L_2(2k) & \le -A\sqrt{T}+\tfrac18\sqrt{T}+\sqrt{r}\kappa\sqrt{T}+\Con \log T \mbox{ for all } k\in \ll1,T\rr \bigg) \ge \beta_1^2-\tfrac2T.
		\end{align*}
		Recall that $A=1+2\sqrt{r}\kappa$ from \eqref{adef}. Thus for large enough $T$ we can ensure that
		$$\big(-\tfrac18\sqrt{T}-\sqrt{r}\kappa\sqrt{T} -\Con \log T\big)-\big(-A\sqrt{T}+\tfrac18\sqrt{T}+\sqrt{r}\kappa\sqrt{T}+\Con \log T\big) \ge \tfrac12\sqrt{T},$$
		and $\beta_1^2-\frac2T \ge \frac12\beta_1^2$. Thus  for large enough $T$ we have
		\begin{align*}
			\bar{\Pr}^{(0,-A\sqrt{T})}\Big(L_1(2k-1)\wedge L_1(2k+1) \ge L_2(2k)+\tfrac12\sqrt{T}, \mbox{ for all }k\in \ll1,T-1\rr\Big) \ge \tfrac12\beta_1^2.
		\end{align*}
		Following the definition of $\wcr$ from \eqref{def:wcr} we thus get
		\begin{align*}
			\bar{\Ex}^{(0,-A\sqrt{T})}[\wcr] \ge \tfrac12\beta_1^2\cdot \exp\big(-2(T-1)e^{-\frac12\sqrt{T}}\big) \ge \tfrac12\beta_1^2\beta_2.
		\end{align*}
		By Lemma \ref{l:LCrit}, this forces
		\begin{align*}
			\Pr_{\alpha_1}^{(0,-A\sqrt{T}),(-\infty)^{T};2,T}\big(|L_2(2)|\ge M\sqrt{T}\big)  & = \frac{\bar{\Ex}^{(0,-A\sqrt{T})}\big[\wcr\ind_{|L_2(2)|\ge M\sqrt{T}}\big]}{\bar{\Ex}^{(0,-A\sqrt{T})}[\wcr]} \\ & \le 2\beta_1^{-2}\beta_2^{-1}\cdot\bar{\Pr}^{(0,-A\sqrt{T})}\big(|L_2(2)|\ge M\sqrt{T}\big).
		\end{align*}
		Under $\bar{\Pr}^{(0,-A\sqrt{T})}$, $L_2(2)$ has variance $T\cdot\operatorname{Var}(V_{1}(1))$ and mean $-A\sqrt{T}+(T-1)\Ex[\log U_{1}(1)]$. One can thus choose $M$ large enough so that the last term in the above equation is at most $\e$. Similarly one can show $\Pr_{\alpha_2}^{(0,-A\sqrt{T}),(-\infty)^{T};2,T}\big(|L_1(1)|\ge M\sqrt{T}\big)\le \e$ for all large enough $M$. This proves \eqref{eq:l5.1} for $p=1$. For \eqref{eq:l5.w}, observe that by Lemma \ref{l:LCrit} (a) and the Markov inequality one can take $\til{M}$ large enough so that have \begin{align*}
			\Pr_{\alpha_1}^{0,(-\infty)^{T};1,T}\big(|L_1(1)|\ge \til{M}\sqrt{T}\big) & \le \tfrac1{\til{M}^2T}\big(T\cdot\operatorname{Var}(V_{1}(1))+(|(T-1)\Ex[\log U_1(1)])^2\big) \le \e.
		\end{align*}
	\end{proof}
	\begin{proof}[Proof of Proposition \ref{l:rpass} in the case $p=1$ (critical case)]  We continue with the same notations as in Lemma \ref{l:LCrit} with $T\mapsto 2T$. Set $T:=\lfloor rN^{2/3}\rfloor$. Consider the collection of random variables $(L_i(j))_{(i,j)\in \mathcal{K}_{2,2T}}$ defined in \eqref{def:LCrit2} with $y_1=-MN^{\frac13}$. By Lemma \ref{l:LCrit} (a),  we get that $$(L_1(j))_{(i,j)\in \mathcal{K}_{1,2T}} \sim \Pr_{\alpha_1}^{-MN^{1/3};(-\infty)^{2T};1,2T}.$$ We may assume $V_{1}(j)$ are defined in a probability space that includes a Brownian motion $B_1(\cdot)$ such that \eqref{coupl} holds. Recall that given a standard Brownian motion $B$ and an open set $\mathcal{U} \subset C([0,1])$ with $\{f: f(0)=0\} \subset \mathcal{U}$, we have $\Pr(B|_{[0,1]}\in \mathcal{U})>0$. Thus by the scale invariance of Brownian motion, there exists $\phi(\theta,\mu,r,M)>0$ such that
		\begin{align*}
			\Pr\Big( 0 \le \sigma B_1(x) -(16M+5\kappa r)N^{1/3} \le {MN^{1/3}} \mbox{ for all } x\in [\tfrac{T}2,2T]\Big) \ge 2\phi.
		\end{align*}
		Here $\kappa=\frac14|\mu|\Psi'(\frac12\theta)\ge 0$. Now for $y=-MN^{\frac13}$ we have $|y+(k-1)\Ex[\log U_1(1)]| \le (M+\kappa r)N^{\frac13}$ for all $k\le 2T$. For large enough $N$ we also have $\Con \log 2T \le MN^{1/3}$ where $\Con$ comes from \eqref{coupl}. Thus in view of \eqref{def:LCrit2} and \eqref{coupl} we have
		\begin{equation}\label{lwrbdA}
			\begin{aligned}
				& \Pr_1\bigg((14M+4\kappa r)N^{\frac13} \le L_1(4T-1-2k)  \\ & \hspace{3.5cm} \le (19M+6\kappa r)N^{\frac13} \mbox{ for all } k\in \ll\tfrac{T}2,2T-1\rr\bigg) \ge 2\phi-\tfrac1{2T},
			\end{aligned}
		\end{equation}
		where for simplicity we write $\Pr_1:=\Pr_{\alpha_1}^{-MN^{\frac13},(-\infty)^{2T};1,2T}$. Let us set
		\begin{align*}
			\m{A} & :=\left\{(14M+4\kappa r)N^{\frac13} \le L_1(4T-1-2k)  \le (19M+6\kappa r)N^{\frac13} \mbox{ for all } k\in \ll\tfrac{T}2,2T-1\rr\right\}, \\
			\m{B}(k) & :=\left\{ |L_1(2k-1)-L_1(2k)|, |L_1(2k+1)-L_1(2k)| \ge  2(5M+2\kappa r)N^{\frac13} \right\}.
		\end{align*}
		Recall the event $\m{RP}_{1,M}$ from \eqref{defrp}. Observe that
		\begin{align*}
			\m{RP}_{1,M} \supset \m{A}\cap \bigcap_{k\in \ll1,3T/2-1\rr}  \m{B}(k)
		\end{align*}
		Thus by applying the union bound we get
		\begin{align}
			\Pr_1(\m{RP}_{1,M})  \ge \Pr_1\bigg(\m{A}\cap \bigcap_{k\in \ll1,3T/2-1\rr}  \m{B}(k)\bigg) \ge \Pr_1(\m{A})-\sum_{k\in \ll1,3T/2-1\rr}\Pr_1\bigg(\m{A}\cap  \neg \m{B}(k)\bigg) \label{lwrbd1}
		\end{align}
		Let us denote $\mathcal{F}_{\operatorname{odd}}:=\sigma\big((L_1(2k-1))_{k=1}^{2T}\big)$. Note that the event $\m{A}$ is measurable with respect to $\mathcal{F}_{\operatorname{odd}}$. On the event $\m{A}$, $|L_1(2k+1)-L_1(2k-1)| \le (5M+2\kappa r)N^{\frac13}$ for all $k\in \ll1,3T/2-1\rr$. Recall that the distribution of even points of $L_1$ conditioned on $\mathcal{F}_{\operatorname{odd}}$ are given by the $\qo$-distributions (see \eqref{qdist} and Lemma \ref{l:LCrit}). Applying the tail bound for the $\qo$-distribution from Lemma \ref{qlemma} we have
		\begin{align*}
			\ind_{\m{A}}\Ex_1\big(\ind_{\neg \m{B}(k)}\mid \mathcal{F}_{\operatorname{odd}}\big) \le \ind_{\m{A}}\cdot \exp\left(-\Con (5M+2\kappa r)N^{\frac13}\right),
		\end{align*}
		for all $k\in \ll1,3T/2-1\rr$. Taking another expectation  with respect to $\mathcal{F}_{\operatorname{odd}}$ above and then plugging the bound back in \eqref{lwrbd1}, along with the lower bound of $\Pr_1(\m{A})$ from \eqref{lwrbdA} we get that
		\begin{align*}
			\Pr_1(\m{RP}_{1,M}) \ge 2\phi-\tfrac1{2T}-3r N^{\frac23} \exp\left(-\Con (5M+2\kappa r)N^{\frac13}\right).
		\end{align*}
		For large enough $N$, the right side of above equation is always larger than $\phi$.
	\end{proof}

	\subsection{Proof of Propositions \ref{lem:ep} and \ref{l:rpass} in the supercritical phase} \label{sec5.5} Recall that Lemma \ref{l:LSCrit} establishes that the law of $\Pr_{\alpha_2}^{\vec{y},(-\infty)^{T};2,T}$ is related to the law of weighted paired random walk ($\m{WPRW}$) defined in Definition \ref{prb}. We will start by developing a few important properties of the paired random walk and weighted paired random walk measures before going into the proof of Propositions \ref{lem:ep} and \ref{l:rpass} in the $p=2$ case (supercritical phase).
	
\begin{lemma}[Properties of the increments] \label{as:in} The densities $\fa$ and $\ga$, defined in \eqref{def:faga}, enjoy the following properties.
		\begin{enumerate}
			\item The density $\fa$ is symmetric and $\log \fa$ is concave.
			
			\item Let $\psi$ denote the characteristic function corresponding to $\fa$. $|\psi|$ is integrable. Given any $\delta>0$, there exists $\eta$ such that $\sup_{t\ge \delta} |\psi(t)|=\eta <1$.

			\item For any $a<b$, $\inf_{x\in [a,b]} \fa(x) >0$ and $\inf_{x\in [a,b]} \ga(x) >0$.
			
			\item There exists a constant $\Con>0$ such that $\fa(x)\le \Con e^{-|x|/\Con}$ and $\ga(x)\le \Con e^{-|x|/\Con}$. In particular, this implies that if $X \sim \fa$ and $Y\sim \ga$, there exists $v>0$ such that and
			\begin{align*}		\sup_{|t|\le v}\big[\Ex[e^{tX}]+\Ex[e^{tY}]\big] <\infty.
			\end{align*}
			In other words $X$ and $Y$ are subexponential random variables.
		\end{enumerate}
	\end{lemma}
	\begin{proof} Recall that $\fa=G_{\theta,+1}\ast G_{\theta,-1}$. Thus the random variable corresponding to $\fa$ can be viewed as difference of two independent random variables drawn from $G_{\theta,+1}$. Hence symmetricity claim of part~(1) follows. Concavity of $\log\fa$ can be checked by computing the second derivative explicitly. {For the characteristic function from \cite[Formula 5.8.3]{NIST:DLMF} one has
			\begin{align}\label{eq:chint}
				\psi(t)=\left|\frac{\Gamma(\theta+it)}{\Gamma(\theta)}\right|^2=\prod_{n=0}^{\infty} \left(1+\frac{t^2}{(\theta+n)^2}\right)^{-1} \le \left(1+\frac{t^2}{\theta^2}\right)^{-1}.
			\end{align}
			From here, one can verify part~(2) of the above lemma.} Part~(3) and (4) follows from the explicit form of the $G$ function.
	\end{proof}

The following corollary allows us to use estimates developed in Appendix \ref{app2} regarding non-intersection probabilities for random walks and bridges with general jump distributions.

\begin{corollary}\label{rm:con} The density $\fa$ defined in \eqref{def:faga} satisfies Assumption \ref{asp:in}.
\end{corollary}
\begin{proof}
This follows immediately from Lemma \ref{as:in}.
\end{proof}

	Recall the $\m{PRW}$ law from Definition \ref{prb}. Let $\bar\fa^{(n)}$ be the density of $\frac{X(1)+\cdots+X(n)}{\sqrt{n}}$ where $X(i)$ are i.i.d.~drawn from $\fa$. Assume $U(n),V(n) \stackrel{i.i.d.}{\sim} \bar\fa^{(n-1)}$.   Note that for any Borel set $A\subset \R^2$ and $x,y\in \R$ we have
	\begin{align}\label{wrn}
		\Pr_{\operatorname{PRW}}^{n;(x,y)}\big((\ise,\iise)\in A\big)= \frac{\Ex\Big[\ga(\sqrt{n}(V(n)-U(n)-\frac{x-y}{\sqrt{n}}))\ind_{(U(n)+x,V(n)+y)\in {n}^{-1/2}A}\Big]}{\Ex[\ga(\sqrt{n}(V(n)-U(n)-\frac{x-y}{\sqrt{n}}))]}.
	\end{align}
	The above formula is the guiding principle for extracting tail estimates of various kinds of functions of $(\ise, \iise)$. We list few of them that are indispensable for our later analysis.
	
	\begin{lemma}[Tail estimates for the entrance law] \label{as:el} Fix any $M>0$, $n\ge 1$, and consider $x_n,y_n\in \R$ with $|x_n|, |y_n|\le M\sqrt{n}$. Fix two open intervals $I_1,I_2>0$. Under the above setup, there exists a constant $\Con=\Con(M,I_1,I_2)>1$ such that for all $n\ge 1$ and $\tau\ge 1$, we have	
		\begin{align} \label{t1}
			& \Pr_{\operatorname{PRW}}^{n;(x_n,y_n)}\big(|\ise|\ge \tau\sqrt{n}\big) \le \Con e^{-\frac1\Con \tau}, \\
			& \Pr_{\operatorname{PRW}}^{n;(x_n,y_n)}\big(|\ise-\iise|\ge \tau\big) \le \Con e^{-\frac1\Con \tau}, \label{t2} \\  & \Pr_{\operatorname{PRW}}^{n;(x_n,y_n)}\big(\ise-\iise \in I_1, \ise \in \sqrt{n}I_2\big) \ge \tfrac1\Con. \label{t3}
		\end{align}
		
	\end{lemma}

	\begin{proof}[Proof of Lemma \ref{as:el}] For simplicity let us write $z_n:=\frac{x_n-y_n}{\sqrt{n}}$. It is enough to prove the Lemma \ref{as:el} for large enough $n$. So, throughout the proof we will assume $n$ is large enough. We first claim that the denominator of the right-hand side of \eqref{wrn} is of the order $n^{-1/2}$, i.e. there exists a $\Con>1$ such that for all large enough $n$ we have
		\begin{align}\label{410}
			\frac1{\Con \sqrt{n}} \le \Ex[\ga(\sqrt{n}(V(n)-U(n)-z_n))]  \le \frac{\Con}{\sqrt{n}}
		\end{align}
	Fix any $\tau\ge 0$. Using the exponential tails of $\ga$ (part~(4) from Lemma \ref{as:in}) we have
	\begin{equation}
		\label{g21}\begin{aligned}
			& \Ex\bigg[\ga(\sqrt{n}(V(n)-U(n)-z_n))\ind_{|\sqrt{n}(V(n)-U(n)-z_n)|\ge \tau}\bigg] \\ & \le  \Ex\bigg[\ga(\sqrt{n}(V(n)-U(n)-z_n))\ind_{|\sqrt{n}(V(n)-U(n)-z_n)|\ge \tau+\sqrt{n}}\bigg] \\ & \hspace{3cm}+\sum_{p\in \ll \tau,\tau+\sqrt{n}\rr} \Ex\bigg[\ga(\sqrt{n}(V(n)-U(n)-z_n))\ind_{|\sqrt{n}(V(n)-U(n)-z_n)|\in [p,p+1]}\bigg] \\ & \le \Con e^{-\frac1{\Con}(\sqrt{n}+\tau)}+ \sum_{p\in \ll \tau,\tau+\sqrt{n}\rr} \Con  e^{-\frac1{\Con}p} \cdot \Pr\big(\tfrac{p}{\sqrt{n}} \le |V(n)-U(n)-z_n|\le \tfrac{p+1}{\sqrt{n}}\big).
		\end{aligned}
	\end{equation}
		Note that $\bar\fa^{(n)}(x)=\sqrt{n}\fa^{\ast n}(x\sqrt{n})$ where $\fa^{\ast n}$ denotes the $n$-fold convolution of $\fa$. As $n\to \infty$, we know by central limit theorem that this should converge to a Gaussian density with appropriate variance. Lemma \ref{sharp} (recall Corollary \ref{rm:con}) records a sharp quantitative version of this convergence. Indeed, the estimate from Lemma \ref{sharp} (recall Corollary \ref{rm:con}) ensures that given any interval $B:=[\frac{p}{\sqrt{n}},\frac{p+1}{\sqrt{n}}]\subset [-2,2]$, for all large enough $n$, we have $$\Pr\big((V(n)-U(n)-z_n)\in B\big)=(1+o(1))\Pr\big(Z_1-Z_2-z_n\in B\big)$$ where $Z_1,Z_2$ are independent Gaussian random variables with same variance as of $\fa$. By Gaussian computations we can ensure that for all large enough $n$ we have
$$\Pr\big((V(n)-U(n)-z_n)\in B\big) \in [R^{-1}/\sqrt{n},R/\sqrt{n}]$$ for some $R>1$ depending only on $M$. This ensures $\Pr\big(\tfrac{p}{\sqrt{n}} \le |V(n)-U(n)-z_n|\le \tfrac{p+1}{\sqrt{n}}\big) \le \frac{R}{\sqrt{n}}$ for all $p\in \ll0,\sqrt{n}\rr$. Plugging this bound back in r.h.s.~\eqref{g21} leads to
		\begin{align}\label{411}
			\Ex\bigg[\ga(\sqrt{n}(V(n)-U(n)-z_n))\ind_{|\sqrt{n}(V(n)-U(n)-z_n)|\ge \tau}\bigg] \le \frac{\Con}{\sqrt{n}}e^{-\tau/\Con}.
		\end{align}
		Taking $\tau=0$ leads to the upper bound in \eqref{410}. For the lower bound we note
		\begin{equation}
			\label{g1}
			\begin{aligned}
				\Ex[g(\sqrt{n}(V(n)-U(n)-z_n))]  & \ge \Ex\bigg[g(\sqrt{n}(V(n)-U(n)-z_n))\ind_{V_n-U_n-x_n+y_n \in (\frac1{\sqrt{n}},\frac2{\sqrt{n}})}\bigg]  \\ & \ge  \Pr\bigg(V(n)-U(n)-z_n \in (\tfrac1{\sqrt{n}},\tfrac2{\sqrt{n}})\bigg)\cdot \inf_{x\in [1,2]} \ga(x) \\ & \ge  \tfrac{R^{-1}}{\sqrt{n}}\cdot \inf_{x\in [1,2]} \ga(x),
			\end{aligned}
		\end{equation}
		which is bounded below by $1/[\Con\sqrt{n}]$, by the property of $\ga$ from part~(3) of Lemma \ref{as:in}. This proves the lower bound in \eqref{410}.
		
		\medskip
		
		Let us now prove the inequalities in Lemma \ref{as:el} one by one. Inserting the upper bound in \eqref{411} and lower bound in \eqref{410} in the formula \eqref{wrn} leads to \eqref{t2}.	 For \eqref{t3} notice that due to \eqref{wrn} and \eqref{g1} we have
		\begin{align*}
			& \Pr\big(\ise-\iise \in I_1, \ise \in \sqrt{n}I_2\big) \\ & \ge \Con_1^{-1} \sqrt{n} \cdot \inf_{x\in I_1} g(-x) \cdot \Pr\Big(U(n)+\tfrac{x_n}{\sqrt{n}} \in \sqrt{n}I_2, U(n)+z_n-V(n)\in {n}^{-1/2}I_1\Big).
		\end{align*}
		Using Lemma \ref{sharp} (recall Corollary \ref{rm:con}) again, the probability above can be shown lower bounded by $\frac{\Con_2^{-1}}{\sqrt{n}}$ for some $\Con_2$ depending on $M,I_1,I_2$ but free of $n$. This proves \eqref{t3}. For \eqref{t1} we observe
		\begin{align*}
			& \Ex\big[g(\sqrt{n}(V(n)-U(n)-z_n))\ind_{|U(n)|\ge \tau}\big] \\ & \le \Con e^{-\frac{\sqrt{n}}{\Con}}+  \sum_{p \in \ll 0,\sqrt{n}\rr} \Con  e^{-\frac{p}{\Con}} \Pr\Big(\tfrac{p}{\sqrt{n}}\le |V(n)-U(n)-z_n|\le \tfrac{p+1}{\sqrt{n}}, |U(n)|\ge \tau\Big).
		\end{align*}
		By a union bound followed by tower property of conditional expectation, we have
		\begin{align*}
			& \Pr\Big(\tfrac{p}{\sqrt{n}}\le |V(n)-U(n)-z_n|\le \tfrac{p+1}{\sqrt{n}}, |U(n)|\ge \tau\Big) \\ & \le \Ex\left[\ind_{\tau \le |U(n)|\le (\log n)^{3/2}} \Pr \left(\tfrac{p}{\sqrt{n}}\le |V(n)-U(n)-z_n|\le \tfrac{p+1}{\sqrt{n}} \mid U(n)\right) \right]+\Pr\big(|U(n)|\ge (\log n)^{3/2}\big).
		\end{align*}
		By Lemma \ref{sharp} (recall Corollary \ref{rm:con}), the conditional probability above can be uniformly bounded above by $\frac{\Con_3}{\sqrt{n}}$ for some $\Con_3$ independent of $p$ and $n$. Exponential tail estimates of $U(n)$, which follows from sub-exponential property (part~(4) of Lemma \ref{as:in}) of $\fa$ (see Theorem 2.8.1 from \cite{vers}), show that the right-hand side of the above equation is at most $\frac{\Con}{\sqrt{n}}e^{-\frac1\Con\tau}$. Combining these estimates yields
$$\Ex\left[g(\sqrt{n}(V(n)-U(n)-z_n))\ind_{|U(n)|\ge \tau}\right] \le \tfrac{\Con}{\sqrt{n}}e^{-\frac1\Con\tau}.$$ Using the lower bound for the denominator from \eqref{g1}, in view of \eqref{wrn}, we get \eqref{t1}.
	\end{proof}

	In order to deal with the $\m{WPRW}$ law, the weighted version of the $\m{PRW}$ law (see Definition \ref{prb}), we next analyze the $\wsc$ weight defined in \eqref{defw}. We  record a convenient lower bound for $\wsc$ that will be useful in our later analysis. Fix any $p,q\in \Z_{\ge 1}$ with $p+q\le n-1$. Given any $\beta>0$, we consider several `Gap' events:
	\begin{align*}
		\m{Gap}_{1,\beta} & :=\{\iks-\iiks \ge \beta k^{1/4} \mbox{ for all } k\in \ll2,p\rr\}, \\
		\m{Gap}_{2,\beta} & :=\{\iks-\iiks \ge \beta (n-k)^{1/4} \mbox{ for all } k\in \ll n-q,n-1\rr\}, \\
		\m{Gap}_{3,\beta} & :=\{\iks-\iiks \ge  n^{1/4} \mbox{ for all } k\in \ll p+1,n-q\rr\}, \\
		\m{Gap}_{4,\beta} & :=\{\se{1}{k-1}-\iks \le \beta^{-1} k^{1/8} \mbox{ for all } k\in \ll2,p\rr\}, \\
		\m{Gap}_{5,\beta} & :=\{\se{1}{k-1}-\iks \le \beta^{-1} (n-k+1)^{1/8} \mbox{ for all } k\in \ll n-q+1,n\rr\}, \\
		\m{Gap}_{6,\beta} & :=\{|\iks-\se{1}{k-1}| \le \beta^{-1}(\log n) \mbox{ for all } k\in \ll p+1,n-q\rr\}.
	\end{align*}
	The events depend on $p$ and $q$ as well, but we have suppressed it from the notation. $\m{Gap}_{1,\beta},\m{Gap}_{2,\beta},$ and $\m{Gap}_{3,\beta}$ requires $\iks-\iiks$ to be bigger than a threshold pointwise in the left ($\ll2,p\rr$), right ($\ll n-q,n-1\rr$), and middle ($\ll p+1,n-q\rr$) region respectively. The type of threshold depends on the region. $\m{Gap}_{4,\beta},\m{Gap}_{5,\beta},$ and $\m{Gap}_{6,\beta}$ controls the increments of $\iks$.	Set
	\begin{align}\label{def:gp}
		\m{Gap}_{\beta}:=\bigcap_{i=1}^6 \m{Gap}_{i,\beta}.
	\end{align}

	We have the following deterministic inequality for $\wsc$.
	\begin{lemma}\label{wgap} Recall $\wsc$ from \eqref{defw}. Given any $\beta>0$, there exists $a_{\beta}>0$ such that for all $n\ge 1$,
		$$\wsc \ge a_{\beta} \cdot \ind_{\m{Gap}_{\beta}\cap \{|\ise-\iise|\le \beta^{-1}\}}.$$
		where $\wsc$ is defined in \eqref{defw}.
	\end{lemma}
	
	\begin{proof} Assume $\m{Gap}_{\beta}$ holds. For $k\in \ll2,n-1\rr$ we have
		\begin{align*}
			\iiks-\iks & \le -\min (\beta k^{1/4}, \beta (n-k)^{1/4}, n^{1/4})=:\tau^{(n)}(k).
		\end{align*}
		Clearly $\sum_{k=2}^{n-1}e^{\iiks-\iks} \le \sum_{k=1}^{n-1} e^{\tau^{(n)}(k)}$ is uniformly bounded in $n$ and hence can be bounded by some constant ${C}_\beta\in (0,\infty)$.	Similarly for $k\in \ll2,n-1\rr$ we have
		\begin{align*}
			\iiks-\se{1}{k+1} & = \iiks-\iks+\iks-\se{1}{k+1} \\ & \le \tau_n^{(1)}(k)+\beta^{-1} \max ((k+1)^{1/8},(n-k)^{1/8},(\log n)) =:\til\tau^{(n)}(k).
		\end{align*}
		Clearly $\sum_{k=2}^{n-1}e^{\iiks-\se{1}{k+1}} \le \sum_{k=1}^{n-1} e^{\til\tau^{(n)}(k)}$ is uniformly bounded in $n$ and hence can be bounded by some constant $\til{C}_\beta\in (0,\infty)$. Finally,
		$$\se{2}{1}-\se{1}{1}= \se{1}{1}-\se{2}{2}+\se{2}{2}-\se{1}{2} \le 3\beta^{-1}$$
		on the event $\{|\se{1}{1}-\se{2}{1}|\le \beta^{-1}\}\cap \m{Gap}_{4,\beta}$.	Thus from the definition of $\wsc$ in \eqref{defw} we have
		\begin{align*}
			\wsc \ge \ind_{\m{Gap}_{\beta}\cap \{|\ise-\iise|\le \beta^{-1}\}} \cdot \exp(-e^{3\beta^{-1}}-C_{\beta}-\til{C}_\beta).
		\end{align*}
		Taking $a_\beta:=\exp(-e^{3\beta^{-1}}-C_{\beta}-\til{C}_\beta)$ completes the proof.
	\end{proof}

Note that upon conditioning on the values of $\se{1}{1}$ and $\se{2}{1}$, a paired random walk  (recall Definition \ref{def:rws}) has the law of two independent $n$-step random bridges (recall Definition \ref{def:rws}) starting from $(\se{1}{1},\se{2}{1})$ and ending at $(x,y)$.

We now introduce modified random bridge above as they are easier to work with than random bridges. Indeed, as described in the proof idea section of the introduction, the main advantage of working with modified random bridges is that they have a (true) random walk portion and one can appeal to  classical non-intersection probability estimates available for the random walks. On the other hand, the laws of random  bridge and modified random bridge can be compared with the help of Lemma \ref{lem:compare} below.
Figure \ref{fmrb2} contains an illustration of such a bridge for $p=q=\lfloor n/4 \rfloor$.

	\begin{definition}[$(n;p,q)$-modified random bridge] \label{def:mrb} Fix $n\ge 1$, and $p,q\in \ll0,n\rr$ with $p+q\le n$ and $p\neq 0$. Take any $a,b\in \R$. Let $(X(i),Y(i))_{i\in \Z_{\ge 1}} \stackrel{i.i.d.}{\sim} \fa$ where $\fa$ is defined in \eqref{def:faga}. Set $\se{}{1}:=a$ and $\se{}{n}:=b$. Set For $k\in \ll2,p\rr$, set $$\se{}{k}:=a+\sum_{j=1}^{k-1} X_j \mbox{ for }k\in \ll2,p\rr, \quad  \se{}{n-k}:=b-\sum_{j=1}^k Y_j \mbox{ for }k\in \ll1,q\rr.$$ Conditioned on $(\se{}{k})_{k\in \ll1,p\rr\cup \ll n-q,n\rr}$, set $(\se{}{k})_{k=p}^{n-q} \sim \pr{n-p-q+1}{\til{a}}{,\til{b}}$ where $\til{a}:=\se{}{p}$, $\til{b}:=\se{}{n-q}$, and $\pr{m}{a}{,b}$ is a $m$-step random walk from $a$ to $b$. We call $(\se{}{k})_{k\in \ll1,n\rr}$ a $(n;p,q)$-modified random bridge of length $n$ starting at $a$ and ending at $b$ and denote its law as $\tpr{(n;p,q)}{a}{,b}$. Again, we shall often consider two independent $(n;p,q)$-modified random bridges starting from $(a_1,a_2)$ and ending at $(b_1,b_2)$. Such bridges can be viewed as a measure on $(\Omega^2_n,\mathcal{F}^2_n)$ space introduced in Definition \ref{prb}. We write $\widetilde{\mathbb{P}}^{(n;p,q);(a_1,a_2),(b_1,b_2)}$ to denote its law.
	\end{definition}

	\begin{lemma}[Comparison Lemma] \label{lem:compare} Fix any $M, \til{M}>0$ and $\delta_1,\delta_2\in [0,1/2)$, and $n\ge 1$. Set $p=\lfloor n\delta_1 \rfloor$ and $q=\lfloor n\delta_2\rfloor$. For $\vec{x}\in \R^{n-2}$, let $V_{a,b}(\vec{x})$ and $\til{V}_{a,b}(\vec{x})$ be the joint density of a $n$-step random bridge and $(n;p,q)$-modified random bridge starting at $a$ and ending at $b$.  Suppose $a,b\in\R$ with $|a-b|\le M\sqrt{n}$.  Then, there exists two constants $\Con_1=\Con_1(M,\delta_1,\delta_2)>0$ and $\Con_2=\Con_2(M,\til{M},\delta_1,\delta_2)>0$ such that for all $\vec{x}\in \R^{n-2}$ and all $a,b\in \R$ with $|a-b|\le M\sqrt{n}$ we have
		\begin{align}\label{eb7}
		&	V_{a,b}(\vec{x}) \le \Con_1 \!\cdot\! \til{V}_{a,b}(\vec{x}),\\
		\label{eb8}
		&	V_{a,b}(\vec{x})\!\cdot\!\ind_{|x_{p}-x_{n-q}|\le \til{M}\sqrt{n}} \ge \Con_2^{-1}\!\cdot\! \til{V}_{a,b}(\vec{x})\!\cdot\!\ind_{|x_{p}-x_{n-q}|\le \til{M}\sqrt{n}}.
		\end{align}
	\end{lemma}
	\begin{proof} We have
		\begin{align*}
			V_{a,b}(\vec{x}):=\frac{\prod_{j=0}^{n-2} \fa(x_{j+1}-x_{j})}{\fa^{\ast (n-1)}(b-a)}, \quad \til{V}_{a,b}(\vec{x}):=\frac{\prod_{j=0}^{n-2} \fa(x_{j+1}-x_{j})}{ \fa^{\ast(n-p-q)}(x_{p}-x_{n-q})}.
		\end{align*}
		where $x_{0}:=a$ and $x_{n-1}:=b$. We thus have
		\begin{align}\label{lratio}
			\frac{V_{a,b}(\vec{x})}{\til{V}_{a,b}(\vec{x})}=\frac{ \fa^{\ast(n-p-q)}(x_{p}-x_{n-q})}{\fa^{\ast (n-1)}(b-a)}.
		\end{align}
		By \cite[Theorem 2, Chapter XV.5]{feller} we know
		\begin{align}
			\label{feller}
			\sup_{z\in \R}\bigg|{\sqrt{k}}\fa^{\ast k}(z)-\tfrac1{\sqrt{2\pi \sigma}}e^{-\frac{z^2}{2k\sigma^2}}\bigg|\xrightarrow{k\to\infty} 0,
		\end{align}
		where $\sigma=\int_{\R} x^2\fa(x)dx$. Thus, there exist a constant $\til\Con_1>1$ depending on $M,\delta_1,\delta_2$ such that
		\begin{align*}
			\sqrt{n}\fa^{\ast (n-p-q)}(z) \le \til\Con_1 \mbox{ for all }z\in \R, \quad \tfrac1{\til\Con_1} \le \sqrt{n}\fa^{\ast (n-1)}(b-a) \le {\til\Con_1}.
		\end{align*}
		for all large enough $n$. Inserting these bounds in the numerator and denominator of r.h.s.~\eqref{lratio} we get the \eqref{eb7} by setting $\Con_1:=\til{\Con}_1^2$. When $|x_p-x_{n-q}|\le \til{M}\sqrt{n}$, we may utilize the limit result in \eqref{feller} to obtain a new constant $\til\Con_2>0$ depending on $\til{M},\delta_1,\delta_2$ such that
		\begin{align*}
			\sqrt{n}\fa^{\ast (n-p-q)}(x_p-x_{n-q}) \ge \tfrac1{\til\Con_2} \mbox{ whenever } |x_p-x_{n-q}|\le \til{M}\sqrt{n},
		\end{align*}
	for all large enough $n$. Using this bound and the upper bound for $\sqrt{n}\fa^{\ast (n-1)}(b-a)$ we get the desired result. We arrive at \eqref{eb8} by setting $\Con_2:=\til{\Con}_2\cdot\til{\Con}_1$.
	\end{proof}

With all the preparatory results in hand, we are now ready to prove Propositions \ref{lem:ep} and \ref{l:rpass}. Recall that in the introduction we gave a proof sketch for Proposition \ref{lem:ep} (that does not appeal to Proposition \ref{l:rpass}). In what follows, we shall use the techniques outlined in that sketch to establish more sophisticated intermediate results (such as Lemma \ref{lem:tre}). These results will allow us to prove Proposition \ref{l:rpass} first. Then using those intermediate results we shall then establish Proposition \ref{lem:ep}.

	\begin{proof}[Proof of Proposition \ref{l:rpass} for the $p=2$ case (supercritical phase)] We split the proof into several steps.
		
		\medskip
		
		\noindent\textbf{Step 1.} In this step, we reduce the proof of Proposition \ref{l:rpass} to the claim around \eqref{toshow1}. Fix $r>0$. Set $T:=\lfloor rN^{2/3}\rfloor$ and $n=2T$. Recall $y_i$'s and the event $\m{RP}_{2,M}$ from the statement of the lemma. 
		Since $\m{RP}_{2,M}$ is a monotone event, by Proposition \ref{p:gmc} we have
		\begin{align*}
			\Pr_{\alpha_2}^{\vec{y},(-\infty)^{2T};2,2T}(\m{RP}_{2,M})\ge \Pr_{\alpha_2}^{\vec{x},(-\infty)^{2T};2,2T}(\m{RP}_{2,M}).
		\end{align*}
		where $x_1=-2MN^{1/3}$, $x_2=-2MN^{1/3}-\sqrt{n}$.  By translation invariance (Lemma \ref{obs1} \ref{traninv}), we may lift the Gibbs measure by $2MN^{1/3}$ units so that the boundary conditions changes from $(x_1,x_2)$ to $(0,-\sqrt{n})$. The $\m{RP}_{2,M}$ event now requires the second curve to be above the (lifted) barrier $4MN^{1/3}$ under this new boundary condition. Since $4MN^{1/3} \le 8Mr^{-1/2}\sqrt{n}$. it thus suffices to show that there exists $\phi=\phi(r,M)>0$ such that
		\begin{align*}
			 \Pr_{\alpha_2}^{(0,-\sqrt{n}),(-\infty)^{n};2,n}\bigg(\inf_{i\in \ll1,n\rr} L_2(i) \ge 8Mr^{-1/2}\sqrt{n}\bigg) \ge \phi.
		\end{align*}
		for all large enough $n$. Towards this end we claim that there exists $\phi=\phi(r,M)>0$ such that
		\begin{align}\label{toshow1}
			\liminf_{n\to \infty}\Pr_{\alpha_2}^{(0,-\sqrt{n}),(-\infty)^{2n};2,n}(\m{D}_m) \ge 2\phi,
		\end{align}
		where
		\begin{align*}
			\m{D}_m:=\bigg\{ (L_1(2i-1),L_2(2i)) \in ( 10m\sqrt{n},11m\sqrt{n})^2 \mbox{ for all }i\in \ll1,n/2\rr\bigg\},
		\end{align*}
		and $m:=Mr^{-1/2}$. Let us complete the proof assuming \eqref{toshow1}. Note that \eqref{toshow1} controls the even points of the second curve. By Lemma \ref{l:LSCrit}, we know conditioned on the even points,  $L_2(2k+1)\sim \qo_{\theta,\theta;-1}^{(L_2(2k),L_2(2k+2))}$ for $k=1,2,\ldots,2n-1$. In view of Lemma \ref{qlemma}, on the event $\m{D}_m$ we have $$\Ex\left[\ind_{L_2(2k+1) \le 8m\sqrt{n}} \mid \sigma\big(L_2(2k),L_2(2k+2)\big)\right] \le \Con e^{-\frac1{\Con} m\sqrt{n}}.$$  By Lemma \ref{l:LSCrit}, $L_2(1) \sim G_{\alpha_2+\theta,1}+L_2(2)$. Since the density  $G_{\alpha_2+\theta,1}$ have exponential tails, we see that on the event $\m{D}_m$ we have $\Pr(L_2(1) \le 8m\sqrt{n} \mid L_2(2)) \le \Con e^{-\frac1{\Con} m\sqrt{n}}.$
		Thus by a union bound,
		\begin{align*}
			\Pr_{\alpha_2}^{(0,-\sqrt{n}),(-\infty)^{2n};2,n}\bigg(\inf_{i\in \ll1,n\rr} L_2(i) \ge 8m\sqrt{n}\bigg) \ge \Pr_{\alpha_2}^{(0,-\sqrt{n}),(-\infty)^{2n};2,n}(\m{D}_m) - \Con \cdot n e^{-\frac1{\Con} m\sqrt{n}} \ge \phi,
		\end{align*}
		for large enough $n$. This establishes Proposition \ref{l:rpass} for $p=2$, modulo \eqref{toshow1}.
		
		\medskip
		
			\noindent\textbf{Step 2.} In this and subsequent steps we prove \eqref{toshow1}.  Recall the $\m{PRW}$ and $\m{WPRW}$ laws from Definition \ref{prb}. Recall from Lemma \ref{l:LSCrit} that $(L_1(2i-1),L_2(2i))_{i\in \ll1,n\rr} \sim \wprw{n}{(0,-\sqrt{n})}$.  We the terminology from Definition \ref{prb} to write
		\begin{align}\label{toshow4}
			\wprw{n}{(0,-\sqrt{n})}(\m{D}_m) = \frac{\erw{n}{(0,-\sqrt{n})}[\wsc\ind_{\m{D}_m}]}{\erw{n}{(0,-\sqrt{n})}[\wsc]}
		\end{align}
		where $\wsc$ is defined in \eqref{defw} and $\m{D}_m$ is now defined  as
		\begin{align*}
			\m{D}_m:= \big\{ (\se{1}{i},\se{2}{i}) \in (10m\sqrt{n},11m\sqrt{n})^2 \mbox{ for all } i\in \ll 1,n/2\rr \big\}.
		\end{align*}
		For the remainder of the proof we write $\Pr$ and $\Ex$ for $\prw{n}{(0,-\sqrt{n})}$ and $\erw{n}{(0,-\sqrt{n})}$ respectively. We claim that there exists constants $\Con>0$ and $\til{\Con}=\til{\Con}(m)>0$ such that
		\begin{align}\label{twobdsa}
			 \Ex[\wsc\ind_{\m{D}_m}] \le \tfrac{\til\Con^{-1}}{\sqrt{n}}, \qquad \Ex[\wsc] \ge \tfrac{\Con}{\sqrt{n}}.
		\end{align}
		Clearly plugging these bounds back in \eqref{toshow4} verifies \eqref{toshow1}. Let us thus focus on proving \eqref{twobdsa}.	For the upper bound we use the following lemma.
		\begin{lemma}\label{lem:tre} There exists a constant $\Con>0$ such that for all Borel sets $A\subset \R^2$ we have
			\begin{align*}
				\Ex\left [\wsc\ind_{\m{A}}\right] \le \tfrac{\Con}n+\tfrac{\Con}{\sqrt{n}}\Ex\bigg[\ind_{\m{A}}\Big((\ise-\iise+1)\vee 1\Big)\Big(\tfrac{|\ise|}{\sqrt{n}}\vee 2\Big)^{3/2}\bigg],
			\end{align*}
			where $\m{A}:=\{(\ise,\iise)\in A\}$.
		\end{lemma}
		Note that taking $A=\R^2$ in Lemma \ref{lem:tre} and utilizing the exponential tail estimates of $|\se{1}{1}-\se{2}{1}|$ and $|\se{1}{1}|/\sqrt{n}$ from Lemma \ref{as:el}, the upper bound in \eqref{twobdsa} follows.
		\begin{proof}[Proof of Lemma \ref{lem:tre}] As in \eqref{defnipab}, define
			\begin{align}\label{def:nip}
				\ni_p:=\left\{ \iks-\iiks \ge -p \mbox{ for all } k\in \ll2,n-1\rr\right\}.
			\end{align}
			We set $\ni:=\ni_0$. Here $\m{NI}$ stands for non-intersection.	Observe that for any $q\in \Z_{\ge 1}$ we have
			\begin{align}
				\ind_{\ni}+\sum_{p=0}^q\ind_{\ni_{p+1}\cap \ni_p^c} +\ind_{\ni_q^c}=1.
			\end{align}
			Thus, taking $q=\lfloor \log\log n\rfloor$ we have
			\begin{align}\nonumber
				\Ex[\wsc\ind_{\m{A}}] & = \Ex[\wsc\ind_{\m{A}\cap\ni}]+\sum_{p=0}^{\lfloor\log\log n\rfloor-1}\Ex[\wsc \ind_{\m{A}\cap\ni_{p+1}\cap \ni_{p}^c}]+\Ex[\wsc\ind_{\m{A}\cap \ni_{\lfloor\log\log n\rfloor}^c}] \\ & \le \sum_{p=0}^{\lfloor\log\log n\rfloor}e^{-e^{p-1}} \Pr\left(\m{A}\cap\ni_{p}\right) + \tfrac1n, \label{retyd}
			\end{align}
			where the above inequality follows by noting that on $\ni_{p}^c$ we have $W \le e^{-e^{p-1}}$. For the probability term above we condition on $\mathcal{F}:=\sigma(\ise, \iise)$ and write $\Pr(\m{A}\cap\ni_p)=\Ex\left[\ind_{\m{A}}\Ex[\ind_{\ni_p}\mid \mathcal{F}]\right]$. Upon conditioning on $\se{1}{1}, \se{2}{1}$, the paired random walk $\Pr_{\operatorname{PRW}}^{n;(0,-\sqrt{n})}$ law is equal to the law of two independent $n$-step random walks starting from $(\se{1}{1},\se{2}{1})$ to $(0,-\sqrt{n})$. Thanks to this observation, we may now appeal to Lemma \ref{l:nipp} to conclude that $\Ex[\ind_{\ni_p}\mid \mathcal{F}] \le e^{\Con p}\Ex[\ind_{\ni}\mid \mathcal{F}]$ holds for some deterministic constant $\Con>0$. This allows us to conclude
			$$\Pr(\m{A}\cap\ni_p) \le e^{\Con p}\cdot \Pr(\m{A}\cap\ni).$$
			Plugging this  into \eqref{retyd} and observing that the series $\sum_{p\ge 0} e^{-e^{p-1}}\cdot e^{\Con p}$ is summable shows
			\begin{align}\label{ttg}
				\Ex[\wsc\ind_{A}] \le \tfrac1n+\Con \cdot \Pr(\m{A}\cap \ni).
			\end{align}
			Thus to suffices to bound $\Pr(\m{A}\cap \ni)$. Towards this end, we first define the event
			\begin{align*}
				\mathsf{B}:= \left\{|\se{i}{1}|\le (\log n)^{3/2}\sqrt{n} \mbox{ for }i=1,2, \mbox{ and } |\ise-\iise|\le (\log n)^{3/2} \right\}.
			\end{align*}
			By the union bound we have $\Pr(\m{A}\cap\ni) \le \Pr(\m{A}\cap\mathsf{B}\cap\ni)+\Pr(\mathsf{B}^c).$
			For the second term note that by tail estimates from Lemma \ref{as:el} we have
			\begin{align}\label{retyh}
				\Pr(\mathsf{B}^c) \le \sum_{i=1}^2\Pr(|\se{i}{1}|\ge (\log n)^{3/2})+\Pr\left(|\ise-\iise|\ge (\log n)^{3/2} \right) \le \tfrac{\Con}{n}.
			\end{align}
			For the first term we write $\Pr(\m{A}\cap\mathsf{B}\cap\ni)=\Ex\left[\ind_{\m{A}\cap\mathsf{B}}\Ex[\ind_{\ni} \mid \mathcal{F}]\right]$. Again, since upon conditioning on $\se{1}{1}, \se{2}{1}$, the paired random walk law is equal to the law of of two independent $n$-step random walks starting from $(\se{1}{1},\se{2}{1})$ to $(0,-\sqrt{n})$, we may use random bridge estimates from Appendix \ref{app2}. In particular, Lemma \ref{genni} (recall Corollary \ref{rm:con}) shows there exists a constant $\Con>0$ such that
			\begin{align*}
				\ind_{\m{B}}\Ex[\ind_{\ni} \mid \mathcal{F}] \le \ind_{\m{B}} \cdot \tfrac{\Con}{\sqrt{n}}\Big((\ise-\iise+1)\vee 1\Big)\Big(\tfrac{|\ise|}{\sqrt{n}}\vee 2\Big)^{3/2}.
			\end{align*}
			for all $n$. Taking expectation  with respect to  $\mathcal{F}$ on both sides of the above equation and then using the fact that $\ind_{\m{A}\cap\m{B}}\le \ind_{\m{A}}$ leads to
			\begin{align}
				\Pr(\m{A}\cap\m{B}\cap\ni)\le \tfrac{\Con}{\sqrt{n}}\cdot \Ex\left[\ind_{\m{A}}  \Big((\ise-\iise+1)\vee 1\Big)\Big(\tfrac{|\ise|}{\sqrt{n}}\vee 2\Big)^{3/2}\right].
			\end{align}
			Inserting this bound along with the bound in \eqref{retyh} back in \eqref{ttg}, we arrive at the desired bound stated in the statement of the lemma.
		\end{proof}

	\medskip
	
	\noindent\textbf{Step 3.} In this step we prove the lower bound  in \eqref{twobdsa}. Towards this end, consider the event
	\begin{align*}
		\m{E}_m:= \big\{ 1 \le \ise-\iise \le 2, \mbox{ and } \ise, \iise \in (\tfrac{41}4m\sqrt{n},\tfrac{43}4m\sqrt{n})\big\},
	\end{align*}
	and $\sigma$-algebra $\mathcal{F}:=\sigma(\ise,\iise)$.  Fix any $\beta>0$ and recall $\m{Gap}_{\beta}$ from \eqref{def:gp}. We have
	\begin{align}\label{r2r}
		\Ex[\wsc\ind_{\m{D}_m}] \ge \Ex[\wsc\ind_{\m{E}_m}\ind_{\m{D}_m}\ind_{\m{Gap}_{\beta}}] \ge a_{\beta}\Ex[\ind_{\m{E}_m}\Ex\big[\ind_{\m{D}_m\cap\m{Gap}_{\beta}}\mid {\mathcal{F}}]\big]
	\end{align}
	where the second inequality above follows by noting that $W \ge a_{\beta}$ on $\m{Gap}_{\beta}\cap \m{E}$ (Lemma \ref{wgap}). As mentioned in the proof of Lemma \ref{lem:tre}, upon conditioning on $\mathcal{F}$, the paired random walk $\Pr_{\operatorname{PRW}}^{n;(0,-\sqrt{n})}$ law is equal to the law of two independent $n$-step random walks starting from $(\se{1}{1},\se{2}{1})$ to $(0,-\sqrt{n})$. For simplicity set $b_1=0$, $b_2=-\sqrt{n}$. We shall use the comparison between random bridge to $(n;n\rho,0)$-modified random bridge from Lemma \ref{lem:compare}, for a special $\rho\in (0,1)$ coming from Corollary \ref{tailni2} (recall Corollary \ref{rm:con}). Using the lower bound in \eqref{eb8} we get
	\begin{align}\label{rrr}
		\ind_{\m{E}_m}\cdot\Ex[\ind_{\m{D}_m\cap\m{Gap}_{\beta}} \mid \mathcal{F}] & \ge \ind_{\m{E}_m}\Con^{-1}\cdot\til{\Pr}_\rho^{(\ise,\iise)}(\m{D}_m\cap\m{Gap}_{\beta}) \\ & =\ind_{\m{E}_m}\Con^{-1}\cdot\til{\Pr}_\rho^{(\ise,\iise)}(\m{D}_m\cap\m{Gap}_{\beta}\mid \ni)\til{\Pr}_\rho^{(\ise,\iise)}(\ni) \nonumber
	\end{align}
	for some $\Con>0$ depending on $m$ and $\rho$. Here $\til{\Pr}_\rho^{(\ise,\iise)}:=\tpr{(n;n\rho,0)}{(\ise,\iise)}{,(0,\sqrt{n})}$ is the joint law of two independent $(n;n\rho,0)$-modified random bridge from $\se{i}{1}$ to $b_i$ defined in Definition \ref{def:mrb}.
	We now claim that there exists $\til{\phi}=\til{\phi}(m)>0$ such that
	\begin{align}\label{4prov}
		\til{\Pr}_\rho^{(a_1,a_2)}(\m{D}_m\mid \ni) \ge 2\til{\phi}.
	\end{align}
	uniformly over all $(a_1,a_2)\in \mathcal{P}_1$ where we define
	$$\mathcal{P}_1:=\big\{(z_1,z_2)\in (\tfrac{41}{4}m\sqrt{n},\tfrac{43}{4}\sqrt{n}) : 1\le z_1-z_2\le 2\big\}.$$
We postpone the proof of this claim to the next step. Let us complete the proof of the lemma assuming it. Since, under $\til{\Pr}_{\rho}^{(a_1,a_2)}$,  $S_1(\ll1,n\rho\rr), S_2(\ll1,n\rho\rr)$ are two independent random walks, we may use non-intersection type estimates for random walks from Appendix \ref{app2}.  In particular, using Lemmas \ref{l:class} and \ref{tailni} (recall Corollary \ref{rm:con}) we can get constants $\delta>0, M_2>0$ and $\Con_1>0$ all depending on $m$ and $\rho$ such that uniformly over $(a_1,a_2)\in\mathcal{P}_1$ we have
	$$\til{\Pr}_\rho^{(a_1,a_2)}\big(\iks\ge \iiks\mbox{ for all }k\in \ll2,n\rho\rr, \se{1}{n\rho}-\se{2}{n\rho}\ge \delta\sqrt{n}, |\se{i}{n\rho}|\le M_2\sqrt{n}\big) \ge \tfrac{\Con_1^{-1}}{\sqrt{n}}.$$
	Set $\m{G}:=\{\se{1}{n\rho}-\se{2}{n\rho}\ge \delta\sqrt{n},|\se{i}{n\rho}|\le M_2\sqrt{n}\}$. Recall from the definition of $(n;n\rho,0)$-modified random bridge that on $\ll n\rho,n\rr$ the modified random bridge is just a random bridge from $\se{i}{n\rho}$ to $b_i$. Applying Lemma \ref{unini} (recall Corollary \ref{rm:con}) it follows that
	$$\ind_{\m{G}}\cdot\til{\Pr}_{\rho}^{(a_1,a_2)}\big(\iks\ge \iiks\mbox{ for all }k\in \ll n\rho,n-1\rr\big) \ge \ind_{\m{G}}\cdot{\Con_2^{-1}},$$
	for some constant $\Con_2>0$ depending on $m$ and $\rho$ only. Thus we get $\til{\Pr}_\rho^{(a_1,a_2)}(\ni) \ge \frac{\Con_4^{-1}}{\sqrt{n}}$ uniformly on $\m{E}_m$ for some deterministic constant $\Con_4$ depending on $m$ and $\rho$ only. By Lemma \ref{l:gapev} (recall Corollary \ref{rm:con}), we may choose $\beta$ small enough depending on $m$ and $\rho$ such that $\til{\Pr}_\rho^{(a_1,a_2)}(\m{Gap}_{\beta}\mid \ni) \ge 1-\til{\phi}$ uniformly over $(a_1,a_2)\in \mathcal{P}_1$. Plugging this estimates back in \eqref{rrr}, we see that $$	\ind_{\m{E}_m}\cdot\Ex[\ind_{\m{D}_m\cap\m{Gap}_{\beta}} \mid \mathcal{F}] \ge 	\ind_{\m{E}_m}\cdot\til{\phi}\cdot\frac{\Con_4^{-1}}{\sqrt{n}}.$$ Now, by Lemma \ref{as:el} (equation \eqref{t3} in particular) we know that $\Pr(\m{E}_m)\ge \Con_5^{-1}>0$ for some $\Con_5$ depending on $m$. Plugging this back in \eqref{r2r} we see that
	\begin{align}\label{dest}
		\Ex[\wsc\ind_{\m{D}_{m}}]  \ge a_{\beta} \cdot \Pr(\m{E}_m) \cdot \til{\phi} \cdot \tfrac{\Con_4^{-1}}{\sqrt{n}} =: \tfrac{\til\Con^{-1}}{\sqrt{n}}.
	\end{align}
	where $\til{\Con}>0$ is a constant depending only on $m$ and $\rho$.

	\medskip
	
	\noindent\textbf{Step 4.} In this step we prove \eqref{4prov}. By equation \eqref{b21} in Lemma \ref{tailni} (recall Corollary \ref{rm:con}), we know there exists $\delta \in (0,\frac18(m\wedge1))$ small enough depending only on $\rho$ such that
	\begin{align}
		\label{l21}
		\til{\Pr}_{\rho}^{(a_1,a_2)}\big(\se{1}{n\rho}-\se{2}{n\rho} \ge \delta \sqrt{n}\mid \ni\big) \ge \tfrac{15}{16}
	\end{align}
	uniformly over $(a_1,a_2)\in \mathcal{P}_1$. We shall now choose $\rho$ as $\rho(\frac1{16},\frac{m\wedge 1}{8})$ where the latter is a constant depending on $m$ and comes from Corollary \ref{tailni2} (recall Corollary \ref{rm:con}). In view of this choice of $\rho$, applying Corollary \ref{tailni2} (recall Corollary \ref{rm:con}), we see that uniformly over $(a_1,a_2)\in \mathcal{P}_1$ we have
	\begin{align}\label{l22}
		\til{\Pr}_\rho^{(a_1,a_2)}\left(\sup_{k\in \ll1,n\rho\rr,i=1,2} |\kis-\se{i}{1}|\le \tfrac{m\wedge 1}{8}\sqrt{n} \mid \ni \right) \ge \tfrac{15}{16}.
	\end{align}
	Since on $\mathcal{P}_1$ we also have $(a_1,a_2) \in (\tfrac{41}4m\sqrt{n},\tfrac{43}4m\sqrt{n})$, combining \eqref{l21} and \eqref{l22} we get
	\begin{align*}
		\til{\Pr}_\rho^{(a_1,a_2)}\bigg(\big\{\se{1}{n\rho}-\se{2}{n\rho} \ge \delta \sqrt{n}\big\}\cap\m{K}_1 \mid \ni\bigg) \ge \tfrac{7}{8},
	\end{align*}
	where
	\begin{align*}
		\m{K}_1:= \left\{ \iks,\iiks \in (\tfrac{81}8m\sqrt{n},\tfrac{87}8m\sqrt{n}) \mbox{ for all } k \in \ll1,n\rho\rr, \mbox{ and } |\se{1}{n\rho}-\se{2}{n \rho}|\le \tfrac{\sqrt{n}}2\right\}.
	\end{align*}
	Following the definition of $(n;n\rho,0)$-modified random bridge, to prove \eqref{4prov} it suffices to show
	\begin{align}\label{lpts}
		\pr{n-n\rho+1}{(c_1,c_2)}{,(b_1,b_2)}\bigg(\big\{\se{1}{k},\se{2}{k} \in (10m\sqrt{n},11m\sqrt{n}) \mbox{ for all } k \in \ll1,n/2\rr\big\} \cap \ni\bigg) \ge \tfrac{16}{7}\til{\phi},
	\end{align}
	uniformly over $(c_1,c_2)\in \mathcal{P}_2$ where
	$$\mathcal{P}_2:=\{(c_1,c_2)\in \R^2: c_i\in (\tfrac{81}8m\sqrt{n},\tfrac{87}8m\sqrt{n}), \mbox{ and }\tfrac12\sqrt{n} \ge c_1-c_2 \ge \delta\sqrt{n}\}.$$
	and $\pr{n-n\rho+1}{(c_1,c_2)}{,(b_1,b_2)}$ is the law of two independent random bridges of length $n-n\rho+1$ starting at $(c_1,c_2)$ and ending at $(b_1,b_2)$. For simplicity set $u:=n-n\rho+1 \ge \frac34n$. By the KMT coupling of random bridges \cite{xd} we may assume there are two independent Brownian bridges $B_1,B_2$ (with variance $\int x^2\fa(x)dx$) on the same probability space such that
	\begin{align}
		\label{kmtbr2}
		\pr{u}{(c_1,c_2)}{,(b_1,b_2)}\left(\sup_{k\in \ll1,u\rr,i=1,2}| \se{i}{k}-\sqrt{u}B_{i}(k/u)-c_i-\tfrac{k}{u}(b_i-c_i)| \ge \Con \log n\right) \le \tfrac1n.
	\end{align}
	Let $r_{n,i}(x)$ be the piece-wise linear function interpolated by three points: $r_{n,i}(0)=r_{n,i}(1)=0$ and $r_{n,i}(3/4)=\frac{3}{4\sqrt{u}}(b_i-c_i)$. Let $\mathcal{U}_i$ be the $L^{\infty}$ open ball of $r_{n,i}(x)$ of radius $\frac1{4}\delta$ (this is the same $\delta$ that was chosen at the beginning of Step 4). By properties of Brownian bridge, there exists a $\til{\phi}=\til{\phi}(m)>0$ such that for all $(c_1,c_2) \in \mathcal{P}_2$, we have
	\begin{align*}
		\pr{u}{(c_1,c_2)}{,(b_1,b_2)}(B^{(i)}\in \mathcal{U}_i \mbox{ for }i=1,2) \ge \tfrac{32}{7}\til{\phi}.
	\end{align*}
	Note that the above equation along with \eqref{kmtbr2} implies that with probability $\tfrac{32}{7}\til{\phi}-\frac1n$,  for all $n$ large enough (and hence $u$ large enough) under the law $\pr{u}{(c_1,c_2)}{,(b_1,b_2)}$ we have the following items simultaneously:
	\begin{itemize}[leftmargin=18pt]
		\item 	For all $k\in \ll 1,3/4u\rr$
		$$|\se{i}{k}-c_i|\le \Con \log n+ \tfrac14\sqrt{u}\delta \le \tfrac{m}8\sqrt{u} <\tfrac{m}{8}\sqrt{n}.$$
		\item For all $k\in \ll1,u\rr$ we have
		\begin{align*}
			\se{1}{k} & \ge \sqrt{u}r_{n,1}(\tfrac{k}u)+c_1+\tfrac{k}u(b_1-c_1)-\tfrac14\sqrt{u}\delta-\Con \log n \\ & \ge \se{2}{k}+ \sqrt{u}(r_{n,1}(\tfrac{k}u)-r_{n,2}(\tfrac{k}u))-\tfrac12\sqrt{u}\delta +c_1-c_2+\tfrac{k}u(b_1-b_2-c_1+c_2)-2\Con \log n
			\\ & \ge \sqrt{u}(r_{n,1}(k/u)-r_{n,2}(k/u))+\tfrac12\sqrt{u}\delta -2\Con \log n+\se{2}{k}.
		\end{align*}
		We have $r_{n,1}(x)\ge r_{n,2}(x)$ by construction, and $c_1-c_2+\tfrac{k}u(b_1-b_2-c_1+c_2)\ge \sqrt{u}\delta$ for all $(c_1,c_2)\in \mathcal{P}_2$. Thus for all large enough $n$, $\se{1}{k}>\se{2}{k}$ for all $k\in \ll0,u\rr$.
	\end{itemize}
	For $n$ large enough $\tfrac{32}{7}\til{\phi}-\frac1n \ge \tfrac{16}{7}\til{\phi}$. This establishes \eqref{lpts} and hence also Proposition \ref{l:rpass}.	
	\end{proof}

	\begin{corollary}\label{corb} There exists an absolute constant $\Con>0$ such that for all $n\ge 1$.
		\begin{align*}
			\erw{n}{(0,-\sqrt{n})}[\wsc] \ge \tfrac{\Con^{-1}}{\sqrt{n}}.
		\end{align*}
	\end{corollary}
	\noindent The above corollary follows from \eqref{dest} as $\erw{n}{(0,-\sqrt{n})}[\wsc]\ge \erw{n}{(0,-\sqrt{n})}[\wsc\ind_{\m{D}_1}]$. We remark that here it is important that the endpoints are $O(\sqrt{n})$ apart to get the precise order of $\Ex[\wsc]$. We expect a different order if the endpoints are closer or lie in a reversed order. Later, in Lemma \ref{crude}, we shall prove a different lower bound for $\erw{n}{(x,y)}[\wsc]$ that is uniform over all possible endpoints $(x,y)$ in a specific window.
	\begin{proof}[Proof of Proposition \ref{lem:ep} in the $p=2$ case (supercritical phase)] Given the machinery developed in the above proof, proof of Proposition \ref{lem:ep} follows easily. By Lemma \ref{l:LSCrit} we have
		\begin{align}
			\nonumber	\Pr_{\alpha_2}^{(0,-\sqrt{T}),(-\infty)^{T};2,T}(|L_i(i)|\ge M\sqrt{T}) & =\wprw{T}{(0,-\sqrt{T})}(|\se{i}{1}| \ge M\sqrt{T}) \\ & = \frac{\erw{T}{(0,-\sqrt{T})}[\wsc\ind_{{|\se{i}{1}| \ge M\sqrt{T}}}]}{\erw{T}{(0,-\sqrt{T})}[\wsc]}. \label{try}
		\end{align}
		Now by Corollary \ref{corb} we have $\erw{T}{(0,-\sqrt{T})}[\wsc] \ge \frac{\Con}{\sqrt{T}}$ and by Lemma \ref{lem:tre} we have
		\begin{align*}
			& \erw{T}{(0,-\sqrt{T})}[\wsc\ind_{|\se{i}{1}| \ge M\sqrt{T}}] \\ & \le \tfrac{1}T+\tfrac{\Con}{\sqrt{T}}\erw{T}{(0,-\sqrt{T})}\bigg[\ind_{|\se{i}{1}| \ge M\sqrt{T}}\Big((S_1(1)-S_2(1)+1)\vee 1\Big)\Big(\tfrac{|\se{i}{1}|}{\sqrt{T}}\vee 2\Big)^3\bigg] \\  & \le \tfrac{1}T+\tfrac{\Con}{\sqrt{T}}\sqrt{\erw{T}{(0,-\sqrt{T})}\left[\Big((\se{1}{1}-\se{2}{1}+1)\vee 1\Big)^2\right]}\sqrt{\erw{T}{(0,-\sqrt{T})}\left[\ind_{|\se{i}{1}| \ge M\sqrt{T}}\Big(\tfrac{|\se{1}{1}|}{\sqrt{T}}\vee 2\Big)^6\right]},
		\end{align*}
		where the last inequality follows from Cauchy-Schwarz.	Taking $T$ and $M$ large enough, in view of the tail estimates from Lemma \ref{as:el}, it follows that \eqref{try} can be made arbitrarily small. By a union bound, we can thus make $\Pr_{\alpha_2}^{(0,-\sqrt{T}),(-\infty)^T;2,T}(|L_1(1)|+|L_2(2)|\ge M\sqrt{T})$ arbitrarily small by taking $T$ and $M$ large enough.
	\end{proof}

	\section{Modulus of continuity: proof of Theorem \ref{t:main0}}
	\label{sec:mc}
	
	In this section we prove our main theorem, Theorem \ref{t:main0}, about spatial tightness of $\hslg$ polymers. Due to the relation in \eqref{impi}, Theorem \ref{t:main0} essentially follows by controlling modulus of continuity of the first curve of log-gamma line ensemble. Towards this end, we recall the definition of modulus of continuity function. Given a continuous function $f: \Z_{\ge 1}\to \R$ and $U>1$, we define the modulus of continuity function as
	\begin{align}\label{eq:modcont}
		\omega_{\delta}^N(f;\ll1,U\rr) := \sup_{\substack{i_1,i_2\in \ll1,U\rr \\ |i_1-i_2|\le \delta N^{2/3}}} |f(i_1)-f(i_2)|.
	\end{align}
	We have the following result that is proved in Section \ref{sec:mcd}.
	\begin{proposition}\label{e:mcd}
		Fix $r,\gamma >0$ and $p\in \{1,2\}$. Set $\alpha=\alpha_p$ according to \eqref{acric}. We have
		\begin{align}\label{ee:mcd}
			\lim_{\delta\downarrow 0}\limsup_{N\to\infty}\Pr_{\alpha_p}\left(\omega_{\delta}^N(\L_1^N,\ll1,2\lfloor rN^{2/3}\rfloor-1\rr) \ge \gamma N^{1/3}\right)=0.
		\end{align}
	\end{proposition}

\begin{proof}[Proof of Theorem \ref{t:main0}]
By a standard criterion for functional tightness \cite[Theorem 7.3]{bill}, Proposition \ref{e:mcd} along with endpoint tightness from Theorem \ref{thm:eptight} implies tightness of the probability law of $N^{-1/3}\L_1^N(\ll 1,2\lfloor rN^{2/3}\rfloor-1\rr)$. In light of the matching in distribution in Theorem \ref{thm:conn}  (i), this immediately translates into tightness of the measure $\mathbb{P}^N_\alpha$ desired to prove Theorem \ref{t:main0}. Note that cases (1) and (2) of Theorem \ref{t:main0} follow by using the $p=1$ and $p=2$ cases of Proposition \ref{e:mcd} and Theorem \ref{thm:eptight}.
\end{proof}

The rest of this section is devoted to proving Proposition \ref{e:mcd}. This relies on the following technical result which deals with the modulus of continuity for the \btf\ measure.
	
	\begin{proposition}\label{pimc} Fix any $M, V, k_1, k_2, \gamma>0$ with $k_2>k_1$. For each $N>0$, define the sets
		\begin{align*}
I_{1,M}:=\{y\in \R, |y|\le 2MN^{1/3}\}, \mbox{  and  }		I_{2,M}:=\{(y_1,y_2)\in \R^2 : y_i\in I_{1,M/2}, y_1-y_2\ge -(\log N)^{7/6}\}.
		\end{align*}
		For each $p\in \{1,2\}$, there exist $\delta=\delta(M,V,k_1,k_2,\gamma,\e)>0$ and $N_0=N_0(M,V,k_1,k_2,\gamma,\e)>0$ such that for all $\vec{x}\in I_{p,M}$, $T\in \ll k_1N^{2/3},k_2N^{2/3}\rr$, and $N\ge N_0$ we have	
		\begin{align*}
			\sum_{i=1}^p\Pr_{\alpha_p}^{\vec{x},(-\infty)^{T};p,T}\left(\sum_{k=1}^i|L_k(k)|\le VN^{1/3},\mbox{ and } \omega_{\delta}^N(L_i, \ll1,T/4+i-2\rr) \ge \gamma 	N^{1/3}\right)\le \e.
		\end{align*}
	\end{proposition}
	We postpone the proof of Proposition \ref{pimc} to Section \ref{sec:pimc}. Section \ref{sec:prep} contains a few lemmas that are used in the proof of Proposition \ref{e:mcd} later in Section \ref{sec:mcd}.

	\subsection{Preparatory Lemmas} \label{sec:prep} We first discuss a few consequences of Proposition \ref{lem:ep} that form the preparatory tools for our modulus of continuity analysis.
	
	\begin{lemma}\label{l:x} Fix any $\e\in (0,\frac12)$ and $T\ge 2$. Let $(X(i))_{i=1}^{2T-1}$ be a random vector with $X(1)=X(2T-1)\equiv 0$ and density at $(X_i)_{i=2}^{2T-2} = (u_i)_{2=1}^{2T-2}$ proportional to
$$\prod_{i=1}^{2T-2} G_{\theta,(-1)^{i+1}}(u_i-u_{i+1})$$
respectively where $u_1=0$ and $u_{2T-1}=0$ and $G$ is defined in \eqref{def:gwt}. Similarly, define an independent random vector  $(Y(i))_{i=1}^{2T-1}$  precisely as with $X$ except that $G_{\theta,(-1)^{i+1}}$ is replaced by $G_{\theta,(-1)^{i}}$. Then, there exists $M_0(\e)>0$ such that for all $T\ge 2$ we have
		\begin{align}\label{exe}
			\Pr\bigg(\sup_{i\in\ll1,2T-1\rr}\big(|X(i)|+|Y(i)|\big)\ge M_0\sqrt{T}\bigg) \le \e.
		\end{align}
	\end{lemma}
		
	We refer to Figure \ref{fig:order2} for graphical representation of the distributions appearing in Lemma \ref{l:x}.
	
	\begin{figure}[h!]
		\centering
		\begin{tikzpicture}[line cap=round,line join=round,>=triangle 45,x=2cm,y=1cm]
			\foreach \x in {0,1,2}
			{
				\draw[line width=1pt,blue,{Latex[length=2mm]}-]  (\x,0) -- (\x-0.5,-0.5);
				\draw[line width=1pt,blue,{Latex[length=2mm]}-] (\x,0) -- (\x+0.5,-0.5);
				\draw [fill=blue] (\x,0) circle (2pt);
				\draw [fill=blue] (\x-0.5,-0.5) circle (2pt);
			}
			\draw [fill=blue] (2.5,-0.5) circle (2pt);
			\node at (-0.6,-0.5) {$0$};
			\node at (2.6,-0.5) {$0$};	
			\foreach \x in {4,5,6}
			{
				\draw[line width=1pt,blue,{Latex[length=2mm]}-]  (\x-0.5,0)--(\x,-0.5);
				\draw[line width=1pt,blue,{Latex[length=2mm]}-] (\x+0.5,0)--(\x,-0.5);
				\draw [fill=blue] (\x,-0.5) circle (2pt);
				\draw [fill=blue] (\x-0.5,0) circle (2pt);
			}
			\draw [fill=blue] (6.5,0) circle (2pt);
			\node at (3.4,0) {$0$};
			\node at (6.6,0) {$0$};		
		\end{tikzpicture}
		\caption{Graphical representation of $X$ (left) and $Y$ (right) distribution from Lemma \ref{l:x}.}
		\label{fig:order2}
	\end{figure}
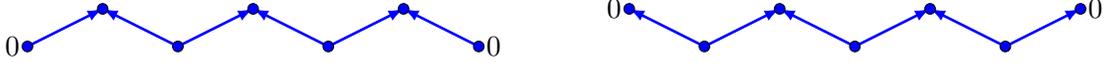
	
	\begin{proof} Fix $\e\in (0,1)$. Note that $(X(2i-1))_{i=1}^{T}$ forms a random bridge from $0$ to $0$ with increment from $G_{\theta,+1}\ast G_{\theta,-1}$. By the KMT coupling for random bridges \cite{xd} along with Brownian bridge estimates, there exists a constant $M>0$ such that (here we temporarily use $\Pr$ and $\Ex$ for the probability and expectation for the $X$ and $Y$ vectors)
		\begin{align*}
			\Pr(\m{A})\le \tfrac\e4, \mbox{ where }	\m{A}:=\left\{\sup_{i\in\ll1,T\rr} |X(2i-1)|\ge M\sqrt{T}\right\}.
		\end{align*}
		Let us write $\mathcal{F}:=\sigma\big((X(2i-1))_{i=1}^T\big)$. By a union bound we have
		\begin{align}\label{ant}
			\Pr\Big(\sup_{i\in\ll1,2T-1\rr} |X(i)|\ge 5M\sqrt{T}\Big) \le \frac\e4 +\sum_{i=1}^{T-1}\Ex\left[\ind_{\neg\m{A}}\cdot \Ex\left[\ind_{|X(2i)|\ge 5M\sqrt{T}}\mid \mathcal{F}\right]\right].
		\end{align}
		Note that the distribution of even points given the odd points are given by the $\qo$-distribution introduced in \eqref{qdist}. Observe that by Lemma \ref{qlemma},
		\begin{align*}
			\ind_{X(2i-1),X(2i+1) \in (-M\sqrt{T},M\sqrt{T})}\cdot \Ex\left[\ind_{|X(2i)|\ge 5M\sqrt{T}}\mid \mathcal{F}\right] \le \Con \exp(-\tfrac1\Con\sqrt{T}),
		\end{align*}
		for some absolute constant $\Con>0$. Plugging the above bound back in \eqref{ant} and taking $T$ large enough we get that r.h.s.~\eqref{ant} is at most $\frac\e2$. Similarly we see that  $\Pr\big(\sup_{i\in\ll1,2T-1\rr} |Y(i)|\ge 5M\sqrt{T}\big) \le \frac\e2$. Adjusting $M$, we arrive at \eqref{ety}.
	\end{proof}

	\begin{lemma}\label{l:sup} Recall for $p\in \{1,2\}$, $\alpha:=\alpha_p$ from \eqref{acric}. Fix any $r \ge 1$ and $\e >0$. Set $T=\lfloor rN^{2/3}\rfloor$. There exists $M=M(\e)>0$ and $N_0(\e)>0$ such that for all $N\ge N_0$ we have
		\begin{align}\label{eq:l5.t}
			& \Pr_{\alpha_1}^{0,(-\infty)^{T};1,T}\bigg(\sup_{i\in\ll1,2T-1\rr}|L_1(i)|\ge M\sqrt{T}\bigg)\le \e, \\  & \label{eq:l5.u}	\Pr_{\alpha_2}^{(0,-\sqrt{T}),(-\infty)^{T};2,T}\bigg(\sup_{i\in\ll1,2T-1\rr}|L_1(i)|+\sup_{j\in\ll1,2T\rr}|L_2(j)|\ge M\sqrt{T}\bigg)\le \e,
		\end{align}
		where the bottom free law $\Pr_{\alpha_p}^{\vec{x},(-\infty)^{2T};2,T}$ is defined in Definition \ref{def:btf}.
	\end{lemma}

	\begin{proof}  For clarity we divide the proof into two steps.
		
		\medskip
		
		\noindent\textbf{Step 1.} Fix any $\e\in (0,\frac12)$ and consider $M_0(\e)$ from Lemma \ref{l:x}. In this step we prove \eqref{eq:l5.t}.  From Proposition \ref{lem:ep} choose $M_1(\e)>0$ such that for all large enough $T$ we have
		\begin{align}\label{ety}
			\Pr_{\alpha_1}^{0,(-\infty)^T;1,T}\big(|L_1(1)|\ge M_1\sqrt{T}\big) \ge \e, \quad 	\Pr_{\alpha_2}^{(0,-\sqrt{T}),(-\infty)^T;2,T}\big(|L_1(1)|+|L_2(2)|\ge M_1\sqrt{T}\big) \ge \e.
		\end{align}
We will use the first bound immediately, and the second  a bit later.
		Set $M_3:=2M_0+M_1+1$, and
		\begin{align*}
		\m{A}_1:=\left\{\sup_{i\in \ll1,2T-1\rr} L_1(i)\ge (M_3+M_0)\sqrt{T}\right\},\qquad
\m{A}_2:=\left\{\sup_{i\in \ll2,2T\rr} L_2(i)\ge (M_0+M_1)\sqrt{T}\right\}.
		\end{align*}
Finally, introduce shorthand notation $\Pr_1$ for $\Pr_{\alpha_1}^{(0,-\sqrt{T}),(-\infty)^{T};1,T}$ and $\Pr_2$ for $\Pr_{\alpha_2}^{(0,-\sqrt{T}),(-\infty)^{T};2,T}$ (and likewise for $\Ex$).
		In view of \eqref{ety}, by a union bound we have
		\begin{align*}
			\Pr_1(\m{A}_1) \le \e+\Ex_1\left[\ind_{L_1(1)\le M_1\sqrt{T}}\Ex_1\left[\ind_{\m{A}_1}\mid \sigma(L_1(1))\right]\right].
		\end{align*}
		As $\m{A}_1$ is an increasing event with respect to the boundary data, due to stochastic monotonicity (Proposition \ref{p:gmc}), increasing the boundaries will only increase the conditional probability. Thus to get an upper bound, we may assume $L_1(1)=L_1(2T-1)=M_1\sqrt{T}$. But note that under this boundary condition we have $(L_1(i)-M_1\sqrt{T})_{i=1}^{2T-1}\stackrel{(d)}{=}(X(i))_{i=1}^{2T-1}$, where $X(\cdot)$ is defined in Lemma \ref{l:x}. Thus, owing to \eqref{exe}, almost surely we have
		\begin{align*}
			\ind_{L_1(1)\le M_1\sqrt{T}}\Ex_1\left[\ind_{\m{A}_1}\mid \sigma(L_1(1))\right] \le \Pr\bigg(\sup_{i\in\ll1,2T-1\rr}|X(i)|\ge (2M_0+M)\sqrt{T}\bigg) \le \e.
		\end{align*}
		This implies $\Pr_1(\m{A}_1) \le 2\e$. Following similar calculations one can show $$\Pr_1\bigg(\inf_{i\in \ll1,2T-1\rr} L_1(i)\le -(M_3+M_0)\sqrt{T}\bigg) \le 2\e.$$
		This proves \eqref{eq:l5.t} with $M=M_3$ for $\e \mapsto 2\e$.
		
		\medskip
		
		\noindent\textbf{Step 2.} In this step we prove \eqref{eq:l5.u}. At this point we encourage the readers to look at Figure \ref{glt} and its caption for an overview of the proof idea.
				
		\begin{figure}[h!]
			\centering
			\begin{overpic}[width=7cm]{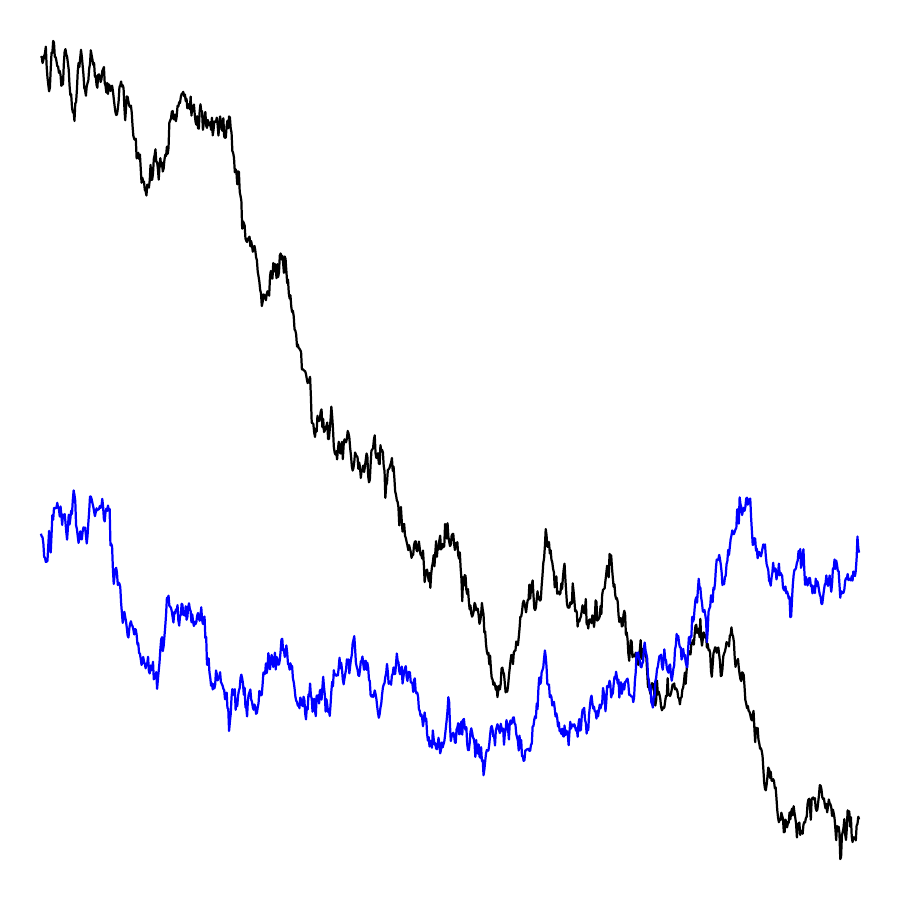}
				\put(2.7,-5){
					\begin{tikzpicture}
						\draw[dashed] (0,0) -- (0,7.5);
				\end{tikzpicture}}
				\put(93.3,-5){
					\begin{tikzpicture}
						\draw[dashed] (0,0) -- (0,7.5);
				\end{tikzpicture}}
				\put(0,50){
					\begin{tikzpicture}
						\draw[dashed] (0.4,0)--(6.8,0);
				\end{tikzpicture}}
			\end{overpic}
			\begin{overpic}[width=7cm]{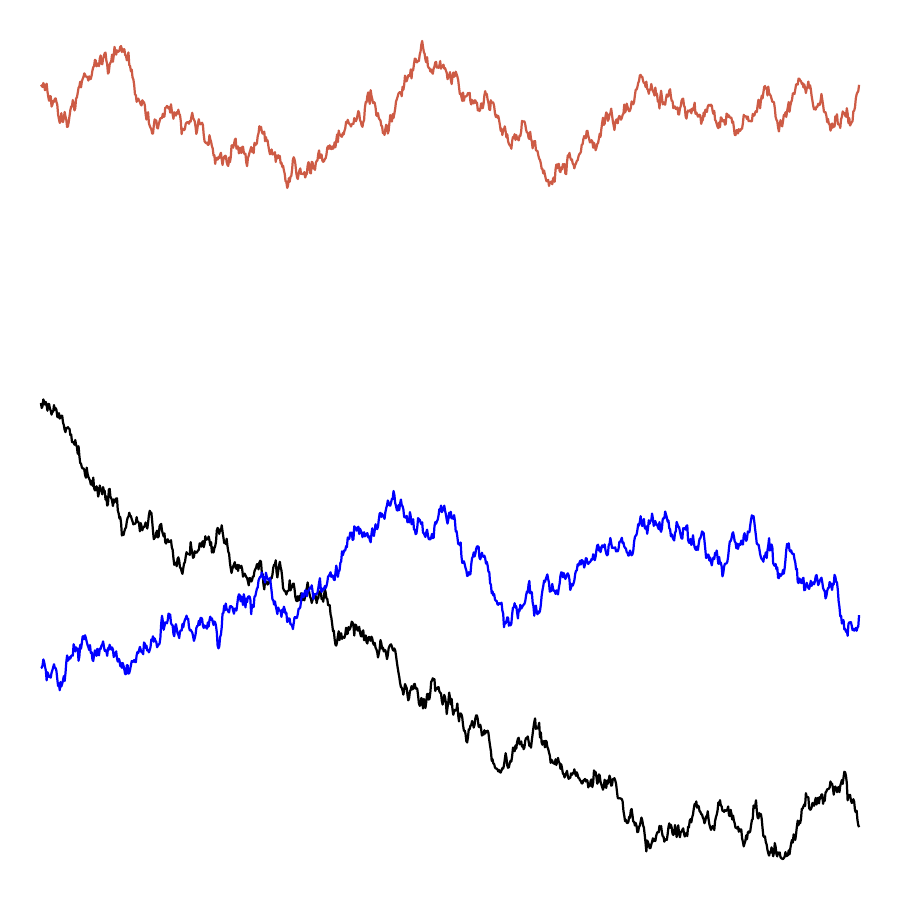}
				\put(2.7,-5){
					\begin{tikzpicture}
						\draw[dashed] (0,0) -- (0,7.5);
				\end{tikzpicture}}
				\put(93.3,-5){
					\begin{tikzpicture}
						\draw[dashed] (0,0) -- (0,7.5);
				\end{tikzpicture}}
				\put(0,50){
					\begin{tikzpicture}
						\draw[dashed] (0.4,0)--(6.8,0);
				\end{tikzpicture}}
				\put(4,54){\begin{tikzpicture}
						\draw[fill=black] (0,0) circle (2pt);
				\end{tikzpicture}}
				\put(94,10){\begin{tikzpicture}
						\draw[fill=black] (0,0) circle (2pt);
				\end{tikzpicture}}
				\put(4,90){\begin{tikzpicture}
						\draw[fill=red] (0,0) circle (2pt);
				\end{tikzpicture}}
				\put(94,90){\begin{tikzpicture}
						\draw[fill=red] (0,0) circle (2pt);
				\end{tikzpicture}}
			\end{overpic}
			\vspace{0.5cm}
			
			\caption{In the above figure, we have plotted $L_1\ll1,2T-1\rr$ (black curve) and $L_2\ll2,2T\rr$ (blue curve). Endpoint tightness, Proposition \ref{lem:ep}, ensures that $L_1(1), L_2(2) \in (-M_0\sqrt{T},M_0\sqrt{T})$. Assuming this, in order to seek an uniform upper bound for the blue curve, by stochastic monotonicity we may push the black curve all the way to $+\infty$. The resulting law for the blue curve is given by $Y(\cdot)$ (upto a translation) introduced in  Lemma \ref{l:x}. A uniform upper bound for the resulting law for the blue curve law can then be estimated by Lemma \ref{l:x}. The upper bound is shown in the dashed line above. Once we have a uniform upper bound for the blue curve, we may elevate the endpoints of black curve much higher (from black points to red points in the above right figure) so that the curve no longer feels the effect of the blue curve. The red curve above denotes a sample for $L_1$ from this elevated end points. Without interaction with the blue curve, its law (upto a translation) equals to $X(\cdot)$ in Lemma \ref{l:x}. A uniform upper bound for the red curve can then be estimated by Lemma \ref{l:x}.}
			\label{glt}
		\end{figure}
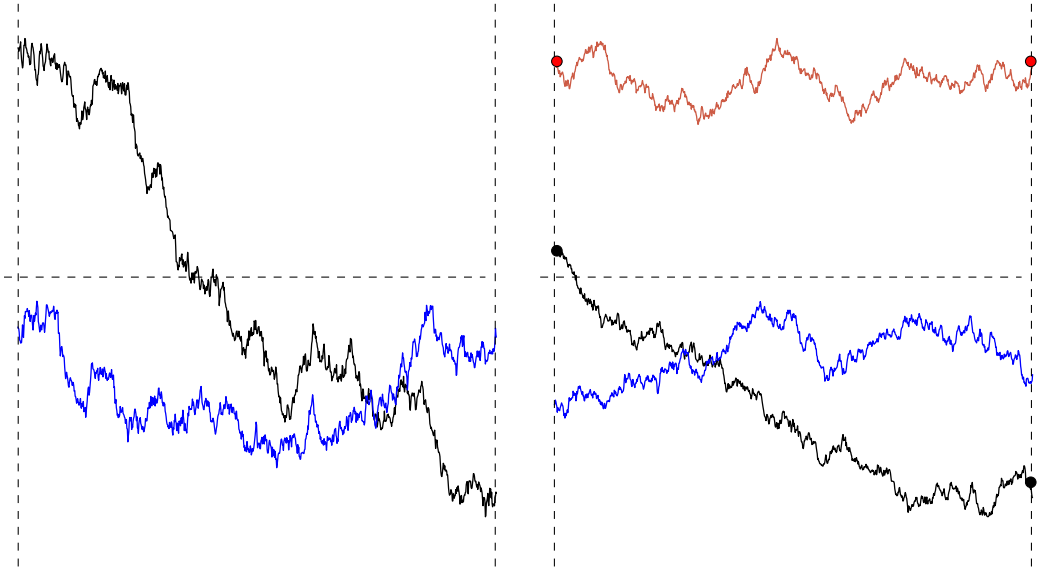

		Let us set $\mathcal{F}_1= \sigma(L_2(2),(L_1(i))_{i=1}^{2T-1})$ and $\mathcal{F}_2=\sigma(L_1(1),(L_2(i))_{i=2}^{2T})$. In view of \eqref{ety}, by a union bound
		\begin{align*}
			\Pr_2(\m{A}_2) & \le \e+\Pr_2\bigg(\{L_2(2) \le M_0\sqrt{T}\} \cap \m{A}_2\bigg)\le\e +\Ex_2\left[\ind_{L_2(2)\le M_0\sqrt{T}}\Ex_2\left[\ind_{\m{A}_2} \mid  \mathcal{F}_1\right]\right].
		\end{align*}
		As $\m{A}_2$ is an increasing event with respect to the boundary data, due to stochastic monotonicity (Proposition \ref{p:gmc}), increasing the boundaries will only increase the conditional probability of $\m{A}_2$. Thus, to get an upper bound, we may assume $L_2(2T)=L_2(2)=M_1\sqrt{T}$, and $L_1(i)=+\infty$ for all $i\in \ll1,2T-1\rr$. Under this boundary condition we have $(L_2(i+1)-M_1\sqrt{T})_{i=1}^{2T-1}\stackrel{(d)}{=}(Y(i))_{i=1}^{2T-1}$ where $Y(\cdot)$ is defined in Lemma \ref{l:x}. Thus,  almost surely we have (recall $\Pr$ is the law of $Y$ below)
		\begin{align*}
			\ind_{L_2(2)\le M_0\sqrt{T}}\cdot\Ex_2\left[\ind_{\m{A}_2} \mid  \mathcal{F}_1\right] \le \Pr\bigg(\sup_{i\in\ll1,2T-1\rr}|X(i)|\ge M_0\sqrt{T}\bigg) \le \e.
		\end{align*}
		Thus $\Pr_2(\m{A}_2)\le 2\e$. In view of this bound, applying a union bound we have
		\begin{align*}
			\Pr_2(\m{A}_1) \le 3\e +\Ex_2\left[\ind_{\{L_1(1)\le M_1\sqrt{T}\}\cap \neg\m{A}_2}\Ex_2\left[\ind_{\m{A}_1}\mid\mathcal{F}_2\right]\right].
		\end{align*}
		As $\m{A}_1$ is an increasing event with respect to the boundary data, due to stochastic monotonicity (Proposition \ref{p:gmc}), increasing the boundaries will only increase the conditional probability. Thus, to get an upper bound we may assume $L_1(1)=L_1(2T-1)=M_3\sqrt{T}$ and $L_2(i)=(M_0+M_1)\sqrt{T}$ for all $i\in \ll2,2T\rr$. From the definition of the Gibbs measure, almost surely we have
		\begin{align*}
			\ind_{\{L_1(1)\le M_1\sqrt{T}\}\cap \neg\m{A}_2}\Ex_2\left[\ind_{\m{A}_1}\mid\mathcal{F}_2\right] \le  \frac1{\Ex[\Delta]}\Ex\left[\Delta\cdot \ind_{\m{A}_1}\right],
		\end{align*}
		where on the right-hand side, $\m{A}_1$ is defined as the event $\{\sup_{i\in\ll1,2T-1\rr} X(i) \ge M_0\sqrt{T}\}$ and
$$\Delta=\exp\left(-\sum_{i=1}^{T-1} \left(e^{-(M_0+1)\sqrt{T}-X_1(2i-1)}+e^{-(M_0+1)\sqrt{T}-X_1(2i+1)}\right)\right).$$
		As $\Delta\le 1$, by \eqref{exe}, $\Ex\left[\Delta\!\cdot\! \ind_{\m{A}_1}\right] \le \Ex[\ind_{\m{A}_1}] \le \e$. By \eqref{exe} we have  $\Ex[\Delta] \ge (1-\e)\cdot e^{-2(T-1)e^{-\sqrt{T}}} \ge \beta$ for some absolute constant $\beta>0$. Thus, $\Pr_2(\m{A}_1) \le (3+\beta^{-1})\e$. Similarly one can show
		\begin{align*}
			& \Pr_2\left(\inf_{i\in \ll2,2T\rr} L_2(i)\le -(M_3+M_0)\sqrt{T}\right) \le (3+\beta^{-1})\e, \\	& \Pr_2\left(\inf_{i\in \ll1,2T-1\rr} L_1(i)\le -(M_0+M_1)\sqrt{T}\right) \le 2\e.
		\end{align*}
		Thus adjusting the constants we can find $\til{M}$ such that
		\begin{align*}
			\Pr_{2}\bigg(\sup_{i\in\ll1,2T-1\rr}|L_1(i)|+\sup_{j\in\ll2,2T\rr}|L_2(j)|\ge (M-1)\sqrt{T}\bigg)\le \e/3.
		\end{align*}
		Finally via Lemma \ref{l:LSCrit} we know $L_2(1)-L_2(2)\sim G_{\alpha_2+\theta,1}$. Thus, by a union bound, for all $T$ large enough we have
		$\Pr_2(|L_2(1)|\ge M\sqrt{T}) \le \e/3 +\Pr_2(|L_2(1)-L_2(2)|\ge \sqrt{T}) \le 2\e/3$. By another union bound, we arrive at \eqref{eq:l5.u}.
	\end{proof}

	Recall the normalizing constant $V_p^T(\vec{y},\vec{z})$ from  \eqref{defvkt}. One can easily obtain a lower bound for this normalizing constant as a consequence of the Lemma \ref{l:sup}.
	\begin{corollary}\label{cora} Fix any $r>0$ and for each $N>0$ set $T=\lfloor rN^{2/3}\rfloor$. Fix any $p\in \{1,2\}$ and set $\alpha=\alpha_p$ according to \eqref{acric}. Recall $V_{p}^T(\vec{y},\vec{z})$ from \eqref{defvkt}. There exists $Q_0=Q_0(r)>0, N_0=N_0(r)>0$ such that for all $Q\ge Q_0$ and $N\ge N_0$, $V_p^T(\vec{y},\vec{z})\ge \tfrac12$ for all $\vec{z}\in \R^{T}$ with $z_i\le QN^{1/3}$ for $i\in \ll 1,T\rr$ and $\vec{y}\in \R^p$ with $y_i\ge (2Q-1)N^{1/3}$ for $i\in \ll 1,p\rr$. Here we assume $L_p(2T+1):=\infty$.
	\end{corollary}
	\begin{proof} Consider the event $$\m{A}:=\left\{\inf_{j\in \ll1,T\rr} L_p(2j-1) \ge (Q+1)N^{1/3}\right\}.$$
		Observe that
		\begin{align*}
			V_p^T(\vec{y},\vec{z}) & \ge \Ex_{\alpha_p}^{\vec{y},(-\infty)^{T};p,T}\left[\ind_{\m{A}}\prod_{j=1}^T W(z_j;L_p(2j+1),L_p(2j-1))\right] \\ & \ge \exp(-2Te^{-N^{1/3}}) \cdot \Pr_{\alpha_p}^{\vec{y},(-\infty)^{T};p,T}\left(\m{A}\right).
		\end{align*}	
		Taking $N$ large enough ensures $\exp(-2Te^{-N^{1/3}}) \ge 1/\sqrt{2}$. Since $\m{A}$ is an increasing event with respect to the boundary data, applying stochastic monotonicity (Proposition \ref{p:gmc}) and translation invariance (Lemma \ref{obs1} \ref{traninv}) of $\hslg$ Gibbs measures we have
		\begin{align*}
			\Pr_{\alpha_p}^{\vec{y},(-\infty)^{T};p,T}\left(\m{A}\right)=\Pr_{\alpha_p}^{\vec{x},(-\infty)^{T};p,T}\left(\inf_{j\in \ll1,T\rr} L_p(2j-1) \ge -(Q-2)N^{1/3}\right)
		\end{align*}
		where $\vec{x}=0$ if $p=1$ and $\vec{x}=(0,-\sqrt{T})$ if $p=2$. Appealing to Lemma \ref{l:sup} we may choose $Q$ large enough so that the above probability is at least $1/\sqrt{2}$.
	\end{proof}

	\subsection{Proof of Proposition \ref{e:mcd}} \label{sec:mcd} For clarity we divide the proof into three steps.
	
	\medskip
	
	\noindent\textbf{Step 1.}  In this step, we give a roadmap of the proof of \eqref{e:mcd} leaving the technical details to later steps.	Fix $r,\e, \delta>0$ and $p\in \{1,2\}$. Fix $N\ge 3$ large enough so that $T=8\lfloor rN^{2/3}\rfloor \ge 24$.  Set $\alpha=\alpha_p$ according to \eqref{acric} and consider the $\hslg$ line ensemble $\L^N$ from Definition \ref{l:nz} with parameters $(\alpha,\theta)$.  Consider the modulus of continuity event
	\begin{align*}
		\m{MC}_{\delta} := \big\{\omega_{\delta}^N(\L_1^N,\ll1,T/4-1\rr) \ge \gamma N^{1/3}\big\}.
	\end{align*}		
	By Theorem \ref{thm:eptight}, there exists $V=V(\e)>0$ such that
	\begin{align}\label{opti}
		\Pr(\m{A}_{1}) \ge 1-\e,\,\,  \mbox{   where   }\,\,  \m{A}_{1}:= \left\{N^{-1/3}|\L_1^N(1)|+N^{-1/3}|\L_2^N(2)|\le V\right\}.
	\end{align}
	By Proposition \ref{p:high}, there exists $M_1(\e)>0$ such that for all large enough $N$
	\begin{align}\label{e1low}
		\Pr\left(\L_1^N(2T-1) \ge M_1N^{1/3}\right) \le \e.
	\end{align}
	We claim that there exists $M_2(r,\e)>0$ such that for all large enough $N$
	\begin{align}\label{e:l2h}
		\Pr\left(\L_p^N(2T+p-2) \le -M_2N^{1/3}\right) \le \e.
	\end{align}
	We shall prove \eqref{e:l2h} in \textbf{Step 2}. Let us assume it for now. Set $M=\max\{M_1,M_2,4\}$ and consider the events
	\begin{align*}
		\m{B}_{1} & :=\big\{|\L_1^N(2T-1)| \le 2MN^{1/3}\big\}, \\
		\m{B}_{2} & :=\Big\{\L_2^N(2T) \ge -MN^{\frac13}, \,\,\L_1^N(2T-1)\le MN^{\frac13}, \,\,\L_1^N(2T-1) \ge \L_1^N(2T)-(\log N)^{7/6}\Big\},
	\end{align*}
	For each $\beta>0$  we define
	\begin{align}\label{c61}
		\m{C}_{}(p,\beta) & := \Big\{V_p^{T}\left((\L_j^N(2T+j-2))_{j\in \ll1,p\rr};(\L_{p+1}^N(2k))_{k=1}^T\right) \ge \beta\Big\},
	\end{align}
	where $V_p^T(\cdot,\cdot)$ is defined in \eqref{defvkt}. We now claim that there exists $\beta(r,\e)>0$ such that
	\begin{align}\label{lastpart}
		\Pr\left(\neg\m{C}_{}(p,\beta)\right) \le \e.
	\end{align}
	We work with this choice of $\beta$ for the rest of this step. We postpone the proof of \eqref{lastpart} to \textbf{Step 3}.  Let us now complete the proof of Proposition \ref{e:mcd} assuming it. Consider the following $\sigma$-algebra:
	\begin{align}\label{sigmaf}
		\mathcal{F}_{p,k}:=\sigma\left((\L_i^N\ll 1,2N-2i+2\rr)_{i\ge p+1}, (\L_i^N(j))_{j\ge 2k+i-2, i\in \ll1,p\rr}\right).
	\end{align}
	Clearly $\m{B}_{p}\cap \m{C}_{}(p,\beta)$ is measurable with respect to $\mathcal{F}_{p,T}$. By union bound and tower property of conditional expectation we have
	\begin{equation}
		\label{3terms}
		\begin{aligned}
			\Pr(\m{MC}_{\delta}) &  \le \Pr(\neg \m{A}_{1})+\Pr\left(\neg \m{B}_{p}\right)+\Pr\left(\neg \m{C}_{}(p,\beta)\right) \\ & \hspace{2cm}+\Ex\left[\ind_{\m{B}_{p}\cap\m{C}_{}(p,\beta)}\Ex\left[\ind_{\m{A}_{1}\cap \m{MC}_{\delta}}\mid \mathcal{F}_{p,T}\right]\right].
		\end{aligned}
	\end{equation}
	
We bound the four terms on the right-hand side of the above equation separately.

\smallskip\noindent\textbf{$\m{A}_{1}$ event:} We have $\Pr(\neg \m{A}_{1}) \le \e$ due to \eqref{opti}.

\smallskip\noindent\textbf{$\m{B}_{p}$ event:} Note that for large enough $N$, $\m{B}_{2} \subset \m{B}_{1}$. Combining \eqref{e1low}, \eqref{e:l2h}, and Theorem \ref{t:order} (with $\rho\mapsto\frac12, M\mapsto M$), by a union bound we see that for all large enough $N$,
		\begin{align*}
			\Pr(\neg \m{B}_{p}) & \le \Pr(\neg \m{B}_{2}) \\ & \le \Pr( \L_2^N(2T) \le -MN^{1/3})+\Pr(\L_1^N(2T-1)\ge MN^{1/3}) \\ & \hspace{1cm}+ \Pr\left(\L_1^N(2T-1) \le \L_1^N(2T)-(\log N)^{7/6}\right) \\ & \le 2\e+2^{-N} \le 3\e.
		\end{align*}

\smallskip \noindent \textbf{$\m{C}_{}(p,\beta)$ event:} We have $\Pr(\neg \m{C}_{}(p,\beta)) \le \e$ due to \eqref{lastpart}.

\smallskip \noindent\textbf{Conditional probability:} By Theorem \ref{thm:conn} and \eqref{e2gib} we have
		\begin{align}\label{e45}
			\Ex\left[\ind_{\m{A}_{1}\cap \m{MC}_{\delta}}\mid \mathcal{F}_{p,T}\right]=\frac{\Ex_{\alpha_p}^{\vec{y},(-\infty)^{2T};p,T}\left[V_p^T\left(\vec{y};(\L_{p+1}^N(2i))_{i=1}^T\right) \cdot\ind_{\m{A}_{1}\cap \m{MC}_{\delta}}\right]}{V_p^{T}\left(\vec{y};(\L_{p+1}^N(2i))_{i=1}^T\right)}
		\end{align}
		where $\vec{y}:=(\L_j^N(2T+j-2))_{j\in \ll1,p\rr}$ and $V_{p}^T(\cdot;\cdot)$ is defined in \eqref{defvkt}. From definition we have $V_p^{T}\left(\vec{y};(\L_{p+1}^N(2i))_{i=1}^T\right)\in [0,1]$. On $\m{C}_{}(p,\beta)$ we  have
		\begin{align*}
			\ind_{\m{C}_{}(p,\beta)} \cdot \mbox{r.h.s.~\eqref{e45}} \le \ind_{\m{C}_{}(p,\beta)} \cdot  \beta^{-1}   \cdot  \Pr_{\alpha_p}^{\vec{y},(-\infty)^{2T};p,T}\left(\m{A}_{1}\cap \m{MC}_{\delta}\right).
		\end{align*}
		Observe that the event $\m{B}_{p}$ ensures $\vec{y}\in I_{p,M}$ where the set $I_{p,M}$ is defined in the statement of Proposition \ref{pimc}. We can thus apply Proposition \ref{pimc} with $\e\mapsto \beta\cdot \e$, to get a $\delta>0$ such that
		\begin{align*}
			\ind_{\m{B}_{p}}\cdot \Pr_{\alpha_p}^{\vec{y},(-\infty)^{2T};p,T}\left(\m{A}_{1}\cap \m{MC}_{\delta}\right) \le \ind_{\m{B}_{p}}\cdot \e,
		\end{align*}
		for all large enough $N$. Thus overall we have
		\begin{align*}
			\Ex\big[\ind_{\m{B}_{p}\cap\m{C}_{}(p,\beta)}\Ex\left[\ind_{\m{A}_{1}\cap \m{MC}_{\delta}}\mid \mathcal{F}_{p,T}\right]\big]\le \e.
		\end{align*}
	Plugging the above four estimates  into r.h.s.~\eqref{3terms} and taking limsup $N\to \infty$, then $\delta\downarrow 0$, yields
	\begin{align*}
		\limsup_{\delta\downarrow 0}\limsup_{N\to \infty} 	\Pr(\m{MC}_{\delta}) \le 6\e.
	\end{align*}
	As $\e$ is arbitrary, we thus have \eqref{ee:mcd}, completing the proof.
	
	\medskip

	\noindent\textbf{Step 2.} In this step we prove \eqref{e:l2h}. We write $\Pr_{\alpha_p}$ instead of $\Pr$ to stress the fact that the $\hslg$ line ensemble has boundary parameter $\alpha_p$, defined in \eqref{acric}. We claim that there exists $M_2(r,\e)$ such that for all large enough $N$
	\begin{align}\label{e:l2hn}
		\Pr_{\alpha_p}\left(\m{Fall}_{p}^{(M_2)}\right) \le \tfrac{\e}{4}, \quad \m{Fall}_{p}^{(M_2)}:=\left\{\inf_{j\in \ll1,4T+4\rr,i\in \ll1,p\rr} \L_i^N(j) \le -M_2N^{1/3}\right\}.
	\end{align}
	Note that as $\{\L_p^N(2T+p-2)\le -M_2N^{1/3}\} \subset \m{Fall}_{p}^{(M_2)}$, \eqref{e:l2hn} implies \eqref{e:l2h}. To show \eqref{e:l2hn}, we first define a few more  events. For each $R\ge 32r+1$ we define
	\begin{align*}
		\m{B}_{i}^{(R,j)} & :=\left\{\L_i^N(2j+i-2) \ge -R^2N^{1/3}\right\}, \quad \til{\m{B}}_{i}^{(R,j)} :=  \m{B}_{i}^{(R,j)}\cap \bigcup_{k\in \ll j+1,RN^{2/3}\rr} \neg \m{B}_{i}^{(R,k)}, \\ \m{B}_{i}^{(R)} & :=\bigcup_{j\in \ll4T+4,RN^{2/3}\rr}\m{B}_{i}^{(R,j)}=\bigsqcup_{j\in \ll4T+4,RN^{2/3}\rr}\til{\m{B}}_{i}^{(R,j)}=\bigg\{\sup_{j\in \ll 4T+4,RN^{2/3}\rr} \L_i^N(2j+i-2) \ge -R^2N^{1/3}\bigg\}, \\
		\m{Dif}^{(R)} & :=\left\{ \L_1^N(2j-1)\ge \L_2^N(2j)+(\log N)^2 \mbox{ for all }j\in \ll 1,RN^{2/3}\rr\right\}.
	\end{align*}
	By Theorem \ref{t:order}, Theorem \ref{p:high2}, and Proposition \ref{p:high}, we can find a $R=R(r,\e)\ge 1$ such that for all large enough $N$, and for $v\in \{1,2\}$
	\begin{align}\label{eqbrav}
		\Pr_{\alpha_v}\left(\neg \m{B}_{1}^{(R)}\right)+\Pr_{\alpha_v}\left(\neg \m{B}_{2}^{(R)}\right)+\Pr_{\alpha_v}\left(\neg \m{Dif}^{(R)}\right) \le \tfrac{\e}{8}.
	\end{align}
	We fix this choice of $R$.  Observe that for large enough $N$, we have
	\begin{align*}
		\til{\m{B}}_{2}^{(R,i)}\cap \m{Dif}^{(R)} \subset \til{\m{B}}_{2}^{(R,i)}\cap \m{B}_{1}^{(2R,i)},
	\end{align*}
	uniformly for all $i\in \ll 4T+4,RN^{2/3}\rr$. For $p=2$, by the union bound and the tower property of conditional expectation, in view of \eqref{eqbrav}, we have
	\begin{equation}
		\label{2fall}
		\begin{aligned}
			\Pr_{\alpha_2}\left(\m{Fall}_{2}^{(M_2)}\right) &  \le \Pr_{\alpha_2}\left(\neg \m{B}_{2}^{(R)}\right)+\Pr\left(\neg \m{Dif}^{(R)}\right) +\sum_{j\in \ll4T+4,RN^{2/3}\rr} \Pr_{\alpha_2}\left(\til{\m{B}}_{2}^{(R,j)}\cap \m{B}_{1}^{(2R,j)} \cap\m{Fall}_{2}^{(M_2)}\right) \\ &  \le \tfrac{\e}{8} +\sum_{j\in \ll4T+4,RN^{2/3}\rr} \Ex\left[\ind_{\til{\m{B}}_{2}^{(R,j)}\cap \m{B}_{1}^{(2R,j)}}\Ex_{\alpha_2}\left[\ind_{\m{Fall}_{2}^{(M_2)}}\mid \mathcal{F}_{2,j}\right]\right],
		\end{aligned}
	\end{equation}
	where $\mathcal{F}_{p,k}$ is defined in \eqref{sigmaf}. For $p=1$, applying union bound and using \eqref{eqbrav} we have
	\begin{equation}
		\label{1fall}
		\begin{aligned}
			\Pr_{\alpha_1}\left(\m{Fall}_{1}^{(M_2)}\right) &  \le \Pr_{\alpha_1}\left(\neg \m{B}_{1}^{(R)}\right)+\sum_{j\in \ll4T,RN^{2/3}\rr} \Pr_{\alpha_1}\left(\til{\m{B}}_{1}^{(R,j)} \cap\m{Fall}_{1}^{(M_2)}\right) \\ &  \le \tfrac{\e}{8} +\sum_{j\in \ll4T+4,RN^{2/3}\rr} \Ex\left[\ind_{\til{\m{B}}_{1}^{(R,j)}}\Ex_{\alpha_1}\left[\ind_{\m{Fall}_{1}^{(M_2)}}\mid \mathcal{F}_{1,j}\right]\right].
		\end{aligned}
	\end{equation}
	We now proceed to control the conditional expectation $\Pr_{\alpha_p}\left(\m{Fall}_{p}^{(M_2)}\mid \mathcal{F}_{p,j}\right)$ separately for $p=1$ and $p=2$. Applying the Gibbs property (Theorem \ref{thm:conn}),  we have
	\begin{align*}
		\ind_{\til{\m{B}}_{2}^{(R,j)}\cap \m{B}_{1}^{(2R,j)}}\cdot	\Ex_{\alpha_2}\left[\m{Fall}_{2}^{(M_2)}\mid \mathcal{F}_{2,j}\right] & = \ind_{\til{\m{B}}_{2}^{(R,j)}\cap \m{B}_{1}^{(2R,j)}}\cdot\Pr_{\alpha_2}^{\vec{y},\vec{z};2,j}\bigg(\m{Fall}_{2}^{(M_2)}\bigg) \\ & \le \ind_{\til{\m{B}}_{2}^{(R,j)}\cap \m{B}_{1}^{(2R,j)}}\cdot\Pr_{\alpha_2}^{(0,-\sqrt{j}),(-\infty)^{j};2,j}\bigg(\m{Fall}_{2}^{(M_2-4R^2)}\bigg).
	\end{align*}
	Here $\vec{y}=(\L_1^N(2j-1),\L_2^N(2j))$ and $\vec{z}=(\L_{3}^N(2m))_{m=1}^{j}$. Let us briefly explain the above inequality. Note that on $\til{\m{B}}_{2}^{(R,j)}\cap \m{B}_{1}^{(2R,j)}$ we have $y_i \ge (-4R^2N^{1/3}-(i-1)\sqrt{j})$ for $i=1,2$. Furthermore $\m{Fall}_{2}^{(M_2)}$ is an event which decreases with respect to boundary data. Thus to obtain an upper bound, by stochastic monotonicity (Proposition \ref{p:gmc}), we may take the boundary data from $(y_1,y_2)$ to $(-4R^2N^{1/3},-4R^2N^{1/3}-\sqrt{j})$ and $\vec{z}$ to $(-\infty)^{j}$. The above inequality then follows by translation invariance (see Lemma \ref{obs1} \ref{traninv}). Similar applications of the Gibbs property and stochastic monotonicity yield that on $\til{\m{B}}_{1}^{(R,j)}$ we have
	\begin{align*}
		\Ex_{\alpha_1}\left[\ind_{\m{Fall}_{1}^{(M_2)}}\mid \mathcal{F}_{2,j}\right] \le \Pr_{\alpha_1}^{0,(-\infty)^{j};1,j}\left(\m{Fall}_{1}^{(M_2-4R^2)}\right).
	\end{align*}
	We now claim that one can choose $M_2(r,\e)>0$ large enough such that for all $j\in \ll4T+4,RN^{2/3}\rr$,
	\begin{align}\label{ver}
		\Pr_{\alpha_p}^{\vec{x},(-\infty)^{j};p,j}\left(\m{Fall}_{p}^{(M_2-4R^2)}\right) \le \tfrac{\e}{8},
	\end{align}
	where $\vec{x}:=0$ (if $p=1$) or $\vec{x}:=(0,-\sqrt{j})$ (if $p=2$). Plugging the above bound back in \eqref{1fall} and \eqref{2fall} and using the fact that $\{\til{\m{B}}_{p}^{(R,j)}\}_{j\in \ll4T+4,RN^{2/3}\rr}$ is a disjoint collection of events we arrive at the bound in \eqref{e:l2hn}. Thus we are left to verify \eqref{ver} in this step. Observe that
	\begin{align*}
		\Pr_{\alpha_p}^{\vec{x},(-\infty)^{j};p,j}\left(\m{Fall}_{p}^{(M_2-4R^2)}\right) \le \Pr_{\alpha_p}^{\vec{x},(-\infty)^{j};p,j}\left(\inf_{k\in \ll1,2j+i-2\rr, i\in \ll1,p\rr} L_i(k) \le -(M_2-4R^2)N^{1/3}\right)
	\end{align*}
	By Lemma \ref{l:sup}, one can choose $M_2$ large enough such that the above expression is bounded above by $\e/8$ for all $j\in \ll4T,RN^{2/3}\rr$. This proves \eqref{ver} completing our work for this step.
	
	\medskip

	\noindent\textbf{Step 3.} In this step we prove \eqref{lastpart}.
	For each $Q>0$ consider the event
	\begin{align}\label{defam}
		\m{D}_{Q} & :=\left\{ \sup_{i\in \ll 1,4T+4\rr}\L_{p+1}^N(i) \le QN^{1/3}, \,\, \inf_{j\in \ll1,p\rr}\L_j^N(4T+j+2) \ge -QN^{1/3}+\sqrt{2T+1}\right\}.
	\end{align}	
	By Theorem \ref{t:order}, Proposition \ref{l:lhigh}, and \eqref{e:l2hn} there exists $Q(r,\e)>0$ large enough such that $\Pr(\neg\m{D}_{p,Q}) \le \frac{\e}{2}$. 
	Consider $\mathcal{F}_{p,2T+2}$ from \eqref{sigmaf}. Recall the event ${\m{C}}_{}(p,\beta)$ from \eqref{c61}. By union bound and the tower-property of the expectation, we have
	\begin{align}\label{eq:twtmr}
		\Pr(\neg\m{C}_{}(p,\beta))  & \le \Pr(\neg\m{C}_{}(p,\beta) \cap \m{D}_Q)+\tfrac{\e}{2}  = \Ex\Big[\ind_{\m{D}_{Q}}\Ex\big[\ind_{\m{C}_{}(p,\beta)}\mid \mathcal{F}_{p,2T+2}\big]\Big]+\tfrac{\e}{2}.
	\end{align}
	Applying the Gibbs property and \eqref{e2gib} we have
	\begin{align*}
		\Ex[\ind_{\neg\m{C}_{}(p,\beta)}\mid \mathcal{F}_{p,2T+2}] & =\Pr_{\alpha_p}^{\vec{y};	\vec{z};p,2T+2}\big(\neg\m{C}_{}(p,\beta)\big)
	\end{align*}
	with $\vec{y}=(y_1,\ldots,y_p)$ and $y_j=\L_j^N(4T+j+2)$ for $j\in \ll1,p\rr$, and $\vec{z}=(\L_{p+1}^N(2k))_{k=1}^{2T+2}$. Set $\vec{x}=(-QN^{1/3}+\sqrt{2T+1})^{p}$. We claim that there exists $Q_0(r,\e)>0$, $N_0(r,\e)>0$ and $\beta(r,\e)>0$, such that for all $N\ge N_0$, $Q\ge Q_0$, $y_i\ge x_i$ and $\vec{z}\in \R^{2T+2}$ with $\sup_{i\in \ll 1,2T+2\rr} z_i\le QN^{1/3}$ we have
	\begin{align}\label{eq:clm1}
		\Pr_{\alpha_p}^{\vec{y};	\vec{z};p,2T+2}\left(\neg\til{\m{C}}_{}(p,\beta)\right) \le \tfrac{\e}{2}, \mbox{ where } \til{\m{C}}_{}(p,\beta):=\left\{\mathcal{V}_p\ge\beta\right\},
	\end{align}
	where we set (see \eqref{defvkt})
\begin{equation}\label{eq:Vp}
\mathcal{V}_p:=V_p^T\big((L_i(2T+i-2))_{i\in\ll1,p\rr},(z_1,\ldots,z_T)\big).
\end{equation}
Clearly in view of the definition of $\m{D}_{Q}$ from \eqref{defam}, the above claim shows that r.h.s.~\eqref{eq:twtmr} is at most $\e$ \eqref{lastpart}. Thus, to complete our proof it suffices to check \eqref{eq:clm1}. Towards this end, we first claim that for all $\vec{y}\in \R^p, \vec{z}\in \R^{2T+2}$
	\begin{align}\label{eq:clm2}
		\Pr_{\alpha_p}^{\vec{y};	\vec{z};p,2T+2}\left(\neg\til{\m{C}}_{}(p,\beta)\right)= \frac{\Ex_{\alpha_p}^{\vec{y};\vec{w};p,2T+2}\left[\ind_{\til{\m{C}}_{}(p,\beta)}\cdot \mathcal{R}_p\cdot \mathcal{V}_p\right]}{\Ex_{\alpha_p}^{\vec{y};	\vec{w};p,2T+2}\left[\mathcal{R}_p\cdot\mathcal{V}_p\right]}, \quad \mbox{ where } \,\,\mathcal{R}_p:=e^{-e^{z_T-L_2(2T+1)}\ind_{p=2}},
	\end{align}
and where $\vec{w}\in [-\infty,\infty)^{2T+2}$ is defined by setting $w_i=-\infty$ for $i\le T$ and $w_i=z_i$ for $i>T$. We postpone the proof of \eqref{eq:clm2} to the next step.

	\medskip
	
	Assuming \eqref{eq:clm2}, to prove \eqref{eq:clm1}, we provide upper and lower bounds for the numerator and denominator of r.h.s.~\eqref{eq:clm2} respectively. Consider the events
	\begin{align*}
		\m{R}_1 & :=\left\{L_1(2T-1)\ge 2QN^{1/3}\right\}, \\
		\m{R}_2 & :=\left\{L_2(2T)\ge 2QN^{1/3}, L_2(2T+1)\ge 2QN^{1/3}, L_1(2T-1) \ge (2Q-1)N^{1/3}\right\}
	\end{align*}
	Note that
	\begin{align}
		\nonumber
		\Ex_{\alpha_p}^{\vec{y},	\vec{w};p,2T+2}\left[\mathcal{R}_p\cdot\mathcal{V}_p\right] & \ge \Ex_{\alpha_p}^{\vec{y},	\vec{w};p,2T+2}\left[\ind_{ \m{R}_p}\cdot \mathcal{R}_p\cdot \mathcal{V}_p\right] \\ \nonumber & \ge \tfrac12 \exp(-e^{-QN^{1/3}})\cdot \Pr_{\alpha_p}^{\vec{y},\vec{w};p,2T+2}\big(\m{R}_{p}\big) \\ & \ge \tfrac12 \exp(-e^{-QN^{1/3}})\cdot \Pr_{\alpha_p}^{\vec{x},(-\infty)^{2T+2};p,2T+2}\big(\m{R}_{p}\big). \label{anoth}
	\end{align}
	where the penultimate inequality follows from the definition of $\mathcal{R}_p$ and Corollary \ref{cora} and the final inequality follows via stochastic monotonicity (Proposition \ref{p:gmc}) as $\m{R}_p$ is an increasing event with respect to the boundary data (recall $y_i\ge x_i$). To lower bound the above expression, we proceed into two cases depending on the value of $p$.
	
	\medskip
	
\noindent\textbf{Case 1. $p=1$.} Note that $\m{R}_1 \supset \m{RP}_{1,Q}$ event defined in \eqref{defrp}.  By Proposition \ref{l:rpass}, we have $\Pr_{\alpha_1}^{\vec{x},(-\infty)^{2T+2};1,2T+2}(\m{R}_1)\ge \Pr_{\alpha_1}^{-QN^{1/3},(-\infty)^{2T+2};1,2T+2}(\m{RP}_{1,Q}) \ge \phi_1>0$ for some $\phi_1$ free of $N$.
	
	\medskip
	
\noindent\textbf{Case 2. $p=2$.} Let $\vec{u}:=(-QN^{1/3}+\sqrt{2T+2},-QN^{1/3})$. Let us use the shorthand notation $\Pr_{2}^{\gamma_1,\gamma_2}$ for $\Pr_{\alpha_2}^{(\gamma_1,\gamma_2),(-\infty)^{2T+2};2,2T+2}$. Note that by stochastic monotonicity and union bound we have
	\begin{equation}
		 \label{2termss}
			\begin{aligned}
			\Pr_{2}^{\vec{x}}(\m{R}_2) \ge \Pr_{2}^{\vec{u}}(\m{R}_2) & \ge \Pr_{2}^{\vec{u}}\Big(\{L_2(2T) \ge 2QN^{1/3}\}\cap \{L_2(2T+1)\ge 2QN^{1/3}\}\Big)  \\ & \hspace{2cm}-\Pr_{2}^{\vec{u}}\big(L_1(2T-1)\le L_2(2T)-N^{1/3}\big).
		\end{aligned}
	\end{equation}

	Note that $\m{RP}_{2,Q} \subset \{L_2(2T) \ge 2QN^{1/3}\}\cap \{L_2(2T+1)\ge 2QN^{1/3}\}$ (with $T$ replaced by $T+1$ in \eqref{defrp}). Applying stochastic monotonicity (Proposition \ref{p:gmc}) and Proposition \ref{l:rpass} with $p\mapsto 2$ and $T\mapsto T+1$, we see that the first term in the above equation can be bounded as
	\begin{align}\label{1termss}
		& \Pr_{2}^{\vec{u}}\Big(\{L_2(2T) \ge 2QN^{1/3}\}\cap \{L_2(2T+1)\ge 2QN^{1/3}\}\Big)  \ge \Pr_{2}^{(-QN^{1/3},-(Q+1)N^{1/3})}\big(\m{RP}_{2,Q}\big) \ge \phi_2,
	\end{align}
	for some $\phi_2>0$ free of $N$. As for the second term in r.h.s.~\eqref{2termss}, by translation invariance (Lemma \ref{obs1} \ref{traninv}) we have
	\begin{align*}
		\Pr_{2}^{\vec{u}}\big(L_1(2T-1)\le L_2(2T)-N^{1/3}\big)  & =\Pr_{2}^{(0,-\sqrt{2T+1})}\big(L_1(2T-1)\le L_2(2T)-N^{1/3}\big) \\ &  =\frac{\erw{2T+2}{(0,\sqrt{2T+1})}\left[\wsc\ind_{\se{1}{T-1}\le \se{2}{T-1}-N^{1/3}}\right]}{{\erw{2T+2}{(0,\sqrt{2T+1})}[\wsc]}},
	\end{align*}
	where the last equality follows from Lemma \ref{l:LSCrit} (recall the PRW law from Definition \ref{prb} and $\wsc$ from \eqref{defw}). Now by Corollary \eqref{corb}, $\erw{2T+2}{(0,\sqrt{2T+1})}[\wsc]\ge \Con/\sqrt{2T+2}$ for some absolute constant $\Con>0$. However on the event $\{\se{1}{T-1}\le \se{2}{T-1}-N^{1/3}\}$, $\wsc \le \exp(-e^{N^{1/3}})$. Thus, $$\Pr_{2}^{\vec{u}}\big(L_1(2T-1)\le L_2(2T)-N^{1/3}\big) \to 0$$ as $N\to \infty$. Hence, inserting \eqref{1termss} back in r.h.s.~\eqref{2termss}, we see that for all large enough $N$,
$$\Pr_{2}^{\vec{x}}(\m{R}_{2})\ge \phi_2-\Pr_{2}^{\vec{u}}\big(L_1(2T-1)\le L_2(2T)-N^{1/3}\big) \ge \frac12\phi_2.$$
	
	\bigskip
	
	Summarizing the above two cases, for all large enough $N$, \eqref{anoth} is lower bounded by some $\phi>0$ free of $N$.	 For the numerator in r.h.s.~\eqref{eq:clm2} observe that as $\mathcal{R}_p\le 1$, by definition of the event  $\til{\m{C}}_{}(p,\beta)$, we have	$\ind_{\neg\til{\m{C}}_{}(p,\beta)}\cdot \mathcal{R}_p\cdot \mathcal{V}_p \le \beta$. Let us now choose $\beta=\phi \e$. Plugging these bounds back in r.h.s.~\eqref{eq:clm2} yields \eqref{eq:clm1}.
	\medskip
	
	\noindent\textbf{Step 5.} All that remains is to prove \eqref{eq:clm2}. We will do this for the $p=2$ case. The $p=1$ case is done analogously. Fix any $\vec{y}\in \R^2$, $\vec{z}\in \R^{2T+2}$ and define $\vec{w}\in [-\infty,\infty)^{2T+2}$ such that $w_i=-\infty$ for $i\le T$ and $w_i=z_i$ for $i>T$. Assume $(L_1\ll1,4T+3\rr,L_2\ll1,4T+4\rr) \sim \Pr_{\alpha_2}^{\vec{y},\vec{z};2,2T+2}$. Let $\mathcal{G}:=\sigma(L_i\ll 2T+i-2,4T+i+2\rr)_{i\in \ll1,2\rr}$. Fix any event $\m{F}$ measurable with respect to $\mathcal{G}$. Set $L_2(4T+1)=\infty$. We claim that (recall $W(a;b,c)$ from \eqref{def:wfn}, $\mathcal{V}_p$ from \eqref{eq:Vp}, and $\mathcal{R}_2$ from \eqref{eq:clm2})
	\begin{equation}
		\label{longiden}
		\begin{aligned}
			& \Ex_{\alpha_2}^{\vec{y},(-\infty)^{2T+2};2,2T+2}\left[\ind_{\m{F}} \cdot\prod_{j=1}^{2T+2}{ W(z_{j};L_2(2j+1),L_2(2j-1))}\right] \\ & = \Ex_{\alpha_2}^{\vec{y},(-\infty)^{2T+2};2,2T+2}\left[\ind_{\m{F}} \cdot \mathcal{R}_2 \cdot \mathcal{V}_2\prod_{j=1}^{2T+2}{ W(w_{j};L_2(2j+1),L_2(2j-1))} \right].
		\end{aligned}
	\end{equation}	
Assuming \eqref{longiden} we can finish the proof of \eqref{eq:clm2} (for $p=2$) via the following string of equalities:
\begin{align*}
	 & \Pr_{\alpha_2}^{\vec{y};\vec{z};2,2T+2}\bigg(\neg\til{\m{C}}(2,\beta)\bigg) \\ & =\frac{\Ex_{\alpha_2}^{\vec{y};(-\infty)^{2T+2};2,2T+2}\bigg[\ind_{\neg\til{\m{C}}(2,\beta)}\cdot \prod\limits_{j=1}^{2T+2} W(z_{j};L_2(2j+1),L_2(2j-1))\bigg]}{\Ex_{\alpha_2}^{\vec{y};(-\infty)^{2T+2};2,2T+2}\bigg[\prod\limits_{j=1}^{2T+2}{ W(z_{j};L_2(2j+1),L_2(2j-1))}\bigg]} \\ & = \frac{\Ex_{\alpha_2}^{\vec{y};(-\infty)^{2T+2};2,2T+2}\bigg[\ind_{\neg\til{\m{C}}(2,\beta)}\cdot \mathcal{R}_2 \cdot \mathcal{V}_2 \prod\limits_{j=1}^{2T+2} W(w_{j};L_2(2j+1),L_2(2j-1))\bigg]}{\Ex_{\alpha_2}^{\vec{y};(-\infty)^{2T+2};2,2T+2}\bigg[\mathcal{R}_2 \cdot \mathcal{V}_2\prod\limits_{j=1}^{2T+2}{ W(w_{j};L_2(2j+1),L_2(2j-1))}\bigg]} \\ & = \frac{\Ex_{\alpha_2}^{\vec{y};\vec{w};2,2T+2}\bigg[\ind_{\neg\til{\m{C}}(2,\beta)}\cdot \mathcal{R}_2 \cdot \mathcal{V}_2 \bigg]}{\Ex_{\alpha_2}^{\vec{y};\vec{w};2,2T+2}\bigg[\mathcal{R}_2 \cdot \mathcal{V}_2\bigg]}.
\end{align*}
Let us briefly explain the above equalities. The first equality is due to \eqref{e2gib} and \eqref{defvkt}. In the second equality we have applied \eqref{longiden} to the numerator and denominator by taking  $\m{F}=\neg\m{C}(2,\beta)$ and $\m{F}=\Omega$ (the full set, i.e., $\ind_{\m{F}}=1$) respectively. The last equality follows by applying \eqref{e2gib} and \eqref{defvkt} again. This proves \eqref{eq:clm2} modulo \eqref{longiden}.

To see why \eqref{longiden} holds, observe that
	\begin{align*}
		& \Ex_{\alpha_2}^{\vec{y},(-\infty)^{2T+2};2,2T+2}\left[\ind_{\m{F}} \cdot\prod_{j=1}^{2T+2}{ W(z_{j};L_2(2j+1),L_2(2j-1))}\right] \\ & = \Ex_{\alpha_2}^{\vec{y},(-\infty)^{2T+2};2,2T+2}\left[\ind_{\m{F}} \cdot \mathcal{R}_2\prod_{j=T+1}^{2T+2}{ W(z_{j};L_2(2j+1),L_2(2j-1))} \right. \\ & \hspace{2.5cm} \left.\cdot\Ex_{\alpha_2}^{\vec{y},(-\infty)^{2T+2};2,2T+2}\left[\mathcal{R}_2\prod_{j=1}^{T-1}{ W(z_{j};L_2(2j+1),L_2(2j-1))} \mid \mathcal{G}\right]\right].
	\end{align*}
	By the Gibbs property, the inner expectation, when viewed as a random variable, is almost surely equal to $\mathcal{V}_2$ defined in \eqref{defvkt}. On the other hand, we have
	$$\prod_{j=T+1}^{2T+2}{ W(z_{j};L_2(2j+1),L_2(2j-1))} = \prod_{j=1}^{2T+2} W(w_{j};L_2(2j+1),L_2(2j-1)).$$
	Combining the above two observations, leads to \eqref{longiden}.

	\subsection{Proof of Proposition \ref{pimc}} \label{sec:pimc} As with the proof of Propositions \ref{lem:ep} and \ref{l:rpass}, we divide the proof of Proposition \ref{pimc} into two parts depending on $p=1$ (critical) or $p=2$ (supercritical).
	
	\begin{proof}[Proof of Proposition \ref{pimc} in the $p=1$ case (critical phase)] Fix any $T\in \ll k_1N^{\frac23},k_2N^{\frac23}\rr$. Fix any $\delta \le \gamma/{6\kappa}$. We recall the representation of \btf\ law in $p=1$ case from Lemma \ref{l:LCrit}. Consider the Brownian motion $B_1$ obtained via the KMT coupling that satisfies \eqref{coupl}. Define
		\begin{align*}
			\m{A}_{\delta} & :=\bigg\{\sup_{\substack{i_1,i_2\in \ll1,T\rr \\ |i_1-i_2|\le \frac{\delta}{2}N^{2/3}}} |L_1(2i_1-1)-L_1(2i_2-1)| \ge \tfrac16\gamma N^{\frac13}\bigg\}, \\
			\m{B}_{}(k) & :=\left\{ |L_1(2k-1)-L_1(2k)|, |L_1(2k+1)-L_1(2k)| \ge  \tfrac13\gamma N^{\frac13} \right\}.
		\end{align*}
		Fix any $x\in \R$ and write $\Pr_1:=\Pr_{\alpha_1}^{x,(-\infty)^{T};1,T}$. Observe that by union bound we have
		\begin{align}\label{unbf}
			\Pr_1\left(\omega_{\delta}^N(L_1, \ll1,2T-1\rr) \ge \gamma 	N^{1/3}\right)\le \Pr_1(\m{A}_{\delta})+\sum_{k=1}^{T-1}\Pr_1(\neg\m{A}_{\delta}\cap \m{B}_{}(k)).
		\end{align}
		We now proceed to bound each of the above term separately. For the first term, by \eqref{def:LCrit2} and \eqref{coupl}, in view of the estimate in \eqref{logu} we have for all large enough $N$ that
		\begin{align*}
			\Pr_1(\m{A}_{\delta}) & \le \Pr_1\bigg(\sup_{\substack{i_1,i_2\in \ll1,T\rr \\ |i_1-i_2|\le \frac{\delta}{2}N^{2/3}}}\sigma|B_1{(T-i_1-1)}-B_1{(T-i_2-1)}| \ge \tfrac{\gamma}{12}N^{1/3}-2\Con \log T\bigg) \\ & \le \Pr_1\bigg(\sup_{\substack{i_1,i_2\in \ll1,T\rr \\ |i_1-i_2|\le \frac{\delta}{2}N^{2/3}}}\sigma|B_1{(i_1)}-B_1{(i_2)}| \ge \tfrac{\gamma}{24}N^{1/3}\bigg).
		\end{align*}
		By modulus of continuity of Brownian motion, the right-hand side of the above equation can be made smaller than $\frac12\e$ by choosing $\delta$ small enough depending on $\mu,\theta,\gamma,k_1,k_2$. For the second term in r.h.s.~\eqref{unbf} we use Lemma \ref{qlemma} to get
		\begin{align*}
			\Pr_1(\neg\m{A}_{\delta}\cap \m{B}_{}(k)) \le \Con e^{-\tfrac1\Con \gamma N^{\frac13}}.
		\end{align*}
		Plugging the bounds back in \eqref{unbf} and taking $N$ large enough we get the desired result.
	\end{proof}
	
	\begin{proof}[Proof of Proposition \ref{pimc} in the $p=2$ case (supercritical phase)] Fix any $(x_1,x_2) \in I_{2,M}$, and $T\in \ll k_1N^{\frac23},k_2N^{\frac23}\rr$. Set $n:=T$. Recall the law paired random walk and weighted paired random walk defined in Definition \ref{prb}.  We recall from Lemma \ref{l:LSCrit} that the \btf\ law $\Pr_{\alpha_2}^{(x_1,x_2),(-\infty)^T;2,T}$ is equal to $\wprw{n}{(x_1,x_2)}$ for the supercritical case. At this point is it also good to recall the random walk measures from Definition \ref{def:rws}.
		
A key to this proof is the following estimate for $\erw{n}{(x_1,x_2)}[\wsc]$ (recall $\wsc$ from \eqref{defw}).
		
		\begin{lemma}\label{crude} There exist constants $\Con_1, \Con_2>0$, depending on $M$, such that for all $(x_1,x_2)\in I_{2,M}$ we have
			\begin{align}\label{eq:crude}
				\erw{n}{(x_1,x_2)}[\wsc] \ge \tfrac1{\sqrt{n}}{\Con_1^{-1}}\cdot\pr{\lfloor n/4 \rfloor}{(x_1,x_2)}{}(\til{\ni}) \ge {\Con_2^{-1}}e^{-\Con_2 (\log n)^{5/4}},
			\end{align}
			where $\til{\ni}:=\{S_1(k) \ge S_2(k) \mbox{ for all } k\in \ll 1,n/4\rr\}$ and $S_1,S_2$ are random walks under the law $\pr{\lfloor n/4 \rfloor}{(x_1,x_2)}{}$.
		\end{lemma}

		Before proving Lemma \ref{crude} we complete the proof of Proposition \ref{pimc} in the following two steps.
		
		\medskip
		
		\noindent\textbf{Step 1.} Fix any $V,\gamma>0$.  Set $v=\gamma/\sqrt{k_2}, u=V/\sqrt{k_1}$, and $t=\lfloor 2\log\log n \rfloor$. Let $\mathcal{F}:=\sigma(\ise,\iise)$.
		Consider the events
		\begin{align*}
			\m{MC}_{\delta} & := \Big\{|\ise|+|\iise|\le u\sqrt{n}, \omega_{\delta}^N(\se{i}{\cdot}, \ll0,\tfrac{n}8\rr) \ge \tfrac16v \sqrt{n}, \mbox{ for }i=1,2\Big\}, \mbox{ for } \delta>0.
		\end{align*}
		We claim that given $\e>0$, there exists $\delta$ small enough and $N$ large enough such that
		\begin{align}\label{mcde}
			\wprw{n}{(x_1,x_2)}(\m{MC}_{\delta}) \le \e.
		\end{align}
		 We finish the proof of the lemma assuming \eqref{mcde}. Lemma \ref{l:LSCrit} implies that $\big(L_1(2j+1),L_2(2j+2)\big)_{j=0}^{n}$ is distributed as $\m{WPRW}$. Write $L_2^{\m{even}}(k):=L_2(2k)$ and $\Pr_2:=\Pr_{\alpha_2}^{\vec{x},(-\infty)^{T};2,T}$. Then  \eqref{mcde} implies
		\begin{align}\label{mcde2}
			\Pr_2\big(\{|L_1(1)|+|L_2(2)|\le VN^{1/3}\}\cap\m{A}\big) \le \e, \quad \m{A}:=\Big\{ \omega_{\delta}^N(L_2^{\m{even}}, \ll 1,T/8\rr) \ge \tfrac16\gamma N^{1/3}\Big\}.
		\end{align}
	On the event $\neg \m{A}$ the increments of $L_2^{\m{even}}$ are well controlled. By Lemma \ref{l:LSCrit}, conditioned on the even points of $L_2$, the distribution of the odd points of $L_2$ are given by $\qo$-distributions defined in \eqref{qdist}. Once we have a bound on  the increments of $L_2^{\m{even}}$, we may invoke the tails estimates of $\qo$-distributions from Lemma \ref{qlemma} to control increments of $L_2$. In particular, due to Lemma \ref{qlemma},
		$$\ind_{\neg\m{A}}\cdot \Ex_2\left[\ind_{|L_2(2k+1)-L_2(2k)|,|L_2(2k+1)-L_2(2k+2)|\ge \frac13\gamma N^{1/3}}\mid \sigma\big(L_2^{\m{even}}\ll 1,T/8\rr\big)\right] \le \Con \exp(-\tfrac1\Con \gamma N^{1/3})$$
	for all $k\ge 1$. For the first point in $L_2$, i.e., $L_2(1)$, we recall from Lemma \ref{l:LSCrit} that $L_2(1)\sim X+L_2(2)$ where $X\sim G_{\alpha_2+\theta,1}$. The explicit form of $G_{\alpha_2+\theta,1}$ from \eqref{def:gwt} allow us to derive that	
	$$\Pr_2\bigg(|L_2(1)-L_2(2)| \ge \tfrac16\gamma N^{1/3} \mid \sigma\big(L_2^{\m{even}}\ll 1,T/8\rr\big)\bigg) \le \Con \exp(-\tfrac1\Con \gamma N^{1/3}).$$
		Thus, in view of \eqref{mcde2}, by the union bound
		\begin{align*}
			\Pr_2\left(|L_1(1)|+|L_2(2)| \le VN^{\frac13}, \omega_{\delta}^N(L_2, \ll1,T/4\rr) \ge \gamma 	N^{\frac13}\right) \le \e+\Con\cdot k_2N^{\frac23} \exp(-\tfrac1\Con \gamma N^{\frac13})
		\end{align*}
		which can be made arbitrarily small taking $N$ large enough. A similar argument shows that
		\begin{align*}
			\Pr_2\left(|L_1(1)|+|L_2(2)| \le VN^{\frac13}, \omega_{\delta}^N(L_1, \ll1,T/4-1\rr) \ge \gamma 	N^{\frac13}\right)
		\end{align*}
	can  be made arbitrarily small as well taking $N$ large enough.	This proves  Proposition \ref{pimc}.
		\medskip
		
		\noindent\textbf{Step 2.} In this step we prove \eqref{mcde}. First, recall that due to \eqref{wscint},
$$\wprw{n}{(x_1,x_2)}(\m{MC}_{\delta})=\frac{\erw{n}{(x_1,x_2)}[\wsc\ind_{\m{MC}_{\delta}}]}{\erw{n}{(x_1,x_2)}[\wsc]}$$
where $\wsc$ is defined in \eqref{defw}.
 We first define a few more necessary events.
		\begin{align*}
			\m{G}_{1} & := \{|\ise|+|\iise|\le u\sqrt{n}, |\ise-\iise|\le (\log n)^{3/2}\}, \\
			\m{G}_{2} & := \{|\ise|+|\iise|\le u\sqrt{n}, 1\le \ise-\iise\le 2\}.
		\end{align*}
		Recall the non-intersection event $\ni_p$ from \eqref{def:nip}. Let us temporarily set $t=\lceil 2\log\log n\rceil$. As $\wsc\le 1$, we write
		\begin{align*}
			\erw{n}{(x_1,x_2)}[\wsc\ind_{\m{MC}_{\delta}}] \le \underbrace{\erw{n}{(x_1,x_2)}[\wsc\ind_{\m{MC}_{\delta}\cap \m{G}_{1} \cap \ni_{t}}]}_{(\mathbf{I})}+\underbrace{{\erw{n}{(x_1,x_2)}[\wsc\ind_{\neg\m{NI}_{t}}]+\erw{n}{(x_1,x_2)}[\ind_{\neg\m{G}_{1}}]}}_{(\mathbf{II})}.
		\end{align*}
		For $(\mathbf{II})$, note that on $\neg \ni_t$, we have $\wsc\le e^{-e^t}\le e^{-(\log n)^2}$ and by Lemma \ref{as:el}, $\prw{n}{(x_1,x_2)}(\neg\m{G}_{1}) \le \Con e^{-\Con^{-1}(\log n)^{3/2}}$. Thus, $(\mathbf{II})\le \Con e^{-\Con^{-1}(\log n)^{3/2}}$. In view of Lemma \ref{crude}, $(\erw{n}{(x_1,x_2)}[\wsc])^{-1}\cdot (\mathbf{II}) \to 0$. For $(\mathbf{I})$, note that
		\begin{equation}
			\label{rty3}
			\begin{aligned}
				(\mathbf{I}) & = \erw{n}{(x_1,x_2)}[\wsc\ind_{\m{MC}_{\delta}\cap \m{G}_{1}  \cap \ni_{0}}]+\sum_{p=1}^{t} \erw{n}{(x_1,x_2)}[\wsc\ind_{\m{MC}_{\delta}\cap \m{G}_{1}  \cap \ni_{p}\cap \neg\ni_{p-1}}] \\ & \le  \sum_{p=0}^{t} \Con e^{-e^p} {\erw{n}{(x_1,x_2)}}\big[\ind_{\m{G}_{1}}\erw{n}{(x_1,x_2)}[\ind_{\m{MC}_{\delta}\cap\ni_p}\mid\mathcal{F}]\big].
			\end{aligned}
		\end{equation}
		Upon conditioning on $\mathcal{F}$, the conditional law is the law of two independent $n$-step random  bridges from $(\ise,\iise)$ to $(x_1,x_2)$. We may lift the bridges by $p$ units. The modulus of continuity event remains unchanged and $\ni_p$ event turns into $\ni$. Now we apply the comparison trick between random bridges and modified random bridges via Lemma \ref{lem:compare}. By Lemma \ref{lem:compare}, there exists a constant $\Con$ depending only on $u$ such that
		\begin{equation}
			\label{rty1}
		\begin{aligned}
			\ind_{\m{G}_1}\erw{n}{(x_1,x_2)}[\ind_{\m{MC}_{\delta}\cap\ni_p}\mid\mathcal{F}] & =\ind_{\m{G}_1}\Pr^{n;(\ise,\iise),(x_1,x_2)}(\m{MC}_{\delta}\cap \ni_p) \\ & =\ind_{\m{G}_1}\Pr^{n;(\ise+p,\iise),(x_1+p,x_2)}(\m{MC}_{\delta}\cap \ni) \\ & \le \Con \cdot \ind_{\m{G}_1}\til\Pr_{p}(\m{MC}_{\delta}\cap\ni) = \Con \cdot \ind_{\m{G}_1}\til\Pr_{p}(\m{MC}_{\delta}\mid \ni) \cdot \til\Pr_{p}(\ni)
		\end{aligned}
	\end{equation}
			where $\til{\Pr}_{p}$ denote the law of a $(n;\lfloor n/4\rfloor,\lfloor n/4\rfloor)$-modified random bridge defined in Definition \ref{def:mrb} starting from $(\ise+p,\iise)$ to $(x_1+p,x_2)$.  Observe that $\til\Pr_{p}(\ni)$ is $\mathcal{F}$-measurable. By Lemmas \ref{l:class} and \ref{l:nipp} (recall Corollary \ref{rm:con}),
		\begin{align}\label{rty2}
			\ind_{\m{G}_1}\cdot\til{\Pr}_{p}(\ni) \le \ind_{\m{G}_1}\cdot \tfrac\Con{\sqrt{n}} \cdot e^{\Con p} \cdot \max\{\ise-\iise,1\}\cdot \Pr^{\lfloor n/4\rfloor ,(x_1,x_2)}\big(\til{\ni}\big).
		\end{align}
		We plug the estimates from \eqref{rty1} and \eqref{rty2} back in \eqref{rty3}. Thus setting $\Con_3:=\sum_{r=1}^\infty 2\Con_1 \Con^3 e^{\Con r}e^{-e^r}$ (with $\Con_1$ coming from Lemma \ref{crude}) and utilizing the lower bound for $\erw{n}{(x_1,x_2)}[\wsc]$ from Lemma \ref{crude}, we have
		\begin{align*}
			(\erw{n}{(x_1,x_2)}[\wsc])^{-1}\cdot (\mathbf{I})\le \Con_3 \cdot  \erw{n}{(x_1,x_2)}\left[\ind_{\m{G}_{1}} \cdot \max\{\ise-\iise,1\} \cdot \sup_{p\in \ll 0,t\rr} \til{\Pr}_{p}(\m{MC}_{\delta}\mid \ni)\right]
		\end{align*}
		Now we claim that one can choose $\delta$ sufficiently small such that
		\begin{align}\label{efrfr}
			\erw{n}{(x_1,x_2)}\left[\ind_{\m{G}_{1}} \cdot \max\{\ise-\iise,1\} \cdot \sup_{p\in \ll 0,t\rr} \til{\Pr}_{p}(\m{MC}_{\delta}\mid \ni)\right]   \le \tfrac12\Con_1^{-1}\e.
		\end{align}
		We write $\m{G}_{1}=\m{G}_{1,M_2}\cup \til{\m{G}}_{1,M_2}$, where
		\begin{align*}
			\m{G}_{1,M_2}:=\{|\ise|+|\iise|\le u\sqrt{n}, |\ise-\iise|\le M_2\}, \quad \til{\m{G}}_{1,M_2}:=\m{G}_{1} \cap \neg \m{G}_{1,M_2}.
		\end{align*}
		Given the tail estimates, one can choose $M_2$ large enough such that
		\begin{align*}
			\erw{n}{(x_1,x_2)}\left[\ind_{\til{\m{G}}_{1,M_2}} \cdot \max\{\ise-\iise,1\} \right]   \le \tfrac14\Con_1^{-1}\e.
		\end{align*}
		This fixes our choice for $M_2$. Now note that the event $\m{MC}_{\delta}$ depends only on the first $\lfloor n/8\rfloor$ points of the two $(n;\lfloor n/4\rfloor ,\lfloor n/4\rfloor)$-modified random bridges. By definition, the first $\lfloor n/4\rfloor$ points of a $(n;\lfloor n/4\rfloor,\lfloor n/4\rfloor)$-modified random bridge is just a random walk. Thus, in view of Lemma \ref{mret} (recall Corollary \ref{rm:con}), one can then choose $\delta$ small enough and $N$ large enough such that on uniformly on $\m{G}_{1,M_2}$ we have $$\sup_{p\in \ll 0,t\rr}\til{\Pr}_p(\m{MC}_{\delta}\mid \ni) \le \tfrac14\Con_1^{-1} M_2^{-1} \e.$$ Thus,
		we have
		\begin{align*}
			\mbox{l.h.s.~\eqref{efrfr}} & \le \erw{n}{(x_1,x_2)}\Big[\ind_{\til{\m{G}}_{1,M_2}} \!\!\!\cdot \max\{\ise-\iise,1\} \Big] + M_2 \!\cdot\!\erw{n}{(x_1,x_2)}\Big[\ind_{\m{G}_{1,M_2}} \!\cdot\!\sup_{p\in \ll 0,t\rr} \til{\Pr}_{p}(\m{MC}_{\delta}\mid \ni)\Big] \\ & \le \tfrac14\Con_1^{-1}\e+M_2\!\cdot\!\tfrac14\Con_1^{-1}M_2^{-1}\e =\tfrac12\Con_1^{-1}\e,
		\end{align*}
		verifying the inequality in \eqref{efrfr}.
	\end{proof}
	
	\begin{proof}[Proof of Lemma \ref{crude}] Recall the definition of $(n,p,q)$-modified random bridge from Definition \ref{def:mrb}, in particular that $\tpr{(n;\lfloor n/4\rfloor,\lfloor n/4\rfloor)}{(a_1,a_2)}{,(x_1,x_2)}$ denotes the law of two independent $(n,\lfloor n/4\rfloor,\lfloor n/4\rfloor)$-modified random bridge started at $(a_1,a_2)$ and ended at $(x_1,x_2)$. We shall use the shorthand $\til{\Pr}^{(a_1,a_2)}$ for $\tpr{(n;\lfloor n/4\rfloor,\lfloor n/4\rfloor)}{(a_1,a_2)}{,(x_1,x_2)}$.  Also recall the notation $\pr{m}{(b_1,b_2)}{}$ from Definition \ref{def:rws} to denote the law of two independent random walks of length $m$ started at $(b_1,b_2)$ with same increment law as the modified bridges.
	
		\medskip
		
		Recall the events $\ni$ and $\m{Gap}_{\beta}$ from \eqref{def:nip} and \eqref{def:gp} respectively.  Invoking Lemma \ref{l:gapev} we first fix a $\beta=\beta(M) \le \frac12$ small enough so that it satisfies $$\til{\Pr}^{(a_1,a_2)}(\m{Gap}_{\beta}\mid \ni) \ge \tfrac34,$$ for all $|a_i|\le \sqrt{n}$ with $1\le a_1-a_2\le 2$. Next by Lemma \ref{tailni}, we fix $\xi=\xi(M)>0$ so that
$$
\pr{\lfloor n/4\rfloor}{(b_1,b_2)}{}\Big(|\se{1}{\lfloor n/4\rfloor}|,|\se{2}{\lfloor n/4\rfloor}| \le \xi \sqrt{n}\mid \til{\ni}\Big) \ge \sqrt{\tfrac34}
$$
 for all $|b_i|\le (M+1)\sqrt{n}$. Here $\til{\ni}:=\{S_1(k) \ge S_2(k) \mbox{ for all } k\in \ll 2,n/4\rr\}$ is the non-intersection event over $\lfloor n/4\rfloor$ points.
		
		We consider the following events
		\begin{align*}
			\m{G}_{3} & := \big\{ |\se{i}{1}| \le \sqrt{n} \mbox{ for }i=1,2, \  1\le \ise-\iise \le 2\big\},  \\
			\m{T}_{\xi} & := \big\{|\se{i}{\lfloor n/4\rfloor }|,|\se{i}{n-\lfloor n/4\rfloor }| \le \xi\sqrt{n} \mbox{ for }i=1,2\big\},
		\end{align*}
		where $\m{T}$ stands for tightness. Observe that by Lemma \ref{wgap} we have
		\begin{equation}
			\label{wr1}
			\begin{aligned}
				\erw{n}{(x_1,x_2)}[\wsc]  \ge \erw{n}{(x_1,x_2)}[\wsc\ind_{\m{Gap}_{\beta}\cap\m{G}_{3}\cap\m{T}_{\xi}}]  & \ge \tfrac1\Con \prw{n}{(x_1,x_2)}(\m{Gap}_{\beta}\cap\m{G}_{3}\cap\m{T}_{\xi}) \\ &  = \tfrac1\Con  \erw{n}{(x_1,x_2)}\left[\ind_{\m{G}_{3}}\erw{n}{(x_1,x_2)}[\ind_{\m{Gap}_{\beta},\m{T}_{\xi}}\mid\mathcal{F}]\right]
			\end{aligned}
		\end{equation}
		where $\mathcal{F}:=\sigma(\ise,\iise)$. Under the event $\m{G}_{3}$ and $\m{T}_{\xi}$ we may invoke Lemma \ref{lem:compare} to get
		\begin{align}\label{wr2}
			\ind_{\m{G}_3}\cdot	\erw{n}{(x_1,x_2)}[\ind_{\m{Gap}_{\beta}\cap\m{T}_{\xi}}\mid\mathcal{F}] \ge \Con^{-1} \cdot \ind_{\m{G}_3}\cdot \til{\Pr}^{(\se{1}{1},\se{1}{2})}(\m{Gap}_{\beta}\cap\m{T}_{\xi})
		\end{align}
		almost surely. By Corollary \ref{l:niexp} (recall Corollary \ref{rm:con}),
		\begin{equation}\label{wr3}
			\begin{aligned}
			& \ind_{\m{G}_3}\cdot	\til{\Pr}^{(\se{1}{1},\se{1}{2})}({\m{Gap}_{\beta}\cap\m{T}_{\xi}}) \\ & = \ind_{\m{G}_3}\cdot	\til{\Pr}^{(\se{1}{1},\se{1}{2})}({\m{Gap}_{\beta}\cap\m{T}_{\xi}}\mid \ni)\til{\Pr}_{(a_1,a_2)}(\ni) \\ & \ge \Con^{-1} \ind_{\m{G}_3}\cdot	\til{\Pr}^{(\se{1}{1},\se{1}{2})}({\m{Gap}_{\beta}\cap\m{T}_{\xi}}\mid \ni) \cdot \pr{\lfloor n/4 \rfloor}{(\se{1}{1},\se{1}{2})}{}(\til{\ni}) \pr{\lfloor n/4 \rfloor}{(x_1,x_2)}{}(\til{\ni}).
			\end{aligned}
		\end{equation}
		By our choice of $\beta$ and $\xi$, we have $\ind_{\m{G}_3}\cdot	\til{\Pr}^{(\se{1}{1},\se{1}{2})}({\m{Gap}_{\beta},\m{T}_{\xi}}\mid\ni) \ge \frac12\ind_{\m{G}_3}$ almost surely.	By Lemma \ref{l:class} (recall Corollary \ref{rm:con}), we have $\ind_{\m{G}_3}\cdot	\pr{\lfloor n/4 \rfloor}{(\se{1}{1},\se{1}{2})}{}(\til{\ni}) \ge \frac{\Con^{-1}}{\sqrt{n}}$ almost surely. Thus combining \eqref{wr1}, \eqref{wr2}, and \eqref{wr3} we have
		\begin{align*}
			\erw{n}{(x_1,x_2)}[\wsc] \ge \tfrac{1}{\sqrt{n}}\Con^{-1}\cdot \pr{\lfloor n/4 \rfloor}{(x_1,x_2)}{}(\til{\ni}) \cdot \prw{n}{(x_1,x_2)}(\m{G}_{3}).
		\end{align*}
		By Lemma \ref{as:el} (\eqref{t3} in particular), $\prw{n}{(x_1,x_2)}(\m{G}_{3}) \ge \Con^{-1}$. Plugging this back in the above equation we get the first inequality in \eqref{eq:crude}.	For the second inequality, we consider the event:
		\begin{align*}
			\m{G}_{4} & := \{|\se{1}{2}-x_1|\le 1,\,\, |\se{2}{2}-\min\{x_1-3,x_2\}| \le 1\}.
		\end{align*}
		Observe that
		\begin{align*}
			\pr{\lfloor n/4 \rfloor}{(x_1,x_2)}{}(\til{\ni}) \ge \pr{\lfloor n/4 \rfloor}{(x_1,x_2)}{}\bigg(\m{G}_{4} \cap \{S_1(j) \ge S_2(j) \mbox{ for all }j\in \ll3,n/4\rr\}\bigg).
		\end{align*}
		By the tail bounds of the increments from Lemma \ref{tailf}, and given the condition $x_1-x_2\ge -(\log N)^{7/6}$, we have $\pr{\lfloor n/4 \rfloor}{(x_1,x_2)}{}(\m{G}_{4}) \ge  \Con^{-1}\exp(-\Con (\log n)^{7/6})$ (recall $n \ge k_1N^{2/3}-1$). Furthermore, on $\m{G}_{4}$ we must have $\se{1}{2}\ge \se{2}{2}$. By Lemma \ref{l:class} (recall Corollary \ref{rm:con}), we have $$\pr{\lfloor n/4 \rfloor-1}{(a_1,a_2)}{}\big(S_1(j) \ge S_2(j) \mbox{ for all }j\in \ll2,n/4-1\rr \big) \ge \Con^{-1}/\sqrt{n}$$ for all $a_1\ge a_2$. Thus we have
		$$\pr{\lfloor n/4 \rfloor}{(x_1,x_2)}{}\bigg(\m{G}_{4} \cap \{S_1(j) \ge S_2(j) \mbox{ for all }j\in \ll3,n/4\rr\}\bigg) \ge \Con^{-1}\exp(-\Con (\log n)^{7/6})\cdot \tfrac{1}{\sqrt{n}}.$$
		Adjusting the constant we get the second inequality in \eqref{eq:crude}. 	\end{proof}

\appendix

\section{Stochastic monotonicity}\label{appc}
The goal of this section is to prove the stochastic monotonicity of $\hslg$ Gibbs measure (Proposition \ref{p:gmc}). Let $\Lambda= \{(i,j) : k_1\le i\le k_2, a_i\le j\le b_{i}\}$. Let $w_1,\ldots,w_{|\Lambda|}$ be the enumeration of points in $\Lambda$ in the lexicographic order. Set $\Lambda_r=\{w_1,w_2,\ldots,w_r\}$, so that $\Lambda_{|\Lambda|}=\Lambda$. Let $E_r:= E(\Lambda_r\cup \partial\Lambda_r)$ (the edges in $\Z^2$ connecting points in $\Lambda_r\cup\partial\Lambda_r$), and, recalling the weights $W_e$ from \eqref{def:wfn}, let
\begin{equation}\label{eq:Hrden}
	H_r(x;(u_v)_{v\in \partial \Lambda_r}):=\int_{\R^{|\Lambda_{r-1}|}} \prod_{e=\{v_1\to v_2\}\in E_r} W_e(u_{v_1}-u_{v_2})\prod_{v\in \Lambda_{r-1}}du_v,
\end{equation}
where $u_{w_r}=x$.  The proof of Proposition \ref{p:gmc} relies on the following technical lemma.

\begin{lemma}\label{lem45} Fix $r\in \ll 1,|\Lambda|\rr$. For each $v\in \partial \Lambda_r$, fix any $u_v,u_v'\in \R$ with $u_v\le u_v'$. For all $s\ge t$
	\begin{align}\label{c1}
		H_r\big(s;(u_v)_{v\in \partial \Lambda_r}\big)H_r\big(t;(u_v')_{v\in \partial \Lambda_r}\big) \le H_r\big(s;(u_v')_{v\in \partial \Lambda_r}\big)H_r\big(t;(u_v)_{v\in \partial \Lambda_r}\big)
	\end{align}
\end{lemma}
We prove Lemma \ref{lem45} at the end of this section and now complete the proof of Proposition \ref{p:gmc}.
\begin{proof}[Proof of Proposition \ref{p:gmc}]
	Fix $r\in \ll 1,|\Lambda|\rr$. We first claim that for all boundary conditions $(u_v)_{v\in \partial\Lambda_r}$ and $(u_v')_{v\in \partial\Lambda_r}$ with $u_v\le u_v'$ for all $v\in \partial\Lambda_r$, and $s\in \R$,
	\begin{equation}\label{eq:PrLwr}
		\Pr\big(L(w_r)\leq s \mid L(v)=  u_v \textrm{ for all } v\in \partial\Lambda_r\big)\geq \Pr\big(L(w_r)\leq s \mid L(v)=  u_v' \textrm{ for all } v\in \partial\Lambda_r\big).
	\end{equation}
	To show this, observe that $H_r(x;(u_v)_{v\in\partial \Lambda_r})$ in \eqref{eq:Hrden}  is proportional to the conditional density at $x$ of $L(w_r)$ given $\big(L(v)\big)_{v\in \partial\Lambda_{r}} =(u_v)_{v\in\partial \Lambda_r}$. Thus,
	\begin{equation}\label{PrLwdef}
		\Pr\big(L(w_r)\leq s \mid L(v)=  u_v \textrm{ for all } v\in \partial\Lambda_r\big)= F_r\big(s;(u_v)_{v\in\partial \Lambda_r}\big) :=\frac{\int_{-\infty}^s H_r(x;(u_v)_{v\in \partial \Lambda_r})dx}{\int_{-\infty}^{\infty} H_r(x;(u_v)_{v\in \partial \Lambda_r})dx}
	\end{equation}
	To prove \eqref{eq:PrLwr} observe that owing to Lemma \ref{lem45}, the derivative of
	$$
	\log \int_{-\infty}^s H_r(x;(u_v)_{v\in \partial \Lambda_r})dx-\log\int_{-\infty}^s H_r(x;(u_v')_{v\in \partial \Lambda_r})dx.
	$$
	is non-positive for all $s$. This implies for $s'\geq s$ we have
	\begin{align*}
		\frac{\int_{-\infty}^s H_r(x;(u_v)_{v\in \partial \Lambda_r})dx}{\int_{-\infty}^s H_r(x;(u_v')_{v\in \partial \Lambda_r})dx} \ge \frac{\int_{-\infty}^{s'} H_r(x;(u_v)_{v\in \partial \Lambda_r})dx}{\int_{-\infty}^{s'} H_r(x;(u_v')_{v\in \partial \Lambda_r})dx}
	\end{align*}
	Taking $s'\to \infty$ and cross-multiplying yields the desired inequality \eqref{eq:PrLwr}, in light of \eqref{PrLwdef}.
	
	Given $(u_v)_{v\in \partial\Lambda}\in \R^{|\partial \Lambda|}$, we now define a sequence of random variables according to the following algorithm. Note that below, $x\leftarrow y$ means to assign the value $y$ to the variable $x$.
	
	\begin{algorithm}[H]
		\caption{Defining the random vectors}
		\begin{algorithmic}
			\STATE Generate $U_1,\ldots,U_{|\Lambda|}$ i.i.d.~random variables from $U[0,1]$
			\STATE $Y_{|\Lambda|} \leftarrow (u_v)_{v\in \partial\Lambda}$
			\STATE $r\leftarrow |\Lambda|$
			\WHILE{$r\ge 1$}
			\STATE $L(w_r;(u_w)_{v\in \partial\Lambda}) \leftarrow F_r^{-1}(U_r;Y_r)$
			\STATE $\til{u}_v \leftarrow u_v$ for all $v\in \partial\Lambda_{r-1}\cap\partial\Lambda_{r}$
			\STATE $\til{u}_{w_r}\leftarrow L(w_r;(u_w)_{v\in \partial\Lambda})$
			\STATE $Y_{r-1} \leftarrow (\til{u}_v)_{v\in \partial\Lambda_{r-1}}$
			\STATE $r\leftarrow r-1$
			\ENDWHILE
		\end{algorithmic}
	\end{algorithm}
	This defines a collection of random variables
	$L(w_i;(u_v)_{v\in \partial\Lambda})$ indexed by  $i\in\ll 1, |\Lambda|\rr$ and $(u_v)_{v\in \partial\Lambda}\in \R^{|\partial\Lambda|}$,
	all on the common probability space on which $U_1,\ldots,U_{|\Lambda|}$ are defined. It is clear from the definition that for each $(u_v)_{v\in \partial\Lambda}\in \R^{|\partial\Lambda|}$, the law of $\big(L(w_i;(u_v)_{v\in \partial\Lambda})\big)_{i\in \ll 1 ,|\Lambda|\rr}$ is given by the $\hslg$ Gibbs measure on the domain $\Lambda$ with boundary condition $(u_v)_{v\in \partial\Lambda}$. Take two boundary conditions $(u_v)_{v\in \partial\Lambda}$ and $(u_v')_{v\in \partial\Lambda}$ with $u_v\le u_v'$ for all $v\in \partial\Lambda$. As each $F_r$ is stochastically increasing with respect to the boundary condition, i.e., \eqref{eq:PrLwr}, sequentially we obtain that with probability 1 on our probability space $L(w_r;(u_v)_{v\in \partial\Lambda}) \le  L(w_r;(u_v')_{v\in \partial\Lambda})$ for all $r$, thus completing the proof.
\end{proof}

\begin{proof}[Proof of Lemma \ref{lem45}] Let us begin with a few pieces of notations. Fix any $1\le r\le |\Lambda|$. Set $e_r:=\{w_r\to (w_{r}+(0,1)),(w_{r}+(0,1))\to w_r\}\cap E_r$. In words, this is the directed blue edge (see Figure \ref{fig20} A) with $w_r$ as the left point of $e_r$.

	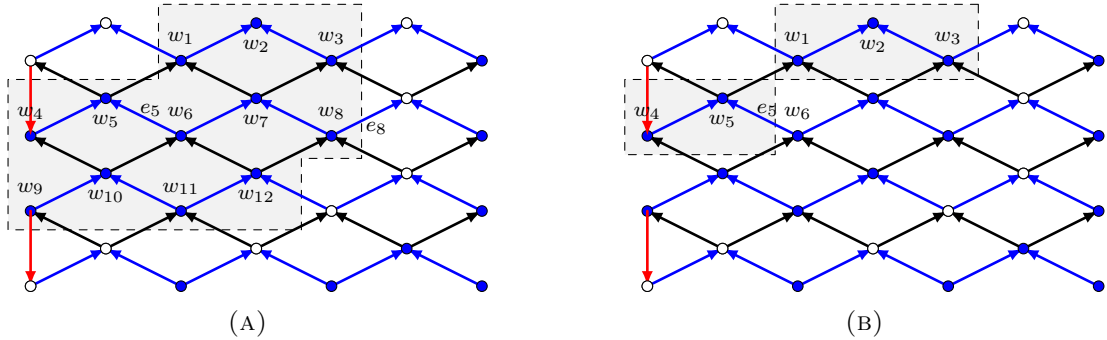
\begin{figure}[h!]
		\centering
		\begin{subfigure}[b]{0.45\textwidth}
			\centering
			\begin{tikzpicture}[line cap=round,line join=round,>=triangle 45,x=2cm,y=1cm]
				\draw[fill=gray!10,dashed] (1.7,0.25)--(0.35,0.25)--(0.35,-0.75)--(-0.65,-0.75)--(-0.65,-2.75)--(1.3,-2.75)--(1.3,-1.8)--(1.7,-1.8)--(1.7,0.25);
				\foreach \x in {0,1,2}
				{
					\draw[line width=1pt,blue,{Latex[length=2mm]}-]  (\x,0) -- (\x-0.5,-0.5);
					\draw[line width=1pt,blue,{Latex[length=2mm]}-] (\x,0) -- (\x+0.5,-0.5);
					\draw[line width=1pt,black,{Latex[length=2mm]}-] (\x-0.5,-0.5) -- (\x,-1);
					\draw[line width=1pt,black,{Latex[length=2mm]}-] (\x+0.5,-0.5) -- (\x,-1);
					\draw[line width=1pt,blue,{Latex[length=2mm]}-]  (\x,-1) -- (\x-0.5,-1.5);
					\draw[line width=1pt,blue,{Latex[length=2mm]}-] (\x,-1) -- (\x+0.5,-1.5);
					\draw[line width=1pt,black,{Latex[length=2mm]}-] (\x-0.5,-1.5) -- (\x,-2);
					\draw[line width=1pt,black,{Latex[length=2mm]}-] (\x+0.5,-1.5) -- (\x,-2);
					\draw[line width=1pt,blue,{Latex[length=2mm]}-]  (\x,-2) -- (\x-0.5,-2.5);
					\draw[line width=1pt,blue,{Latex[length=2mm]}-] (\x,-2) -- (\x+0.5,-2.5);
					\draw[line width=1pt,black,{Latex[length=2mm]}-] (\x-0.5,-2.5) -- (\x,-3);
					\draw[line width=1pt,black,{Latex[length=2mm]}-] (\x+0.5,-2.5) -- (\x,-3);
					\draw[line width=1pt,blue,{Latex[length=2mm]}-]  (\x,-3) -- (\x-0.5,-3.5);
					\draw[line width=1pt,blue,{Latex[length=2mm]}-] (\x,-3) -- (\x+0.5,-3.5);
				}
				\foreach \x in {0,1,2}
				{
					\draw [fill=blue] (\x,0) circle (2pt);
					\draw [fill=blue] (\x,-1) circle (2pt);
					\draw [fill=blue] (\x,-2) circle (2pt);
					\draw [fill=blue] (\x-0.5,-0.5) circle (2pt);
					\draw [fill=blue] (\x-0.5,-1.5) circle (2pt);
					\draw [fill=blue] (\x-0.5,-2.5) circle (2pt);
					\draw [fill=blue] (\x-0.5,-3.5) circle (2pt);
					\draw [fill=blue] (\x,-3) circle (2pt);
				}
				\draw[line width=1pt,red,{Latex[length=2mm]}-] (-0.5,-1.5) -- (-0.5,-0.5);
				\draw[line width=1pt,red,{Latex[length=2mm]}-] (-0.5,-3.5) -- (-0.5,-2.5);
				\draw [fill=blue] (2.5,-0.5) circle (2pt);
				\draw [fill=blue] (2.5,-1.5) circle (2pt);
				\draw [fill=blue] (2.5,-2.5) circle (2pt);
				\draw [fill=blue] (2.5,-3.5) circle (2pt);
				\draw [fill=white] (2,0) circle (2pt);
				\draw[fill=white] (2,-1) circle (2pt);
				\draw[fill=white] (2,-2) circle (2pt);
				\draw[fill=white] (1.5,-2.5) circle (2pt);
				\draw[fill=white] (0,-3) circle (2pt);
				\draw[fill=white] (1,-3) circle (2pt);
				\draw[fill=white] (0,0) circle (2pt);
				\draw[fill=white] (-0.5,-0.5) circle (2pt);
				\draw[fill=white] (-0.5,-3.5) circle (2pt);
				\begin{scriptsize}
					\node at (0.5,-0.2) {$w_1$};
					\node at (1,-0.3) {$w_2$};
					\node at (1.5,-0.2) {$w_3$};
					\node at (-0.5,-1.2) {$w_4$};
					\node at (0,-1.3) {$w_5$};
					\node at (0.5,-1.2) {$w_6$};
					\node at (1,-1.3) {$w_7$};
					\node at (1.5,-1.2) {$w_8$};
					\node at (-0.5,-2.2) {$w_{9}$};
					\node at (0,-2.3) {$w_{10}$};
					\node at (0.5,-2.2) {$w_{11}$};
					\node at (1,-2.3) {$w_{12}$};
					\node at (0.3,-1.15) {$e_5$};
					\node at (1.8,-1.4) {$e_8$};
				\end{scriptsize}
			\end{tikzpicture}
			\caption{}
		\end{subfigure}\qquad
		\begin{subfigure}[b]{0.45\textwidth}
			\centering
			\begin{tikzpicture}[line cap=round,line join=round,>=triangle 45,x=2cm,y=1cm]
				\draw[fill=gray!10,dashed] (1.7,0.25)--(0.35,0.25)--(0.35,-0.75)--(-0.65,-0.75)--(-0.65,-1.75)--(0.35,-1.75)--(0.35,-0.75)--(1.7,-0.75)--(1.7,0.25);
				\foreach \x in {0,1,2}
				{
					\draw[line width=1pt,blue,{Latex[length=2mm]}-]  (\x,0) -- (\x-0.5,-0.5);
					\draw[line width=1pt,blue,{Latex[length=2mm]}-] (\x,0) -- (\x+0.5,-0.5);
					\draw[line width=1pt,black,{Latex[length=2mm]}-] (\x-0.5,-0.5) -- (\x,-1);
					\draw[line width=1pt,black,{Latex[length=2mm]}-] (\x+0.5,-0.5) -- (\x,-1);
					\draw[line width=1pt,blue,{Latex[length=2mm]}-]  (\x,-1) -- (\x-0.5,-1.5);
					\draw[line width=1pt,blue,{Latex[length=2mm]}-] (\x,-1) -- (\x+0.5,-1.5);
					\draw[line width=1pt,black,{Latex[length=2mm]}-] (\x-0.5,-1.5) -- (\x,-2);
					\draw[line width=1pt,black,{Latex[length=2mm]}-] (\x+0.5,-1.5) -- (\x,-2);
					\draw[line width=1pt,blue,{Latex[length=2mm]}-]  (\x,-2) -- (\x-0.5,-2.5);
					\draw[line width=1pt,blue,{Latex[length=2mm]}-] (\x,-2) -- (\x+0.5,-2.5);
					\draw[line width=1pt,black,{Latex[length=2mm]}-] (\x-0.5,-2.5) -- (\x,-3);
					\draw[line width=1pt,black,{Latex[length=2mm]}-] (\x+0.5,-2.5) -- (\x,-3);
					\draw[line width=1pt,blue,{Latex[length=2mm]}-]  (\x,-3) -- (\x-0.5,-3.5);
					\draw[line width=1pt,blue,{Latex[length=2mm]}-] (\x,-3) -- (\x+0.5,-3.5);
				}
				\foreach \x in {0,1,2}
				{
					\draw [fill=blue] (\x,0) circle (2pt);
					\draw [fill=blue] (\x,-1) circle (2pt);
					\draw [fill=blue] (\x,-2) circle (2pt);
					\draw [fill=blue] (\x-0.5,-0.5) circle (2pt);
					\draw [fill=blue] (\x-0.5,-1.5) circle (2pt);
					\draw [fill=blue] (\x-0.5,-2.5) circle (2pt);
					\draw [fill=blue] (\x-0.5,-3.5) circle (2pt);
					\draw [fill=blue] (\x,-3) circle (2pt);
				}
				\draw[line width=1pt,red,{Latex[length=2mm]}-] (-0.5,-1.5) -- (-0.5,-0.5);
				\draw[line width=1pt,red,{Latex[length=2mm]}-] (-0.5,-3.5) -- (-0.5,-2.5);
				\draw [fill=blue] (2.5,-0.5) circle (2pt);
				\draw [fill=blue] (2.5,-1.5) circle (2pt);
				\draw [fill=blue] (2.5,-2.5) circle (2pt);
				\draw [fill=blue] (2.5,-3.5) circle (2pt);
				\draw [fill=white] (2,0) circle (2pt);
				\draw[fill=white] (2,-1) circle (2pt);
				\draw[fill=white] (2,-2) circle (2pt);
				\draw[fill=white] (1.5,-2.5) circle (2pt);
				\draw[fill=white] (0,-3) circle (2pt);
				\draw[fill=white] (1,-3) circle (2pt);
				\draw[fill=white] (0,0) circle (2pt);
				\draw[fill=white] (-0.5,-0.5) circle (2pt);
				\draw[fill=white] (-0.5,-3.5) circle (2pt);
				\begin{scriptsize}
					\node at (0.3,-1.15) {$e_5$};
					\node at (0.5,-0.2) {$w_1$};
					\node at (1,-0.3) {$w_2$};
					\node at (1.5,-0.2) {$w_3$};
					\node at (-0.5,-1.2) {$w_4$};
					\node at (0,-1.3) {$w_5$};
					\node at (0.5,-1.2) {$w_6$};
				\end{scriptsize}
			\end{tikzpicture}
			\caption{}
		\end{subfigure}
		
		\caption{(A)  A possible domain $\Lambda$ includes all the vertices in the shaded region. $w_i$'s are the vertices of $\Lambda$ enumerated in lexicographic order. Directed edges $e_r$ going are shown above for $r=5$ and $r=8$. These are the blue edges with $w_r$ as the left point of $e_r$. (B) The domain $\Lambda_5$ includes the vertices in the shaded region. $Q_5$ is the set of all red and black edges that have one vertex as $w_6$ and one vertex in $\partial\Lambda_6$. In the above figure, $Q_5$ is composed of two black edges that points toward $w_6$.}
		\label{fig20}
	\end{figure}

	Define
	\begin{align*}
		h_r\big(x;(u_v)_{v\in \partial\Lambda_r}\big):=\int_{\R^{|\Lambda_{r-1}|}} \prod_{e=\{v_1\to v_2\}\in E_r\setminus \{e_r\}} W_e(u_{v_1}-u_{v_2})\prod_{v\in \Lambda_{r-1}} du_v
	\end{align*}
	with the convention $u_{w_r}=x$. Observe that the difference between $H_r$ from  \eqref{eq:Hrden} and $h_r$ above is that the weight of the directed blue edge $e_r$ is included in the former but not in the latter. Note that the vertices of $e_r$ are not in $\Lambda_{r-1}$. Thus in the definition of $H_r$, the edge weight function corresponding to $e_r$ can be pulled out of the integrand leading to
	\begin{align}\label{Hrt}
		H_r\big(x;(u_v)_{v\in \partial\Lambda_r}\big)=h_r\big(x;(u_v)_{v\in \partial\Lambda_r}\big)\cdot F_r(u_{w_r+(0,1)}-x)
	\end{align}
	where $F_r(y)$ is the directed blue edge weight corresponding to $e_r$, i.e., $F_r(y):=e^{\vartheta_r y-e^{y}}$ or $F_r(y)=e^{-\vartheta_r y-e^{-y}}$ depending on the direction of the $e_r$ edge between $w_r$ and $w_r+(0,1)$. Here $\vartheta_r$ is the parameter linked to the blue edge $e_r$.
	
	With the above introduced notation, we now turn towards the proof of \eqref{c1}. Note that given a function $P(x)=e^{-R(x)}$ with $R$ being convex, we have
	\begin{align}\label{pineq}
		P(\delta-\beta)P(\gamma-\alpha)\ge P(\delta-\alpha)P(\gamma-\beta)
	\end{align}
	for all $\alpha,\beta,\gamma,\delta\in \R$ with $\alpha\le \beta$ and $\gamma\le \delta$. All our weight functions in \eqref{def:wfn} are of this type. In particular, this implies that \eqref{pineq} holds for $P=F_r$. In view of this and the relation \eqref{Hrt}, to show \eqref{c1} it suffices to show the same holds for $h_r$ replacing $H_r$, i.e.,
	\begin{align}\label{hrc1}
		h_r\big(s;(u_v)_{v\in \partial \Lambda_r}\big)h_r\big(t;(u_v')_{v\in \partial \Lambda_r}\big) \le h_r\big(s;(u_v')_{v\in \partial \Lambda_r}\big)h_r\big(t;(u_v)_{v\in \partial \Lambda_r}\big).
	\end{align}
	We shall prove \eqref{hrc1} via induction. Note that
	\begin{align*}
		h_1(x;(u_v)_{v\in \partial\{w_1\}})= \prod_{e=\{v_1\to v_2\}\in E_1\setminus \{e_1\}} W_e(u_{v_1}-u_{v_2})
	\end{align*}
	is the product of edge weights without any integration and with the convention $u_{w_1}=x$. Applying \eqref{pineq} to each such weight function yields \eqref{hrc1} for $r=1$.
	Observe the recursion relation for  $h_{r}$:
	\begin{align*}
		h_{r+1}\big(x;(u_v)_{v\in \partial\Lambda_{r+1}}\big)=d_r\big(x;(u_v)_{v\in \partial\Lambda_{r+1}}\big)\cdot\int_{\R} h_{r}\big(y;(u_v)_{v\in \partial\Lambda_{r}}\big)F_{r}(x-y)dy
	\end{align*}
	where by convention we set $u_{w_{r+1}}=x$ and where we define
	\begin{align*}
		d_r(x;(u_v)_{v\in \partial\Lambda_{r+1}})=\prod_{e=\{v_1\to v_2\}\in Q_r} W_e(u_{v_1}-u_{v_2})
	\end{align*}
	with $Q_r$ being the set of all red and black edges that have one vertex as $w_{r+1}$ and another vertex in $\partial \Lambda_{r+1}$, see Figure \ref{fig20} (B). Note that the blue edge $e_{r+1}$ between $w_{r+1}$ and $w_{r+1}+(0,1)$ is excluded from $Q_r$. Appealing to \eqref{pineq} again, we have
	\begin{align}\label{dr}
		d_r(s;(u_v)_{v\in \partial\Lambda_{r+1}})d_r(t;(u_v')_{v\in \partial\Lambda_{r+1}}) \le d_r(s;(u_v')_{v\in \partial\Lambda_{r+1}})d_r(t;(u_v)_{v\in \partial\Lambda_{r+1}})
	\end{align}
	for all $s\ge t$ and for all $u_v'\ge u_v$ with $v\in \partial\Lambda_{r+1}$. Under same conditions we claim that
	\begin{equation}
		\label{eqt}
		\begin{aligned}
			& \int_{\R^2} h_{r}(y;(u_v)_{v\in \partial\Lambda_{r}})F_{r}(s-y) h_{r}(x;(u_v')_{v\in \partial\Lambda_{r}})F_{r}(t-x)dx dy \\ & \hspace{3cm}\le \int_{\R^2} h_{r}(y;(u_v')_{v\in \partial\Lambda_{r}})F_{r}(s-y) h_{r}(x;(u_v)_{v\in \partial\Lambda_{r}})F_{r}(t-x)dx dy.
		\end{aligned}
	\end{equation}
	Combining the above inequality with \eqref{dr} we have \eqref{hrc1} completing the proof.	To see why \eqref{eqt} holds, we split the integrals in \eqref{eqt} over $\{x<y\}$ and $\{y<x\}$ and swap the $x,y$ labels in the region $\{y<x\}$ to get that \eqref{eqt} is equivalent to
	$$
	\int_{x<y} A(y)Y(y) B(x)X(x)+C(x)Z(x)D(y)W(y) \le\int_{x<y} D(y)Y(y)C(x)X(x)+B(x)Z(x)A(y)Z(y)
	$$
	where we let
	$$A(y)=h_r(y;(u_v)_{v\in \partial\Lambda_{r}}), B(x)=h_r(x;(u_v')_{v\in \partial\Lambda_{r}}), C(x)=h_r(x;(u_v)_{v\in \partial\Lambda_{r}}), D(y)=h_r(y;(u_v')_{v\in \partial\Lambda_{r}}),$$
	$$X(x)=F_{r}(t-x), Y(y)=F_{r}(s-y), W(y)=F_{r}(t-y), Z(x)=F_{r}(s-x).$$
	The integral above can be rewritten as $\int_{x<y}\big(A(y)B(x)-C(x)D(y)\big)\big(X(x)Y(y)-W(y)Z(x)\big)$ and thus it suffices to show for each $x\leq y$ the integrand is non-positive.
	By induction hypothesis, $A(y)B(x)\le C(x)D(y)$ for all $x\leq y$ and  since the weight function $F_r$ satisfies \eqref{pineq} (with $P=F_r$), we also have $X(x)Y(y)\ge W(y)Z(x)$. This proves \eqref{eqt}, completing the proof of the lemma.
\end{proof}

\section{Basic properties of log-gamma type random variables}

\label{app1}

In this section we collect some basic facts about log-gamma type random variables. Towards this end, for each $\theta,\kappa >0$, and $m\in \Z_{\ge 1}$ we consider the following function:
\begin{align*}
	H_{\theta,(-1)^m,\kappa}(y):=\frac{\kappa^{\theta}}{\Gamma(\theta)}\exp(\theta(-1)^m y-\kappa e^{(-1)^my}).
\end{align*}
It is plain to check $H$ is a valid probability density function. Observe that $H_{\theta,(-1)^m,1}\equiv G_{\theta,(-1)^m}$ where $G$ is defined in \eqref{def:gwt}. The following lemma collects some useful properties of $H$. Its proof follows via straightforward computations and is hence omitted.

\begin{lemma} Suppose $X \sim H_{\theta,1,\kappa}$. Then $-X \sim H_{\theta,-1,\kappa}$. For all $\alpha>-\theta$ we have $\Ex[e^{\alpha X}]= \frac{\Gamma(\alpha+\theta)}{\kappa^{\alpha}\Gamma(\theta)}$.
\end{lemma}

We next define generalized $\hslg$ $\Theta$-Gibbs measures in the same vein as $\hslg$ $\Theta$-Gibbs measures (see Definition \ref{def:hslggibbs}) but by considering the weight function
\begin{align*}
	\til{W}_{e}(x)=\begin{cases}
		\exp(\vartheta x-\kappa e^x) & \mbox{ if $e$ is \blue{\color{blue}($\vartheta$)}} \\
		\exp(-\gamma e^x) & \mbox{ if $e$ is \black} \\
		\exp(-\alpha x) & \mbox{ if $e$ is \red.}
	\end{cases}
\end{align*}
instead of $W$ defined in \eqref{def:wfn}. $\kappa=\gamma=1$ in above weights lead to the usual Gibbs measures. The following result ensures that generalized $\hslg$ $\Theta$-Gibbs measures (and hence the usual ones from Definition \ref{def:hslggibbs}) are well-defined.

\begin{lemma}\label{b2} Fix any $\gamma,\kappa>0,$ $\Theta:=\{\vartheta_{m,n}>0 : (m,n)\in \Z_{\ge 1}^2\}$ and $\alpha>-\min \{\vartheta_{m,n} : (m,n)\in \Z_{\ge 1}^2\}$. Recall the graph $G$ from Section \ref{sec:1.2} used in defining $\hslg$ $\Theta$-Gibbs measures. Given a domain $\Lambda$ and a boundary condition $\{u_{i,j} : (i,j)\in \partial\Lambda\}$, we have
	\begin{align*}
		\int_{\R^{|\Lambda|}}	\prod_{e=\{v_1\to v_2\}\in E(\Lambda\cup \partial\Lambda)} \til{W}_{e}(u_{v_1}-u_{v_2})\prod_{v\in \Lambda} du_v <\infty.
	\end{align*}
	Let us suppose $|u_{i,j}| \le R$ for all $(i,j)\in \partial\Lambda$. Let us assume $\Lambda= \mathcal{K}_{k,T}$ or $\mathcal{K}'_{k,T}$ defined in \eqref{def:kkt}. There exists a constant $\Con$ that depends only on $\gamma,\kappa,\theta,$ and $\alpha$ such that 	
	\begin{align*}
		\int_{\R^{|\Lambda|}}	\prod_{e=\{v_1\to v_2\}\in E(\Lambda\cup \partial\Lambda)} \til{W}_{e}(u_{v_1}-u_{v_2})\prod_{v\in \Lambda} du_v \le \Con^{kT+R}.
	\end{align*}
\end{lemma}

\begin{proof} We shall prove this lemma only for the homogeneous case, i.e.~$\vartheta_{m,n}\equiv \theta>0$. The general case is notationally more cumbersome but follows in an exact same manner as the homogeneous case. First note that for red edges $\{v_1\to v_2\}$ the corresponding weight function $W_e(u_{v_1}-u_{v_2})$ factors out as $e^{-\alpha u_{v_1}} \cdot e^{\alpha u_{v_2}}$. Hence they can be viewed as vertex weight functions. More specifically, at each vertex $(k,1)$ we can associate the vertex weight function $V_{k}(u):=e^{(-1)^{k}\alpha u}$. They replace the role of red edge weights. We denote this vertex weights as red circles in Figure \ref{fig:g1}. We now divide our analysis into two cases based on the value of $\alpha$.
	
	\medskip

	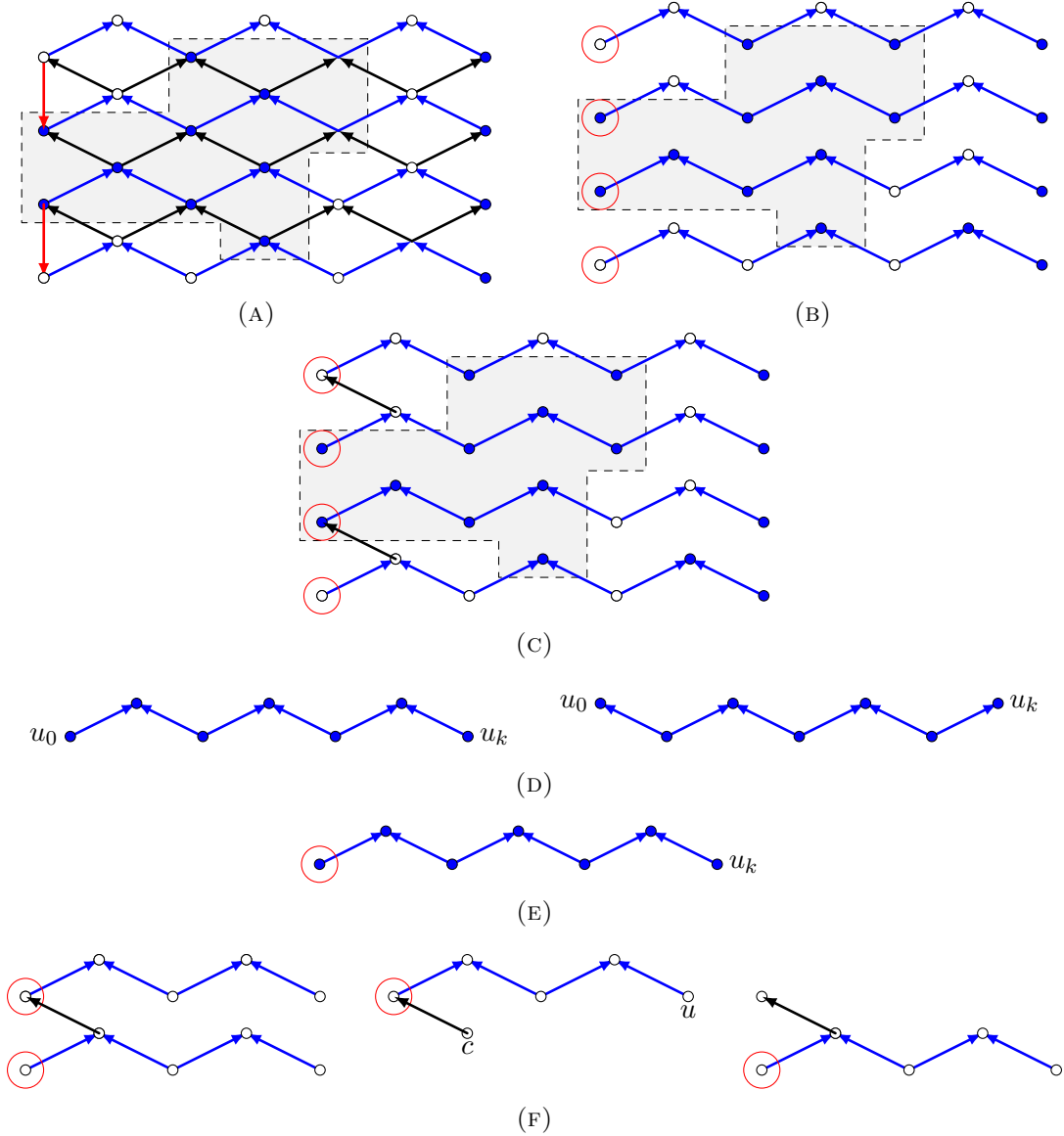
\begin{figure}[h!]
		\centering
		\begin{subfigure}[b]{0.45\textwidth}
			\centering
			\begin{tikzpicture}[line cap=round,line join=round,>=triangle 45,x=2cm,y=1cm]
				\draw[fill=gray!10,dashed] (1.7,-0.25)--(0.35,-0.25)--(0.35,-1.25)--(-0.65,-1.25)--(-0.65,-2.75)--(0.7,-2.75)--(0.7,-3.25)--(1.3,-3.25)--(1.3,-1.8)--(1.7,-1.8)--(1.7,-0.25);
				\foreach \x in {0,1,2}
				{
					\draw[line width=1pt,blue,{Latex[length=2mm]}-]  (\x,0) -- (\x-0.5,-0.5);
					\draw[line width=1pt,blue,{Latex[length=2mm]}-] (\x,0) -- (\x+0.5,-0.5);
					\draw[line width=1pt,black,{Latex[length=2mm]}-] (\x-0.5,-0.5) -- (\x,-1);
					\draw[line width=1pt,black,{Latex[length=2mm]}-] (\x+0.5,-0.5) -- (\x,-1);
					\draw[line width=1pt,blue,{Latex[length=2mm]}-]  (\x,-1) -- (\x-0.5,-1.5);
					\draw[line width=1pt,blue,{Latex[length=2mm]}-] (\x,-1) -- (\x+0.5,-1.5);
					\draw[line width=1pt,black,{Latex[length=2mm]}-] (\x-0.5,-1.5) -- (\x,-2);
					\draw[line width=1pt,black,{Latex[length=2mm]}-] (\x+0.5,-1.5) -- (\x,-2);
					\draw[line width=1pt,blue,{Latex[length=2mm]}-]  (\x,-2) -- (\x-0.5,-2.5);
					\draw[line width=1pt,blue,{Latex[length=2mm]}-] (\x,-2) -- (\x+0.5,-2.5);
					\draw[line width=1pt,black,{Latex[length=2mm]}-] (\x-0.5,-2.5) -- (\x,-3);
					\draw[line width=1pt,black,{Latex[length=2mm]}-] (\x+0.5,-2.5) -- (\x,-3);
					\draw[line width=1pt,blue,{Latex[length=2mm]}-]  (\x,-3) -- (\x-0.5,-3.5);
					\draw[line width=1pt,blue,{Latex[length=2mm]}-] (\x,-3) -- (\x+0.5,-3.5);
				}
				\foreach \x in {0,1}
				{
					\draw [fill=blue] (\x,0) circle (2pt);
					\draw [fill=blue] (\x,-1) circle (2pt);
					\draw [fill=blue] (\x,-2) circle (2pt);
					\draw [fill=blue] (\x-0.5,-0.5) circle (2pt);
					\draw [fill=blue] (\x-0.5,-1.5) circle (2pt);
					\draw [fill=blue] (\x-0.5,-2.5) circle (2pt);
					\draw [fill=blue] (\x-0.5,-3.5) circle (2pt);
					\draw [fill=blue] (\x,-3) circle (2pt);
				}
				\draw[line width=1pt,red,{Latex[length=2mm]}-] (-0.5,-1.5) -- (-0.5,-0.5);
				\draw[line width=1pt,red,{Latex[length=2mm]}-] (-0.5,-3.5) -- (-0.5,-2.5);
				\draw [fill=blue] (2.5,-0.5) circle (2pt);
				\draw [fill=blue] (2.5,-1.5) circle (2pt);
				\draw [fill=blue] (2.5,-2.5) circle (2pt);
				\draw [fill=blue] (2.5,-3.5) circle (2pt);
				\draw [fill=white] (2,0) circle (2pt);
				\draw[fill=white] (2,-1) circle (2pt);
				\draw[fill=white] (2,-2) circle (2pt);
				\draw[fill=white] (1.5,-2.5) circle (2pt);
				\draw[fill=white] (1.5,-3.5) circle (2pt);
				\draw[fill=white] (0.5,-3.5) circle (2pt);
				\draw[fill=white] (-0.5,-3.5) circle (2pt);
				\draw[fill=white] (0,-3) circle (2pt);
				\draw[fill=white] (0,0) circle (2pt);
				\draw[fill=white] (0,-1) circle (2pt);
				\draw[fill=white] (-0.5,-0.5) circle (2pt);
				\draw[fill=white] (1,0) circle (2pt);
			\end{tikzpicture}
			\caption{}
		\end{subfigure}
		\begin{subfigure}[b]{0.45\textwidth}
			\centering
			\begin{tikzpicture}[line cap=round,line join=round,>=triangle 45,x=2cm,y=1cm]
				\draw[fill=gray!10,dashed] (1.7,-0.25)--(0.35,-0.25)--(0.35,-1.25)--(-0.65,-1.25)--(-0.65,-2.75)--(0.7,-2.75)--(0.7,-3.25)--(1.3,-3.25)--(1.3,-1.8)--(1.7,-1.8)--(1.7,-0.25);
				\foreach \x in {0,1,2}
				{
					\draw[line width=1pt,blue,{Latex[length=2mm]}-]  (\x,0) -- (\x-0.5,-0.5);
					\draw[line width=1pt,blue,{Latex[length=2mm]}-] (\x,0) -- (\x+0.5,-0.5);
					\draw[line width=1pt,blue,{Latex[length=2mm]}-]  (\x,-1) -- (\x-0.5,-1.5);
					\draw[line width=1pt,blue,{Latex[length=2mm]}-] (\x,-1) -- (\x+0.5,-1.5);
					\draw[line width=1pt,blue,{Latex[length=2mm]}-]  (\x,-2) -- (\x-0.5,-2.5);
					\draw[line width=1pt,blue,{Latex[length=2mm]}-] (\x,-2) -- (\x+0.5,-2.5);
					\draw[line width=1pt,blue,{Latex[length=2mm]}-]  (\x,-3) -- (\x-0.5,-3.5);
					\draw[line width=1pt,blue,{Latex[length=2mm]}-] (\x,-3) -- (\x+0.5,-3.5);
				}
				\foreach \x in {0,1,2}
				{
					\draw [fill=blue] (\x,0) circle (2pt);
					\draw [fill=blue] (\x,-1) circle (2pt);
					\draw [fill=blue] (\x,-2) circle (2pt);
					\draw [fill=blue] (\x-0.5,-0.5) circle (2pt);
					\draw [fill=blue] (\x-0.5,-1.5) circle (2pt);
					\draw [fill=blue] (\x-0.5,-2.5) circle (2pt);
					\draw [fill=blue] (\x-0.5,-3.5) circle (2pt);
					\draw [fill=blue] (\x,-3) circle (2pt);
				}
				\draw [fill=blue] (2.5,-0.5) circle (2pt);
				\draw [fill=blue] (2.5,-1.5) circle (2pt);
				\draw [fill=blue] (2.5,-2.5) circle (2pt);
				\draw [fill=blue] (2.5,-3.5) circle (2pt);
				\draw [fill=white] (2,0) circle (2pt);
				\draw[fill=white] (2,-1) circle (2pt);
				\draw[fill=white] (2,-2) circle (2pt);
				\draw[fill=white] (1.5,-2.5) circle (2pt);
				\draw[fill=white] (1.5,-3.5) circle (2pt);
				\draw[fill=white] (0.5,-3.5) circle (2pt);
				\draw[fill=white] (-0.5,-3.5) circle (2pt);
				\draw[fill=white] (0,-3) circle (2pt);
				\draw[fill=white] (0,0) circle (2pt);
				\draw[fill=white] (0,-1) circle (2pt);
				\draw[fill=white] (-0.5,-0.5) circle (2pt);
				\draw[fill=white] (1,0) circle (2pt);
				\draw[red] (-0.5,-0.5) circle (7pt);
				\draw[red] (-0.5,-1.5) circle (7pt);
				\draw[red] (-0.5,-2.5) circle (7pt);
				\draw[red] (-0.5,-3.5) circle (7pt);
			\end{tikzpicture}
			\caption{}
		\end{subfigure}
		
		\begin{subfigure}[b]{0.45\textwidth}
			\centering
			\begin{tikzpicture}[line cap=round,line join=round,>=triangle 45,x=2cm,y=1cm]
				\draw[fill=gray!10,dashed] (1.7,-0.25)--(0.35,-0.25)--(0.35,-1.25)--(-0.65,-1.25)--(-0.65,-2.75)--(0.7,-2.75)--(0.7,-3.25)--(1.3,-3.25)--(1.3,-1.8)--(1.7,-1.8)--(1.7,-0.25);
				\foreach \x in {0,1,2}
				{
					\draw[line width=1pt,blue,{Latex[length=2mm]}-]  (\x,0) -- (\x-0.5,-0.5);
					\draw[line width=1pt,blue,{Latex[length=2mm]}-] (\x,0) -- (\x+0.5,-0.5);
					\draw[line width=1pt,blue,{Latex[length=2mm]}-]  (\x,-1) -- (\x-0.5,-1.5);
					\draw[line width=1pt,blue,{Latex[length=2mm]}-] (\x,-1) -- (\x+0.5,-1.5);
					\draw[line width=1pt,blue,{Latex[length=2mm]}-]  (\x,-2) -- (\x-0.5,-2.5);
					\draw[line width=1pt,blue,{Latex[length=2mm]}-] (\x,-2) -- (\x+0.5,-2.5);
					\draw[line width=1pt,blue,{Latex[length=2mm]}-]  (\x,-3) -- (\x-0.5,-3.5);
					\draw[line width=1pt,blue,{Latex[length=2mm]}-] (\x,-3) -- (\x+0.5,-3.5);
				}
				\foreach \x in {0,1,2}
				{
					\draw [fill=blue] (\x,0) circle (2pt);
					\draw [fill=blue] (\x,-1) circle (2pt);
					\draw [fill=blue] (\x,-2) circle (2pt);
					\draw [fill=blue] (\x-0.5,-0.5) circle (2pt);
					\draw [fill=blue] (\x-0.5,-1.5) circle (2pt);
					\draw [fill=blue] (\x-0.5,-2.5) circle (2pt);
					\draw [fill=blue] (\x-0.5,-3.5) circle (2pt);
					\draw [fill=blue] (\x,-3) circle (2pt);
				}
				\draw [fill=blue] (2.5,-0.5) circle (2pt);
				\draw [fill=blue] (2.5,-1.5) circle (2pt);
				\draw [fill=blue] (2.5,-2.5) circle (2pt);
				\draw [fill=blue] (2.5,-3.5) circle (2pt);
				\draw [fill=white] (2,0) circle (2pt);
				\draw[fill=white] (2,-1) circle (2pt);
				\draw[fill=white] (2,-2) circle (2pt);
				\draw[fill=white] (1.5,-2.5) circle (2pt);
				\draw[fill=white] (1.5,-3.5) circle (2pt);
				\draw[fill=white] (0.5,-3.5) circle (2pt);
				\draw[fill=white] (-0.5,-3.5) circle (2pt);
				\draw[fill=white] (0,-3) circle (2pt);
				\draw[fill=white] (0,0) circle (2pt);
				\draw[fill=white] (0,-1) circle (2pt);
				\draw[fill=white] (-0.5,-0.5) circle (2pt);
				\draw[fill=white] (1,0) circle (2pt);
				\draw[red] (-0.5,-0.5) circle (7pt);
				\draw[red] (-0.5,-1.5) circle (7pt);
				\draw[red] (-0.5,-2.5) circle (7pt);
				\draw[red] (-0.5,-3.5) circle (7pt);
				\draw[line width=1pt,black,{Latex[length=2mm]}-] (-0.5,-0.5) -- (0,-1);
				\draw[line width=1pt,black,{Latex[length=2mm]}-] (-0.5,-2.5) -- (0,-3);
			\end{tikzpicture}
			\caption{}
		\end{subfigure}

		\vspace{0.3cm}

		\begin{subfigure}[b]{0.9\textwidth}
			\centering
			\begin{tikzpicture}[line cap=round,line join=round,>=triangle 45,x=1.8cm,y=0.9cm]
				\foreach \x in {0,1,2}
				{
					\draw[line width=1pt,blue,{Latex[length=2mm]}-]  (\x,0) -- (\x-0.5,-0.5);
					\draw[line width=1pt,blue,{Latex[length=2mm]}-] (\x,0) -- (\x+0.5,-0.5);
					\draw [fill=blue] (\x,0) circle (2pt);
					\draw [fill=blue] (\x-0.5,-0.5) circle (2pt);
				}
				\draw [fill=blue] (2.5,-0.5) circle (2pt);
				\node at (-0.7,-0.5) {$u_0$};
				\node at (2.7,-0.5) {$u_k$};	
				\foreach \x in {4,5,6}
				{
					\draw[line width=1pt,blue,{Latex[length=2mm]}-]  (\x-0.5,0)--(\x,-0.5);
					\draw[line width=1pt,blue,{Latex[length=2mm]}-] (\x+0.5,0)--(\x,-0.5);
					\draw [fill=blue] (\x,-0.5) circle (2pt);
					\draw [fill=blue] (\x-0.5,0) circle (2pt);
				}
				\draw [fill=blue] (6.5,0) circle (2pt);
				\node at (3.3,0) {$u_0$};
				\node at (6.7,0) {$u_k$};		
			\end{tikzpicture}
			\caption{}
		\end{subfigure}
		
		\vspace{0.3cm}

		\begin{subfigure}[b]{0.45\textwidth}
			\centering
			\begin{tikzpicture}[line cap=round,line join=round,>=triangle 45,x=1.8cm,y=0.9cm]
				\foreach \x in {0,1,2}
				{
					\draw[line width=1pt,blue,{Latex[length=2mm]}-]  (\x,0) -- (\x-0.5,-0.5);
					\draw[line width=1pt,blue,{Latex[length=2mm]}-] (\x,0) -- (\x+0.5,-0.5);
					\draw [fill=blue] (\x,0) circle (2pt);
					\draw [fill=blue] (\x-0.5,-0.5) circle (2pt);
				}
				\draw [fill=blue] (2.5,-0.5) circle (2pt);
				\node at (2.7,-0.5) {$u_k$};	
				\draw[red] (-0.5,-0.5) circle (7pt);
			\end{tikzpicture}
			\caption{}
		\end{subfigure}
		
		\vspace{0.3cm}
		
		\begin{subfigure}[b]{0.9\textwidth}
			\centering
			\begin{tikzpicture}[line cap=round,line join=round,>=triangle 45,x=2cm,y=1cm]
				\foreach \x in {0,1}
				{
					\draw[line width=1pt,blue,{Latex[length=2mm]}-]  (\x,0) -- (\x-0.5,-0.5);
					\draw[line width=1pt,blue,{Latex[length=2mm]}-] (\x,0) -- (\x+0.5,-0.5);
					\draw[line width=1pt,blue,{Latex[length=2mm]}-]  (\x,-1) -- (\x-0.5,-1.5);
					\draw[line width=1pt,blue,{Latex[length=2mm]}-] (\x,-1) -- (\x+0.5,-1.5);
				}
				\foreach \x in {0,1}
				{
					\draw [fill=white] (\x,0) circle (2pt);
					\draw [fill=white] (\x,-1) circle (2pt);
					\draw [fill=white] (\x-0.5,-0.5) circle (2pt);
					\draw [fill=white] (\x-0.5,-1.5) circle (2pt);
				}
				\draw [fill=white] (1.5,-0.5) circle (2pt);
				\draw[fill=white] (1.5,-1.5) circle (2pt);
				\draw[fill=white] (0,0) circle (2pt);
				\draw[fill=white] (0,-1) circle (2pt);
				\draw[fill=white] (-0.5,-0.5) circle (2pt);
				\draw[fill=white] (1,0) circle (2pt);
				\draw[red] (-0.5,-0.5) circle (7pt);
				\draw[red] (-0.5,-1.5) circle (7pt);
				\draw[line width=1pt,black,{Latex[length=2mm]}-] (-0.5,-0.5) -- (0,-1);
				
				\foreach \x in {2.5,3.5}
				{
					\draw[line width=1pt,blue,{Latex[length=2mm]}-]  (\x,0) -- (\x-0.5,-0.5);
					\draw[line width=1pt,blue,{Latex[length=2mm]}-] (\x,0) -- (\x+0.5,-0.5);
				}
				\foreach \x in {2.5,3.5}
				{
					\draw [fill=white] (\x,0) circle (2pt);
					\draw [fill=white] (\x-0.5,-0.5) circle (2pt);
					
				}
				\draw [fill=white] (4,-0.5) circle (2pt);
				\draw[fill=white] (2.5,-1) circle (2pt);
				\draw[red] (2,-0.5) circle (7pt);
				\draw[line width=1pt,black,{Latex[length=2mm]}-] (2,-0.5) -- (2.5,-1);
				
				\foreach \x in {5,6}
				{
					\draw[line width=1pt,blue,{Latex[length=2mm]}-]  (\x,-1) -- (\x-0.5,-1.5);
					\draw[line width=1pt,blue,{Latex[length=2mm]}-] (\x,-1) -- (\x+0.5,-1.5);
				}
				\foreach \x in {5,6}
				{
					\draw [fill=white] (\x,-1) circle (2pt);
					
					\draw [fill=white] (\x-0.5,-1.5) circle (2pt);
				}
				\draw[fill=white] (6.5,-1.5) circle (2pt);
				\draw [fill=white] (5-0.5,-0.5) circle (2pt);
				\draw[fill=white] (5,-1) circle (2pt);
				\draw[red] (4.5,-1.5) circle (7pt);
				\draw[line width=1pt,black,{Latex[length=2mm]}-] (4.5,-0.5) -- (5,-1);
				\node at (2.5,-1.2)  {$c$};
				\node at (4,-0.7) {$u$};
			\end{tikzpicture}
			\caption{}
		\end{subfigure}
		
		\caption{(A) A possible domain $\Lambda$. (B) Reduction in the case of $\alpha\in (-\theta,\theta)$. (C) Reduction in the case of $\alpha>0$. (D) Type I Gibbs measures. The figure shows two of them of even length. It may also have odd length with one edge at either of the end removed. (E) Type II Gibbs measures. It may also have odd length with one edge at right end removed. (F) Few examples of Type III Gibbs measures.}
		\label{fig:g1}
	\end{figure}

	Suppose $\alpha \in (-\theta,\theta)$. As black edge weights are less than $1$, we may drop all of them to get a Gibbs measure based on the blue and red edge weights only, see Figure \ref{fig:g1} (B). The integral of the reduced Gibbs measure can be viewed as a product of integrals of several smaller Gibbs measures that are two types: Type I and Type II as in Figure \ref{fig:g1} (D) and (E) respectively. Type I Gibbs measures are those for which red vertex weights do not appear. The integral corresponding to Type I takes the following form:
	\begin{align*}
		\left(\kappa^{\theta}(\Gamma(\theta))^{-1}\right)^k \int_{\R^{k-1}}\prod_{i=1}^k H_{\theta,\kappa,(-1)^{i+m}}(u_{i-1}-u_{i})\prod_{i=1}^{k-1}du_i
	\end{align*}
	where $u_0$ and $u_k$ are in $\partial\Lambda$. In this case, we may use $H_{\theta,\kappa,(-1)^{i+m}}(u_{k-1}-u_{k}) \le \Con$ and the fact that $H$ is a probability density function to get that the integral is bounded by $\Con \cdot \left(\kappa^{\theta}(\Gamma(\theta))^{-1}\right)^k$. Type II Gibbs measures are the ones where red vertex weights are present. The integral corresponding to the Type II Gibbs measures takes the form
	\begin{align*}
		\int_{\R^{k}}\prod_{i=1}^k e^{(-1)^{m}\alpha u_0}\cdot e^{(-1)^{i+m}\theta(u_{i-1}-u_i)-\kappa e^{(-1)^{i+m} (u_{i-1}-u_i)}}\prod_{i=1}^{k}du_i.
	\end{align*}
	The integrand can be manipulated to show that the above integral is equal to
	\begin{align*}
		& e^{(-1)^{m+k-1}\alpha u_k}\prod_{i=1}^{k-1}(\Gamma(\theta+(-1)^{m+i+1}\alpha))\kappa^{-\theta+(-1)^{m+i}\alpha} \int_{\R^{k}}\prod_{i=1}^k  H_{\theta+(-1)^{m+i+1}\alpha,\kappa,(-1)^{i-1}}(x_i)\prod_{i=1}^{k}dx_i \\ & = e^{(-1)^{m+k-1}\alpha u_k}\prod_{i=1}^{k-1}(\Gamma(\theta+(-1)^{m+i+1}\alpha))\kappa^{-\theta+(-1)^{m+i}\alpha}.
	\end{align*}
	Since we factor out the measure into these independent pieces and all integrals are finite, the claim follows for $\alpha\in (-\theta,\theta)$.
	
	\medskip
	
	For $\alpha>0$, we remove all the black edges except the ones connecting $(2i-1,1)$ to $(2i,1)$ (again since weights of black edges are atmost $1$, removal of them only increases the integrand). This leads to a reduced Gibbs measures shown in Figure \ref{fig:g1} (C). The reduced Gibbs measure decomposes into several Type I Gibbs measures and Type III Gibbs measures. Type III Gibbs measures are those for which red vertex weights do appear. Because of the presence of black edge in this case, Type III Gibbs measures are different from Type II. A few of the possible Type III Gibbs measures are shown in Figure \ref{fig:g1} (F).
	
	\begin{itemize}[leftmargin=18pt]
		\item If a Type III Gibbs measure has two red vertices in its domain or boundary, we may use the fact that the weight of the figure
		
		\begin{center}
			\begin{tikzpicture}[line cap=round,line join=round,>=triangle 45,x=2cm,y=1cm]
				\draw[line width=1pt,blue,{Latex[length=2mm]}-]  (0,-1) -- (-0.5,-1.5);
				\draw [fill=white] (-0.5,-1.5) circle (2pt);
				\draw[fill=white] (0,-1) circle (2pt);
				\draw[fill=white] (-0.5,-0.5) circle (2pt);
				\draw[red] (-0.5,-0.5) circle (7pt);
				\draw[red] (-0.5,-1.5) circle (7pt);
				\draw[line width=1pt,black,{Latex[length=2mm]}-] (-0.5,-0.5) -- (0,-1);
				\node at (-0.7,-0.5) {$a$};
				\node at (-0.7,-1.5) {$b$};
				\node at (0.2,-1) {$c$};
			\end{tikzpicture}
		\end{center}	
		is $e^{(\theta+\alpha)(b-c)-\kappa e^{b-c}}\cdot e^{\alpha (c-a)-\gamma e^{c-a}} \le e^{(\theta+\alpha)(b-c)-\kappa e^{b-c}}\cdot \left(\sup_{x\in \R} e^{\alpha x-\gamma e^x}\right) \le \Con \cdot e^{(\theta+\alpha)(b-c)-\kappa e^{b-c}}$.
		
		\item  If a Type III Gibbs measure has only one red vertex in its domain or boundary, then it must contain either of the two following figures

		\begin{center}
			\begin{tikzpicture}[line cap=round,line join=round,>=triangle 45,x=2cm,y=1cm]
				\draw[line width=1pt,blue,{Latex[length=2mm]}-]  (0,-1) -- (-0.5,-1.5);
				\draw [fill=white] (-0.5,-1.5) circle (2pt);
				\draw[fill=white] (0,-1) circle (2pt);
				\draw[red] (-0.5,-1.5) circle (7pt);
				\node at (-0.7,-1.5) {$b$};
				\node at (0.2,-1) {$c$};
				\draw[fill=white] (2,-1.5) circle (2pt);
				\draw[fill=white] (1.5,-1) circle (2pt);
				\draw[red] (1.5,-1) circle (7pt);
				\draw[line width=1pt,black,{Latex[length=2mm]}-] (1.5,-1) -- (2,-1.5);
				\node at (1.3,-1) {$a$};
				\node at (2.2,-1.5) {$c$};
			\end{tikzpicture}
		\end{center}
		with $c\in \partial\Lambda$. The corresponding weights are $e^{\alpha c} \cdot e^{(\theta +\alpha)(b-c)-\kappa e^{b-c}} \le \Con' e^{\alpha c}$ and $e^{-\alpha c} \cdot e^{\alpha(c-a)-\gamma e^{c-a}} \le \Con' e^{-\alpha c}$ respectively where $\Con':=\sup_{x\in \R} e^{\alpha x-\gamma e^x}$.
	\end{itemize}
	
	Based on the kind of Type III Gibbs measures, we may insert the above bounds on the Gibbs weights in the integrand of this type of Gibbs measures. The resulting integral can then be computed explicitly to yield a bound of the form $\Con^{V} e^{|\alpha c|}$ where $V$ is the number of vertices in the Gibbs measures. For example, for the middle figure in Figure \ref{fig:g1} (F) we have (with $u_4:=u$)
	\begin{align*}
		& \left(\kappa^{-\theta}\Gamma(\theta)\right)^4\int_{\R^4} e^{-\alpha u_0} e^{-\gamma e^{c-u_0}}\prod_{i=0}^3 H_{\theta,\kappa,(-1)^{i}}(u_{i}-u_{i+1})du_i
		\\ & \le \left(\kappa^{-\theta}\Gamma(\theta)\right)^4\cdot \Con e^{-\alpha c}\int_{\R^4}\prod_{i=0}^3 H_{\theta,\kappa,(-1)^{i}}(u_{i}-u_{i+1})du_i  \le \left(\kappa^{-\theta}\Gamma(\theta)\right)^4\cdot \Con e^{|\alpha c|}.
	\end{align*}

	This establishes the lemma for $\alpha>0$.				
\end{proof}

We end this section we two lemmas concerning with the tail properties of $\fa$ and $\qo$-distributions defined in \eqref{def:faga} and \eqref{qdist} respectively.
\begin{lemma}\label{tailf} For all $x\in \R$ we have $$e^{-2e} \le (\Gamma(\theta))^2\fa(x)e^{\theta|x|} \le \Gamma(2\theta),$$
	where $\fa$ is defined in \eqref{def:faga}.
\end{lemma}

\begin{proof} Since $\fa$ is symmetric, it suffices to show the lemma for $x>0$. We have
	\begin{align*}
		\big(\Gamma(\theta)\big)^2\fa(x)=\int_{\R} e^{\theta y-e^{y}+\theta(y-x)-e^{y-x}}dy=e^{-\theta x}\int_{\R} e^{2\theta y-e^{y}-e^{y-x}}dy.
	\end{align*}
	Now for the lower bound we observe
	\begin{align*}
		\int_{\R} e^{2\theta y-e^{y}-e^{y-x}}dy \ge \int_{0}^1 e^{2\theta y-e^{y}-e^{y-x}}dy \ge e^{-2e},
	\end{align*}
	whereas for the upper bound we have
	\begin{align*}
		\int_{\R} e^{2\theta y-e^{y}-e^{y-x}}dy \le \int_{\R} e^{2\theta y-e^{y}}dy = \Gamma(2\theta).
	\end{align*}
\end{proof}

\begin{lemma}\label{qlemma} Fix any $\theta_0>1$. For any $\theta_1,\theta_2 \in [\theta_0^{-1},\theta_0]$ and $a,b\in \R$, define the random variable $X_{\theta_1,\theta_2;\pm 1}^{(a,b)}\sim \qo_{\theta_1,\theta_2;\pm 1}^{(a,b)}$ where $ \qo_{\theta_1,\theta_2;\pm 1}^{(a,b)}$ is defined in \eqref{qdist}. There exists a constant $\Con>0$ depending only $\theta_0$ such that for all $\theta_1,\theta_2 \in [\theta_0^{-1},\theta_0]$, for all $a,b\in \R$ and for all  $r\ge |a-b|$ we have
	\begin{align*}
		\Pr\Big(X_{\theta_1,\theta_2;\pm 1}^{(a,b)} \notin \big[\min\{a,b\}-2r,\max\{a,b\}+2r]\Big) \le \Con e^{-\frac1\Con r}.
	\end{align*}
\end{lemma}
\begin{proof} Fix any $\theta_1,\theta_2 \in [\theta_0^{-1},\theta_0]$. We shall prove the bound for $X_{\theta_1,\theta_2;1}^{(a,b)}$. The proof for the case $X_{\theta_1,\theta_2;-1}^{(a,b)}$ is analogous.  Without loss of generality assume $b\le a$. Observe that
\begin{equation}\label{twre}
\Pr\Big(X_{\theta_1,\theta_2;\pm 1}^{(a,b)} \notin \big[\min\{a,b\}-2r,\max\{a,b\}+2r\big]\Big)\le \frac{\int_{(-\infty,b-2r]\cup [a+2r,\infty)} G_{\theta_1,1}(a-x)G_{\theta_2,1}(b-x)}{\int_{a}^{a+1} G_{\theta_1,1}(a-x)G_{\theta_2,1}(b-x)}.
\end{equation}
Note that
	\begin{align*}
		\int_{a}^{a+1} G_{\theta_1,1}(a-x)G_{\theta_2,1}(b-x) \ge \tfrac1\Con \cdot e^{-\max\{\theta_1,\theta_2\} \cdot (a-b)},
	\end{align*}
	where we have used the fact that $G_{\beta,1}(-y)\ge \Con^{-1} e^{-\beta y}$ (recall $G$ from \eqref{def:gwt}). Similarly
	\begin{align*}
		\int_{x\le b-2r} G_{\theta_1,1}(a-x)G_{\theta_2,1}(b-x)\d x+	\int_{x\ge a+2r} G_{\theta_1,1}(a-x)G_{\theta_2,1}(b-x)\d x \le \Con \cdot e^{-2(\theta_1+\theta_2)r}.
	\end{align*}
	Thus as long as $r\ge a-b$, inserting the above two bounds back in \eqref{twre} and adjusting the constant $\Con$ we get the desired result.
\end{proof}

\section{Estimates for non-intersection probability}

\label{app2}

In this section, we study non-intersection probability of random walks and random bridges (defined in Definition \ref{def:rws}), and modified random bridges (defined in Definition \ref{def:mrb}). Throughout this section we shall assume the increments are drawn from a density $\ffa$ that satisfies the following assumptions. It is worth recalling that due to Corollary \ref{rm:con}, $\fa$ defined in \eqref{def:faga} satisfies the conditions of Assumption \ref{asp:in} and hence all results of this section can be applied to random walks with that increment law.

	\begin{assumption}[Assumption on the increments] \label{asp:in} The density $\ffa$ satisfies the following properties.
	\begin{enumerate}
		\item The density $\ffa$ is symmetric and $\log \ffa$ is concave.
		
		\item Let $\psi$ denote the characteristic function corresponding to $\ffa$. $|\psi|$ is integrable. Given any $\delta>0$, there exists $\eta$ such that $\sup_{t\ge \delta} |\psi(t)|=\eta <1$.
		
		\item There exists a constant $\Con>0$ such that $\ffa(x)\le \Con e^{-|x|/\Con}$. In particular, this implies that if $X \sim \ffa$, there exists $v>0$ such that and
		\begin{align*}		\sup_{|t|\le v}\big[\Ex[e^{tX}]\big] <\infty.
		\end{align*}
		In other words $X$ is an subexponential random variable.
	\end{enumerate}
\end{assumption}

The following lemma concerns with sharp rate of convergence of the probability density function of $(X(1)+X(2)+\cdots+X(n))/\sqrt{n}$, where $X(i)\stackrel{i.i.d.}{\sim} \ffa$, to the Gaussian density with appropriate variance.

\begin{lemma}\label{sharp} Let $\ffa^{*n}$ be the $n$-fold convolution of $\ffa$. There exists a constant $\Con>0$ such that
	\begin{align*}
		\sup_{|x|\le (\log n)^2} \left|\frac{\sqrt{n}\ffa^{*n}(x\sqrt{n})}{\phi_{\sigma}(x)}-1\right|\le  \Con \cdot n^{-3/4}.
	\end{align*}
	where $\phi_{\sigma}(x):=\frac1{\sqrt{2\pi \sigma^2}}e^{-\frac{x^2}{2\sigma^2}}$ and $\sigma^2:=\int x^2\ffa(x)dx$.
\end{lemma}
\begin{proof} This proof is adapted from Theorem 5 in Chapter XV in \cite{feller}. In what follows we shall use the big $O$-notation and write $a_n=O(b_n)$ if $a_n/b_n$ is uniformly bounded above by some universal constant. Let $\psi$ denote the characteristic function of $\fa$. The explicit form of $\psi$ was given in \eqref{eq:chint}. In particular, $|\psi|$ is integrable. In what follows, for simplicity we will assume $\sigma^2=1$. 	Set $f_n(x):=\int_{\R} e^{i tx}\psi^n(t/\sqrt{n})dt$. By the Fourier inversion formula, $f_n(x)$ is the density of $(X(1)+X(2)+\cdots+X(n))/\sqrt{n}$ where $X(i)\stackrel{i.i.d.}{\sim} \ffa$. Hence we have $\sqrt{n}\ffa^{*n}(x/\sqrt{n})=f_n(x)$.  Since $\ffa$ is symmetric and has all finite moments, by Taylor expansion we have
	\begin{align*}
		\psi(t/\sqrt{n})= 1-\tfrac{t^2}{2n}+O(\tfrac{t^4}{n^2}).
	\end{align*}
	Set $\alpha=1/16$. Thus for $|t|\le n^{\alpha}$, we have $\psi(t/\sqrt{n})=1-\frac{t^2}{2n}+O(n^{4\alpha-2})=e^{-t^2/2n+O(n^{4\alpha-2})}$. Thus $\psi^n(t/\sqrt{n})=e^{-t^2/2}(1+O(n^{-3/4}))$, where the $O$ term is free of $t$ in that specified range. Thus,
	\begin{align*}
		f_n(x) & = (1+O(n^{-3/4}))\int_{|t|\le n^{\alpha}} e^{i tx}e^{-t^2/2}dt+\int_{|t|\ge n^{\alpha}} e^{i tx}\psi^n(t/\sqrt{n})dt \\ & = (1+O(n^{-3/4}))\int_{\R} e^{i tx}e^{-t^2/2}dt+\int_{|t|\ge n^{\alpha}} e^{i tx}\psi^n(t/\sqrt{n})dt- (1+O(n^{-3/4}))\int_{|t|\ge n^{\alpha}} e^{i tx}e^{-t^2/2}dt.
	\end{align*}
	We next compute the order of the last two integrals above. Clearly $\int_{|t|\ge n^{\alpha}} e^{-t^2/2} dt  \le \Con e^{-c n^{2\alpha}}$. For the second one, we  choose $\delta>0$ small enough such that $|\psi(t)|\le e^{-t^2/4}$ for all $|t|\le \delta$. This implies
	\begin{align*}
		\int_{n^{\alpha}\le t\le \sqrt{n}\delta} |\psi^n(t/\sqrt{n})|dt  \le \Con e^{-c n^{2\alpha}}.
	\end{align*}
	For $|t|\ge \sqrt{n}\delta$, we know $\sup_{t\ge \delta}|\psi(t)|=\eta < 1$ by part~(2) of Lemma \ref{as:el}. This forces
	\begin{align*}
		\int_{|t|\ge \sqrt{n}\delta} |\psi^n(t/\sqrt{n})| dt  \le \eta^{n-1}\sqrt{n}\int_{\R} \psi(t) dt.
	\end{align*}
	Thus the error integrals are at most $\Con e^{-\frac1\Con n^{1/8}}$ in absolute value uniform in $x$. Furthermore if we assume $|x|\le (\log n)^2$, $\phi_1(x)\ge \frac1{\sqrt{2\pi}}e^{-(\log n)^2/2}$, which dominates the error coming from the integrals. Hence we may divide $\phi_1(x)$ and still obtain that the errors are going to zero.
\end{proof}

For any $p\in [0,\infty)$, $s,t\in \ll1,n\rr$, we set
\begin{align*}
	\ni_p\ll s,t\rr:=\{\se{1}{k}-\se{2}{k}\ge-p, \mbox{ for all }k\in \ll s,t\rr\}.
\end{align*}
When $s=2, t=n-1$ we write $\ni_p:=\ni_p\ll2,n-1\rr$ so that it coincides with the $\ni_p$ event defined in \eqref{defnipab}. When $a_1-a_2=O(1)$, it is well known that $\pr{n}{(a_1,a_2)}{}(\ni)=O(n^{-1/2})$. We record this classical fact in the following lemma.

\begin{lemma}\label{l:class} For all $(a_1,a_2)\in \R^2$ we have $\pr{n}{(a_1,a_2)}{}(\ni) \le \Con\frac{\max\{a_1-a_2,1\}}{\sqrt{n}}$ for some absolute constant $\Con>0$. If in addition $a_1\ge a_2$, we have $\pr{n}{(a_1,a_2)}{}(\ni) \ge \frac{\Con^{-1}}{\sqrt{n}}$.
\end{lemma}
\begin{proof}
	The first part is  \cite[Theorem A]{kozlov} and the second part is \cite[Theorem 3.5]{spit}.
\end{proof}

We again remind the readers that the above result, as well as all the results stated below within this section, the random walks/bridges or the modified random bridges are assumed to have increments drawn from a density $\ffa$ satisfying Assumption \ref{asp:in}. In many of our arguments below, we shall often appeal to stochastic monotonicity of non-intersecting random walks or bridges  with respect to boundary data. We record this result below.	

\begin{proposition}[Stochastic monotonicity of random bridges and random walks] \label{smrw}
Fix $n\in \Z_{\ge 1}$.
et $a_i^{(j)},b_i^{(j)} \in [-\infty,\infty]$ for $i,j\in \{1,2\}$. Suppose $a_i^{(1)} \ge a_i^{(2)}$ and $b_i^{(1)} \ge b_i^{(2)}$ for $i\in \{1,2\}$.
		\begin{enumerate}[label=(\alph*), leftmargin=18pt]
			\item There exists a probability space that supports a collection of random variables
			\begin{align*}
				\big(S_1^{(j)}(k),S_2^{(j)}(k) : j\in \{1,2\}, k\in \ll1,n\rr \big)
			\end{align*}
			such that $S_i^{(1)}(k) \ge S_i^{(2)}(k)$ for all $i\in \{1,2\}$ and $k\in \ll1,n\rr$, and marginally $(S_1^{(j)}(\cdot),S_2^{(j)}(\cdot)) \sim \pr{n}{(a_1^{(j)},a_2^{(j)})}{,(b_1^{(j)},b_2^{(j)})}(\cdot \mid \ni)$ for each $j\in \{1,2\}$.
			\item There exists a probability space that supports a collection of random variables
			\begin{align*}
				\big(S_1^{(j)}(k),S_2^{(j)}(k) : j\in \{1,2\}, k\in \ll1,n\rr \big)
			\end{align*}
			such that $S_i^{(1)}(k) \ge S_i^{(2)}(k)$ for all $i\in \{1,2\}$ and $k\in \ll1,n\rr$, and marginally $(S_1^{(j)}(\cdot),S_2^{(j)}(\cdot)) \sim \pr{n}{(a_1^{(j)},a_2^{(j)})}{}(\cdot \mid \ni)$ for each $j\in \{1,2\}$.
		\end{enumerate}

	\end{proposition}
\begin{proof}
This proposition of as a discrete analogue of Lemma 2.6 in \cite{ch14} and is true under log-concavity assumption on $\ffa$. Instead of giving the full details, we explain the two possible ways in proving this proposition. One is via Markov Chain arguments as done, for example, in \cite{ch14,wu,serio} previously. In \cite{serio} for example, the stochastic monotonicity was proved for non-intersecting random  bridges under discrete bounded increment assumption. One can take a discrete to continuous limit of the increments to obtain the above result. The second route is via direct construction argument is the style of \cite{bcd,xd1}. We have, in fact, already adapted that technique in proving stochastic monotonicity for $\hslg$ Gibbs measures in Appendix \ref{appc} and a similar argument can be carried out to prove Proposition \ref{smrw}.
\end{proof}
	
	\medskip
	
Next we study diffusive properties of the random walks under the non-intersecting event.
\begin{lemma}\label{tailni} Given any $\e>0$ there exists a constant $\delta(\e)>0$ such that for all $n\in \Z_{\ge 1}$ and $(a_1,a_2)\in \R^2$  we have
	\begin{align} \label{b21}
		& \pr{n}{(a_1,a_2)}{}\bigg( \se{1}{n}-\se{2}{n} \ge \delta\sqrt{n} \mid \ni\bigg) \ge 1-\e, \\ & \label{b22} \pr{n}{(a_1,a_2)}{}\bigg( \sup_{k\in \ll1,n\rr} (\se{1}{k}-\se{2}{k}) \le \delta^{-1}\sqrt{n}+\max\{a_1-a_2,0\} \mid \ni\bigg) \ge 1-\e, \\
		& \pr{n}{(a_1,a_2)}{}\bigg( \inf_{k\in \ll1,n\rr} \se{1}{k}-a_1 \ge -\delta^{-1}\sqrt{n}\mid \ni\bigg) \ge 1-\e, \label{b23} \\
		& \pr{n}{(a_1,a_2)}{}\bigg( \sup_{k\in \ll1,n\rr} \se{2}{k}-a_2 \le \delta^{-1}\sqrt{n} \mid \ni\bigg) \ge 1-\e. \label{b24}
	\end{align}
\end{lemma}
We remark that Lemma \ref{tailni} holds if $\ni=\ni_0\ll2,n-1\rr$ is replaced by $\ni_0\ll2,n\rr$. The same argument presented below essentially works when the conditional event is the latter one instead.
\begin{proof} \textbf{Proof of Eq.~\eqref{b21}.} Set $U(k):=\se{1}{k}-\se{2}{k}$. Under $\pr{n}{(a_1,a_2)}{}$, $(U(k))_{k=1}^n$ is a random walk starting from $a_1-a_2$ with increments drawn from $\ffa\ast\ffa$. The non-intersection condition for $(\se{1}{k},\se{2}{k})_{k=1}^n$ translates to $(U(k))_{k=1}^n$ staying non-negative. If $a_1\ge a_2$, since $\{U(n)\ge \delta\sqrt{n}\}$ is an increasing event  with respect to the boundary conditions, we have
	\begin{align*}
		\pr{n}{(a_1,a_2)}{}\bigg( \se{1}{n}-\se{2}{n} \ge \delta\sqrt{n} \mid \ni\bigg) \ge \pr{n}{(0,0)}{}\bigg( \se{1}{n}-\se{2}{n} \ge \delta\sqrt{n} \mid \ni\bigg)
	\end{align*}
		But under $\pr{n}{(0,0)}{}$, it is known from \cite{igl} that the random walk $(U(k))_{k=1}^{n}$, conditioned to stay non-negative converges weakly to a Brownian meander under diffusive scaling. Since the endpoint of a Brownian meander is a strictly positive continuous random variable, we thus have \eqref{b21}. If $a_1\le a_2$, the argument is a bit more involved. We first write the (complement of the) conditional probability as a ratio:
		\begin{align*}
			\pr{n}{(a_1,a_2)}{}\big( \se{1}{n}-\se{2}{n} \le \delta\sqrt{n} \mid \ni\big) =\frac{\pr{n}{(a_1,a_2)}{}\big( \big\{\se{1}{n}-\se{2}{n} \le \delta\sqrt{n}\big\} \cap \ni\big)}{\pr{n}{(a_1,a_2)}{}\big(\ni\big)}.
		\end{align*}
	For the denominator we condition on $(\se{1}{2},\se{2}{2})$ and use the lower bound from Lemma \ref{l:class}:
	\begin{align}
		\pr{n}{(a_1,a_2)}{}(\ni) & =\Ex^{n;(a_1,a_2)}\left[\ind_{\se{1}{2}\ge \se{2}{2}}\Ex^{n;(a_1,a_2)}\left[\ind_{\ni_0\ll 3,n-1\rr}\mid \sigma(\se{1}{2},\se{2}{2})\right]\right] \\ & \ge \frac{\Con^{-1}}{\sqrt{n}}  \Pr^{n;(a_1,a_2)}({\se{1}{2}\ge \se{2}{2}}). \label{b21m}
	\end{align}
	For the numerator, we again condition on $(\se{1}{2},\se{2}{2})$ to get
		\begin{align}\nonumber
			& \pr{n}{(a_1,a_2)}{}\big( \big\{\se{1}{n}-\se{2}{n} \le \delta\sqrt{n}\big\} \cap \ni\big) \\ & =\Ex\left[\ind_{\se{1}{2}\ge \se{2}{2}}\Ex\left[\ind_{\{\se{1}{n}-\se{2}{n} \le \delta\sqrt{n}\}\cap\ni_0\ll 3,n-1\rr}\mid \sigma(\se{1}{2},\se{2}{2})\right]\right]. \label{b211}
		\end{align}
	Upon conditioning on $(\se{1}{2},\se{2}{2})$, the random walks starts at $(\se{1}{2},\se{2}{2})$. For any $b_1\ge b_2$, utilizing the upper bound from Lemma \ref{l:class} and stochastic monotonicity we have
	\begin{align*}
	&	\Pr^{n-1;(b_1,b_2)}\left(\{\se{1}{n-1}-\se{2}{n-1} \le \delta\sqrt{n}\}\cap\ni_0\ll 2,n-2\rr\right) \\ & = \Pr^{n-1;(b_1,b_2)}(\ni_0\ll2,n-2\rr)\Pr^{n-1;(b_1,b_2)}\left(\{\se{1}{n-1}-\se{2}{n-1} \le \delta\sqrt{n}\}\mid\ni_0\ll 2,n-2\rr\right) \\ & \le \frac{\Con}{\sqrt{n}}\max\{b_1-b_2,1\}\cdot \Pr^{n-1;(0,0)}\left({\{\se{1}{n-1}-\se{2}{n-1} \le \delta\sqrt{n}\}\mid\ni_0\ll 2,n-2\rr}\right)
	\end{align*}
	Taking $b_1=\se{1}{2}$ and $b_2=\se{2}{2}$, we insert the above bound back in \eqref{b211} to get
	\begin{align}
	\nonumber & \pr{n}{(a_1,a_2)}{}\big( \big\{\se{1}{n}-\se{2}{n} \le \delta\sqrt{n}\big\} \cap \ni\big)	\\ & \label{b212} \hspace{1cm}\le \frac{\Con}{\sqrt{n}}\Ex^{n;(a_1,a_2)}\left[\ind_{\se{1}{2}\ge \se{2}{2}}\max\{\se{1}{2}-\se{2}{2},1\}\right] \\ \nonumber & \hspace{4cm}\cdot \Pr^{n-1;(0,0)}\left({\{\se{1}{n-1}-\se{2}{n-1} \le \delta\sqrt{n}\}\mid\ni_0\ll 2,n-2\rr}\right)
	\end{align}
	Note that under $\Pr^{n;(a_1,a_2)}$, $\se{1}{2}-\se{2}{2}\stackrel{(d)}{=} Z-b$ where $Z\sim \ffa\ast\ffa$ and $b=a_2-a_1\ge 0$. We claim that there exists a constant $\Con>0$ such that for all $b\ge 0$ we have
	\begin{align}\label{b213}
		{\Ex\left[\max\{Z-b,1\}\ind_{Z\ge b}\right]} \le \Con \cdot \Pr(Z\ge b).
	\end{align}
Plugging this bound in the expectation in \eqref{b212}  we have
\begin{align*}
	& \pr{n}{(a_1,a_2)}{}\big( \big\{\se{1}{n}-\se{2}{n} \le \delta\sqrt{n}\big\} \cap \ni\big)	\\ & \le \frac{\Con}{\sqrt{n}}\Pr^{n;(a_1,a_2)}({\se{1}{2}\ge \se{2}{2}}) \!\cdot\!\Pr^{n-1;(0,0)}\left({\{\se{1}{n-1}-\se{2}{n-1} \le \delta\sqrt{n}\}\mid\ni_0\ll 2,n-2\rr}\right)
\end{align*}
Combining this with the lower bound on the denominator from \eqref{b21m} we get that
\begin{align*}
	\pr{n}{(a_1,a_2)}{}\big( \se{1}{n}-\se{2}{n} \le \delta\sqrt{n} \mid \ni\big) \le	\Con \!\cdot\! \Pr^{n-1;(0,0)}\left({\{\se{1}{n-1}-\se{2}{n-1} \le \delta\sqrt{n}\}\mid\ni_0\ll 2,n-2\rr}\right)\!.\!
\end{align*}
From here, we can again appeal to \cite{igl} and Brownian meander properties to show that the above bound can be made arbitrarily small by choosing $\delta$ small enough. Thus we are left to \eqref{b213}.

	Suppose $Z\sim \ffa\ast \ffa$. Observe that for any $b\in \R$ we have
\begin{align*}
	\Ex\left[\max\{Z-b,1\}\ind_{Z\ge b}\right] & =\sum_{k=1}^{\infty} \Ex\left[\max\{Z-b,1\}\ind_{Z\in [b+k-1,b+k]}\right] \\ & \le \Pr(Z\ge b) \left[1+\sum_{k=2}^{\infty} k\cdot \frac{\Pr(Z\ge b+k-1)}{\Pr(Z\ge b)}\right]
\end{align*}
If we assume $b\ge 0$ additionally, using exponential tail bounds for $\ffa$ (and hence $\ffa\ast\ffa$), we may get a constant $\Con>0$ free of $b$, such that $\frac{\Pr(Z\ge b+k-1)}{\Pr(Z\ge b)} \le \Con e^{-k/\Con}$ for all $k\ge 2$. This ensure the infinite sum above can be bounded uniformly over $b\in [0,\infty)$. This proves \eqref{b213}.

	\medskip
	
	\noindent\textbf{Proof of Eq.~\eqref{b22}.} Set $U(k):=\se{1}{k}-\se{2}{k}$. To obtain \eqref{b22}, observe the following inequalities
	\begin{align*}
		& \pr{n}{(a_1,a_2)}{}\bigg( \sup_{k\in \ll1,n\rr}U(k) \le \delta^{-1}\sqrt{n}+\max\{a_1-a_2,0\} \mid \bigcap_{k=2}^{n}\{U(k)\ge 0\}\bigg) \\ & \ge \pr{n}{(\max\{a_1,a_2\},a_2)}{}\bigg(\sup_{k\in \ll1,n\rr}U(k) \le \delta^{-1}\sqrt{n}+\max\{a_1-a_2,0\} \mid \bigcap_{k=2}^{n}\{U(k)\ge 0\}\bigg) \\ & \ge \pr{n}{(\max\{a_1,a_2\},a_2)}{}\bigg( \sup_{k\in \ll1,n\rr}U(k) \le \delta^{-1}\sqrt{n}+\max\{a_1-a_2,0\} \mid \bigcap_{k=2}^{n}\{U(k)\ge \max\{a_1-a_2,0\}\}\bigg) \\ & =  \pr{n}{(a_2,a_2)}{}\bigg( \sup_{k\in \ll1,n\rr}U(k) \le \delta^{-1}\sqrt{n} \mid \bigcap_{k=2}^{n}\{U(k)\ge 0\} \bigg) \ge 1-\e.
	\end{align*}
	Let us briefly explain the above inequalities that imply \eqref{b22}. The first inequality follows from stochastic monotonicity applied to the boundary point. We are conditioning on the event that requires the random walk $(U(k))_{k=1}^n$ to stay above the barrier zero. By stochastic monotonicity, increasing this barrier will only decrease the conditional probability. This implies the second inequality. The equality in the last line follows by translating the random walk. The final inequality follows by taking $\delta$ small enough due to the tightness of the random walk paths conditioned to stay positive (when scaled by diffusively) \cite{igl}.
	\medskip
	
	\noindent\textbf{Proof of Eq.~\eqref{b23} and Eq.~\eqref{b24}} Note that due to stochastic monotonicity (Proposition \ref{smrw}), taking $a_2\downarrow -\infty$ we get
	\begin{align*}
		\pr{n}{(a_1,a_2)}{}\bigg(\inf_{k\in \ll1,n\rr}\se{1}{k}-a_1 \ge -\delta^{-1}\sqrt{n} \mid \ni\bigg) & \ge \pr{n}{(a_1,-\infty)}{}\bigg(\inf_{k\in \ll1,n\rr}\se{1}{k}-a_1 \ge -\delta^{-1}\sqrt{n} \mid \ni\bigg) \\ & = \pr{n}{(a_1,-\infty)}{}\bigg(\inf_{k\in \ll1,n\rr}\se{1}{k}-a_1 \ge -\delta^{-1}\sqrt{n}\bigg)\\ & = \pr{n}{(a_1,a_2)}{}\bigg(\inf_{k\in \ll1,n\rr}\se{1}{k}-a_1 \ge -\delta^{-1}\sqrt{n}\bigg).
	\end{align*}
The first equality above is due to the fact that $\ni$ happens almost surely when the second walk starts at $-\infty$. The second equality follows from noting that $\se{1}{\cdot}$ and $\se{2}{\cdot}$ are independent and hence the probability is independent of the starting point of the second walk. Thus the non-intersecting condition makes $\se{1}{\cdot}$ stochastically larger than a usual random walk.	By diffusive behavior of random walks one can choose $\delta$ small enough so that the above quantity is at least $1-\e$. Similarly the non-intersecting condition makes $\se{2}{\cdot}$ stochastically smaller than a usual random walk. Combining this with the diffusive behavior of random walks leads to \eqref{b24}.
\end{proof}

\begin{corollary}\label{tailni2} Fix any $n\in \Z_{\ge 2}$. Suppose $a_1,a_2 \in \R$ with $|a_1-a_2|\le n/\log n$. Given any $\e,\gamma>0$ there exists a constant $\rho(\e,\gamma) \in (0,\frac14]$ such that for all large enough $n$ we have
	\begin{align*}
		\pr{n}{(a_1,a_2)}{}\bigg( \sup_{k\in \ll1,n\rho\rr,i=1,2} \big|\se{i}{k}-a_{i}\big| \ge \gamma\sqrt{n} \mid \ni\bigg) \ge 1-\e.
	\end{align*}
\end{corollary}

\begin{proof} Let us focus only on $\se{1}{\cdot}$. We may control lower drift of $\se{1}{\cdot}$ around $a_1$, i.e., $\inf_{k\in \ll1,n\rho\rr} [\se{1}{k}-a_1]$ by an argument similar to the proof of \eqref{b23}. For upper drift we use
	\begin{align*}
		\sup_{k\in \ll1,n\rho\rr} [\se{1}{k}-a_1] \le a_2-a_1+\sup_{k\in \ll1,n\rho\rr} [\se{1}{k}-\se{2}{k}]+\sup_{k\in \ll1,n\rho\rr} [\se{2}{k}-a_2]
	\end{align*}
The second and third term can be controlled by an argument similar to the proof of \eqref{b22} and \eqref{b24} respectively. Note that by diffusive properties all the fluctuations are of the order $\sqrt{n\rho}$. Hence one can choose $\rho$ small enough so that
\begin{align*}
	\pr{n}{(a_1,a_2)}{}\bigg( \sup_{k\in \ll1,n\rho\rr} \big|\se{1}{k}-a_{1}\big| \ge \gamma\sqrt{n} \mid \ni\bigg) \ge 1-\e.
\end{align*}
\end{proof}

\smallskip

We now study non-intersecting probabilities for random  bridges $\pr{n}{(a_1,a_2)}{,(b_1,b_2)}$ defined in Definition \ref{def:rws} (increments drawn from $\ffa$).
 The following lemma shows that when the starting points and endpoints are far apart in the diffusive scale, non-intersection probability is bounded from zero.
\begin{lemma}\label{unini} Fix $\delta>0$. For each $n\in \Z_{\ge4}$, consider the set \begin{align}\label{rnd}
		R_{n,\delta}:=\{(x_1,x_2): |x_i|\le 2\sqrt{n}(\log n)^{3/2}, x_1-x_2 \ge \delta\sqrt{n}\}
	\end{align}
	There exists $\phi=\phi(\delta)>0$ such that for all $n$ large enough and all $(a_1,a_2),(b_1,b_2)\in R_{n,\delta}$ we have
	\begin{align*}
		\pr{n}{(a_1,a_2)}{,(b_1,b_2)}\left(\inf_{k\in \ll1,n\rr} \big[\iks-\iiks\big]\ge \tfrac14\delta\sqrt{n}\right) \ge \phi.
	\end{align*}
\end{lemma}
\begin{proof} Fix any $(a_1,a_2),(b_1,b_2)\in R_{n,\delta}$. For simplicity let us write $\Pr$ for $\pr{n}{(a_1,a_2)}{,(b_1,b_2)}$. Note that $|b_i-a_i| \le 4\sqrt{n}(\log n)^{3/2}$. By the KMT coupling for Brownian bridges (Theorem 2.3 in \cite{xd} with $z=b_i-a_i$ and $p=0$), there exists a constant $\Con>0$ such that for all $n\in \Z_{\ge1}$ we have
	\begin{equation}
		\label{kmtbr1}
			\begin{aligned}
			& \pr{n}{(a_1,a_2)}{,(b_1,b_2)}\bigg(\neg\m{SC}_{(a_1,a_2)}^{(b_1,b_2)}\bigg) \le \tfrac1{n}, \\ & \hspace{0.5cm} \mbox{where }\m{SC}_{(a_1,a_2)}^{(b_1,b_2)}:=\bigg\{\sup_{k\in \ll1,n\rr,i=1,2} \left|\kis-\sqrt{n}B_i(k/n)-a_i-\tfrac{k}{n}(b_i-a_i)\right| \le \Con \log^3 n\bigg\},
		\end{aligned}
	\end{equation}
and where $B_1, B_2$ are Brownian bridges on the same probability space with variance $\int x^2\ffa(x)dx$ ($\m{SC}$ stands for `strong coupling'). By Brownian bridge properties, there exists $\phi=\phi(\delta)>0$ so that
	\begin{align*}
		\pr{n}{(a_1,a_2)}{,(b_1,b_2)}\bigg(\sup_{x\in [0,1]} (|B_1(x)|+|B_2(x)|) \le \tfrac18\delta\bigg) \ge 2\phi.
	\end{align*}
	Combining the previous two math displays we see that with probability $2\phi-\frac1n$ we have
	\begin{align*}
		\iks-\iiks & \ge a_1-a_2+\tfrac{k}{n}(b_1-a_1-b_2+a_2)-2\Con (\log n)^{3}-\tfrac14\delta \sqrt{n} \\ & = \tfrac{n-k}{n}a_1-a_2+\tfrac{k}{n}(b_1-b_2)-2\Con (\log n)^{3}-\tfrac14\delta \sqrt{n} \\ & \ge -2\Con (\log n)^{3}+\tfrac12\delta \sqrt{n} > \tfrac14\delta \sqrt{n}
	\end{align*}
	for all large enough $n$. Taking $n$ large enough ensures $2\phi-\frac1n \ge \phi$ completing the proof.
\end{proof}
Our next lemma gives a crude bound for the weak non-intersection probability in terms of true non-intersection probability.
\begin{lemma}\label{l:nipp} There exists $\Con>0$ such that for all $p \in [0,\infty)$, $(a_1,a_2), (b_1,b_2)\in \R^2$, $n\in \Z_{\ge 1}$
	\begin{align*}
		\pr{n}{(a_1,a_2)}{,(b_1,b_2)}(\ni_p) \le e^{\Con p} \cdot \pr{n}{(a_1,a_2)}{,(b_1,b_2)}(\ni).
	\end{align*}
	
\end{lemma}
\begin{proof}
	By lifting the first random bridge by $p$ units we see that
	\begin{align*}
		\pr{n}{(a_1,a_2)}{,(b_1,b_2)}(\ni_p)=\pr{n}{(a_1+p,a_2)}{,(b_1+p,b_2)}(\ni).
	\end{align*}
	Conditioning on the second point and the penultimate point of both the random bridges we get
	\begin{align}\label{b15}
		\pr{n}{(a_1+p,a_2)}{,(b_1+p,b_2)}(\ni) = \frac{\int_{x_1\ge x_2,y_1\ge y_2} \Lambda_{x_1,x_2}^{n;(y_1,y_2)}(\ni) \Upsilon_p(x_1,x_2;y_1,y_2) dx_1dx_2dy_1dy_2}{\ffa^{\ast (n-1)}(a_1-b_1)\ffa^{\ast (n-1)}(a_2-b_2)}.
	\end{align}
	where
	\begin{align*}
		& \Upsilon_p(x_1,x_2;y_1,y_2):=\ffa(a_1+p-x_1)\ffa(a_2-x_2)\ffa(y_1-b_1-p)\ffa(y_2-b_2), \\
		& \Lambda_{x_1,x_2}^{n;(y_1,y_2)}(\ni) := \int_{x_{j,1}\ge x_{j,2}, j\in \ll3,n-2\rr} \prod_{j=2}^{n-2} \ffa(x_{j,1}-x_{j+1,1})\ffa(x_{j,2}-x_{j+1,2})\prod_{j=3}^{n-2} dx_{j,1}dx_{j,2}.
	\end{align*}
	Here in the above integration we set $x_{2,1}:=x_1$, $x_{2,2}:=x_2$, $x_{n-1,1}:=y_1$, $x_{n-1,2}:=y_2$. From Lemma \ref{tailf}, we have that $\Upsilon_p(x_1,x_2;y_1,y_2) \le e^{\Con p}\Upsilon_0(x_1,x_2;y_1,y_2)$, where the $\Con>0$ depends only on $\theta$. Plugging this bound back in \eqref{b15} we get the desired result.
\end{proof}
The following technical lemma, which can be thought of as the bridge analog of Lemma \ref{l:class}, studies the non-intersection probability for random bridges when the starting points are close.

\begin{lemma}\label{genni} Fix $M>0$ and $n\in \Z_{\ge 2}$.  There exist a constant $\Con=\Con(M)>0$ such that for all $|a_i|\le \sqrt{n}(\log n)^{3/2}$ with $|a_1-a_2|\le (\log n)^{3/2}$, and $|b_i|\le M\sqrt{n}$ with $b_1\ge b_2$ we have
	\begin{align*}
			\pr{n}{(a_1,a_2)}{,(b_1,b_2)}(\ni) \le \Con\tfrac1{\sqrt{n}}{\max\{a_1-a_2,1\}\cdot \max\left\{\tfrac1{\sqrt{n}}|a_1-b_1|,2\right\}^{3/2}}.
	\end{align*}
\end{lemma}
\begin{proof} It suffices to prove the lemma only for large enough $n$ (since we can always choose the $\Con$ large enough). Set $r=\max\{\frac1{\sqrt{n}}|a_1-b_1|,2\}$ and $p=\lfloor nr^{-3}\rfloor$. We first claim that there exists $m(M)>0$ such that
	\begin{align}\label{toshow6}
		\pr{n}{(a_1,a_2)}{,(b_1,b_2)}(\ni) \le 2\cdot\pr{n}{(a_1,a_2)}{,(b_1,b_2)}\left(\{|\se{i}{p}-a_i|\le m\sqrt{n}r^{-1} \mbox{ for } i=1,2\} \cap \ni\right).
	\end{align}
Let us first complete the proof of the lemma assuming \eqref{toshow6}.
	Note that the density of $\se{i}{p}$ at $x$ is given by $\frac{\ffa^{*(p-1)}(x-a_i)\ffa^{*(n-p)}(b_i-x)}{\ffa^{*(n-1)}(b_i-a_i)}$. By Lemma \ref{sharp}, we may replace the $n$-fold convolution with Gaussian densities at the expensive of a multiplicative factor close to $1$. In particular for large enough $n$ we have
	\begin{align*}
		\sup_{|x-a_i|\le m\sqrt{n}r^{-1}}\frac{\ffa^{*(n-p)}(b_i-x)}{\ffa^{*(n-1)}(b_i-a_i)} & =2\exp\left(\tfrac1{2\sigma^2}\left(r^2-\tfrac{(r-mr^{-1})^2}{1-r^{-3}}\right)\right) \\ &  =2\exp\left(\tfrac1{2\sigma^2(1-r^{-3})}\left(-r^{-1}-m^2r^{-2}+2m\right)\right) \le 2e^{2m/\sigma^2}.
	\end{align*}
	Thus $\frac{\ffa^{*(p-1)}(x-a_i)f^{*(n-p)}(b_i-x)}{\ffa^{*(n-1)}(b_i-a_i)} \le  2e^{2m/\sigma^2} \cdot \ffa^{*(p-1)}(x-a_i)$ whenever $|x-a_i|\le m\sqrt{n}r^{-1}$. This allows us to go from random bridge laws to random walk laws. We thus have
	\begin{align*}
		& \pr{n}{(a_1,a_2)}{,(b_1,b_2)}\left(\{|\se{i}{p}-a_i|\le m\sqrt{n}r^{-1} \mbox{ for } i=1,2\} \cap \ni\right) \\ & \le \pr{n}{(a_1,a_2)}{,(b_1,b_2)}\left(\{|\se{i}{p}-a_i|\le m\sqrt{n}r^{-1} \mbox{ for } i=1,2\} \cap \bigcap_{k=1}^p \{\iks\ge \iiks\}\right)  \\ & \le 2e^{2m/\sigma^2}\cdot \pr{n}{(a_1,a_2)}{}\left(\{|\se{i}{p}-a_i|\le m\sqrt{n}r^{-1} \mbox{ for } i=1,2\} \cap \bigcap_{k=1}^p \{\iks\ge \iiks\}\right) \\ & \le   2e^{2m/\sigma^2} \cdot \pr{p}{(a_1,a_2)}{}\left(\bigcap_{k=1}^p \{\iks\ge \iiks\}\right) \le \tfrac{\Con}{\sqrt{n}} r^{3/2}\cdot \max\{a_1-a_2,1\}.
	\end{align*}
	where the last inequality uses Lemma \ref{l:class}. This completes the proof modulo \eqref{toshow6}. The rest of the proof is devoted to showing \eqref{toshow6}.
	
	\medskip

	We claim that
	\begin{align}\label{bf1}
		& \pr{n}{(a_1,a_2)}{,(b_1,b_2)}(\se{1}{p}-a_1\le -m\sqrt{n} r^{-1} \mid \ni) \le \tfrac18, \\ & \label{bf2} \pr{n}{(a_1,a_2)}{,(b_1,b_2)}(\se{1}{p}-a_1\ge m\sqrt{n} r^{-1} \mid \ni) \le \tfrac18, \\ \nonumber
		& \pr{n}{(a_1,a_2)}{,(b_1,b_2)}(\se{2}{p}-a_2\le -m\sqrt{n} r^{-1} \mid \ni) \le \tfrac18, \\  \nonumber &  \pr{n}{(a_1,a_2)}{,(b_1,b_2)}(\se{2}{p}-a_2\ge m\sqrt{n} r^{-1} \mid \ni) \le \tfrac18,
	\end{align}
	for all large enough $n$. Applying an union bound, leads to \eqref{toshow6}. We shall prove the first two inequalities: \eqref{bf1} and \eqref{bf2}, the remaining two follows in a similar fashion.
	
\noindent\textbf{Proof of Eq.~\eqref{bf2}.}	Similar to the proof of \eqref{b23} and \eqref{b24}, by stochastic monotonicity for random bridges (Proposition \ref{smrw}) we have
	\begin{align}\label{eeb9}
		\pr{n}{(a_1,a_2)}{,(b_1,b_2)}(\se{1}{p}-a_1\le -m\sqrt{n} r^{-1} \mid \ni) \le \pr{n}{(a_1,a_2)}{,(b_1,b_2)}(\se{1}{p}-a_1\le -m\sqrt{n} r^{-1})
	\end{align}
We invoke the KMT coupling 	for random bridges \cite{xd} to define Brownian bridge $B_1, B_2$ on $[0,1]$ on a common probability space such that \eqref{kmtbr1} holds. By \eqref{kmtbr1}, with probability $1-\frac1n$,
\begin{align*}
	\se{1}{p}-a_1 & \ge \sqrt{n}B_1(p/n)+\tfrac{p}{n}(b_1-a_1)-\Con \log^3 n \\ & = \sqrt{n}B_1(p/n)-\sqrt{n} r^{-2}-\Con \log^3 n \ge \sqrt{n}B_1(p/n)-2\sqrt{n}r^{-1}.
\end{align*}
for large enough $n$. Since $p/n$ is of the order $r^{-3}$, $B_{1}(p/n)$ fluctuates of the order $r^{-3/2}$. By Brownian bridge one point tail estimates, there exists a constant $c>0$ such that for all $m\ge 3$
$$\pr{n}{(a_1,a_2)}{,(b_1,b_2)}(B_{1}(p/n)> -(m-2)r^{-1}) \ge 1-e^{-c m^2r}.$$
Thus by an union bound we have
\begin{align}\label{efre}
	\pr{n}{(a_1,a_2)}{,(b_1,b_2)}(\se{1}{p}-a_1> -m\sqrt{n} r^{-1}) \ge 1-\tfrac1n-e^{-cm^2r}.
\end{align}
Taking $m,n$ are large enough, ensure that $1-\frac1n-e^{-cm^2r}\ge \frac78$. This verifies \eqref{bf1}.

\medskip

\noindent\textbf{Proof of Eq.~\eqref{bf2}.}  By stochastic monotonicity (Proposition \ref{smrw}) at the starting points,
	\begin{align}\nonumber
	&	\pr{n}{(a_1,a_2)}{,(b_1,b_2)}(\se{1}{p}-a_1\ge m\sqrt{n} r^{-1} \mid \ni) \\ \nonumber & \le \pr{n}{(a_1+\sqrt{n}r^{-1},a_2)}{,(b_1,b_2)}(\se{1}{p}-a_1\ge m\sqrt{n} r^{-1} \mid \ni) \\ & \le \frac{\pr{n}{(a_1+\sqrt{n}r^{-1},a_2)}{,(b_1,b_2)}(\se{1}{p}-a_1\ge m\sqrt{n} r^{-1})}{\pr{n}{(a_1+\sqrt{n}r^{-1},a_2)}{,(b_1,b_2)}(\ni)}. \label{eeb0}
	\end{align}
Using an argument similar to the derivation of \eqref{efre}, we find that
	\begin{align}\label{eeb1}
		\pr{n}{(a_1+\sqrt{n}r^{-1},a_2)}{,(b_1,b_2)}(\se{1}{p}-a_1\ge m\sqrt{n} r^{-1}) \le \tfrac1n+e^{-cm^2r}.
	\end{align}
	This gives an upper bound for the numerator of \eqref{eeb0}. For the denominator, recall the event $\m{SC}_{(a_1,a_2)}^{(b_1,b_2)}$ and the Brownian bridges $B_1, B_2$ from \eqref{kmtbr1}. Note that on the event
$$\m{SC}_{(a_1+\sqrt{n}r^{-1},a_2)}^{(b_1,b_2)}\cap \{\inf_{x\in [0,1]} (B_1(x)-B_2(x)) \ge -\frac12r^{-1}\},$$
for large enough $n$ we have
	\begin{align*}
		\iks & \ge \sqrt{n}B_1(k/n)+a_1+\sqrt{n}r^{-1}+\tfrac{k}{n}(b_1-a_1)-\Con (\log n)^3 \\ & \ge \sqrt{n}B_2(k/n)+\tfrac12\sqrt{n}r^{-1}+a_2+\tfrac{k}{n}(b_2-a_2)-2\Con (\log n)^3 \\ & \ge \iiks+\tfrac12\sqrt{n}r^{-1}-3\Con(\log n)^3 \ge \iiks.
	\end{align*}
	where we used that $|a_1-a_2|\le (\log n)^{3/2}$, $b_1\ge b_2$, and $r\le (\log n)^{3/2}$. Thus for large enough $n$,
	\begin{align*}
		\pr{n}{(a_1+\sqrt{n}r^{-1},a_2)}{,(b_1,b_2)}(\ni) & \ge
		\pr{n}{(a_1+\sqrt{n}r^{-1},a_2)}{,(b_1,b_2)}\bigg(\inf_{x\in [0,1]} (B_1(x)-B_2(x))\ge -\tfrac12r^{-1}\bigg) \\ & \quad -\pr{n}{(a_1+\sqrt{n}r^{-1},a_2)}{,(b_1,b_2)}\bigg(\neg\m{SC}_{(a_1+\sqrt{n}r^{-1},a_2)}^{(b_1,b_2)}\bigg) \\ & \ge \Con r^{-2}-\tfrac1n  \ge \tfrac12\Con r^{-2},
	\end{align*}
	where the penultimate inequality follows from \eqref{kmtbr1} and Brownian bridge calculations (see Lemma 2.11 in \cite{ch16} for example). Combining \eqref{eeb1} and the above lower bound we have
	\begin{align*}
		\mbox{r.h.s.~\eqref{eeb0}} \le \tfrac{2}{\Con}\left(\tfrac{r^{2}}{n}+r^2e^{-cm^2r}\right) \le \tfrac18,
	\end{align*}
	for all large enough $n$ and $m$ (as $r\le (\log n)^{3/2}$).
\end{proof}

\begin{corollary}\label{l:niexp}  Fix any $M>0$ and $n\ge 1$. Suppose $|a_i|,|b_i| \le M\sqrt{n}$ for $i=1,2$. There exists a constant $\Con=\Con(M)>0$ such that
	\begin{align*}
	& \pr{n}{(a_1,a_2)}{,(b_1,b_2)}(\ni) \le \Con \!\cdot\! \pr{\lfloor n/4 \rfloor}{(a_1,a_2)}{}(\til{\ni})\pr{\lfloor n/4 \rfloor}{(b_1,b_2)}{}(\til{\ni}) \\ &  \pr{n}{(a_1,a_2)}{,(b_1,b_2)}(\ni) \ge \tfrac1\Con \!\cdot\! \pr{\lfloor n/4 \rfloor}{(a_1,a_2)}{}(\til{\ni})\pr{\lfloor n/4 \rfloor}{(b_1,b_2)}{}(\til{\ni}).
	\end{align*}
where $\til{\ni}:=\{\se{1}{k}\ge \se{1}k \mbox{ for all }k\in \ll2,n/4\rr\}.$
\end{corollary}	
\begin{proof} The upper bound follows by applying \eqref{eb7} with $\delta_i=\frac14$ and integrating over the non-intersection event. Let us focus on the lower bound. For simplicity we will drop the floor functions from $\lfloor n/4 \rfloor$.  By Lemma \ref{tailni}, we can choose a constant $\til{M}$ depending only on $M$ such that for all $|c_i|\le M\sqrt{n}$ with $|x_i|\le M\sqrt{n}$, we have
	\begin{align}\label{eb12}
		\pr{n/4}{(c_1,c_2)}{}\bigg(\bigcap_{i=1}^2 \{|\se{i}{n/4}| \le \til{M}\sqrt{n} \}\mid \til{\ni}\bigg) \ge \tfrac34.
	\end{align}
	By Lemma \ref{tailni} we can next choose a $\delta=\delta(M)>0$ small enough such that
	\begin{align}\label{eb13}
		\pr{n/4}{(c_1,c_2)}{}\bigg( (\se{1}{n/4},\se{2}{n/4}) \in R_{n,\delta}\mid \til{\ni}\bigg) \ge \tfrac34.
	\end{align}
	where $R_{n,\delta}$ is from \eqref{rnd}. By Lemma \ref{unini},
	there exists $\phi(\delta)>0$ so that for all $(x_1,x_2), (y_1,y_2) \in R_{n,\delta}$
	\begin{align}\label{eb14}
		\pr{n/2}{(x_1,x_2)}{,(y_1,y_2)}\bigg(\bigcap_{k=1}^{n/2} \{\se{1}{k}\ge \se{2}{k}\}\bigg) \ge \phi.
	\end{align}
	We next consider the events
	\begin{align*}
		\m{E}_1:=\left\{|\se{i}{ n/4}| \le \til{M}\sqrt{n} \mbox{ for }i=1,2 \right\}, \quad \m{E}_2:=\left\{|\se{i}{ 3n/4}| \le \til{M}\sqrt{n} \mbox{ for }i=1,2 \right\}.
	\end{align*}
	Using \eqref{eb8} with $\delta=\frac14$ we have
	\begin{align*}
		\pr{n}{(a_1,a_2)}{,(b_1,b_2)}(\ni) \ge \pr{n}{(a_1,a_2)}{,(b_1,b_2)}(\m{E}_1\cap\m{E}_2\cap\ni) & \ge  \Con^{-1}\!\cdot\! \til{\Pr}(\m{E}_1\cap\m{E}_2\cap\ni)\\ & =\Con^{-1} \!\cdot\!\til{\Pr}(\ni)\til{\Pr}(\m{E}_1\cap \m{E}_2 \mid {\ni}).
	\end{align*}
	for some $\Con>0$ depending on $\til{M},M$. Here $\til{\Pr}:=\tpr{(n;n/4,n/4)}{(a_1,a_2)}{,(b_1,b_2)}$ denotes the joint law of two independent $(n;n/4,n/4)$-modified random bridges of length $n$ starting at $(a_1,a_2)$ and ending at $(b_1,b_2)$ (see Definition \ref{def:mrb}). In view of our $\til{M}$ choice and by the definition of modified random bridges, we have $$\til{\Pr}(\m{E}_1\cap \m{E}_2 \mid {\ni}) = 	\pr{n/4}{(a_1,a_2)}{}\bigg(\bigcap_{i=1}^2 \{|\se{i}{n/4}| \le \til{M}\sqrt{n} \}\mid \til{\ni}\bigg)\pr{n/4}{(b_1,b_2)}{}\bigg(\bigcap_{i=1}^2 \{|\se{i}{n/4}| \le \til{M}\sqrt{n} \}\mid \til{\ni}\bigg)$$ which is lower bounded by $(3/4)^2$ from \eqref{eb12} and \eqref{eb13}. Furthermore, in view of \eqref{eb14}, we have
	$$\til{\Pr}(\ni) \ge \phi \cdot \pr{n/4}{(a_1,a_2)}{}(\til{\ni})\pr{n/4}{(b_1,b_2)}{}(\til{\ni}).$$
	We thus have the desired lower bound.
\end{proof}

We now analyze the $\m{Gap}_{\beta}$ event defined in \eqref{def:gp} under modified random bridge law. Fix any $M>0$, $n\ge 1$, and $(a_1,a_2),(b_1,b_2)\in \R^2$. Suppose $|a_i|,|b_i|\le M\sqrt{n}$ and $a_1\ge a_2$. Take $p,q\in \ll0,n\rr$ with $p+q\le n/2$ and $p\neq 0$. Suppose further that there exists $\rho\in (0,1)$ such that either $q\ge n\rho$ or $b_1-b_2\ge \rho\sqrt{n}$. Consider two independent $(n;p,q)$- modified random bridges $(\kis)_{k\in \ll1,n\rr,i=1,2}$ starting and ending at $(a_1,a_2)$ and $(b_1,b_2)$ respectively. We denote its law by $\tpr{(n;p,q)}{(a_1,a_2)}{,(b_1,b_2)}$. The following lemma asserts $\m{Gap}_{\beta}$ event is very likely under non-intersection.
\begin{lemma}\label{l:gapev}  Fix $\e,\rho\in (0,1)$ and $M>0$. There exist $\beta(\e,\rho,M)>0$, $n_0(\e,\rho,M)>0$, such that for all $n\ge n_0$
	\begin{align*}
		\tpr{(n;p,q)}{(a_1,a_2)}{,(b_1,b_2)}\left(\m{Gap}_{\beta} \mid \ni\right) \ge 1-\e.
	\end{align*}
\end{lemma}
\begin{proof}  Recall that $\m{Gap}_{\beta}$ event is intersection of six smaller `Gap' events: $\m{Gap}_{i,\beta}$ defined around \eqref{def:gp}. For simplicity write $\til\Pr$ for $\tpr{(n;p,q)}{(a_1,a_2)}{,(b_1,b_2)}$. We now analyze each `Gap' event separately.

\medskip\noindent \textbf{$\m{Gap}_{1,\beta}$ and $\m{Gap}_{2,\beta}$.} Note that for $k\in \ll1,p\rr$, $\iks-\iiks$ is itself a random walk. The $\ni$ event corresponds to the event of this random walk being non-negative. By classical result about growth of random walks conditioned to stay non-negative (see \cite[Theorem 2]{ritter}) it follows that one can choose $\beta$ small enough such that $\til\Pr(\m{Gap}_{1,\beta}\mid \ni)\ge 1-\frac{\e}{6}$. By the same argument one has $\til\Pr(\m{Gap}_{2,\beta}\mid \ni)\ge 1-\frac{\e}{6}$ for all large enough $n$ by choosing $\beta$ small enough.

\medskip\noindent \textbf{$\m{Gap}_{3,\beta}$.}
		Note that combining \eqref{b22}, \eqref{b23}, and \eqref{b24} from Lemma \ref{tailni} we have tightness of the endpoint of random walks conditioned on non-intersection. Combining this with \eqref{b21}, one can choose $\gamma$ small enough such that \begin{align}\label{png}
			\til\Pr\big((\se{1}{p},\se{2}{p}),(\se{1}{n-q},\se{2}{n-q} \mid \ni) \in \mathcal{P}_{n,\gamma}\big) \ge 1-\tfrac{\e}{12},
		\end{align}
		where
		\begin{align}\label{pnga}
			\mathcal{P}_{n,\gamma}:=\{(z_1,z_2)\in \R^2 : |z_i|\le \gamma^{-1}\sqrt{n}, z_1-z_2 \ge \gamma\sqrt{n}\}.
		\end{align}
		
		In other words, with probability $1-\frac{\e}{12}$, the endpoints of the middle portions of the modified random bridges are in $\mathcal{P}_{\gamma}$ when conditioned upon non-intersection . Thus,
		\begin{align}\label{rty}
			\til\Pr(\m{Gap}_{3,\beta}\mid \ni) \ge (1-\tfrac{\e}{12})\cdot \inf_{(a_1,a_2),(b_1,b_2)\in \mathcal{P}_{n,\gamma}} \pr{n-p-q+1}{(a_1,a_2)}{,(b_1,b_2)}(\m{Gap}_{3,\beta}\mid \ni).
		\end{align}
	Since the increments are drawn from a smooth density, for each fixed $n$, the probability $$\pr{n-p-q+1}{(a_1,a_2)}{,(b_1,b_2)}(\m{Gap}_{3,\beta}\mid \ni)$$ is jointly continuous  with respect to  the starting and ending points of the random bridge. As $\mathcal{P}_{n,\gamma}$ is closed, the infimum in \eqref{rty} is attained at some point $(a_1^*,a_2^*), (b_1^*,b_2^*) \in \mathcal{P}_{n,\gamma}$. Take any subsequential limit of $\frac{1}{\sqrt{n}}(a_1^*,a_2^*), \frac{1}{\sqrt{n}}(b_1^*,b_2^*)$ say $(u_1,u_2), (v_1,v_2)$. Then $|u_i|,|v_i|\le \gamma^{-1}$ and $u_1-u_2, v_1-v_2 \ge \gamma$. By invariance principle for Brownian bridges, this conditional law under diffusive scaling converges to non-intersecting Brownian bridges (with variance $\int x^2\ffa(x)dx$) $(B_1,B_2)$ starting at $(u_1,u_2)$ ending at $(v_1,v_2)$. We have $\Pr(\inf_{x\in [0,1]} (B_1(x)-B_2(x))>0)=1$. This implies along this subsequence the limit of $\pr{n-p-q+1}{(a_1^*,a_2^*)}{,(b_1^*,b_2^*)}(\m{Gap}_{3,\beta}\mid \ni)$ is $1$. Since this holds for all subsequences, we thus see that for all large enough $n$, r.h.s.~\eqref{rty} can be made at least $1-\tfrac{\e}{6}$.
		
\medskip\noindent\textbf{$\m{Gap}_{4,\beta}$ and $\m{Gap}_{5,\beta}$.} We shall first show $\m{Gap}_{4,\beta}$ happens with high probability under non-intersection. Note that this event only depends on the first part of the modified random bridge (which is just two independent pure random walk) independent of the other two parts. Hence	
		\begin{align}
			& \til\Pr(\neg\m{Gap}_{4,\beta}\mid \ni) \\ & = \til\Pr\bigg(\bigcap_{k=2}^p \{\se{1}{k}-\se{1}{k-1}\ge \beta^{-1} k^{1/8}\} \mid \bigcap_{k=2}^p \{\se{1}{k}\ge \se{2}{k}\} \bigg)  \\ & \le \sum_{k=2}^p \frac{\til\Pr\bigg(\{\se{1}{k}-\se{1}{k-1}\ge \beta^{-1} k^{1/8}\} \cap \bigcap_{j\in \{2\}\cup \ll k+2,p\rr} \{\se{1}{k}\ge \se{2}{k}\} \bigg)}{\til\Pr\bigg(\bigcap_{k=2}^p \{\se{1}{k}\ge \se{2}{k}\} \bigg)}, \label{b31}
		\end{align}
	where the above inequality follows via an union bound.
	Since under $\til\Pr$, $(\se{1}{\ell},\se{2}{\ell})_{\ell\in \ll1,p\rr}$ are two independent random walks starting from $(a_1,a_2)$. From \eqref{b21m} we get that
	\begin{align}
		\til\Pr\bigg(\bigcap_{k=2}^p \{\se{1}{k}\ge \se{2}{k}\} \bigg) \ge \frac{\Con^{-1}}{\sqrt{p}}\cdot \til\Pr(\se{1}{2}\ge \se{2}{2}).
	\end{align}
	\begin{align}
		& \til\Pr\bigg(\{\se{1}{k}-\se{1}{k-1}\ge \beta^{-1} k^{1/8}\} \cap \bigcap_{j\in \{2\}\cup \ll k+2,p\rr} \{\se{1}{k}\ge \se{2}{k}\} \bigg) \\ & \le \til\Ex\bigg[\ind_{\se{1}{2}\ge \se{2}{2}}\ind_{\se{1}{k}-\se{1}{k-1}\ge \beta^{-1} k^{\frac18}}\!\cdot\! \til\Ex\bigg[\prod_{j=k+2}^{p}\ind_{\se{1}{j}\ge \se{2}{j}} \mid \sigma\big((\se{1}{\ell},\se{2}{\ell})_{\ell\in \ll1,k+1\rr}\big)\bigg]\bigg] \label{b32}
	\end{align}
	By Lemma \ref{l:class}, we have the following bound for the interior conditional expectation above:
	\begin{align}
	 \til\Ex\bigg[\prod_{j=k+2}^{p}\ind_{\se{1}{j}\ge \se{2}{j}} \mid \sigma\big((\se{1}{\ell},\se{2}{\ell})_{\ell\in \ll1,k+1\rr}\big)\bigg] \le  \frac{\Con\!\cdot\! \max\{\se{1}{k+1}-\se{2}{k+1},1\}}{\sqrt{p-k+1}}. \label{b33}
	\end{align}	
Under $\til\Pr$, the increments of $\se{1}{\cdot}$ and $\se{2}{\cdot}$  are independent and distributed as $\ffa$ which has exponential tails by assumption.  We now claim that
\begin{equation}\label{b34}
	\begin{aligned}
		& \til\Ex\bigg[\ind_{\se{1}{2}\ge \se{2}{2}}\ind_{\se{1}{k}-\se{1}{k-1}\ge \beta^{-1} k^{1/8}}\!\cdot\! \max\{\se{1}{k+1}-\se{2}{k+1},1\}\bigg] \\ & \hspace{2cm}\le \Con \cdot k\cdot e^{-\frac1\Con \beta^{-1}k^{1/8}} \cdot \til\Pr(\se{1}{2}\ge \se{2}{2})
	\end{aligned}
\end{equation}
We shall prove \eqref{b34} later. Assuming it, combining the estimates from \eqref{b32}, \eqref{b33}, and \eqref{b34} we get
\begin{align*}
	\eqref{b31}	\le \Con^2 \sum_{k=2}^p  \sqrt{\tfrac{p}{p-k+1}} \cdot k \cdot e^{-\frac1\Con \beta^{-1}k^{1/8}}.
\end{align*}
Taking $\beta$ small enough, the right-hand side can be made smaller than $\e/6$. Thus, $\til\Pr(\m{Gap}_{4,\beta}\mid \ni) \ge 1-\frac\e{6}$ for all small enough $\beta$. An exact same argument leads to $\til\Pr(\m{Gap}_{5,\beta}\mid \ni) \ge 1-\frac\e{6}$ for all small enough $\beta$ as well.

To prove Eq.~\eqref{b34}, we start by writing $X(k):=\se{1}{k}-\se{1}{k-1}$ and $Y(k):=\se{2}{k}-\se{2}{k-1}$. For $k=2$, observe that
\begin{equation*}
	\begin{aligned}
		\mbox{l.h.s.~\eqref{b34}}& \le  \til\Ex\bigg[\ind_{\se{1}{2}\ge \se{2}{2}}\!\cdot\! \max\{\se{1}{3}-\se{2}{3},1\}\bigg] \\ & \le  \til\Ex\bigg[\ind_{\se{1}{2}\ge \se{2}{2}}\!\cdot\! \max\{\se{1}{2}-\se{2}{2},1\}\bigg] + \til\Ex\bigg[\ind_{\se{1}{2}\ge \se{2}{2}}\!\cdot\! \max\{X(3)-Y(3),1\}\bigg].
	\end{aligned}
\end{equation*}
By \eqref{b213}, the first expectation above is less than $\Con' \cdot \til\Pr(\se{1}{2}\ge \se{2}{2})$. For the second expectation by independence we get
$$\til\Ex\bigg[\ind_{\se{1}{2}\ge \se{2}{2}}\cdot \max\{X(3)-Y(3),1\}\bigg]=\til\Pr(\se{1}{2}\ge \se{2}{2})\cdot \til\Ex[\max\{X(3)-Y(3),1\}].$$
Since $X(3)$ and $Y(3)$ have exponential tails, by adjusting the constant $\Con'$ we get $\mbox{l.h.s.~\eqref{b34}} \le \Con' \cdot \til\Pr(\se{1}{2}\ge \se{2}{2})$. This proves \eqref{b34} for $k=2$ upon adjusting $\Con$. For $k\ge 3$, using the fact that $\max\{\sum_i A_i,1\}\le \sum_i \max\{A_i,1\}$, we get
\begin{equation}\label{b342}
	\begin{aligned}
		\mbox{l.h.s.~\eqref{b34}}  & \le \sum_{i=3}^{k+1}\til\Ex\bigg[\ind_{\se{1}{2}\ge \se{2}{2}}\ind_{X(k)\ge \beta^{-1} k^{1/8}}\!\cdot\! \max\{X(i)-Y(i),1\}\bigg] \\ & \hspace{2cm}+ \til\Ex\bigg[\ind_{\se{1}{2}\ge \se{2}{2}}\ind_{X(k)\ge \beta^{-1} k^{1/8}}\!\cdot\! \max\{\se{1}{2}-\se{2}{2},1\}\bigg] \\ & \le \sum_{i=3}^{k+1}\til\Pr(\se{1}{2}\ge \se{2}{2}) \cdot \til\Ex\bigg[\ind_{X(k)\ge \beta^{-1} k^{1/8}}\!\cdot\! \max\{X(i)-Y(i),1\}\bigg] \\ & \hspace{2cm}+ \til\Ex\bigg[\ind_{\se{1}{2}\ge \se{2}{2}}\!\cdot\! \max\{\se{1}{2}-\se{2}{2},1\}\bigg] \cdot \til\Pr\big(X(k)\ge \beta^{-1} k^{1/8}\big).
	\end{aligned}
\end{equation}
Using \eqref{b213} again, we have $\til\Ex\bigg[\ind_{\se{1}{2}\ge \se{2}{2}}\!\cdot\! \max\{\se{1}{2}-\se{2}{2},1\}\bigg] \le \Con'\cdot \til\Pr(\se{1}{2}\ge \se{2}{2})$. Using exponential tail estimates for $X(\ell), Y(\ell)$ we obtain that
\begin{align*}
	& \til\Ex\bigg[\ind_{X(k)\ge \beta^{-1} k^{1/8}}\!\cdot\! \max\{X(i)-Y(i),1\}\bigg] \le \Con \exp\big(-\tfrac1\Con \beta^{-1}k^{1/8}\big), \\ & \til\Pr\big(X(k)\ge \beta^{-1} k^{1/8}\big) \le  \Con \exp\big(-\tfrac1\Con \beta^{-1}k^{1/8}\big).
\end{align*}
Putting this estimates back in r.h.s.~\eqref{b342} we arrive at \eqref{b34}.

\medskip\noindent \textbf{$\m{Gap}_{6,\beta}$.} From \eqref{png}, we get that the endpoints of the middle part of the modified random walk are in $\mathcal{P}_{n,\gamma}$ (defined in \eqref{pnga}) with probability $1-\frac{\e}{12}$. Whenever the endpoints are in $\mathcal{P}_{n,\gamma}$, by Lemma \ref{unini}, the probability of non-intersection of the middle portion of the walk is lower bounded by some constant $\phi>0$. Under this event, we may use the KMT coupling \cite{xd} on the middle portion bridge of the modified random bridge to deduce that
		$$\Pr^{n-p-q+1;(c_1,c_2),(d_1,d_2)}(|\iks-\se{1}{k-1}|\ge \beta^{-1}\log n) 
		\le \tfrac1{n^2}.$$
		for all small enough $\beta$ and for all $(c_1,c_2),(d_1,d_2)\in \mathcal{P}_{n,\gamma}$. Combining all these estimates, by a union bound we have the desired result.
\end{proof}

We end this section with a modulus of continuity estimate for non-intersecting random walks.
\begin{lemma}\label{mret} Fix $M,\gamma>0$. There exists $n_0(M,\gamma)>0$ and $\delta(M,\gamma)>0$ such that for all $n\ge n_0$ and for all $0\le a_1-a_2\le M+2\log\log n$ we have (recall the modulus of continuity $\omega_{\delta}$ from \eqref{eq:modcont})
	\begin{align*}
		\sum_{i=1}^2\Pr^{n;(a_1,a_2)}\bigg(\omega_{\delta}(\se{i}{\cdot},\ll 1,n\rr) \ge \gamma \sqrt{n}\mid \ni_0\ll2,n\rr\bigg) \le \e,
	\end{align*}

\end{lemma}
\begin{proof} Fix $\gamma>0$. We write $\Pr$ for $\Pr^{n;(a_1,a_2)}$. By Corollary \ref{tailni2} one can choose $\rho$ such that
	\begin{align*}
		\Pr\bigg(\sup_{i=1,2}\omega_{\delta}(\se{i}{\cdot},\ll 1,n\rho\rr) \ge \gamma \sqrt{n}\mid \ni_0\ll2,n\rr\bigg) \le \e.
	\end{align*}
	Thus it suffices to control the modulus of continuity away from zero: on the interval $\ll n\rho/2,n\rr$ (assuming $\delta<\rho/2$). Towards this end let $I_{v}:=\{(x_1,x_2): |x_i|\le v^{-1}\sqrt{n}, x_1-x_2\ge v\sqrt{n}\}$. By Lemma \ref{tailni}, one can choose $v$ small enough to get $\Pr(\m{A}_v\mid \ni_0\ll2,n\rr) \ge 1-\e$ where $\m{A}_v:=\{(\se{1}{n\rho/2},\se{2}{n\rho/2})\in I_v\}$.
	Let $\mathcal{F}:=\sigma\big(\se{1}{n\rho/2},\se{2}{n\rho/2}\big)$. Note that
	\begin{equation}
		\Pr\Big(\omega_{\delta}(\se{i}{\cdot},\ll n\rho/2,n\rr)\ge \gamma\sqrt{n} \mid \ni_0\ll2,n\rr\Big)  \le \e +\frac{\Pr\big(\m{A}_v \cap \{\omega_{\delta}(\se{i}{\cdot},\ll n\rho/2,n\rr)\ge \gamma\sqrt{n}\} \cap \ni_0\ll2,n\rho/2\rr\big)}{\Pr(\ni_0\ll2,n\rr)}. \label{eq:modnonintersect}
	\end{equation}
Note that $\{\omega_{\delta}(\se{i}{\cdot},\ll n\rho/2,n\rr)\ge \gamma\sqrt{n}\}$ is independent of $\mathcal{F}$. By Lemma \ref{unini}, we have $\ind_{\m{A}_v}\cdot\Ex[\ind_{\ni_0\ll n\rho/2,n\rr} \mid \mathcal{F}] \ge \ind_{\m{A}_v}\phi$ for some $\phi>0$. Combining these two facts we get
\begin{align} \nonumber
	& \Pr\left(\m{A}_v \cap \{\omega_{\delta}(\se{i}{\cdot},\ll n\rho/2,n\rr)\ge \gamma\sqrt{n}\} \cap \ni_0\ll2,n\rho/2\rr\right) \\ & \nonumber =\Pr\left(\m{A}_v  \cap \ni_0\ll2,n\rho/2\rr\right)\Pr\left(\omega_{\delta}(\se{i}{\cdot},\ll n\rho/2,n\rr)\ge \gamma\sqrt{n}\right) \\ & \le \nonumber \phi^{-1}\cdot\Ex\left[\ind_{\m{A}_v \cap \ni_0\ll2,n\rho/2\rr}\Ex[\ind_{\ni_0\ll n\rho/2,n\rr}\mid\mathcal{F}]\right]\cdot \Pr\left(\omega_{\delta}(\se{i}{\cdot},\ll n\rho/2,n\rr)\ge \gamma\sqrt{n}\right) \\ & \le \nonumber \phi^{-1}\cdot\Pr(\m{A}_v\cap\ni_0\ll2,n\rr)\cdot \Pr\left(\omega_{\delta}(\se{i}{\cdot},\ll n\rho/2,n\rr)\ge \gamma\sqrt{n}\right) \\ & \le  \phi^{-1}\cdot \Pr(\ni_0\ll2,n\rr)\cdot \Pr\left(\omega_{\delta}(\se{i}{\cdot},\ll n\rho/2,n\rr)\ge \gamma\sqrt{n}\right). \label{b55}
\end{align}
Invoking the modulus of continuity of random walks we can choose $\delta$ small enough such that  $\Pr\left(\omega_{\delta}(\se{i}{\cdot},\ll n\rho/2,n\rr)\ge \gamma\sqrt{n}\right)$ is  at most $\e \phi$ for all large enough $n$. This, implies $$\eqref{b55}\le \e \cdot\Pr(\ni_0\ll2,n\rr).$$ Using this inequality we see that r.h.s.~\eqref{eq:modnonintersect} is at most $2\e$. Hence combining the near zero and away zero modulus of continuity we get the desired result by adjusting $\gamma$ and $\e$.
\end{proof}

\section{Supporting calculations} \label{appd}

In this section we provide a detailed verification of various tedious calculations. We first show how to go from \eqref{e:den1} to \eqref{e:den2}-\eqref{e:den4} under the change of variables $u_{i,j}=\log\big( t_{N+\lfloor j/2\rfloor-i+1,N-\lceil j/2 \rceil-i+2}\big)$ for $(i,j)\in K_N$. This follows from the fact that the  factor $\prod t_{i,j}^{-1}$ in \eqref{e:den1} is absorbed as the Jacobian of the change of variables, as well as the following four relations:
\begin{align}
	\label{tta} & \quad 	\prod_{j=1}^N\left(\frac{\tau_{2N-2j+2}\tau_{2N-2j}}{\tau_{2N-2j+1}^2}\right)^{\theta_j}  =\prod_{i=1}^N \left(e^{-\theta_i u_{i,2N-2i+2}}\prod_{j=1}^{N-i+1} e^{\theta_{N-j+1}(u_{i,2j-1}-u_{i,2j})} \prod_{j=1}^{N-i} e^{\theta_{N-j+1}(u_{i,2j+1}-u_{i,2j})}\right) \\
	\label{ttb}& \quad
	\sum_{i>j} \frac{t_{i-1,j}}{t_{i,j}} = \sum_{i=1}^{N}\sum_{j=1}^{N-i+1}e^{u_{i,2j-1}-u_{i,2j}}+\sum_{i=1}^{N-1}\sum_{j=1}^{N-i}e^{u_{i+1,2j}-u_{i,2j+1}}\\
	\label{ttc}& \quad
	\sum_{i\ge j>1} \frac{t_{i,j-1}}{t_{i,j}} = \sum_{i=1}^{N-1}\sum_{j=1}^{N-i} e^{u_{i,2j+1}-u_{i,2j}}+ \sum_{i=1}^{N-1} \sum_{j=1}^{N-i} e^{u_{i+1,2j}-u_{i,2j-1}} \\
	\label{ttd}& \quad \prod_{j=1}^N t_{j,j}^{(-1)^{N-j+1}\alpha}=\prod_{j=1}^N t_{N-j+1,N-j+1}^{(-1)^{j}\alpha}= \prod_{i=1}^N e^{(-1)^{i}u_{i,1}\alpha}.
\end{align}
While \eqref{ttd} is obvious, \eqref{tta}, \eqref{ttb} and \eqref{ttc} are shown below.  We continue with the same notations as in the proof of Theorem \ref{thm:conn}.

\medskip

\noindent\textbf{Verification of \eqref{tta}.} Note that from the transformation we have
\begin{align*}
	e^{u_{j-i+1,2N-2j+1}}=t_{i+2N-2j,i}, \qquad e^{u_{j-i+1,2N-2j+2}}=t_{i+2N-2j+1,i}.
\end{align*}
This yields
\begin{align*}
	\tau_{2N-2j}^{\theta_j} = \prod_{i=1}^{j} t_{i+2N-2j,i}^{\theta_j} = \prod_{i=1}^j 	e^{\theta_j u_{j-i+1,2N-2j+1}}=\prod_{i=1}^j 	e^{\theta_j u_{i,2N-2j+1}}.
\end{align*}
Similarly we have
\begin{align*}
	\tau_{2N-2j+2}^{\theta_j} =e^{-\theta_j u_{j,2N-2j+3}}\prod_{i=1}^{j} e^{\theta_j u_{i,2N-2j+3}}, \quad \tau_{2N-2j+1}^{\theta_{j}} =\prod_{i=1}^{j} e^{\theta_j u_{i,2N-2j+2}}.
\end{align*}
Thus,
\begin{align*}
	\prod_{j=1}^N \left(\frac{\tau_{2N-2j+2}	\tau_{2N-2j}}{\tau_{2N-2j+1}^2}\right)^{\theta_j} & =\prod_{j=1}^{N}\left(e^{-\theta_j u_{j,2N-2j+3}}\prod_{i=1}^{j}e^{\theta_j(u_{i,2N-2j+1}+u_{i,2N-2j+3}-2u_{i,2N-2j+2})}\right) \\ & = \left(\prod_{i=1}^N e^{-\theta_i u_{i,2N-2i+3}}\right)\cdot \left(\prod_{i=1}^{N}\prod_{j=i}^N  e^{\theta_j(u_{i,2N-2j+1}+u_{i,2N-2j+3}-2u_{i,2N-2j+2})}\right) \\ & {=} \left(\prod_{i=1}^N e^{-\theta_i u_{i,2N-2i+3}}\right)\cdot \left(\prod_{i=1}^{N}\prod_{j=1}^{N-i+1} e^{\theta_{N-j+1}(u_{i,2j-1}+u_{i,2j+1}-2u_{i,2j})}\right),
\end{align*}	
where the last equality follows by changing the dummy variable $j$ has been changed to $N-j+1$. The last term above is clearly equal to the right-hand side of \eqref{tta}.

\medskip

\noindent\textbf{Verification of \eqref{ttb}.} Let us write
\begin{align}\notag
	\sum_{i>j} \frac{t_{i-1,j}}{t_{i,j}}  = \sum_{j=1}^N \sum_{i=j+1}^{2N-j+1} \frac{t_{i-1,j}}{t_{i,j}}  & = \sum_{j=1}^N \sum_{r=1}^{2N-2j+1} \frac{t_{j+r-1,j}}{t_{j+r,j}} \\ & = \sum_{j=1}^N \sum_{r=1}^{N-j} \frac{t_{j+2r-1,j}}{t_{j+2r,j}}+\sum_{j=1}^N \sum_{r=1}^{N-j+1} \frac{t_{j+2r-2,j}}{t_{j+2r-1,j}}. \label{ettb}
\end{align}
Observe that
\begin{align}\label{fttb}
	e^{u_{N-r-j+1,2r+1}}=t_{j+2r,j}, \quad e^{u_{N-r-j+2,2r}}=t_{j+2r-1,j}.
\end{align}
Thus we have
\begin{align*} \eqref{ettb}  & = \sum_{j=1}^N \sum_{r=1}^{N-j} e^{u_{N-r-j+2,2r}-u_{N-r-j+1,2r+1}}+\sum_{j=1}^N \sum_{r=1}^{N-j+1} e^{u_{N-r-j+2,2r-1}-u_{N-r-j+2,2r}} \\ & = \sum_{j=1}^N \sum_{r=1}^{j-1} e^{u_{j-r+1,2r}-u_{j-r,2r+1}}+\sum_{j=1}^N \sum_{r=1}^{j} e^{u_{j-r+1,2r-1}-u_{j-r+1,2r}} \quad (j\mapsto N-j+1) \\ & =\sum_{r=1}^{N-1} \sum_{j=r+1}^{N} e^{u_{j-r+1,2r}-u_{j-r,2r+1}}+\sum_{r=1}^N \sum_{j=r}^{N} e^{u_{j-r+1,2r-1}-u_{j-r+1,2r}}  \\ & =\sum_{r=1}^{N-1} \sum_{i=1}^{N-r} e^{u_{i+1,2r}-u_{i,2r+1}}+\sum_{r=1}^N \sum_{i=1}^{N-r+1} e^{u_{i,2r-1}-u_{i,2r}},
\end{align*}
where $(j\mapsto N-j+1)$ means the dummy variable $j$ has been changed to $N-j+1$ to obtain the equality in the second step. The last equality follows by setting $j-r\mapsto i$ and $j-r\mapsto i-1$ in the first and second sum respectively. A final interchange of sum in each of the two terms leads to the right hand side of \eqref{ttb}.

\medskip

\noindent\textbf{Verification of \eqref{ttc}.} We follow the same above strategy and write
\begin{align}\notag
	\sum_{i\ge j>1} \frac{t_{i,j-1}}{t_{i,j}}  = \sum_{j=2}^N \sum_{i=j}^{2N-j+1} \frac{t_{i,j-1}}{t_{i,j}}  & = \sum_{j=2}^N \sum_{r=1}^{2N-2j+2} \frac{t_{j+r-1,j-1}}{t_{j+r-1,j}} \\ & = \sum_{j=2}^N \sum_{r=1}^{N-j+1} \frac{t_{j+2r-1,j-1}}{t_{j+2r-1,j}} +\sum_{j=2}^N \sum_{r=1}^{N-j+1} \frac{t_{j+2r-2,j-1}}{t_{j+2r-2,j}} \label{ettc}
\end{align}
Due to \eqref{fttb} we have
\begin{align*}
	\eqref{ettc} & = \sum_{j=2}^N \sum_{r=1}^{N-j+1} e^{u_{N-r-j+2,2r+1}-u_{N-r-j+2,2r}} +\sum_{j=2}^N \sum_{r=1}^{N-j+1} e^{u_{N-r-j+3,2r}-u_{N-r-j+2,2r-1}} \\ & =\sum_{j=1}^{N-1} \sum_{r=1}^{j} e^{u_{j-r+1,2r+1}-u_{j-r+1,2r}} +\sum_{j=1}^{N-1} \sum_{r=1}^{j} e^{u_{j-r+2,2r}-u_{j-r+1,2r-1}} \quad (j\mapsto N-j+1) \\ & =\sum_{r=1}^{N-1}\sum_{j=r}^{N-1} e^{u_{j-r+1,2r+1}-u_{j-r+1,2r}} +\sum_{r=1}^{N-1}\sum_{j=r}^{N-1} e^{u_{j-r+2,2r}-u_{j-r+1,2r-1}} \\ & =\sum_{r=1}^{N-1}\sum_{i=1}^{N-r} e^{u_{i,2r+1}-u_{i,2r}} +\sum_{r=1}^{N-1}\sum_{i=1}^{N-r} e^{u_{i+1,2r}-u_{i,2r-1}} \quad (j-r\mapsto i-1).
\end{align*}
A final interchange of sum in each of the two terms leads to the right hand side of \eqref{ttc}. This completes the verification of all three equalities.

\medskip

\noindent\textbf{Verification of \eqref{calcc}.} Note that $\theta_c$ is a function of $p$ defined as a solution of the equation $\Psi'(\theta_c)-p\Psi'(2\theta-\theta_c)=0$. Set $g(p)=\theta_c$. Note that $g(1)=\theta$. By differentiating the equation  with respect to $p$ we get
$$g'(p)\Psi''(g(p))-\Psi'(2\theta-g(p))+p\Psi''(2\theta-g(p))g'(p)=0$$
This implies $g'(1)=\Psi'(\theta)/2\Psi''(\theta)$. Since $p-1=O(N^{-1/3})$. By Taylor expansion around $1$ to first three terms we get
\begin{align*}
	(N-k)f_{\theta,p}& =-(N-k)\big(\Psi(g(p))+p\Psi(2\theta-g(p))\big) \\ & = -(N-k)\Big(2\Psi(\theta)+(p-1)\Psi(\theta)+(p-1)^2\big(\Psi''(\theta)(g'(1))^2-g'(1)\Psi'(\theta)\big)+O(N^{-1})\Big) \\ & =-2N\Psi(\theta)+\frac{k^2(\Psi'(\theta))^2}{\Psi''(\theta)(N-k)}+O(1),
\end{align*}
where in the final line we used the fact that $p-1=2k/(N-k)$ and the formula for $g'(1)$ derived above. Taking $k=MN^{2/3}$ we arrive at the leading orders claimed in the first part of \eqref{calcc}. The second part follows by observing that by Taylor expansion up to first order we have
$$\log \sigma_{\theta,p}=\log\sigma_{\theta,1}+O(p-1)=\log\sigma_{\theta,1}+O(N^{-1/3}).$$
Thus, $\sigma_{\theta,p}/\sigma_{\theta,1}=\sigma_{\theta,p}/(-\Psi''(\theta))^{1/3} \stackrel{N\to\infty}{\to} 1,$ proving the second claim in \eqref{calcc}.
\section*{Glossary}

{\renewcommand{\arraystretch}{1.2}
	\begin{longtable}[t]{lp{0.6\textwidth}l}
		\toprule
		
		\multicolumn{3}{l}{General notation used throughout the text} \\
		\midrule
		$\hslg$ & half-space log-gamma & Sec.~\ref{sec:1.1} \\
		$\mathbb{Z}_{\ge k}$ & set of all integers $\ge k$ & Sec.~\ref{sec:1.1} \\
		$\calI$ &$\big\{(i,j)\in (\mathbb{Z}_{\geq 1})^2: j\leq i\big\}$ &  Sec.~\ref{sec:1.1}\\
        $W_{i,j}$ & inverse-gamma polymer weights & Eq.~\eqref{eq:wt}\\
        $\Pi_{m,n}$ & set of all directed paths from $(1,1)$ to $(m,n)$ in $\calI$ & Sec.~\ref{sec:1.1}\\
        $w(\pi)$ & weight of path $\pi$ & Eq.~\eqref{eq:wz}\\
		$Z_{(\alpha,\vec\theta)}(m,n)=Z(m,n)$ & point-to-point $\hslg$ polymer partition function & Eq.~\eqref{eq:wz}\\
		$\mathcal{F}_N^{\alpha}(s)$  & centered and scaled $\hslg$ free energy process &  Eq.~\eqref{eq:fnalpha}\\
        $\Psi(z)$ & digamma function & Eq.~\eqref{psidef}\\
		$\zl (k)$ & point-to-line $\hslg$ polymer partition function & Eq.~\eqref{rfl} \\
		
		$Z_{\operatorname{sym}}^{(r)}(m,n)$ &  multipath point-to-point symmetrized log-gamma polymer partition function & Eq.~\eqref{zsymr} \\
		
		$\L^N$ & $\hslg$ line ensemble & Def.~\ref{l:nz}\\

		$\mathcal{K}_{k,T}$ and $\mathcal{K}_{k,T}'$ & two important domains for $\hslg$ Gibbs measures & Eq.~\eqref{def:kkt} \\
		
		$\alpha_1$ and $\alpha_2$ & scalings for the boundary parameter & Eq.~\eqref{acric} \\
		$\m{Gap}_{\beta}$ & gap event & Eq.~\eqref{def:gp}\\
		$\omega_{\delta}^N(f;\ll1,U\rr)$ & modulus of continuity & Eq.~\eqref{eq:modcont}\\
		\midrule
		
		\multicolumn{3}{l}{Basic probability densities and distributions} \\
		\midrule
		$\operatorname{Gamma}^{-1}(\beta)$ & inverse-gamma distribution with density against Lebesgue given by $\mathbf{1}\{x>0\}\Gamma^{-1}(\beta)x^{-\beta-1}e^{-1/x}$ &   Sec.~\ref{sec:1.1}\\
		$W_{e}(x)$ & weight function for edges & Eq.~\eqref{def:wfn}\\
		
		\multicolumn{2}{l}{$W(a;b,c):=\exp(-e^{a-b}-e^{a-c}), \ \ a,b,c\in \R$} & Eq.~\eqref{def:wfns} \\
		
		\multicolumn{2}{l}{$G_{\theta,(-1)^m}(x):=e^{\theta (-1)^mx-e^{(-1)^my}}/\Gamma(\theta), \ \ \theta\in \R, m\in \mathbb{Z}_{\ge 0}, y\in \R$} & Eq.~\eqref{def:gwt} \\
		
		\multicolumn{2}{l}{$\fa(x)=\int_{\R} G_{\theta,+1}(y) G_{\theta,-1}(x-y)dy, \ \ \theta\in \R, x\in \R$} & Eq.~\eqref{def:faga} \\
		
		\multicolumn{2}{l}{$\ga(x)=G_{\zeta,+1}(x)$} & Eq.~\eqref{def:faga} \\
		
			{$\qo_{\theta_1,\theta_2;\pm 1}^{(a,b)}(x)$} &  & Eq.~\eqref{qdist} \\		
		\midrule
		
		\multicolumn{3}{l}{Probability distributions on random walks and bridges} \\
		\midrule
$f_{k,T}^{\vec{y},\vec{z}}(\mathbf{u})$ & density of the $\hslg$ Gibbs measure on the domain $\mathcal{K}_{k,T}$ with boundary condition $(\vec{y},\vec{z})$ & Eq.~\eqref{def:fhsgm}					\\
		$\psa$ &  $\hslg$ Gibbs measure on $\mathcal{K}_{k,T}$ with boundary condition $(\vec{y},\vec{z})$ (the $\alpha$ subscript is sometimes dropped when clear)& Def.~\ref{def:hslgfmeas} \\
$\Pr_{\alpha}^{\vec{y};(-\infty)^{T};k,T}$ & bottom free $\hslg$ Gibbs measure on $\mathcal{K}_{k,T}$ with boundary condition $(\vec{y}$ (the $\alpha$ subscript is sometimes dropped when clear)& Def.~\ref{def:btf} \\
$Q_{k,T}^{\vec{y}',\vec{z}}(\mathbf{u})$ &
$\hslg$ Gibbs measure on the domain $\mathcal{K}_{k,T}'$ with boundary condition $(\vec{y},\vec{w})$ & Eq.~\eqref{def:fhsgm2}\\
		  $\Pr_{\operatorname{WPRW}}^{n;(a_1,a_2)}$ and $\Pr_{\operatorname{PRW}}^{n;(a_1,a_2)}$ & law of weighted paired random walk and paired random walk of length $n$ started from $(a_1,a_2)$  & Def.~\ref{prb}\\		
			$\Pr^{n;(a_1,a_2),(b_1,b_2)}$ & law of two independent random  bridges of length $n$ started from $(a_1,a_2)$ and ending at $(b_1,b_2)$ with increments drawn from $\fa$ defined in \eqref{def:faga} & Def.~\ref{def:rws} \\			
			$\til\Pr^{(n;p,q);(a_1,a_2),(b_1,b_2)}$ &  law of two independent $(n;p,q)$-modified random bridges  of length $n$ started from $(a_1,a_2)$ ending at $(b_1,b_2)$ with  increments drawn from $\fa$ & Def.~\ref{def:mrb} \\			
			$\Pr^{n;(a_1,a_2)}$ & law of two independent random walks of length $n$  started from $(a_1,a_2)$ with increments drawn from $\fa$ & Def.~\ref{def:rws} \\
		\bottomrule
	\end{longtable}
}


	\bibliographystyle{alpha}		
	\bibliography{half}
	
\end{document}